%% file: HodgeTypeRZsp-5.tex
\documentclass[a4paper]{amsart}
\input{newmacros}

\title{Rapoport-Zink spaces of Hodge type}

\begin{document}
\begin{abstract}
When $p>2$, we construct a Hodge-type analogue of Rapoport-Zink spaces under the unramifiedness assumption, as formal schemes parametrising ``deformations'' (up to quasi-isogeny) of $p$-divisible groups with certain crystalline Tate tensors. We also define natural rigid analytic towers with expected extra structure, providing more examples of ``local Shimura varieties'' conjectured by Rapoport and Viehmann.
\end{abstract}
\maketitle
\tableofcontents

\section{Introduction}
Let  $(G, \UH)$ be a Hodge-type Shimura datum; i.e., $(G,\UH)$ can be embedded into the Shimura datum associated to some symplectic similitude group (i.e., Siegel Shimura datum). 
By choosing such an embedding, the associated complex Shimura variety $\Sh(G,\UH)_\BC$ obtains a family of abelian varieties (coming from the ambient Siegel modular variety) together with Hodge cycles.

In this paper, we construct, in the unramified case, a natural $p$-adic local analogue of such Shimura varieties; loosely speaking, what we constructed can be regarded as ``moduli spaces'' of $p$-divisible groups equipped with certain ``crystalline Tate tensors''. Since the precise definition is rather technical, let us just indicate the idea. Recall that for a $\Q$-Hodge structure $H$, a Hodge cycle on $H$ can be understood as a morphism $t:\triv \ra H$ of $\Q$-Hodge structures, where $\triv$ is the trivial $\Q$-Hodge structure of rank $1$. Our definition of crystalline Tate tensors is very similar, with $\Q$-Hodge structures replaced by $F$-crystals\footnote{By $F$-crystal, we mean  Tate twists by \emph{any} integers of crystals equipped with nondegenerate Frobenius action. This is to allow the dual of an $F$-crystal to be an $F$-crystal.} equipped with Hodge filtration.\footnote{We only define the notion of crystalline Tate tensors for $p$-divisible groups defined over ``formally smooth'' base rings, to avoid subtleties involving torsions of crystalline Dieudonn\'e theory. See Definition~\ref{def:RZG} for the precise definition over nice enough base rings.}

Let $G$ be a connected \emph{unramified}\footnote{I.e., quasi-split and split over $\Qp^{\ur}$.} reductive group over $\Qp$. We fix a reductive $\Zp$-model of $G$ (which exists by unramifiedness), and also denote it by $G$. We choose an element $b\in G(\wh\Q_p^{\ur})$ which gives rise to a $p$-divisible group $\BX$ over $\Fpbar$ in the following sense: for some finite free $\Zp$-module $\Lambda$ with faithful $G$-action, the $F$-crystal $\bM:=(\wh\Z_p^{\ur}\otimes \Lambda^*, b\circ(\sig\otimes\id))$ gives rise to a $p$-divisible group $\BX$ by the (contravariant) Dieudonn\'e theory. In this case, we can associate to such $b$ an ``unramified Hodge-type local Shimura datum'' $(G,[b],\set{\mu\iv})$ (\emph{cf.} \S\ref{subsec:FilGIsoc}.)

Let $\bM^\otimes$ denote the direct sum of the combinations of tensor products, symmetric and alternating products, and duals of $\bM$. Then the fact that $b\in G(K_0)$ gets encoded as the existence of certain ``Frobenius-invariant tensors'' $(\bft_\alpha)\in\bM^\otimes$; \emph{cf.} Lemma~\ref{lem:CrysTateOverW}, Proposition~\ref{prop:Chevalley}.

Rapoport and Zink constructed a moduli space $\RZ_\BX$ parametrising ``deformations'' of $\BX$ up to quasi-isogeny \cite[Theorem~2.16]{RapoportZink:RZspace}. Here, $\RZ_\BX$ is a formal scheme which is locally formally of finite type over  $\wh \Z_p^{\ur}$. Our main result is roughly of the following form:
\begin{thm*}[\ref{thm:RZHType}]
Let $(G,b)$ and $\BX$ be as above, and assume that $p>2$.\footnote{The assumption is made in order to use the Grothendieck-Messing deformation theory for the nilpotent ideal generated by $p$.} 
Then there exists a closed formal subscheme $\RZ_{G,b}\subset \RZ_\BX$ which classifies deformations (up to quasi-isogeny) of $\BX$ with Tate tensors $(\bft_\alpha)$, such that the Hodge filtration of the $p$-divisible group is \'etale-locally given by some cocharacter in the conjugacy class $\set\mu$. (See Definitions~\ref{def:RZG} and \ref{def:RZGloc} for the precise  conditions that define $\RZ_{G,b}$.) Furthermore, $\RZ_{G,b}$ is formally smooth, is functorial in $(G,b)$, and only depends on the associated unramified Hodge-type local Shimura datum $(G,[b],\set{\mu\iv})$ up to isomorphism (and not on the auxiliary choice of $\BX$).
\end{thm*}

Rapoport and Viehmann conjectured that to any (not necessarily unramified nor Hodge-type) ``local Shimura datum'' $(G,[b].\set{\mu\iv})$, there exists a rigid analytic tower of ``local Shimura varieties'' with suitable extra structure \cite[\S5]{RapoportViehmann:LocShVar}. In \S\ref{sec:ExtraStr}, we construct the rigid analytic tower over the generic fibre of $\RZ_{G,b}$ equipped with suitable extra structure as predicted in \cite[\S5]{RapoportViehmann:LocShVar}; in other words, we construct ``local Shimura varieties'' associated to any unramified Hodge-type local Shimura data when $p>2$.\footnote{Rapoport and Viehmann also conjectured that ``local Shimura varieties'' could be constructed by a purely group-theoretic means. Note that our construction of ``local Shimura varieties'' (in the unramified Hodge-type case) is not purely group-theoretic as we make crucial use of $p$-divisible groups.}

In the case of unramified  EL and PEL type, we also show that $\RZ_{G,b}$ recovers the original construction of Rapoport-Zink space  in \cite[Theorem~3.25]{RapoportZink:RZspace}. See Proposition~\ref{prop:ELnPEL} for the precise statement.
On the other hand, the theorem provides Rapoport-Zink spaces for more general class of groups $G$ that do not necessarily arise from any EL or PEL datum. For example, we may allow $G$ to be the spin similitude group associated to a split quadratic space over $\Qp$, which do not arise from any EL or PEL datum if the rank of the quadratic space is at least $7$. Note also that the ``functoriality'' assertion of the theorem produces some interesting morphisms between EL and PEL Rapoport-Zink spaces, which may not be easily seen from the original construction. See Remark~\ref{rmk:ExcIsom} for such an example involving an ``exceptional isomorphism''.

Recently, Scholze and Weinstein  \cite{ScholzeWeinstein:RZ} constructed the ``infinite-level'' Rapoport-Zink spaces of EL- and PEL- type, which provides a new  approach to study Repoport-Zink spaces. As remarked in the introduction of \cite{ScholzeWeinstein:RZ}, for quite a general ``local Shimura datum'' $(G,[b],\set{\mu\iv})$ -- without requiring $G$ to be unramified, nor $[b]$ to come from a $p$-divisible group -- it should be possible to construct the infinite-level Rapoport-Zink space for $(G,[b],\set\mu)$ using the technique in \cite{ScholzeWeinstein:RZ}. This approach does not require any formal model at the  ``maximal level'' Rapoport-Zink space, nor does it give a natural formal model, while our approach is to start with the  formal scheme at the hyperspecial maximal level and build up the rigid analytic tower from there. Perhaps, having a formal scheme at the hyperspecial maximal level could be useful in some applications; for example,  $p$-adic uniformisation of Hodge-type Shimura varieties; see the next paragraph for more details.
In \S\ref{subsec:InfLevel} we give a construction of $\RZ_{G,b}^\infty$ using ``finite-level'' Rapoport-Zink spaces $\RZ_{G,b}^\KK$ and \cite[Theorem~D]{ScholzeWeinstein:RZ}. Note that there should be more natural ``purely infinite-level'' construction of $\RZ_{G,b}^\infty$, which should work more generally, but we give our ``finite-level'' construction just to link our work with \cite {ScholzeWeinstein:RZ}.

In the PEL case, Rapoport and Zink also showed that certain arithmetic quotients of PEL Rapoport-Zink spaces can be related to PEL Shimura varieties, generalising the theorem of Drinfeld and Cerednik on $p$-adic uniformisation of Shimura curves; \emph{cf.} \cite[Ch.~VI]{RapoportZink:RZspace}. This is a useful tool for studying the mod~$p$ geometry of PEL Shimura varieties -- especially, the basic (i.e. supersingular) locus -- by reducing the question to a purely local problem of studying the corresponding Rapoport-Zink space.
In the sequel of this paper \cite{Kim:Unif}, we give a Hodge-type generalisation of this result at odd good reduction primes. In particular, the result is  applicable to $\GSpin(n,2)$ Shimura varieties for any $n$.

Since the first version of this paper appeared, some alternative constructions of Hodge-type Rapoport-Zink spaces were developed under some mild restrictions.
Ben~Howard and George~Pappas \cite{HowardPappas:GSpin} gave another construction of Hodge-type Rapoport-Zink spaces, compatible wiht ours, in the case when the Hodge-type local Shimura datum $(G,[b],\set{\mu\iv})$ comes from a (global) Hodge-type Shimura datum. Their construction relies on the existence of integral canonical models of Hodge-type Shimura varieties and the Rapoport-Zink uniformisation for Siegel modular varieties. This global construction is simpler than ours and automatically gives the Rapoport-Zink uniformisation of Hodge-type Shimura varieties (recovering \cite{Kim:Unif}), while our construction is purely local (as it should be) and does not require $(G,[b],\set{\mu\iv})$ to come from a global Shimura datum.

Xinwen~Zhu \cite{Zhu:Satake} gave a purely group-theoretic (and local) construction of the perfection of the underlying reduced schemes via mixed characteristic affine grassmannian. Oliver~B\"ultel and George~Pappas \cite{BueltelPappas:G-muDisplays} gave yet another purely group-theoretic (and local) construction of Hodge-type Rapoport-Zink spaces via the theory of $(G,\mu)$-displays if the fixed $p$-divisible group $\BX$ does not have any non-trivial multiplicative or \'etale part.

Let us comment on the proof of the main theorem (Theorem~\ref{thm:RZHType}).
We do not directly extract  defining equations of  $\RZ_{G,b}$  in $\RZ_\BX$, but instead  we  take an indirect approach. 
As a starting point, observe that the candidate for the set of closed points $\RZ_{G,b}(\Fpbar)\subset\RZ_\BX(\Fpbar)$ is given by the affine Deligne-Lusztig set (\emph{cf.} Proposition~\ref{prop:LR}). Furthermore, Faltings constructed a formally smooth closed subspace $(\RZ_{G,b})\wh{_x}$ of the completion $(\RZ_\BX)\wh{_x}$ at $x\in\RZ_{G,b}(\Fpbar)$, which gives a natural candidate for the completion of $\RZ_{G,b}$ at $x$.

We then  construct a formal algebraic space  $\RZ_{G,b}$  by patching together Faltings deformation spaces $(\RZ_{G,b})\wh{_x}$ via Artin's algebraisation technique (\emph{cf.} \S\ref{subsec:Artin}).\footnote{Although Artin's criterion is only for algebraic spaces, not for formal algebraic spaces, we apply Artin's criterion to suitable ``closed subspaces'' which turn out to be algebraic spaces.} The key step is to show that Tate tensors, constructed formal-locally over $(\RZ_{G,b})\wh{_x}$ by Faltings, patch together and spread out to some neighbourhood of $x$ whenever they should (\emph{cf.} Propositions~\ref{prop:descent} and \ref{prop:extension}).

The main reason why we exclude $p=2$ is that the standard PD structure on $pR$ is not nilpotent unless either $p>2$ or $pR=0$. In particular, if $p=2$ then we cannot apply the Grothendieck-Messing deformation theory \cite{messingthesis} for the thickenings $R\thra R/p$. The main results of this paper will be extended to the case when $p=2$ in the author's forthcoming paper.

\subsection*{Structure of the paper}
In \S\ref{sec:housekeeping} we recall and introduce some basic definitions, such as  (iso)crystals with $G$-structures, filtrations, and cocharacters. In \S\ref{sec:FaltingsDefor} we review Faltings's explicit construction of the ``universal deformation'' of $p$-divisible groups with Tate tensors. The theorem on the existence of $\RZ_{G,b}$ is stated in \S\ref{sec:RZ}, which is proved in \S\ref{sec:fpqcDesc} and \S\ref{sec:RepPf}.

In \S\ref{sec:ExtraStr} we define various extra structures on $\RZ_{G,b}$ as predicted in \cite[\S5]{RapoportViehmann:LocShVar}, including rigid analytic tower and the Hecke and quasi-isogeny group actions. In \S\ref{sec:Ccris}, we show the existence and integrality of ``\'etale realisations'' of the crystalline Tate tensors,  which are needed for constructing the rigid analytic tower in \S\ref{sec:ExtraStr}.

\subsection*{Acknowledgement}
The author would like to thank Chris Blake, James Newton, Tony Scholl,  Teruyoshi Yoshida, and Xinwen Zhu for helpful discussions and advice, and especially Brian Conrad and George Pappas for their careful reading of the manuscript and providing helpful suggestions. The author sincerely thanks the anonymous referees who gave extremely detailed readings of the manuscript, pointed out the gap in the earlier version regarding the torsion of PD hulls, and provided many valuable comments that helped improving the exposition.

While the paper was being written and revised, the author was supported by Herchel Smith Postdoctoral Fellowship and the EPSRC (Engineering and Physical Sciences Research Council) in the form of EP/L025302/1.

\section{Definitions and Preliminaries}\label{sec:housekeeping}
\subsection{Notation}

For any ring $R$, an $R$-module $M$, and an $R$-algebra $R'$, we write $M_{R'}:=R'\otimes_R M$. Similarly, if $R$ is a noetherian adic ring and  $\XX$ is a formal scheme over $\Spf R$, then for any continuous morphism of adic rings $R\ra R'$ we write $\XX_{R'}:=\XX\times_{\Spf R}\Spf R'$.

For any ring $\fo$, we let $\Alg_\fo$ denote the category of $\fo$-algebras. Now assume that $\fo$ is a $p$-adic discrete valuation ring, and we let  $\Nilp_\fo$ denote the category of $\fo$-algebras $R$ where $p$ is nilpotent. Let $\art\fo$ denote the category of artin local $\fo$-algebras with residue field $\fo/\m_\fo$. Note that any (not necessarily $p$-adic) locally noetherian formal scheme $\XX$ over $\Spf \fo$  defines a set-valued functor on $\Nilp_\fo$ by $\XX(R):=\Hom_\fo(\Spec R,\XX)$ for $R\in\Nilp_\fo$.

We will work with formal schemes that satisfy the following finiteness condition:
\begin{defnsub}\label{def:ForFType}
Let $\fo$ be a $p$-adic discrete valuation ring. A locally noetherian formal $\fo$-scheme $\XX$  is \emph{locally formally  of finite type} over $\Spf\fo$ if $\XX_\red$ is locally of finite type over $\fo/\m_\fo$ (where $\XX_\red$ is the closed subscheme of $\XX$ defined by the ideal of locally topologically nilpotent sections). 

A formal $\fo$-scheme $\XX$ is   \emph{formally of finite type} if it is quasi-compact and locally formally of finite type over $\Spf \fo$.
\end{defnsub}

\begin{rmksub}\label{rmk:ForFType}
Let us denote $\fo[[u_1,\cdots,u_r]]\langle v_1,\cdots,v_s\rangle$ to be the completion of $\fo[u_1,\cdots,u_r,v_1,\cdots,v_s]$ with respect to $(p,u_1,\cdots,u_r)$.
We say that a topological $\fo$-algebra $R$ is  \emph{formally finitely generated} over $\fo$ if it can be written as the quotient of $\fo[[u_1,\cdots,u_r]]\langle v_1,\cdots,v_s\rangle$. 
Note that the Jacobson radical $J$ of $R$ is the ideal of topological nilpotent elements, and the natural topology on $R$ is the $J$-adic topology, with respect to which $R$ is separated and complete. (In particular, \emph{any maximal ideal of $R$ is necessarily open}.)

Now, it follows that an affine formal scheme $\XX=\Spf R$ is formally of finite type over $\fo$ if and only if $R$ is a formally finitely generated $\fo$-algebra equipped with the natural topology (i.e., the $J$-adic topology, where $J$ is the Jacobson radical).
Similarly, a locally noetherian formal scheme $\XX$ over $\fo$ is locally formally of finite type over $\fo$ if and only if $\XX$ admits an affine open covering $\set{\Spf R_\xi}$ where each $R_\xi$ is formally finitely generated over $\fo$ equipped with the natural topology.

Later in this paper, we work with two (not necessarily equivalent) topologies on formally finitely generated $\fo$-algebras $R$; namely, the natural topology and the $p$-adic topology (\emph{cf.} Definition~\ref{def:RZG}). To minimise confusions, we introduce the following convention:
\begin{enumerate}
\item If the topology on $R$ is not specified, we always endow $R$ with the $J$-adic topology (referred to as the \emph{natural topology}), where $J$ is the Jacobson radical.
An \emph{ideal of definition} $J_0$ of $R$ means an ideal of definition for the natural topology.
\item By $\Spf R$, we always endow $R$ with the natural topology. If there is any risk of confusion, we use the notation $\Spf (R,J_0)$ where $J_0$ is an ideal of definition for the natural topology.
\item A formally finitely generated $W$-algebra $R$ is called \emph{formally smooth (over $W$} if it is formally smooth over $W$ with respect to the natural topology.
\item We write $\Spf (R,(p))$ to denote the formal spectrum of $R$ endowed with the $p$-adic topology. If $p$ is nilpotent in $R$ (so the $p$-adic topology on $R$ is discrete), then we have $\Spec R = \Spf (R,(p))$.
\end{enumerate}
Since any maximal ideal of $R$ is open for the natural topology, the natural maps
$\Spec R \to \Spf (R,(p)) \to \Spf (R,J)$ induce bijections on the set of closed points. 
\end{rmksub}

Let $\cC$ be a pseudo-abelian\footnote{Pseudo-abelian categories are defined in the same way as abelian categories, except that we only require the existence of kernel for idempotent morphisms instead of requiring the existence of kernel and cokernel for any morphism. In practice, the pseudo-abelian categories that we will encounter are the category of filtered or graded objects in some abelian category.} symmetric tensor category such that arbitrary (infinite) direct sum exists. (For definitions in category theory, see \cite{Wedhorn:TannakianDVR} and references therein.) Let $\triv$ denote the unit object for $\otimes$-product in $\cC$ (which exists by axiom of tensor categories).

Let $\cD$ be a full subcategory of $\cC$ which is stable under direct sums, tensor products, and direct factors. Assume furthermore that $\cD$ is rigid; i.e., every object of $\cD$ has a dual. (For example, $\cC$ can be the category of $R$-modules filtered by direct factors over $R$, and $\cD$ can be the full subcategory of finitely generated projective $R$-modules filtered by direct factors.)

\begin{defnsub}\label{def:tensors}
For any object $M\in \cD$, we let
\[M^\otimes\in \cC\]
denote the direct sum of any (finite) combination of tensor products, symmetric products, alternating products, and duals of $M$. Note that  $M^\otimes = (M^*)^\otimes$.
\end{defnsub}

\begin{defnsub}\label{def:reductive}
Throughout this paper, a \emph{reductive group}  $G$ over a scheme $S$ means an affine smooth group scheme such that each geometric fibre is a \emph{connected} reductive group. (Note that we always impose the connectedness as a part of the definition of reductive groups.)
\end{defnsub}

\begin{propsub}\label{prop:Chevalley}
Let $R$ be either a field of characteristic~$0$ or a discrete valuation ring of mixed characteristic.
Let $G$ be a reductive group over $R$, and  $\Lambda$ a finite free $R$-module equipped with a closed immersion of algebraic $R$-groups $G\hra \GL(\Lambda)$. (We identify $G$ with its image in $\GL(\Lambda)$.) Then there exist finitely many  elements $s_\alpha\in \Lambda^\otimes$ such that $G$ coincides with the pointwise stabiliser of $(s_\alpha)$; i.e., for any $R$-algebra $R'$ we have
\[G(R') = \set{g\in\GL(\Lambda_{R'})\text{ such that }g(s_\alpha) = s_\alpha\ \forall \alpha}.\]
\end{propsub}
\begin{proof}
A more general statement is proved in \cite[Proposition~1.3.2]{Kisin:IntModelAbType}.
\end{proof}

\begin{exasub}\label{exa:GSp}
For any ring $R$,  let $\Lambda$ be a finite free $R$-module, and $(,):\Lambda\otimes\Lambda \thra R$ be a perfect alternating form. Then we will construct a  tensor $s\in\Lambda^\otimes$ from $R\starr\cdot (,)$ so that its pointwise stabiliser is $\GSp(\Lambda,(,))$. Let $c:\GSp(\Lambda,(,))\ra\Gm$ be the similitude character, and let $R(c)$ be $R$ as a $R$-module equipped with the $\GSp(\Lambda,(,))$-action via $c$. Then $(,)$ induces a $\GSp(\Lambda,(,))$-equivariant morphism $\Lambda\otimes\Lambda \thra R(c)$. 

Now, since the pairing $(,)$ induces an isomorphism $\Lambda\riso\Lambda^*$ sending $\lambda\in \Lambda$ to $(\lambda,\cdot)\in\Lambda^*$, we can view $(,)$ as a pairing of $\Lambda^*$. Then the the natural action of $\GSp(\Lambda,(,))$ on $\Lambda^*$ preserves this pairing up to similitude, with similitude character $c\iv$; i.e., we obtain the following $\GSp(\Lambda,(,))$-equivariant morphism
\[\Lambda^*\otimes\Lambda^*\thra R(c\iv) = R(c)^*,\]
induced by $(,)$.
By double duality, we now obtain a $\GSp(\Lambda,(,)) $-equivariant section $R(c) \hra \Lambda\otimes\Lambda$. (By choosing an $R$-basis $\{e_i\}$ of $\Lambda$, we can make these maps explicit as follows. Let $\{e_i^*\}$ denote the dual basis of $\Lambda^*$, and we define $e_i^\perp\in\Lambda$ by the condition $(e_i^\perp,e_j) = \delta_{ij}$ for any $j$. Then the pairing of $\Lambda^*$ induced by $(,)$ sends $e_i^*\otimes e_j^*$ to $(e_i^\perp,e_j^\perp)$. By applying double duality, the map $R(c) \hra \Lambda\otimes\Lambda$ sends the basis $1\in R(c)$ to $\sum_{i,j} (e_i^\perp, e_j^\perp)\cdot (e_i\otimes e_j)$.)

Therefore, $(,)$ induces the following   $\GSp(\Lambda,(,)) $-equivariant endomorphism
\begin{equation}\label{eqn:GSp}
\Lambda\otimes\Lambda \thra R(c) \hra \Lambda\otimes\Lambda; \quad e_i\otimes e_j \mapsto (e_i,e_j)\sum_{i',j'} (e_{i'}^\perp,e_{j'}^\perp) \cdot (e_{i'}\otimes e_{j'}).
\end{equation}
(The endomorphism is independent of the choice of $\{e_i\}$.)
We may view this endomorphism as a tensor $s\in \Lambda^{\otimes2}\otimes\Lambda^{*\otimes2}$. Note that replacing $(,)$ with any $R\starr$-multiple does not modify $s$.

We claim  that the pointwise stabiliser of $s$ in $\GL_{\Zp}(\Lambda)$ is  $\GSp(\Lambda,(,)) $. Indeed, any $g\in\GSp(\Lambda,(,))$ fixes each of the maps in (\ref{eqn:GSp}) by construction. Conversely, if $g\in\GL_R(\Lambda)$ fixes $s$, then it should preserve the image of $R(c)\hra \Lambda\otimes\Lambda$ (i.e., the line in generated by $\sum_{i',j'} (e_{i'}^\perp,e_{j'}^\perp) \cdot (e_{i'}\otimes e_{j'})$). This condition means that $g$ fixes the pairing on $\Lambda^*$ induced by $(,)$ up to similitude, which is equivalent to fixing the pairing $(,)$ on $\Lambda$ up to similitude.
\end{exasub}

\subsection{$\set\mu$-filtrations}\label{subsec:muFil}

\begin{defnsub}\label{def:grmu}
Let $R$ be any ring and $M$ a finitely generated projective $R$-module.
For a cocharacter  $\mu:\Gm\ra \GL_R(M)$, we define a grading $\gr^\bullet_\mu M$ so that $\gr^a_\mu( M)$ is the ``weight $(-a)$-part''; in other words, the $\Gm$-action on $M$ via $\mu$ leaves each grading stable, and the resulting $\Gm$-action on $\gr^a_\mu( M)$ is given by
\[ \Gm \xra{z\mapsto z^{-a}} \Gm \xra{z\mapsto z\id} \GL(\gr^a_\mu(M)).\]
We say that a filtration $\Fil^\bullet M$ is induced by $\mu$ if we can write $\Fil^a M = \bigoplus_{a'\geqslant a}\gr^a_\mu M$.
%
\end{defnsub}
In the definition, we chose the sign so that it is compatible with  the standard sign convention in the theory of Shimura varieties (as in \cite{Deligne:CorvalisShimura,Milne:Shimura}).

Let $G$ be a  reductive group over $\Zp$. We choose  $\Lambda\in\prep(G)$ with faithful $G$-action, and finitely many tensors $(s_\alpha)\subset \Lambda^\otimes$ which define $G$ (in the sense of Proposition~\ref{prop:Chevalley}). Let $\XX$ be a locally noetherian formal scheme over $\Spf \Zp$, and $\E$  a finite-rank locally free $\OO_\XX$-module (i.e., a vector bundle on $\XX$). Let $(t_\alpha)\subset \Gamma(\XX,\E^\otimes)$ be finitely many global sections. 

We will introduce a notion of filtrations on $\E$ which \'etale-locally admits a splitting given by some cocharacter (but such a splitting may not be defined globally); see Definition~\ref{def:RZG} for the relevant setting.

Given $\mu:\Gm\to G_W$, let us  recall the definition of parabolic groups and their unipotent radicals over $W$:
\begin{equation}\label{eqn:parabolic}
\begin{aligned}
P_G(\mu)&:=\{g\in G_W|\ \lim_{t\to0}\ad(\mu(t))g \text{ exists.} \} \subset G_W\text{ and}\\
U_G(\mu\iv)&:=\{g\in G_W|\ \lim_{t\to0}\ad(\mu\iv(t))g=\id \} \subset G_W,
\end{aligned}
\end{equation}
where the existence of the limit is defined as in \cite[p.~46]{ConradGabberPrasad:PRedGp2}. Recall that both groups are smooth over $W$; \emph{cf.} \cite[Proposition~2.1.8(3)]{ConradGabberPrasad:PRedGp2}.

\begin{defnsub}[X.~Zhu {\cite[Definition~3.8]{Zhu:Satake}}]\label{def:muFil}
Let $\mu:\Gm\ra G_W$ be a cocharacter, and let $\set\mu$ denote the $G(W)$-conjugacy class of $\mu$.
A filtration $\Fil^\bullet\E$  of $\E$ is called a \emph{$\set\mu$-filtration} (with respect to $(t_\alpha)$) if the following formal scheme
\begin{equation}\label{eqn:PFil}
\cP_{\Fil^\bullet_\E}:=\nf\Isom_{\OO_\XX}\big([\E,(t_\alpha),\Fil^\bullet\E],[\OO_\XX\otimes_{\Zp}\Lambda,(1\otimes s_\alpha), \Fil^\bullet_\mu]\big)
\end{equation}
is a $P_G(\mu)$-torsor. (This notion does not depend on the choice of $\mu\in\{\mu\}$.)
\end{defnsub}

\begin{rmksub}
Note that if there exists a $\set\mu$-filtration of $\E$ with respect to $(t_\alpha)$, then the following formal scheme is a $G$-torsor:
\begin{equation}\label{eqn:P}
\cP:=\nf\Isom_{\OO_\XX}\big([\E,(t_\alpha)],[\OO_\XX\otimes_{\Zp}\Lambda,(1\otimes s_\alpha)]\big).
\end{equation}

If $\cP$ is a $G$-torsor, then a filtration $\Fil^\bullet\E$  of $\E$ is a $\set\mu$-filtration (with respect to $(t_\alpha)$) if and only if the following condition holds: for some \'etale-local section $\varsigma\in\cP(\YY)$ (where $\YY\to\XX$ is an \'etale covering), the filtration $(\Fil^\bullet\E)_{\YY}$ of $\E_{\YY}$ is induced from the conjugate of the cocharacter $\mu$ over $\YY$ by some element $g\in G(\YY)$. In the earlier version of this paper, this was stated as the definition of $\set\mu$-filtration, and Definition~\ref{def:muFil} is the equivalent formulation by X.~Zhu \cite[Definition~3.8]{Zhu:Satake}.
\end{rmksub}

\begin{rmksub}\label{rmk:muFilGLn}
If $G=\GL(\Lambda)$ and $\E$ is a vector bundle over $\XX$ with rank equal to $\rank_{\Zp}(\Lambda)$, then a filtration $\Fil^\bullet\E$ of $\E$ is a $\set\mu$-filtration for some $\set\mu$ if and only if each of the graded pieces $\gr^i\E$ is of constant rank. (And one can write down $\set\mu$ uniquely up to conjugation from the ranks of $\gr^i\E$.)
\end{rmksub}

In the setting of Definition~\ref{def:muFil}, let $\Flag_{G,\set\mu}^{\E,(t_\alpha)}$ denote the functor on formal schemes over $\XX$, which associates to $\YY\xra f\XX$ the set of $\set\mu$-filtrations of $f^*\E$ with respect to $(f^*t_\alpha)$. We write  $\Flag_{\set\mu}^{\E}:=\Flag_{\GL(\Lambda),\set\mu}^{\E,\emptyset}$, and  use the same symbol to denote  the representing formal scheme, which is relatively projective and smooth  over $\XX$.

\begin{lemsub}\label{lem:muFil}
Assume that $\cP$  (\ref{eqn:P}) is a $G$-torsor.
Then $\Flag_{G,\set\mu}^{\E,(t_\alpha)}$ can be represented by a closed formal subscheme of $\Flag_{\set\mu}^{\E}$ which is smooth over $\XX$ with nonempty  geometrically connected fibres.
\end{lemsub}
\begin{proof}
Let us first deduce the lemma when the torsor $\cP_{\Fil^\bullet_\E}$ (\ref{eqn:PFil}) is trivial. In this case, the representing formal scheme is isomorphic to $(G_W/P_G(\mu))\times_{\Spec W} \XX$. Note that the representability of $G_W/P_G(\mu)$ as a proper smooth $W$-scheme can be deduced from  \cite[Proposition~2.1.8(3)]{ConradGabberPrasad:PRedGp2} and the Iwasawa decomposition (\cite[\S4.4]{BruhatTits:RedGp1}).\footnote{Indeed, \cite[Proposition~2.1.8(3)]{ConradGabberPrasad:PRedGp2} together with the Iwasawa decomposition shows that $G_W/P_G(\mu)$ admits an open dense subscheme isomorphic to $U_G(\mu\iv)$ and its translates by $G(W)$ covers $G_W/P_G(\mu)$. And the $W$-scheme $G_W/P_G(\mu)$ satisfies the valuative criterion for properness thanks to the properness of each fibre and the Iwasawa decomposition.} Now, the natural morphism $\Flag_{G,\set\mu}^{\E,(t_\alpha)}\to \Flag_{\set\mu}^{\E}$ is a closed immersion, being a proper monomorphism. (Note that the properness follows from the properness of the source and the target.)

In general, the natural inclusion $\Flag_{G,\set\mu}^{\E,(t_\alpha)} \hra \Flag_{\set\mu}^{\E}$ can be  represented by a closed immersion \'etale-locally on $\XX$, and it respects the \'etale descent datum. Now the lemma follows from  effectivity of \'etale descent for closed immersions. 
\end{proof}

The following corollary is straightforward from Lemma~\ref{lem:muFil}:
\begin{corsub}\label{cor:muFilviaCompl}
Assume that $\XX$ is locally noetherian and  $\cP$  (\ref{eqn:P}) is a $G$-torsor. Then a filtration $\Fil^\bullet \E$ is a $\mu$-filtration if and only if it is so over the formal neighbourhood $\wh\XX_x$ for any closed point $x$ in $\XX$.
\end{corsub}

Let us finish the section with the following trivial but useful lemma:
\begin{lemsub}\label{lem:muFil0}
Let $\E/\XX$ and $(t_\alpha)$ be as in Definition~\ref{def:muFil}, and let $\Fil^\bullet \E$ be a $\set\mu$-filtration on $\E$. Then we have $t_\alpha \in \Gamma(\XX, \Fil^0\E^\otimes)$ for each $\alpha$.
\end{lemsub}
\begin{proof}
Since the claim is \'etale-local on $\XX$, we may assume that the torsor $\cP_{\Fil^\bullet \E}$ (\ref{eqn:PFil}) is trivial. Fixing a trivialisation of $\cP_{\Fil^\bullet \E}$ we may assume that the filtration $\Fil^\bullet\E$ is given by a cocharacter $\mu:\Gm\ra G_\XX$. This means that $\Fil^\bullet\E$ admits a splitting  by the weight decomposition $\gr^\bullet_\mu\E$ using the $\Gm$-action by $\mu$. In particular, $\gr^0_\mu\E^\otimes$ is precisely the $\Gm$-invariants (with respect to $\mu$). On the other hand, the $\Gm$-action fixes each  $t_\alpha$ as the cocharacter $\mu$ factors through the pointwise stabiliser of $(t_\alpha)$. It follows that $t_\alpha \in \gr^0_\mu\E^\otimes$ for each $\alpha$.
\end{proof}

\subsection{Review on $p$-divisible groups and crystalline Dieudonn\'e theory}\label{subsec:DieudonneThy}
Throughout the paper, $\kappa$ be an algebraically closed field of characteristic $p>0$  unless stated otherwise. 
(Most of the time, there is no harm to set $\kappa:=\Fpbar$). We will set $W:= W(\kappa)$ and $K_0:= W[\ivtd p]$. Let $\sig$ denote the Witt vector Frobenius map on $W$ and $K_0$.

We will only consider \emph{compatible} PD thickenings; i.e., the PD structure is required to be compatible with the standard PD structure on $p\Zp$.

If $B$ is a $\Zp$-algebra and $\bb\subset B$ is a PD ideal, then for any $x\in \bb$ we let $x^{[i]}$ denote the $i$th divided power.

\begin{defnsub}
We define the \emph{isogeny category} of $p$-divisible groups over a scheme $\XX$ over $\Spf\Zp$ as follows:
\begin{itemize}
\item objects are $p$-divisible groups over $\XX$
\item morphisms $\iota:X \dra X'$ are  global sections of the following Zariski sheaf over $\XX$:
\[\QHom_R(X,X'):=\Hom(X,X')\otimes_\Z\Q. \]
 If $\XX$ is quasi-compact (for example, $\XX = \Spec R$ for $R\in\Nilp_W$) then morphisms can be understood as equivalence classes of diagrams of the form
\[ X \xla{[p^n]} X \xra{\iota'}X'\]
where the equivalence relation is defined by ``calculus of fractions''. We also say that $\iota$ is \emph{defined up to isogeny}, and often write $\iota=\ivtd{p^n}\iota'$ or $\iota' = p^n\iota$.
\end{itemize}
We will use dashed arrows $X\dra X'$ to denote morphisms defined up to isogeny. By \emph{quasi-isogeny} $\iota:X\dra X'$, we mean an isomorphism in the isogeny category; i.e., an invertible global section of $\QHom_R(X,X')$. Let  $\Qisg_R(X,X')$ be the set of quasi-isogenies. A quasi-isogeny $\iota$ is called an \emph{isogeny} if $\iota$ is an actual map of $p$-divisible groups.

We define the \emph{height} $h(\iota)$ of an isogeny $\iota:X\ra X'$ to be a locally constant function on $\XX$ so that Zariski-locally on $\XX$, the order of $\ker(\iota)$ is $p^{h(\iota)}$. We extend the definition of the \emph{height}  to a quasi-isogenies by
\[h(p^{-n}\iota):= h(\iota) - h([p^n])\]
Zariski-locally on $\XX$, where $\iota$ is an isogeny. 
\end{defnsub}

When $p$ is  nilpotent on the base, $\Qisg_R(X,X')$ satisfies the ``rigidity property'' analogous to rigidity of crystals; namely, for any $B\thra R$ with  nilpotent kernel killed by some power of $p$, and $p$-divisible groups $X$ and $X'$ over $B$, the natural morphism
\begin{equation}\label{eqn:DrinfeldRigidity}
 \Hom_B(X,X')[\ivtd p] \ra \Hom_R(X_R,X'_R)[\ivtd p]
\end{equation}
is bijective.
(\emph{Cf.} \cite[Lemma~1.1.3]{Katz:SerreTate}. Note that it is important to invert $p$ to obtain a bijection, since to lift a morphism $X_R\to X'_R$ over $B$ one often needs to introduce some denominator.)
In particular, for a locally noetherian formal scheme $\XX$ over $\Zp$, the isogeny categories of $p$-divisible groups over  $\XX$ and $\XX_{\red}$ are equivalent via the pull-back functor.

 Let $\XX$ be a where $p$ is nilpotent, and let $\ol\XX$ be the closed subscheme of $\XX$ cut out by the ideal generated by $p$. Let $i_{\CRIS}:=(i_{\CRIS,*},i_{\CRIS}^*):(\ol\XX/\Zp)_{\CRIS}\ra(\XX/\Zp)_{\CRIS}$ be the morphism of topoi induced from the closed immersion $\ol\XX\hra \XX$. Then $i_{\CRIS,*}$ and $i_{\CRIS}^*$ induce quasi-inverse exact equivalences of categories between the categories of crystals of quasi-coherent (respectively, finite locally free) $\cO_{\ol\XX/\Zp}$-modules and $\cO_{\XX/\Zp}$-modules. (This follows from  \cite[Lemma~2.1.4]{dejong:crysdieubyformalrigid}, which can be applied since the natural map $i_{\CRIS,*}\cO_{\ol\XX/\Zp} \ra \cO_{\XX/\Zp}$ is an isomorphism by \cite[\S5.17.3]{Berthelot-Ogus}.) In particular, for any crystal $\DD$ of quasi-coherent $\OO_{\XX/\Zp}$-modules, we define the pull-back by the absolute Frobenius morphism $\sig:\ol\XX\ra\ol\XX$ as follows:
\[
\sig^*\DD := i_{\CRIS,*}(\sig^*_{\CRIS} i_{\CRIS}^*\DD).
\]

For a $p$-divisible group $X$ over $\XX$, we have a contravariant Dieudonn\'e crystal\footnote{See \cite{messingthesis}, \cite{Mazur-Messing}, or \cite{Berthelot-Breen-Messing:DieudonneII} for the construction.} $\DD(X)$ equipped with  a filtration  $(\Lie X)^*\cong\Fil^1_X\subset \DD(X)_\XX$ by a Zariski-local direct factor as a vector bundle on $\XX$, where $\DD(X)_\XX$ is the pull-back of $\DD(X)$ to the Zariski site of $\XX$. We call $\Fil^1_X$  the \emph{Hodge filtration} for $X$. If $\XX = \Spec R$, then  we can regard the Hodge filtration as a filtration on the $R$-sections $\Fil^1_X \subset \DD(X)(R)$.
From the relative Frobenius morphism $F:X_{\ol \XX}\ra\sig^*X_{\ol\XX}$, we obtain the Frobenius morphism $F:\sig^*\DD(X)\ra\DD(X)$.
On tensor products of $\DD(X)$'s, we naturally extend the Frobenius structure and filtration.

We set $\triv:=\DD(\Qp/\Zp)$ and $\triv(-1):=\DD(\mu_{p^\infty})$. Note that $\triv\cong \OO_{\XX/\Zp}$ with the usual Frobenius structure and $\Fil^1 = 0$. We define $\DD(X)^*$ to be the $\OO_{\XX/\Zp}$-linear dual with the dual filtration. (Note that the Frobenius structure on $\DD(X)^*$ is defined ``up to isogeny''.) We set $\triv(0):= \triv$ and
\[\triv(-c):= \triv(-1)^{\otimes c}\ \text{and}\ \triv(c):=\triv(-c)^* \text{ if }c>0.\]
For any crystal $\DD$ with Frobenius structure and Hodge filtration, we set $\DD(r):=\DD\otimes\triv(r)$ for any $r\in \Z$. Note that $\DD(X^\vee)\cong \DD(X)^*(-1)$ by  \cite[\S5.3]{Berthelot-Breen-Messing:DieudonneII}.

We can extend the above definitions for $p$-divisible groups $X$ over a locally noetherian formal scheme $\XX$ over $\Spf\Zp$ as follows. We write $\XX = \varinjlim_n\XX_n$ where $\XX_n$ is a closed subscheme of $\XX$ cut out by some ideal of definition, and we set $X_n:=X_{\XX_n}$. Then we define $\DD(X)$ to be the projective system $\{\DD(X_n)\}$ with respect to the pull-back under the natural inclusion $\XX_n\hra\XX_{n+1}$ (\emph{cf.} \cite[\S2.4]{dejong:crysdieubyformalrigid}), and we carry out all the operations for $\DD(X)$ on the level of the projective system $\{\DD(X_n)\}$ (such as $\sigma^*\DD(X):=\{\sigma^*\DD(X_n)\}$, $\DD(X)^*:=\{\DD(X_n)^*\}$, etc.).
We let $\DD(X)_\XX$ denote the vector bundle on $\XX$ given by the projective system $\{\DD(X_n)_{\XX_n}\}$, and we obtain the Hodge filtration $\Fil^1_X\subset\DD(X)_\XX$ given by the projective system $\{\Fil^1_{X_n}\}$. If $\XX=\Spf R$, then we denote by $\DD(X)(R)$ the global sections of $\DD(X)_\XX$, and identify $\Fil^1_X$ with its global section.

\begin{defnsub}
We define the \emph{category of isocrystals} over $\XX$ as follows:
\begin{itemize}
\item objects are locally free $\OO_{\XX/\Zp}$-modules $\DD$; we write $\DD[\ivtd p]$ if we view $\DD$ as an isocrystal;
\item morphisms are  global sections of the Zariski sheaf $\Hom(\DD,\DD')[\ivtd p]$ over $\XX$. If $\XX$ is quasi-compact (for example, $\XX = \Spec R$ for $R\in\Nilp_W$) then morphisms can be understood as
 equivalence classes of diagrams of the form
\[ \DD \xla{p^n} \DD \xra{\iota'}\DD'\]
where the equivalence relation is defined by ``calculus of fractions''.
\end{itemize}

An \emph{$F$-isocrystal} over $\XX$ is a pair $(\DD[\ivtd p],F)$ where $\DD[\ivtd p]$ is an isocrystal over $\XX$ and $F:\sig^*\DD[\ivtd p]\riso \DD[\ivtd p]$ is an isomorphism of isocrystals. 
\end{defnsub}
For a $p$-divisible group $X$, $\DD(X)[\ivtd p]$ can be naturally viewed as an $F$-isocrystal. Note also that the notion of $F$-isocrystals is closed under direct sums, tensor products, and duality. So $\DD(X)^*[\ivtd p]$ is an $F$-isocrystal although the Frobenius structure may not be defined on $\DD(X)^*$.

We will often let $\triv(-c)$ denote the $F$-isocrystal associated to $\triv(-c)$.

Note that for any morphism up to isogeny $\iota:X\dra X'$ of $p$-divisible groups over $\XX$, we obtain a morphism of isocrystals $\DD(\iota):\DD(X')[\ivtd p]\ra\DD(X)[\ivtd p]$. If $\iota$ is a quasi-isogeny, then $\DD(\iota)$ is an isomorphism of $F$-isocrystals.

We now define $\DD(X)^\otimes$ by setting $\cC$ to be the category of (integral) crystals of quasi-coherent  $\OO_{\XX/\Zp}$-modules and $\cD\subset\cC$ to be the full subcategory of finitely generated locally free objects (\emph{cf.} Definition~\ref{def:tensors}). Then the Hodge filtration on $\DD(X)_\XX$ induces a natural filtration on $\DD(X)^\otimes_\XX$, and the Frobenius morphism on $\DD(X)$ induces an isomorphism of isocrystals $F:\sig^*\DD(X)^\otimes[\ivtd p]\riso \DD(X)^\otimes[\ivtd p]$.
More generally, for any quasi-isogeny $\iota:X\dra X'$ of $p$-divisible groups over $R$, $\DD(\iota)$ extends to
\[\DD(\iota):\DD(X')^\otimes[1/p] \riso \DD(X)^\otimes [1/p].\]

\begin{defnsub}\label{def:sections}
For a ring $R$ where $p$ is nilpotent, a surjection $B \thra R$ is called \emph{$p$-adic PD thickening} if $B=\varprojlim_m B/p^m$  where $ B/p^m$'s form a projective system of PD thickenings compatible with the PD structure for $p\Zp$.

Let $X$ be a $p$-divisible group over $R$. For a $p$-adic PD thickening $B\thra R$, we set
\[
\DD(X)(B):=\varprojlim_m \DD(X)(B/p^m)
\]

For $t:\triv\ra\DD(X)^\otimes$ a morphism of crystals,  we define the \emph{section} $t(B)$ of $t$ over $B$ (or the \emph{$B$-section} of $t$) to be the image of $1$ under the map
\[B=\triv(B)  \xra t \DD(X)(B)^\otimes.\]
\end{defnsub}

Let us review the interpretation of crystals and morphisms thereof as modules with connection and horizontal morphisms thereof, following \cite[\S2.2]{dejong:crysdieubyformalrigid}. Let $R$ be a formally finitely generated $W/p^m$-algebra for some $m$, and we choose a surjection $A \thra R$ from a formally smooth formally finitely generated $W$-algebra $A$. Let $D\thra R$ be the $p$-adic completed PD hull of $A\thra R$. (Note that if $R$ is formally smooth over $W/p^m$, we may choose $A$ so that $A/p^m = R$, in which case we have $D=A$.)

Let $\wh\Omega_{A}$ denote the module of $p$-adically continuous K\"ahler differentials of $A$. By \cite[Remark~1.3.4]{dejong:crysdieubyformalrigid}, $\wh\Omega_{A}$ is isomorphic to the module of continuous K\"ahler differentials with respect to the natural  topology. (In particular, $\wh\Omega_{A}$ is a locally free $A$-module with rank equal to the relative dimension of $A$ over $W$; note that $\wh\Omega_W = 0$ as the residue field of $W$ is perfect, so $\wh\Omega_{A} = \wh\Omega_{A/W}$.) Then the universal $p$-adically continuous derivation $\nabla: A \to \wh\Omega_{A}$ naturally extends to a $p$-adically continuous connection
\[
 \nabla(=\nabla_D): D \to D\otimes_{A}\wh\Omega_{A};
 \]
 see \cite[(2.2.1.2)]{dejong:crysdieubyformalrigid} for the formula.

Let $\E$ be a crystal of quasi-coherent $\OO_{\Spec R/\Zp}$-modules. (Note that we work over $\Spec R$, \emph{not} over $\Spf R$ with the natural adic topology.) Then one can naturally define the connection
\begin{equation}\label{eqn:CrysConn}
\nabla(=\nabla_\E):\E(D) \to \E(D)\otimes_{A}\wh\Omega_{A}
\end{equation}
over $\nabla_D$; \emph{cf.} \cite[Remark~2.2.4~d)]{dejong:crysdieubyformalrigid}. We recall the following result
\begin{lemsub}[{\cite[Proposition~2.2.2]{dejong:crysdieubyformalrigid}}]\label{lem:CrysConn}
The connection $\nabla_\E$ is integrable and topologically quasi-nilpotent (in the sense of \cite[Remark~2.2.4~c)]{dejong:crysdieubyformalrigid}). Furthermore, the assignment $\E\rightsquigarrow (\E(D),\nabla_\E)$ induces an equivalence of categories from the categories of crystals of quasi-coherent $\OO_{\Spec R/\Zp}$-modules to the category of $D$-modules with integrable  and topologically quasi-nilpotent connections, matching locally free crystals and locally free $D$-modules.
\end{lemsub}

\begin{rmksub}\label{rmk:JadicConn}
We choose $R= A/p$, so $D=A$. Let $\E$ be a crystal of locally free $\OO_{\Spec R/\Zp}$-module of  finite rank.
Then although the crystalline connection on $\E(A)$ is a priori only $p$-adically continuous, it turns out to be continuous for the natural $J$-adic topology (where $J$ is an ideal of definition). Indeed, this follows from the $J$-adic continuity of  the universal $p$-adically continuous derivation $\nabla:A\to\wh\Omega_{A}$; \emph{cf.} \cite[\S2.5]{dejong:crysdieubyformalrigid}.
\end{rmksub}

\begin{rmksub}\label{rmk:isoc}
We may extend Lemma~\ref{lem:CrysConn} for isocrystals over $\Spec R$ as follows. Let $\E[\ivtd p]$ be an isocrystal over $\Spec R$. We choose a crystal $\E$ representing $\E[\ivtd p]$, and consider $\E(D)[\ivtd p]$ equipped with  connection $\nabla_{\E[\frac{1}{p}]}:\E(D)[\ivtd p] \to \E(D)[\ivtd p]\otimes_{\wt A}\wh\Omega_{\wt A}$ defined by extending $\nabla_\E$. Then $(\E(D)[\ivtd p],\nabla_{\E[\frac{1}{p}]})$ is independent of the choice of $\E$. Now, Lemma~\ref{lem:CrysConn}  shows that the assignment $\E[\ivtd p]\rightsquigarrow(\E(D)[\ivtd p],\nabla_{\E[\frac{1}{p}]})$ induces a fully faithful functor from the category of isocrystals over $\Spec R$ to the category of $D[\ivtd p]$-modules with integrable topologically quasi-nilpotent connection.
\end{rmksub}

 \begin{rmksub}\label{rmk:Fisoc}
 We choose a lift of Frobenius endomorphism $\sig: \wt A\to \wt A$, which naturally extends to a lift of Frobenius endomorphism $\sig:D\to D$.\footnote{\label{footnote} To see this, $D$ turns out to coincide with the $p$-adically completed PD hull of $\wt A\thra R/p$, and $\sigma:\wt A\to\wt A$ sends $\ker (\wt A\thra R/p)$ into itself. Now we apply the universal property of PD hull to obtain $\sigma:D\to D$.} Given an $F$-isocrystal $\E[\ivtd p]$ over $\Spec R$ (i.e, an isocrystal $\E$ equipped with an isomorphism $F:\sigma^*\E[\ivtd p]\riso\E[\ivtd p]$), we obtain the following horizontal map
 \[
 F: \sigma^*\E(D)[\ivtd p] = D\otimes_{\sigma,D}\E(D)[\ivtd p] \riso \E(D)[\ivtd p]
 \]
induced by $F:\sigma^*\E[\ivtd p]\riso\E[\ivtd p]$. Then Lemma~\ref{lem:CrysConn} induces a fully faithful functor from the category of $F$-isocrystals to certain Frobenius-modules with connection.
 \end{rmksub}

 \begin{rmksub}\label{rmk:CrysBC}
 Let $R\to R'$ is a quotient map of formally finitely generated $W/p^m$-algebras, and let $D'\thra R'$ denote the $p$-adically completed PD hull of $A\thra R\thra R'$. Then we have a natural PD morphism $D\to D'$.

 Let $\E$ be a crystal of quasi-coherent $\OO_{\Spec R/\Zp}$-modules, and $\E'$ its pull-back over $\Spec R'$. Then we have a natural isomorphism
 \[
 \E'(D') \cong \E(D)\wh\otimes_D D',
 \]
 where $-\wh\otimes_D D'$ is the $p$-adic completion of the scalar extension. Furthermore, one can obtain $\nabla_{\E'}$ by naturally extending $\nabla_{\E}$. (This horizontal isomorphism can be obtained by applying the rigidity axiom of crystals \cite[D\'efinition~1.2.1]{Berthelot-Breen-Messing:DieudonneII} to the PD morphism $D/p^{m'}\to D'/p^{m'}$ over $R\to R'$ for any $m'\geqslant m$.) This discussion can be applied to the pull-back of ($F$-)isocrystals via  closed immersions $\Spec R'\hra \Spec R$.

Let $J$ be an ideal of definition of $R$. Then a projective system  $\{\E_n\}$ of crystals over $\Spec R/J^n$ can be interpreted as a projective system $\{\E_n(D_n)\}$ of $D_n$-modules with connection, where $D_n\thra R/J^n$ is the $p$-adically completed PD hull of $\wt A\thra R\thra R/J^n$. This way, one can get an interpretation of ``crystals over $\Spf R$''  in terms of certain modules with connection (which would be different from the descriptions of crystals over $\Spec R$ unless $R$ is finitely generated over $W/p^m$).
 \end{rmksub}

 \subsection{$F$-isocrystals with $G$-structure over $\kappa$}\label{subsec:GIsoc}
 We review some basic definitions and results on ``$F$-isocrystals with $G$-structure''. See \cite[Ch.1]{RapoportZink:RZspace} for a detailed overview.

 Recall that the category of quasi-coherent crystals of $\OO_{\Spec \kappa/W}$-modules is equivalent to the category of $W$-modules by taking sections over $W = W(\kappa)$. Therefore, an $F$-isocrystal over $\kappa$ can be regarded as a pair $(D,F)$, where $D$ is a $K_0$-vector space and $F:\sig^*D\riso D$ is an isomorphism.  By abuse of notation, we also call $(D,F)$ an \emph{$F$-isocrystal} over $\kappa$.

Let  $(D,F)$ be a rank-$n$ $F$-isocrystal over $\kappa$. By choosing a basis of $D$ (or equivalently, by choosing a faithful algebraic action of ${\GLn}$ on $D$), we can find $b\in\GLn(K_0)$ which is the matrix representation of $F$ (i.e., $F$ is the linearisation of $b\sigma$). Choosing a different basis of $D$, $b$ is replaced by a suitable ``$\sig$-twisted conjugate''.
Motivated by this, we make the following definition (\emph{cf.} \cite{Kottwitz:Gisoc1}, \cite[\S1.7]{RapoportZink:RZspace}):
\begin{defnsub}
\label{def:GIsoc} Let $G$ be a linear algebraic group over $\Qp$. We say that $b,b'\in G(K_0)$ are \emph{$\sig$-conjugate in $G(K_0)$} if there exists $g\in G(K_0)$ such that $b'=g b \sig(g)\iv$. Let $[b]\subset G(K_0)$ denote the set of $\sig$-conjugates of $b\in G(K_0)$ in $G(K_0)$. 
\end{defnsub}

Let $\Qprep(G)$ denote the category of finite-dimensional $\Qp$-vector spaces with algebraic $G$-action. Then for any $b\in G(K_0)$ we can functorially associate, to  any $V\in\Qprep(G)$, an $F$-isocrystal $\beta_b(V) = (K_0\otimes_{\Qp}V, F_b)$, where $F_b$ is defined as follows:
\begin{equation}
F_b: \sigma^*(K_0\otimes_{\Qp}V) \cong K_0 \otimes_{\Qp} V \xra{\rho_V(b)} K_0\otimes_{\Qp}V.
\end{equation}
Here, $\rho_V:G\ra\GL(V)$ is the homomorphism of algebraic groups defining the $G$-action on $V$.
For $b,b'\in G(K_0)$, we have $\beta_b\cong \beta_{b'}$ if  $b$ and $b'$ are $\sig$-conjugate in $G(K_0)$.
Note that if $G=\GLn/\Qp$ and $V=\Qp^n$ is the standard representation of $\GLn$, then $b\mapsto \beta_b(V)$ induces a bijection between the set of $\sig$-conjugacy classes in $\GLn(K_0)$ and the set of isomorphism classes $F$-isocrystals of rank $n$ over $\kappa$.

We define the following group valued functor $J_b=J_{G,b}$ on $\Alg_{\Qp}$:
\begin{equation}\label{eqn:J}
J_b(R):=\set{g\in G(R\otimes_{\Qp}K_0)|\ gb\sig(g)\iv = b},\ \forall R\in\Alg_{\Qp}.
\end{equation}

\begin{propsub}[{\cite[Corollary~1.14]{RapoportZink:RZspace}}]
If $G$ is a  reductive group over $\Qp$ then $J_b$ can be represented by  an inner form of some Levi subgroup of $G$.
\end{propsub}

\subsection{Affine Deligne-Lusztig set}\label{subsec:FilGIsoc}

From now on, we let $G$ be a  reductive group which is \emph{unramified}\footnote{I.e., quasi-split and split over $\Qp^{\ur}$; or equivalently, $G$ admits a reductive model over $\Zp$.} over $\Qp$. We fix a reductive model $G_{\Zp}$ over $\Zp$, and will often write $G=G_{\Zp}$ if there is no risk of confusion. For any $\Zp$-algebra $R$, we write $G_R$ the base change of $G$ over $\Spec R$.

Recall that $G$ is split over $W$ (which is a strictly henselian discrete valuation ring). Choosing a maximal torus $T$ of $G_W$, we have a bijection
\begin{multline}\label{eqn:CartanDecomp}
X_*(T)/W(G,T) \riso \Hom_W(\Gm,G_W)/G(W) \\
\cong \Hom_{K_0}(\Gm,G_{K_0})/G(K_0) \riso G(W)\backslash G(K_0)/G(W)
\end{multline}
induced by $\set{\nu}\mapsto G(W)p^\nu G(W)$,  where $p^\nu:=\nu(p)\in G(K_0)$ and $W(G,T)$ is the Weyl group; indeed, the first two bijections are standard for split reductive groups, and the last bijection is the Cartan decomposition.

\begin{defnsub}\label{def:ADL}
To $b\in G(K_0)$ and a $G(W)$-conjugacy class $\set\nu$ of cocharacters $\nu:\Gm\ra G_W$, we associate the \emph{affine Deligne-Lusztig set} as follows:
\begin{multline*}
 X^G_{\set\nu}(b):=\big\{g\in G(K_0) \text{ such that } g\iv b \sig(g)\in G(W)p^\nu G(W)\big\}/G(W)\\
 \subset G(K_0)/G(W).
\end{multline*}
In the intended application, we will choose $\set\nu$ so that $b\in G(W)p^\nu G(W)$. 
If $\set\nu$ is chosen this way, then we write $X^G(b):=X_{\set\nu}^G(b)$ since it only depends on $(G,b)$.
\end{defnsub}

For $\gamma\in G(K_0)$, the left translation $gG(W)\ra \gamma gG(W)$ induces
\begin{equation}\label{eqn:ADLEqCl}
X^G_{\set{\nu}}(\gamma\iv b\sig(\gamma)) \riso X^G_{\set{\nu}}(b).
\end{equation}
In particular, we obtain a natural action $J_b(\Qp)$ on $X^G_{\set\nu}(b)$ (since $J_b(\Qp)\subset G(K_0)$). 

The following properties are straightforward to verify from the definition:
\begin{lemsub}\label{lem:FunctAffDL}
\begin{enumerate}
\item For any morphism $f:G\ra G'$ of  reductive group over $\Zp$, we have a natural map $X^G_{\set\nu}(b) \ra X^{G'}_{\set{f\circ\nu}}(f(b))$ induced by $gG(W)\mapsto f(g)G'(W)$. Furthermore, if $f$ is a closed immersion, then the induced map on the affine Deligne-Lusztig sets is injective.
\item For another  reductive group $G'$ over $\Zp$, $b'\in G'(W)$ and a conjugacy class of cocharacters $\nu':\Gm\ra G'_W$, we have an isomorphism
\[X^{G\times G'}_{\set{(\nu,\nu')}}(b,b')\riso X^G_{\set\nu}(b)\times X^{G'}_{\set{\nu'}}(b')\]
induced by the natural projections.
\end{enumerate}
In particular, $f:G\ra G'$ induces $X^G(b)\ra X^{G'}(f(b))$ and we have a natural isomorphism $X^{G\times G'}(b,b')\riso X^G(b)\times X^{G'}(b')$.
\end{lemsub}
\begin{proof}
The only possibly non-trivial assertion is the injectivity of the map $X^G_{\set\nu}(b) \ra X^{G'}_{\set{f\circ\nu}}(f(b))$  when $f$ is a closed immersion. Indeed, the following map
\[G(K_0)/G(W) \ra G'(K_0)/G'(W),\]
induced by $f$, is injective, from which the desired injectivity follows.
\end{proof}

Later, we will consider pairs $(G,b)$ which can be related to some $p$-divisible group. Let us first spell out the condition when  $G=\GL_n$ and $\Lambda = \Zp^n$ is the standard representation. Then $\rho_{\Lambda^*[\frac{1}{p}]}:\GL_n\to\GL_n$ is the isomorphism sending $b$ to its transpose-inverse. Now, we consider $\beta_b(\Lambda^*[\frac{1}{p}]):=(K_0\otimes_{\Zp}\Lambda^*, F_b)$ (\emph{cf.} Definition~\ref{def:GIsoc}), and let $\bM^\Lambda_b $ denote the $W$-lattice $W\otimes_{\Zp}\Lambda^*$ in $\beta_b(\Lambda^*[\frac{1}{p}])$. Then $\bM^\Lambda_b$ is a Dieudonn\'e module of some $p$-divisible group over $\kappa$ if and only if we have $p\bM^\Lambda_b \subset F_b(\sigma^*\bM^\Lambda_b) \subset \bM^\Lambda_b$, which is equivalent to requiring that $b\in \GL_n(W)p^\nu \GL_n(W)$ where $p^\nu:=\diag(\underbrace{1,\cdots,1}_{n-d}, \underbrace{p^{-1},\cdots,p^{-1}}_d)$ for some $d\in [0,n]$.
\begin{defnsub}\label{def:HodgeAssump}
We fix $b\in G(K_0)$. Assume that there exists a faithful $G$-representation $\Lambda\in \prep(G)$ such that  the $W$-lattice
\[\bM^\Lambda_b:=W\otimes_{\Zp}\Lambda^* \subset \beta_b(\Lambda^*[\tfrac{1}{p}]) \]
satisfies $p\bM^\Lambda_b \subset F_b(\sigma^*\bM^\Lambda_b) \subset \bM^\Lambda_b$, where $\rho_{\Lambda^*[\frac{1}{p} ]}:G\to\GL_{\Qp}(\Lambda^*[\ivtd p])$ is the morphism defining the contragradient $G(K_0)$-action on $K_0\otimes_{\Zp}\Lambda^*$. Then, $\bM^\Lambda_b$ is an $F_b$-stable $W$-lattice in $\beta_b(\Lambda^*[\ivtd p])$, which is a (contravariant) Dieudonn\'e module of some $p$-divisible group over $\kappa$.
(The existence of such $\Lambda$ is a restrictive condition on $G$ and $b$. See Example~\ref{exa:ShVar} for the reason to dualise $\Lambda$ in the definition of $\bM^\Lambda_b$.)

For such $b$ and $\Lambda$, we let $\BX_b^\Lambda$ denote the $p$-divisible group over $\kappa$ such that we have an isomorphism $\DD(\BX_b^\Lambda)(W)\cong \bM^\Lambda_b$ as an $F$-crystal.
\end{defnsub}

Just like the abelian variety which appears as a complex point of some Shimura variety of Hodge type has certain Hodge cycles, the $p$-divisible group $\BX$ as in Definition~\ref{def:HodgeAssump} has certain crystalline tensors from the fact that $b\in G(K_0)$. The following lemma is straightforward.
\begin{lemsub}\label{lem:CrysTateOverW}
Let $(G,b)$  and $\BX^\Lambda_b$ be as in Definition~\ref{def:HodgeAssump}.
Let $s\in\Lambda^\otimes$ be such that for any $\Zp$-algebra $R$, $1\otimes s\in R\otimes \Lambda^\otimes$ is fixed by $G(R)$. (For example, we may take $s=s_\alpha$ for $s_\alpha$ as in Proposition~\ref{prop:Chevalley}.)
Consider $1\otimes s \in W\otimes\Lambda^\otimes = \DD(\BX^\Lambda_b)(W)^\otimes$. (Recall that $\Lambda^\otimes = (\Lambda^*)^\otimes$; \emph{cf.} Definition~\ref{def:tensors}.) Then $1\otimes s\in\DD(\BX^\Lambda_b)(W)^\otimes[\ivtd p]$ is fixed by $F$, where we extend $F$ naturally to $\DD(\BX^\Lambda_b)(W)^\otimes[\ivtd p]$. 
\end{lemsub}

\begin{defnsub}\label{def:bft}
We fix finitely many tensors $(s_\alpha)\subset \Lambda^\otimes$ which defines $G$ in the sense of Proposition~\ref{prop:Chevalley}. We set $(\bft_\alpha):=(1\otimes s_\alpha) \in W\otimes_{\Zp}\Lambda^\otimes\cong \DD(\BX^\Lambda_b)(W)^\otimes$. By Lemma~\ref{lem:CrysTateOverW}, each $\bft_\alpha$ is fixed by the crystalline Frobenius operator in $\DD(\BX^\Lambda_b)(W)^\otimes[\ivtd p]$, which is induced by $b$ via $\bM^\Lambda_b=\DD(\BX^\Lambda_b)(W)$.
\end{defnsub}

\begin{lemsub}\label{lem:X0}
Let $(G,b)$ and $\Lambda$ be as in Definition~\ref{def:HodgeAssump}, and set  $\BX:=\BX_b^\Lambda$. Let $\set\mu$ be a $G(W)$-conjugacy class of cocharacters $\Gm\ra G_W$. Then the following hold:
\begin{enumerate}
\item\label{lem:X0:nu}
For a cocharacter $\mu:\Gm\ra G_W$, the Hodge filtration $\Fil^1_{\BX}\subset \DD(\BX)(\kappa)= \kappa\otimes_{\Zp}\Lambda^*$ is induced by $\mu$ if and only if  $b\in G(W)p^{\sig^*\mu\iv}$. (Recall that we identified $\DD(\BX)(W) = \bM^\Lambda_b = W\otimes_{\Zp}\Lambda^*$.)
\item\label{lem:X0:muFil}
The Hodge filtration $\Fil^1_{\BX}\subset \DD(\BX)(\kappa)$ is a $\set\mu$-filtration with respect to the image of $(\bft_\alpha)$ in $\DD(\BX)(\kappa)^\otimes$  (\emph{cf.} Definition~\ref{def:muFil})  if and only if we have $b\in G(W)p^{\sig^*\mu\iv}G(W)$.
\item\label{lem:X0:tensor}
The image of $(\bft_\alpha)$ in $\DD(\BX)(\kappa)^\otimes$ lies in the $0$th filtration with respect to the Hodge filtration.
\end{enumerate}
\end{lemsub}
\begin{proof}
We write $V:=\Lambda[\ivtd p]$ and let $\rho_{V^*}:G\to\GL(V^*)$ denote the homomorphism defining the $G$-action on $V^*$.

Let us first show (\ref{lem:X0:nu}). Assume that $b\in G(W)p^\nu$ for some $\nu:\Gm\ra G_W$. Since the Hodge filtration $\Fil^1_{\BX}\subset \DD(\BX)(\kappa)$ is the kernel of $\rho_{V^*}(b)(\sigma\otimes\id_{\Lambda^*})$ acting on $\kappa\otimes_{\Zp}\Lambda^* = \DD(\BX)(\kappa)$, we have
\[(\rho_{V^*}(b)\sig)\iv(p\bM^\Lambda_b) = \sig\iv(\rho_{V^*}((p\iv)^\nu)\cdot p\bM^\Lambda_b) = \gr^1_{(\sig\iv)^*\nu\iv}\bM^\Lambda_b + p\bM^\Lambda_b.\]
(To see the second equality, recall that $\gr^i_{(\sig\iv)^*\nu\iv}\bM^\Lambda_b$ is the eigenspace for the action of $\rho_{V^*}((p\iv)^\nu)$ with eigenvalue $p^i$.)
This shows that the Hodge filtration $\Fil^1_\BX$ is the image of $\gr^1_{(\sig\iv)^*\nu\iv}\bM^\Lambda_b$ in $\bM^\Lambda_b/p\bM^\Lambda_b\cong \DD(\BX)(\kappa)$ (i.e.,  $\Fil^1_\BX$ is induced by $\mu$ with $\nu:=\sigma^*\mu\iv$).

Conversely, assume that  $\Fil^1_\BX$ is induced by $\mu:\Gm\to G_W$. By reversing the engineering, we have
\[
\rho_{V^*}(p^\mu)(p\bM^\Lambda_b) = \gr^1_{\mu}\bM^\Lambda_b + p\bM^\Lambda_b = (\rho_{V^*}(b)\sigma)\iv(p\bM^\Lambda_b).
\]
Applying $p\iv \rho_{V^*}(b)\sigma$ to this equation, it follows that $\rho_{V^*}(b\cdot \sigma(p^\mu))\bM^\Lambda_b = \bM^\Lambda_b$. (Note that $\rho_{V^*}$ commutes with $\sigma$ as it is defined over $\Qp$.)
Since  $G(W)\subset G(K_0)$ is the  stabiliser of the $W$-lattice $\bM^\Lambda_b = W\otimes_{\Zp}\Lambda^*$ in $K_0\otimes _{\Qp}V^*$ (via $\rho_{V^*}$), we have $b\cdot \sigma(p^\mu) = b\cdot p^{\sigma^*\mu} \in G(W)$.
This proves (\ref{lem:X0:nu}).

Let us show (\ref{lem:X0:muFil}). Choose a cocharacter $\nu:\Gm\ra G_W$ such that $b\in G(W)p^{\nu}G(W)$; recall that such $\nu$ is unique up to $G(W)$-conjugate. By replacing $\nu$ by a suitable $G(W)$-conjugate, we may assume that $b = gp^\nu$ for some $g\in G(W)$; indeed, if $b =  g_1 p^{\nu} g_2$ for $ g_1,  g_2\in G(W)$, then we take $g = g_1g_2$ and replace $\nu$ with $g_2\iv\nu g_2$. From (\ref{lem:X0:nu}) it follows that the Hodge filtration $\Fil^1_\BX$ is a $\set\mu$-filtration for $\mu = (\sig\iv)^*\nu\iv$.
Conversely, if $\Fil^1_\BX$ is a $\{\mu\}$-filtration, then it suffices, by (\ref{lem:X0:nu}), to show the existence of a cocharacter $\mu:\Gm\to G_W$ (over $W$, not just over $\Fpbar$) that induces $\Fil^1_\BX$. To show the existence of such $\mu$, note that there exists, by definition of $\{\mu\}$-filtration,  a cocharacter $\mu_0\in\{\mu\}$ over $W$ such that $\Fil^1_\BX$ is given by the conjugate of the special fibre $\bar\mu_0:\Gm\to G_{\Fpbar}$ by some $\bar g\in G(\Fpbar)$.  Now, we choose a lift $g\in G(W)$ of $\bar g$ and let $\mu$ denote the conjugate of $\mu_0$ by $g$. This proves (\ref{lem:X0:muFil}).

By Lemma~\ref{lem:muFil0}, (\ref{lem:X0:tensor}) follows from (\ref{lem:X0:muFil}).
\end{proof}

\begin{rmksub}\label{rmk:X0}
We use the notation in Lemma~\ref{lem:X0}. Then Lemma~\ref{lem:X0} asserts that for any $\BX:=\BX^\Lambda_b$ with $b\in G(K_0)$, there exists a unique conjugacy class of cocharacters $\set\mu$ such that the Hodge filtration $\Fil^1_\BX$ is a $\set\mu$-filtration.

Since there is no obstruction of lifting $\set\mu$-filtrations (\emph{cf.} Lemma~\ref{lem:muFil}), there exists a $\set\mu$-filtration $\Fil^1_{\wt \BX}\subset \bM^\Lambda_b$ lifting $\Fil^1_\BX$. If $p>2$, then such a lift $\Fil^1_{\wt \BX}$ gives rise to a $p$-divisible group $\wt\BX$ over $W$. For such a lift $\wt \BX$, the tensors $(\bft_\alpha)\subset \DD(\wt\BX)(W)^\otimes  =(\bM^\Lambda_b)^\otimes$ lie in the $0$th filtration with respect to the Hodge filtration for $\wt\BX$; \emph{cf.} Lemma~\ref{lem:muFil0}.
\end{rmksub}

We consider $(G,b)$ and $\Lambda$ as in Definition~\ref{def:HodgeAssump}, and choose finitely many tensors $(s_\alpha)\subset \Lambda^\otimes$ defining $G\subset\GL(\Lambda)$ as in Proposition~\ref{prop:Chevalley}. Then to $(G,b)$ and $\Lambda$, we have associated $(\BX,(\bft_\alpha))$ and an isomorphism $\varsigma:W\otimes_{\Zp}\Lambda^*\cong\DD(\BX)(W)$, where $\BX:=\BX_b^\Lambda$ is a $p$-divisible group over $\kappa$, $(\bft_\alpha)\subset \DD(\BX)(W)^\otimes$ are $F$-invariant tensors (\emph{cf.} Lemma~\ref{lem:CrysTateOverW}), and $\varsigma$ is a $W$-linear isomorphism which matches $(1\otimes s_\alpha)$ and  $(\bft_\alpha)$. Note that  we can recover $(G,b)$ from $(\BX, (\bft_\alpha), \varsigma)$. In the setting of Example~\ref{exa:GSp} (when $G=\GSp_{2g}$ and $\Lambda = \Zp^{2g}$ is the standard representation), we can interpret $(\BX,(\bft_\alpha))$ as a principally quasi-polarised $p$-divisible group.

We will now interpret $X^G(b) = X^G_{\set{\sig^*\mu\iv}}(b)$ in terms of quasi-isogenies of $p$-divisible groups with $F$-invariant tensors over $\kappa$. For a coset in $G(K_0)/G(W)$ belonging to $X^G(b)$, we pick a representative $g$ and set $b':=g\iv b \sig(g)$. Consider $\bM^\Lambda_{b'} = W\otimes_{\Zp}\Lambda^*$ with $F$ given by $b'\in G(K_0)$, and $F$-invariant tensors $(\bft'_\alpha) = (1\otimes s_\alpha)\subset(\bM^\Lambda_{b'})^\otimes$.  The condition $b'\in G(W)p^{\sig^*\mu\iv}G(W)$ implies  that $\bM^\Lambda_{b'}$ corresponds to  a $p$-divisible group $\BX':=\BX^\Lambda_{b'}$, whose Hodge filtration is a $\set\mu$-filtration with respect to $\bf(t'_\alpha)$ by Lemma~\ref{lem:X0}. We also obtain a quasi isogeny $\iota:\BX\dra\BX'$ corresponding to
\[ \bM^\Lambda_{b'}[\ivtd p] = K_0\otimes_{\Zp}\Lambda^* \xra{g} K_0\otimes_{\Zp}\Lambda^* = \bM^\Lambda_b[\ivtd p],\]
which matches $(\bft'_\alpha)\subset(\bM^\Lambda_{b'})^\otimes$ with $(\bft_\alpha)\subset(\bM^\Lambda_b)^\otimes$. The tuple $(\BX',(\bft'_\alpha),\iota)$ only depends on $gG(W)$ up to isomorphism of $p$-divisible respecting the tensors and quasi-isogeny.

\begin{propsub}\label{prop:LR}
The map defined above gives a bijection from $X^G(b)$ to the set of isomorphism classes of tuples $(\BX', (\bft'_\alpha), \iota)$ which satisfy the following
\begin{itemize}
\item $\BX'$ is a $p$-divisible group over $\kappa$ and $(\bft'_\alpha)\subset \DD(\BX')(W)^\otimes$ are $F$-invariant tensors, such that there exists a $W$-linear isomorphism  $\varsigma':\DD(\BX')(W) \riso W\otimes_{\Zp}\Lambda^*$  that matches $(1\otimes s_\alpha)$ and $(\bft'_\alpha)$, and the Hodge filtration $\Fil^1_{\BX'}$ is a $\set\mu$-filtration with respect to $(t'_\alpha)$.
\item $\iota:\BX\dra\BX'$  is a quasi-isogeny such that $\DD(\iota):\DD(\BX')(W)[\ivtd p] \riso \DD(\BX)(W)[\ivtd p]$ matches $(\bft'_\alpha)$ with $(\bft_\alpha)$.
\end{itemize}
\end{propsub}
\begin{proof}
We will define the inverse map. By construction of $\BX=\BX^\Lambda_b$, we  have an $F$-equivariant isomorphism $\varsigma:\DD(\BX)(W) \riso \bM^\Lambda_b = (W\otimes_{\Zp}\Lambda^*,F_b)$ matching $(\bft_\alpha)$ with $(1\otimes s_\alpha)$.
Let  $(\BX', (\bft'_\alpha), \iota)$ be a tuple as in the statement, and we choose a $W$-linear isomorphism  $\varsigma':\DD(\BX')(W) \riso W\otimes_{\Zp}\Lambda^*$ that matches $(1\otimes s_\alpha)$ and $(\bft'_\alpha)$, which exists by assumption on $(\BX', (\bft'_\alpha), \iota)$. Then one can find a unique $b'\in G(K_0)$ depending on $\varsigma'$, such that 
$\varsigma'$ induces an $F$-equivariant isomorphism $\DD(\BX')(W) \riso \bM^\Lambda_{b'}$. 
Let $g\in G(K_0)$ be an element that defines the middle isomorphism below
\[\DD(\iota):\xymatrix@1{\DD(\BX')(W)[\ivtd p] \ar[r]^-{\cong}_-{\varsigma'} & \beta_{b'}(\Lambda^*[\ivtd p]) \ar[r]^-{\cong}_-{\rho(g)} & \beta_{b}(\Lambda^*[\ivtd p]) \ar[r]^-{\cong}_-{\varsigma\iv} & \DD(\BX)(W)[\ivtd p]}.\]
Note that the underlying $K_0$-vector space of  $\beta_{b'}(\Lambda^*[\ivtd p])$ and $\beta_{b}(\Lambda^*[\ivtd p])$ is $K_0\otimes_{\Zp}\Lambda^*$, so $\rho_{\Lambda^*[\frac{1}{p}]}(g)$ defines the middle isomorphism above. 
Since $\DD(\iota)$ is $F$-equivariant, we obtain $gb'\sigma = b\sig(g)\sigma$  (i.e., $b' = g\iv b\sig(g)$). Lastly, since the Hodge filtration for $\BX'$ is a $\set\mu$-filtration, we have $b'\in G(W)p^{\sig^*\mu\iv}G(W)$.

The choice of tensor-preserving $W$-linear isomorphism  $\varsigma':\DD(\BX')(W) \riso W\otimes_{\Zp}\Lambda^*$ is not canonical, but any other choice of $\varsigma'$ is of the form $\varsigma'\circ h$ for some $h\in G(W)$, which would replace $g$ by $gh$. 
One can easily check that map sending the isomorphism class of $(\BX', (\bft'_\alpha), \iota)$ to $gG(W)$ is well defined, and gives the desired inverse map.
\end{proof}

The following notion, which is the local analogue of Hodge-type Shimura data, turns out to provide a group-theoretic invariant associated to  $X^G(b)$ up to bijection given by (\ref{eqn:ADLEqCl}):

\begin{defnsub}\label{def:LocShData}
Let $G$ be a  reductive group over $\Zp$, and $b\in G(K_0)$. We associate to any $(G,b)$ a tuple $(G,[b],\set{\mu\iv})$, where $[b]$ is  the $\sig$-conjugacy class of $b$ in $G(K_0)$ and $\set{\mu\iv}$ is the unique $G(W)$-conjugacy class of cocharacters $\Gm\ra G_W$ such that $b\in G(W)p^{\sig^*\mu\iv}G(W)$. (The unique existence of such $\set\mu$ is by the Cartan decomposition.) 

If there is a faithful representation $\Lambda\in\prep(G)$ as in Definition~\ref{def:HodgeAssump} (i.e., there exists a $p$-divisible group $\BX^\Lambda_b$ as in Definition~\ref{def:HodgeAssump}), then we call the associated triple  $(G,[b],\set{\mu\iv})$ an \emph{(unramified)  local Shimura datum of Hodge type}. We take the obvious notion of morphism.
\end{defnsub}

To an unramified Hodge-type local Shimura datum, one can easily associate a local Shimura datum as defined in Rapoport and Viehmann by replacing $G$ with $G_{\Qp}$ \cite[Definition~5.1]{RapoportViehmann:LocShVar}. (Since $G$ is split over $W$, geometric conjugacy classes of cocharacters can be viewed as $G(W)$-conjugacy classes of cocharacters defined over $W$.) 

If $(G,[b],\set{\mu\iv})$ is an unramified Hodge-type local Shimura datum via $\Lambda$, then the inclusion $G\hra\GL(\Lambda)$ induces a morphism $(G,[b],\set{\mu\iv})\ra(\GL(\Lambda),[b],\set{\mu\iv})$ of unramified Hodge-type local Shimura data.

Let  $(G',[b'],\set{\mu^{\prime-1}})$ be another unramified local Shimura datum of Hodge type induced by $(G',b')$ and a faithful $G$-representation $\Lambda'$ (giving rise to  a $p$-divisible group $\BX^{\Lambda'}_{b'}$). Then the  product $(G\times G', [(b,b')],\set{(\mu\iv,\mu^{\prime-1})})$  is again an integral unramified Hodge-type local Shimura datum. (Indeed, we can associate the following $p$-divisible group $\BX^{\Lambda\times\Lambda'}_{(b,b')}\cong \BX^\Lambda_b\times \BX^{\Lambda'}_{b'}$.)

\begin{exasub}\label{exa:ShVar}
Assume that $G$ comes from a reductive group over $\Z_{(p)}$, which we also denote by $G$. Assume that there exists  a Hodge-type Shimura datum $(G_\Q,\UH)$.
%
Let $\KK_p:= G(\Zp)$ be the hyperspecial maximal subgroup of $G(\Qp)$. By construction, the integral canonical model $\sS_{\KK_p}(G_\Q,\mathfrak{H})$, when it exists, carries a ``universal'' abelian scheme depending on some auxiliary choices. See  \cite[\S1.4]{Vasiu:GoodRedn1} or \cite[\S2.3]{Kisin:IntModelAbType} for more details on the construction. Pick any point valued in $W:=W(\Fpbar)$, and let $\wt\BX$ be the  $p$-divisible group associated to the corresponding abelian scheme over $W$. Let $\Lambda:= T(\wt \BX)$ denote the (integral) Tate module. Then by construction there exist finitely many Galois-invariant tensors $(s_\alpha)\subset \Lambda^\otimes$ whose pointwise stabiliser is $G_{\Zp}$. By a conjecture of Milne (proved independently in \cite[Main~Theorem~1.2]{Vasiu:ys} and \cite[Proposition~1.3.4]{Kisin:IntModelAbType}) there exists a $W$-isomorphism
\[W\otimes_{\Zp}\Lambda^*\cong \DD(\wt \BX)(W)\]
which takes $(1\otimes s_\alpha)$ to the $F$-invariant tensors $(\bft_{\alpha})$ obtained from $(s_\alpha)$ via crystalline comparison isomorphism. Choosing such an isomorphism, we can extract $b\in G(K_0)$ from the matrix representation of $F$.

As $\bft_{\alpha}:\triv\ra(\bM^\Lambda_b)^\otimes$ are morphisms of ``strongly divisible modules'', we may apply Wintenberger's theory of canonical splitting \cite[Th\'eor\`eme~3.1.2]{Wintenberger:Splitting} and obtain a unique cocharacter $\mu:\Gm\ra G_W$ such that $\mu$ gives the Hodge filtration for $\wt \BX$ and $b\in G(W)p^\nu$ with $\nu = \sigma^*\mu\iv$;  \emph{cf.} Lemma~\ref{lem:X0}. Therefore, the triple $(G_{\Zp},[b],\set{\mu\iv})$ obtained from $\sS_{\KK_p}(G_\Q,\UH)(W)$ is an unramified Hodge-type local Shimura datum in the sense of Definition~\ref{def:HodgeAssump}.
Note that the geometric conjugacy class $\set\mu$ corresponds to the geometric conjugacy class associated to the Shimura datum $(G_\Q,\UH)$.
\end{exasub}

\section{Faltings's construction of universal deformation}\label{sec:FaltingsDefor}

In this section we review Faltings' explicit constructions of a ``universal'' deformation of $p$-divisible groups with crystalline Tate tensors (depending on some auxiliary choices). Furthermore, we also show that the underlying formal scheme of the deformation space of $p$-divisible groups with crystalline Tate tensors is independent of auxiliary choices (such as the choice of  $\Lambda$ and tensors $(s_\alpha)\subset \Lambda^\otimes$) and satisfies some functoriality properties (\emph{cf.} Proposition~\ref{prop:FunctDefor}).
We will crucially use this formal local construction to obtain the natural closed formal subscheme of a Rapoport-Zink space where some natural crystalline Tate tensors are defined. Most results in this section (except Propositions~\ref{prop:FunctDefor} and \ref{prop:LiftingTate}) can be found in \cite[\S7]{Faltings:IntegralCrysCohoVeryRamBase} and \cite[\S4]{Moonen:IntModels}.

Let $\kappa$ be an algebraically closed field of characteristic $p>2$, with $W:=W(\kappa)$. We consider a $p$-divisible group $\BX$ over $\kappa$. 
Recall that $\art{W}$ is the category of artin local $W$-algebras with residue field $\kappa$.
\begin{defn}
We define a functor $\Def_\BX:\art{W} \ra\Sets$ as follows: for any $R\in\art{W}$ we set
\[\Def_\BX(R):=\set{(X_{/R},\iota) :\  \iota: \BX\riso X_\kappa)}_{/\cong},\]
where $X$ is a $p$-divisible group over $R$, and an isomorphism of tuples  $(X,\iota)\riso (X',\iota)$ means an isomorphism $X\riso X'$ lifting $\iota'\circ\iota\iv$. We will often suppress the isomorphism $\iota:\BX\riso X_\kappa$, and write $X\in\Def_\BX(R)$.
\end{defn}

\subsection{Explicit construction in characteristic~$p$}\label{subsec:UnivDeforModp}
The functor $\Def_\BX$ can be prorepresented by the formal power series ring over $W$ with $d\tim d^\vee$ variables, where $d$ and $d^\vee$ are respectively the dimensions of $\BX$ and $\BX^\vee$; \emph{cf.}  \cite[Corollaire~4.8(i)]{Illusie:DeforBT}.
Faltings made such an identification via explicitly describing a ``universal Dieudonn\'e crystal'' when $p>2$; \emph{cf.} \cite[\S4.8]{Moonen:IntModels}, \cite[\S1.5]{Kisin:IntModelAbType}, which we recall now.

For a   $p$-divisible group $\BX$ over $\kappa$, we write $(\bM,F):=\DD(\BX)(W)$ for the contravariant Dieudonn\'e module. Choosing a lift $\wt\BX$  over $\Spf W$, we obtain a direct factor $\Fil^1_{\wt\BX}\subset \DD(\wt\BX)(W)\cong\bM$  from the Hodge filtration for $\wt\BX$. We also fix a splitting of this filtration; or equivalently, we choose a cocharacter $\mu:\Gm\ra \GL_{W}(\bM^*)$ which induces $\Fil^1_{\wt\BX}$ and have weights in $\{0,-1\}$. (Note our sign convention in Definition~\ref{def:grmu}.)   
Using the choice of  splitting,  we can define the ``opposite unipotent subgroup'' $U(\mu\iv)$ (i.e., $U(\mu\iv)$ is the unipotent radical of the  parabolic subgroup opposite to the stabiliser $P(\mu)$ of  $\Fil^1_{\wt\BX}$; note that we have $U(\mu\iv) = U_{\GL(\Lambda)}(\mu\iv)$ using the notation of  (\ref{eqn:parabolic})). Let $A _{\GL}$ be such that $\Spf A _{\GL} \cong \wh U(\mu\iv)$ is the formal completion of $U(\mu\iv)$ at the identity section. We choose an isomorphism $A_{\GL}\cong W[[u_{ij}]]$, and define a lift of Frobenius $\sig : A _{\GL} \ra A _{\GL} $ by $\sig(u_{ij}) = u_{ij}^p$.

Let $u_t\in \wh U(\mu\iv)(A _{\GL})$ be the tautological point and we define the following:
\[
\bM _{\GL}:= A _{\GL}\otimes_{W} \bM;\quad \Fil^1\bM _{\GL}:= A _{\GL}\otimes_{W} \Fil^1_{\wt\BX};
\quad F:=u_{t}\iv\circ(A _{\GL}\otimes F).
\]
More concretely, if we choose an isomorphism $\bM\cong \bM^\Lambda_b$ (with the notation of Definition~\ref{def:HodgeAssump} for $G=\GL(\Lambda)$) then the matrix representation of $F$ on $\bM _{\GL}\cong A _{\GL}\otimes_{\Zp}\Lambda^*$ is $u_t\iv b$.

As discussed in \cite[\S4.5]{Moonen:IntModels}, Faltings constructed a unique integrable connection $\nabla$ on $\bM _{\GL}$ which commutes with $F$. In particular, the tuple $(\bM _{\GL},F,\nabla)$ is a \emph{crystalline Dieudonn\'e module} in the sense of \cite[Definition~2.3.4]{dejong:crysdieubyformalrigid} and gives rise to a $p$-divisible group $\ol X _{\GL}$ over $\Spec A _{\GL}/(p)$ by
\cite[Main~Theorem~1]{dejong:crysdieubyformalrigid}.\footnote{Crystalline Dieudonn\'e modules over $A _{\GL}$ correspond to Dieudonn\'e crystals over $\Spec A _{\GL}/(p)$, not $\Spf A _{\GL}/(p)$. The distinction between $\Spec$ and $\Spf$ will be important, especially for verifying the effectivity property \S\ref{subsec:Artin}(\ref{subsec:Artin:Eff}).} Since  $\sigma^*(\Fil^1\bM _{\GL}/(p))$ is the kernel of $F$ on $\bM _{\GL}/(p)$, it follows that $\Fil^1\bM _{\GL}/(p)$ is the Hodge filtration of $\ol X _{\GL}$.

Faltings also showed that $\ol X _{\GL}$ is a universal mod~$p$ deformation of $\BX$ via the Kodaira-Spencer theory. Implicit in the proof is the following lemma, which compares the tangent space of $A _{\GL}/(p)$ and the deformations over $\kappa[\ep]/(\ep^2)$ given by the Grothendieck-Messing deformation theory. We give a proof of the lemma as we will need it later (\emph{cf.} Proposition~\ref{prop:LiftingTate}).

\begin{lemsub}\label{lem:TangentSp}
Let $B:=\kappa[\ep]/(\ep^2)$, and we give the square-zero PD structure on $\bb:=\ep B$.
For any $\gamma\in \wh U(\mu\iv)(B) =\Hom_{W}(A _{\GL},B)$, we set ${}^\gamma X:=\gamma^* \ol X _{\GL}$. Then ${}^\gamma X$ is the lift of $\BX$ which corresponds, via the Grothendieck-Messing deformation theory, to the lift of the Hodge filtration $\gamma(B\otimes_{\kappa}\Fil^1_{\BX})\subset B\otimes_{\kappa} \DD(\BX)(\kappa) = \DD(\BX)(B)$, where $\gamma(B\otimes_{\kappa}\Fil^1_{\BX})$ is the filtration translated by $\gamma\in\wh U(\mu\iv)(B)$.
\end{lemsub}
\begin{proof}
Let $\wt B:=W[\ep]/(\ep^2)$, and we give the square-zero PD structure on $\ep \wt B$, which is compatible with the standard PD structure on $(p)$. We also define a lift of Frobenius $\sig:\wt B\ra \wt B$ by letting $\sig(\ep):=\ep^p=0$. Since $\wt B$ is a (compatible) PD thickening of both $\kappa$ and $B$, there exists a natural Frobenius-equivariant isomorphism $\DD({}^\gamma X)(\wt B)\riso \DD(\BX)(\wt B)$ such that after reducing modulo~$p$ the Hodge filtration for ${}^\gamma X$ on the left hand side maps to the lift of the Hodge filtration for $\BX$ which corresponds to ${}^\gamma X$ via the Grothendieck-Messing deformation theory. Note that choosing a lift $\tilde\gamma\in\wh U(\mu\iv)(\wt B)$ of $\gamma$, we obtain natural isomorphisms
\begin{equation}\label{eqn:TangentSpVersal}
\DD({}^\gamma X)(\wt B)\cong \wt B\otimes_{\tilde\gamma,A _{\GL}}\bM _{\GL} \cong \wt B\otimes_{W}\DD(\BX)(W)
\end{equation}
and the crystalline Frobenius action on the left hand side corresponds to $\tilde \gamma\iv\circ(\wt B\otimes F)$ on the right hand side. (Recall that we used the \emph{inverse} of the tautological point to define Frobenius-action on $\bM _{\GL}$.)
Thus, the natural Frobenius-equivariant isomorphism $\DD({}^\gamma X)(\wt B)\riso \DD(\BX)(\wt B)$ can be translated as $g\in1+\ep\End_{W}(\bM)$ which makes the following diagram commute:
\[\xymatrix{
\DD({}^\gamma X)(\wt B) \ar[r]^-{\sim}_-{\text{(\ref{eqn:TangentSpVersal})}}&
\wt B\otimes_{W} \DD(\BX)(W) \ar[r]^-{\sim}_-{g}&
\wt B\otimes_{W}\DD(\BX)(W) &
\DD(\BX)(\wt B) \ar[l]_-{\sim}\\
& \wt B\otimes_{\sigma,W} \DD(\BX)(W)\ar[r]^-{\sim}_-{\sigma^*g} \ar[u]_-{\tilde\gamma\iv\circ(\wt B\otimes F)}&
\wt B\otimes_{\sigma,W}\DD(\BX)(W) \ar[u]^-{\wt B\otimes F}&
}.\]
By chasing the top row, the Hodge filtration of $^{\tilde\gamma}X$ in $\DD({}^{\tilde\gamma}X)(\wt B)$ maps to
 $g(\wt B\otimes_W \Fil^1_{\wt\BX})\subset\DD(\BX)(\wt B)$. So to prove the lemma, it suffices to prove that $g=\tilde\gamma$.
Indeed, note that $\sigma^*(g)=\id$ because $g\equiv\id\bmod{(\ep)}$ and $\sig(\ep)=0$, so
the commutative square above induces
\[
(g\cdot\tilde\gamma\iv)\circ(B\otimes\wt F)= (B\otimes\wt F).
\]
Since $(B\otimes\wt F)$ becomes an isomorphism after inverting $p$, we obtain $g=\tilde\gamma$ from the displayed equation.
\end{proof}

\subsection{Explicit construction: lifting Hodge filtration}
Since $\ol X _{\GL}$ is a universal mod~$p$ deformation of $\BX$ (\emph{cf.} Lemma~\ref{lem:TangentSp}), any lift  $X _{\GL}$ of $\ol X _{\GL}$ over $\Spf(A _{\GL},(p))$  is a universal deformation of $\BX$. Note that the Frobenius endomorphism and $\nabla$ on $\DD(X _{\GL})(A _{\GL})$ only depends on $\ol X _{\GL}$; i.e., we naturally have a Frobenius-equivariant horizontal isomorphism $\DD(X _{\GL})(A _{\GL} )\cong \bM _{\GL}$.

We now turn to the Hodge filtration for  $X _{\GL}$. 
Since the Hodge filtration for $\ol X _{\GL}$ is the image of $\Fil^1\bM _{\GL}$ in
$\bM _{\GL}/(p)$,
 the Grothendieck-Messing deformation theory gives us an $A _{\GL}$-lift $X _{\GL}$ of $\ol X _{\GL}$ with Hodge filtration $\Fil^1\bM _{\GL}$ if $p>2$.

To sum up, we have proved the following result:
\begin{thm}[Faltings]\label{thm:UnivDefGLn}
Assume that  $p>2$. Then the $p$-divisible group $X _{\GL}$ over $A _{\GL}$, corresponding to the filtered crystalline Dieudonn\'e module $(\bM _{\GL}, \Fil^1\bM _{\GL}, F,\nabla)$, is a universal deformation of $\BX$.
\end{thm}

\subsection{Deformation with Tate tensors}\label{subsec:FaltingsDeforConstr}
In this section, we return to the ``unramified  Hodge-type'' setting (\emph{cf.} Definition~\ref{def:HodgeAssump}): namely, we let $G\subset\GL(\Lambda)$ be a  reductive subgroup  over $\Zp$, and assume that there exists an isomorphism $\BX\cong \BX^\Lambda_b$ for some $b\in G(K_0)$. (In particular, we assume that there exists an isomorphism $\bM \cong \bM^\Lambda_b$ of Dieudonn\'e modules.) For finitely many tensors $(s_\alpha)\subset\Lambda^\otimes$ whose stabiliser is $G$,  we let  $(\bft_\alpha)\subset \bM^\otimes$ denote the image of $(1\otimes s_\alpha)\in(\bM^\Lambda_b)^\otimes$ via the isomorphism $\bM\cong\bM^\Lambda_b$. Note that $(\bft_\alpha)$ are $F$-invariant after inverting $p$.

Let us now recall Faltings' construction of the ``universal'' deformation of $( \BX, (\bft_\alpha))$ when $p>2$.  We can choose a $W$-lift $\wt\BX$   whose Hodge filtration is induced by some cocharacter $\mu:\Gm\ra G_W$; \emph{cf.} Remark~\ref{rmk:X0}. Then $(\bft_\alpha)$ lies in the $0$th filtration with respect to the Hodge filtration for $\wt\BX$ by Lemma~\ref{lem:muFil0}.

Let $U_G(\mu\iv):=U(-\mu)\cap G_W$ be the scheme-theoretic intersection. Since $\mu$ is valued in $G_W$, this definition coincides with (\ref{eqn:parabolic}), so it follows that $U_G(\mu\iv)$ is a smooth affine $W$-group with connected unipotent fibres (\emph{cf.} \cite[Proposition~2.1.8(3)]{ConradGabberPrasad:PRedGp2}). Let $A _G$ be the quotient of $A _{\GL}$ corresponding to the formal subgroup $\wh U_G(\mu\iv) \subset \wh U(\mu\iv)$. Then, $A _G$  is  a formal power series over $W$. We also choose``coordinates'' for $A _{\GL}$ so that  the kernel of $A _{\GL}\thra A _G$ is stable under $\sig$. (In particular, we get a lift of Frobenius $\sig$ on $A _G$ induced by $\sig$ on $A _{\GL}$.)

Let $X _{G}$ denote the pull-back of $X _{\GL}$ over $\Spf (A _G,(p))$. Then,  $\DD(X _{G})(A _G)$ with the Hodge filtration and Frobenius action corresponds to  the following quotient of  $(\bM _{\GL}, \Fil^1\bM _{\GL}, F)$:
\begin{equation}\label{eqn:UnivDefCrys}
\bM _{G}:= A _G\otimes_{W} \bM;\quad \Fil^1\bM _{G}:= A _G\otimes_{W} \Fil^1_{\wt\BX};
\quad F:=u_{t}\iv\circ(A _G\otimes F),
\end{equation}
where $(\bM,F):= \DD(\BX)(W)$ and $u_t\in U_G(\mu\iv)(A _G)$ is the tautological point. 

From this explicit description, it is immediate that the tensors $(1\otimes \bft_{\alpha})\subset (\bM _{G})^\otimes[\ivtd p]$ are $F$-invariant, and the pointwise stabiliser of $(1\otimes \bft_{\alpha})$ is $G_{A _G}\subset \GL(\bM)_{A _G}$. Since $\Fil^1\bM _{G}$ is  a $\set\mu$-filtration, the tensors $(1\otimes \bft_\alpha)$ lie in the $0$th filtration by Lemma~\ref{lem:muFil0}. It is also known that $(1\otimes \bft_\alpha)$ are horizontal (\emph{cf.} \cite[\S1.5.4]{Kisin:IntModelAbType}).  So for each $\alpha$ we obtain a morphism
\begin{equation}\label{eqn:tuniv}
\hat t^\univ_\alpha:\triv\ra\DD(X _{G})^\otimes
\end{equation}
of crystals over $\Spf(A _G,(p))$ such that $\hat t^\univ_\alpha(A _G)= 1\otimes \bft_\alpha$ on the $A _G$-sections by the usual dictionary \cite[Corollary~2.2.3]{dejong:crysdieubyformalrigid}. 

Now we can rephrase the theorem of Faltings  as follows (\emph{cf.} \cite[\S7]{Faltings:IntegralCrysCohoVeryRamBase}, \cite[Theorem~4.9]{Moonen:IntModels}):
\begin{thm}[Faltings]\label{thm:FaltingsDef}
Let  $A$ be either $W[[u_1,\cdots,u_N]]$ or $W[[u_1,\cdots,u_N]]/(p^m)$ for arbitrary $N\geqslant1$ and $m\geqslant1$, and choose a $p$-divisible group $X$ over $A$ which lifts $\BX$. Let $f: A _{\GL}\ra A$ be the morphism induced by $X$ (via $\Spf A _{\GL}\cong \Def_\BX$). Then $f$  factors through $A _G$ if and only if the map $\triv\ra\bM^\otimes$ sending $1$ to $\bft_\alpha$  has a (necessarily unique) lift to a morphism of crystals over $\Spf (A,(p))$
\[t_\alpha:\triv\ra\DD(X)^\otimes\]
which  is Frobenius-equivariant up to isogeny and has the property that its $A$-section $t_\alpha(A)\in\DD(X)(A)^\otimes$  lies in the $0$th filtration with respect to the Hodge filtration. If this holds, then we necessarily have  $f^*\hat t_\alpha^\univ = t_\alpha$.

Furthermore, the image of the closed immersion   $\Spf A _G\hra \Def_\BX$, given by $X _{G}$, is independent of the choice  of $(\bft_\alpha)$ and $\mu\in\set\mu$.
\end{thm}
\begin{proof}
The universal property for $A _G$ for test rings  of the form $A = W[[u_1,\cdots,u_N]]$ was proved by Faltings (\emph{cf.}  \cite[Theorem~4.9]{Moonen:IntModels}). When $A=W[[u_1,\cdots,u_N]]/(p^m)$, we choose a lift $\wt X$ over $\wt A:=W[[u_1,\cdots,u_N]]$ corresponding to a $\set\mu$-filtration (with respect to $(t_\alpha(\wt A))$) in $\DD(X)(\wt A)$ lifting the Hodge filtration of $X$. Then $t_\alpha$ also defines a unique morphism $\triv\ra\DD(\wt X)^\otimes$ (as it only depends on the mod~$p$ fibre of $\wt X$), and we obtain the desired claim by applying \cite[Theorem~4.9]{Moonen:IntModels} to $(\wt X,(t_\alpha))$.

The closed immersion $\Spf A _G\hra \Def_\BX$ is clearly independent of the choice of $(\bft_\alpha)$, and the independence claim on $\mu\in\set\mu$ follows from the universal property.
\end{proof}

\subsection{Functoriality of deformation spaces}\label{subsec:FunctDefor}
We identify the deformation functor $\Def_\BX$ with the formal spectrum of complete local noetherian ring which prorepresents $\Def_\BX$. 
\begin{defnsub}
Using the notation from Theorem~\ref{thm:FaltingsDef}, we define  $\Def_{\BX,G}$ to be the formally smooth closed formal subscheme of $\Def_\BX$ which classify deformations of $(\BX, (\bft_\alpha))$ over formal power series rings over $W$ or $W/(p^m)$ in the sense of Theorem~\ref{thm:FaltingsDef}. Note that for any cocharacter $\mu:\Gm\ra G_W$ giving rise to the Hodge filtration of $\BX$, we get an isomorphism $\wh U_G(\mu\iv)\riso \Def_{\BX,G}$ induced by $X _{G}$, and the closed formal subscheme $\Def_{\BX,G}\subset\Def_\BX$ is independent of the choice of $(t_\alpha)$.
\end{defnsub}

Note that for any isomorphism $(\BX,(\bft_\alpha))\riso (\BX^\Lambda_b,(1\otimes s_\alpha))$, we have a natural isomorphism $\Def_{\BX,G}\riso \Def_{\BX^\Lambda_b,G}$. Therefore, we fix an identification $\BX = \BX^\Lambda_b$ for the moment, and show that $\Def_{\BX,G}$ only depends on $(G,b)$, not on $\Lambda$, in a canonical way, and that it is functorial with respect to $(G,b)$. (See Remark~\ref{rmk:CosetRep} for the discussion on the choice of isomorphism $(\BX,(\bft_\alpha))\riso (\BX^\Lambda_b,(1\otimes s_\alpha))$.)

We consider another pair $(G',b')$ and $\Lambda'$ as in Definition~\ref{def:HodgeAssump}, and consider  $\BX':=\BX^{\Lambda'}_{b'}$. We also obtain a subfunctor $\Def_{\BX',G'}\subset \Def_{\BX'}$, such that for any cocharacter $\mu':\Gm\ra G'_W$ that induces the Hodge filtration of $\BX'$ we have a natural isomorphism $\wh U_{G'}(\mu'^{-1})\riso \Def_{\BX',G'}$ induced by $X'_{G'}$. We do not assume the existence of any morphism between $\BX$ and $\BX'$.

\begin{propsub}\label{prop:FunctDefor}
In the above setting, the natural monomorphism $\Def_\BX\times\Def_{\BX'} \ra \Def_{\BX\times\BX'}$, defined by taking the product of deformations, induces an isomorphism
\[
\Def_{\BX,G}\times\Def_{\BX',G'} \riso \Def_{\BX\times\BX', G\times G'}.
\]

Let $f:G_W\ra G'_W$ be a homomorphism over $W$ such that $f(b) = b'$. We choose a cocharacter $\mu:\Gm\ra G_W$ inducing the Hodge filtration of $\BX:=\BX^\Lambda_b$. (Since $f(b) = b'$, it follows that $f\circ\mu$ induces the Hodge filtration of $\BX'$.) Then the morphism $\Def_{\BX,G}\ra\Def_{\BX',G'}$, corresponding to $f|_{\wh U_G(\mu\iv)}:\wh U_G(\mu\iv) \ra \wh U_{G'}(f\circ\mu\iv)$,  depends only on $f$, not on the choice of $\mu$.
\end{propsub}
Before we begin the proof, let us make some remarks on the statement.
\begin{rmksub}
One can apply this proposition to the identity map on $(G,b)$ with different choice of $\Lambda$ and $\Lambda'$ to obtain a natural functorial isomorphism $\Def_{\BX^\Lambda_b,G}\riso\Def_{\BX^{\Lambda'}_b,G}$. Under this identification,  the  morphism $\Def_{\BX^\Lambda_b,G}\ra\Def_{\BX^{\Lambda'}_{b'},G'}$ associated to  $f:(G,b)\ra(G',b')$ depends only on $f$, not on the choice of $\Lambda$ and $\Lambda'$.
\end{rmksub}
\begin{proof}[Proof of Proposition~\ref{prop:FunctDefor}]
For the first assertion on the product decomposition, observe that we have  $X_{G\times G'} \cong X _{G}\times X'_{G'}$, which follows from the explicit description (\S\ref{subsec:FaltingsDeforConstr}). The claim now follows.

Let us now show that for a fixed choice of $\Lambda$ and $\Lambda'$ the map $\Def_{\BX,G}\ra\Def_{\BX',G'}$ induced by $f|_{\wh U_G(\mu\iv)}$ is independent of the choice of $\mu$. 
For this, we factor the map $\Def_{\BX,G}\ra\Def_{\BX',G'}$ as follows, and show that each arrow on the top row is independent of the choice of $\mu$:
\[\xymatrix@R=10pt{
\Def_{\BX,G}  \ar[d]^\cong& \ar[l]^-{\pr_1}\Def_{\BX\times\BX',G}  \ar@{^{(}->}[r] \ar[d]^\cong & \Def_{\BX\times\BX',G\times G'}  \ar[d]^\cong &\ar@/_1pc/[lll]_-{\pr_1} \ar[l]^-{\cong}\Def_{\BX,G}\times\Def_{\BX',G'} \ar@{->>}[r]^-{\pr_2} \ar[d]^\cong &\Def_{\BX',G'} \ar[d]^\cong\\
\wh U_G \ar@/_1.2pc/[rrrr]_-{f}&  \ar[l]_\id \wh U_G \ar@{^{(}->}[r]^-{(\id,f)} & \wh U_{G\times G'} & \ar[l]_-{\cong} \wh U_G\times\wh U_{G'}\ar@{->>}[r] &\wh U_{G'}
}.\]
Here, we write $\wh U_G:=\wh U_G(\mu\iv)$, etc.,  and  view $\Lambda\times\Lambda'$ as a faithful $G$-representation by $G \xra{(\id,f)} G\times G'$, so we have $\BX\times\BX' = \BX^{\Lambda\times\Lambda'}_b$.

Note that the third arrow on the top is the isomorphism defined by taking the product of deformations, which is independent of the choice of $\mu$ since the subspaces $\Def_{\BX,G}\subset\Def_\BX$ and $\Def_{\BX',G'}\subset\Def_{\BX'}$ are independent of the choice of cocharacters (\emph{cf.} Theorem~\ref{thm:FaltingsDef}). Similarly, it follows that the projection maps on the top row ($\pr_1$ and $\pr_2$) are independent of the choice of $\mu$, as they are the restrictions of the natural projections $\Def_\BX\times\Def_{\BX'} \thra \Def_\BX $ and $\Def_\BX\times\Def_{\BX'} \thra \Def_{\BX'} $  to a closed subspace independent of the choice of cocharacter.

The second arrow on the top row can be obtained from the universal property for $\Def_{\BX\times\BX',G\times G'}$, hence it is independent of the choice of $\mu$. This shows that the first arrow does not depend on the choice of $\mu$ as it can be obtained as the compositions of maps independent of $\mu$. Furthermore, it is an isomorphism as it corresponds to the identity map of $\wh U_G(\mu\iv)$.
Now, chasing the diagram, we conclude that the map $\Def_{\BX,G}\ra\Def_{\BX',G'}$ does not depend on the choice of $\mu$.
\end{proof}
\begin{rmksub}\label{rmk:CosetRep}
We remark on the effect of different choice of  isomorphism $\BX\cong\BX^\Lambda_b$ in Proposition~\ref{prop:FunctDefor}. For $g\in G(W)$, we write $b':=g\iv b\sig(g)$ and $\BX':=\BX^\Lambda_{b'}$. Then $g$ induces an $F$-equivariant isomorphism $g:\bM^\Lambda_{b'} \ra  \bM^\Lambda_b$, so we get an isomorphism
\[(\BX,(1\otimes s_\alpha))\riso (\BX',(1\otimes s_\alpha)),\]
preserving tensors, where $\BX:=\BX^\Lambda_b$ as before. This induces an isomorphism $\Def_{\BX,G}\riso \Def_{\BX',G}$. We want to give a group-theoretic interpretation of this isomorphism via the explicit construction of universal deformations with Tate tensors.

Choose a cocharacter $\mu:\Gm\ra G_W$ which induces the Hodge filtration of $\BX$. Then, $\mu':=g\iv \mu g$ induces the Hodge filtration of $\BX'$, so we have an isomorphism $\wh U_G(\mu'^{-1})=\Spf A'_G\riso \Def_{\BX',G}$ defined by the deformation $X'_G$ of $\BX'$.  Let $u_t\in \wh U_G(\mu\iv)(A _G)\subset G(A _G)$ and $u_t'\in \wh U_G(\mu'^{-1})(A'_G)\subset G(A'_G)$ be the tautological points.

We have an isomorphism $j_g:\wh U_G(\mu\iv) \riso \wh U_G(\mu'^{-1})$, defined by the conjugation by $g\iv$, and we have
$u_t' = g\iv (j_g^*u_t) g$ as an elements in $G(A'_G)$. So we have
\[
(u_t')\iv (g\iv b \sig(g)) = (g\iv j_g^*(u_t\iv) g)\cdot (g\iv b\sig(g)) = g\iv j_g^*(u_t \iv b) \sig(g).
\]
In particular, by identifying the underlying $A^{\mu'}_G$-modules of $\bM'_G$ and $j_g^*\bM _G$ with $A'_G\otimes_{\Zp}\Lambda^*$, the isomorphism $g:\bM'_G \riso j_g^*\bM _G$ is horizontal, filtered, and $F$-equivariant.
In short, we obtain the following commutative diagram of isomorphisms
\[\xymatrix{
\wh U_G(\mu\iv) \ar[r]^-{\sim}_-{j_g} \ar[d]_{X _G} &
\wh U_G(\mu'^{-1})  \ar[d]^{X'_G} \\
\Def_{\BX,G} \ar[r]^-{\sim}&\Def_{\BX',G}
},\]
where the bottom isomorphism is induced by $(\BX,(1\otimes s_\alpha))\riso (\BX',(1\otimes s_\alpha))$ which corresponds to $g:\bM^\Lambda_{b'} \ra  \bM^\Lambda_b$. 

Let us return to the setting of Proposition~\ref{prop:FunctDefor}, and consider a homomorphism $f:G\ra G'$. Then it follows without difficulty that for any $g\in G(W)$ the following  diagram commutes
\begin{equation}
\xymatrix{
\Def_{\BX^\Lambda_b,G} \ar[d]^\cong \ar[r] &\Def_{\BX^{\Lambda'}_{f(b)},G'}\ar[d]^\cong\\
\Def_{\BX^\Lambda_{g\iv b\sig(g)},G} \ar[r]&\Def_{\BX^{\Lambda'}_{f(g\iv b\sig(g))},G'}
},
\end{equation}
where the vertical isomorphisms are as constructed above associated to $g\in G(W)$ and $f(g)\in G'(W)$, respectively, and the horizontal arrows are associated to $f:G\ra G'$ via Proposition~\ref{prop:FunctDefor}.
\end{rmksub}

We now study deformation theory for points of $\Def_{\BX,G}$ valued in complete local noetherian rings. Let $R$ be a complete local noetherian $W/p^m$-algebra for some $m$, with residue field $\kappa$, and consider a $W$-morphism $f:\Spf R\ra\Def_{\BX,G}$. We set $(X_R,(t_\alpha)):=(f^*X _{G}, (f^*\hat t_\alpha^\univ))$.
Let $B\thra R$ be a square-zero thickening with finitely generated kernel $\bb$ (so that  $B$ is complete local noetherian ring as well), and give the square-zero PD structure on $\bb$; i.e., $a^{[i]}=0$ for any $i>1$ and $a\in\bb$.
Then we can define the $B$-sections $(t_\alpha(B))\subset\DD(X_R)(B)^\otimes$ as in Definition~\ref{def:sections}.
Let $\tilde f:\Spf B\ra\Def_{\BX,G}$ be a lift of  $f$, and set $(X_B,(\tilde t_\alpha)):=(\tilde f^*X _{G}, (\tilde f^*\hat t_\alpha^\univ))$. Then we have a natural isomorphism $\DD(X_B)(B)\cong \DD(X_R)(B)$, which matches $(\tilde t_\alpha(B))$ with $(t_\alpha(B))$.

\begin{prop}\label{prop:LiftingTate}
Assume that $p>2$.
Let $(X_R,(t_\alpha))$ and $B\thra R$ be as above. Then a $B$-lift $X_B$ of $X_R$ defines a $\Spf B$-point of $\Def_{\BX,G}$  if and only if the Hodge filtration
\[
\Fil^1_{X_B}\subset \DD(X_B)(B) \cong \DD(X_R)(B)
\]
is a $\set\mu$-filtration with respect to $(t_\alpha(B))$, where $\mu:\Gm\ra G_W$ is the cocharacter in the definition of $X _{G}$.
\end{prop}

Let us outline the basic strategy of the proof. 
The proposition when $B=\kappa[\ep]\thra R = \kappa$ can be deduced from Lemma~\ref{lem:TangentSp}. When $B\thra R$ is a small thickening (i.e., $\bb$ is of $B$-length~$1$) then we prove the proposition using the fact that the set of $B$-lifts of $f$ is a torsor under the reduced tangent space of $\Def_{\BX,G}$. The general case can be deduced by filtering $B\thra R$ into successive small thickenings.

Before beginning the proof of the proposition, let us review fibre products of rings.
Let $B\thra R$ be a small thickening of rings in $\art W$ with  kernel $\bb\subset B$. Let $\kappa[\bb]$ denote the $\kappa$-algebra whose underlying $\kappa$-vector space is $\kappa\oplus\bb$, such that $\bb$ is the augmentation ideal. If we pick a generator $\ep\in\bb$ then we have $\kappa[\bb] = \kappa[\ep]/\ep^2$.

We have the following isomorphism
\begin{equation}\label{eqn:SmallThickening}
B\times_\kappa \kappa[\bb] \riso B\times_R B =:B' ;\quad(a, \bar a + a')\mapsto (a, a+a') ,
\end{equation}
where $a\in B$, $a'\in \bb$, and $\bar a\in \kappa$ is the image of $a$. The inverse is given by $(a,a')\mapsto (a, \bar a + (a'-a))$.

Let $\cF:\art W\ra\Sets$ be a pro-representable functor. (For example, $\cF=\Def_\BX$ or $\cF = \Def_{\BX,G}$.) Then we have a natural bijection
\[\cF(B\times_R B')\riso \cF(B)\times_{\cF(R)}\cF(B')\]
for any $B,B'\thra R$. So we obtain from (\ref{eqn:SmallThickening})   a natural bijection
\begin{equation}\label{eqn:SmallThAction}
\cF(B)\times \cF(\kappa[\bb]) \riso \cF(B)\times_{\cF(R)}\cF(B),
\end{equation}
which defines an $\cF(\kappa[\bb])$-action on $\cF(B)$ (where $\cF(\kappa[\bb])$ has a natural structure of a $\kappa$-vector space by \cite[Lemma~2.10]{Schlessinger:FunctArtRing}), and makes the set of $\tilde f\in\cF(B)$ lifting a fixed $f\in\cF(R)$ into a  $\cF(\kappa[\bb])$-torsor.

Let us consider the case when $\cF=\Def_\BX$.
For any $R\in\art W$, we set $\bM_{R}:=R\otimes_W\bM$ and  $\Fil^1\bM_{R}:=R\otimes_W\Fil^1_{\wt\BX} \subset \bM_R$.
Via the Grothendieck-Messing deformation theory, we have a natural bijection
\begin{equation}\label{eqn:GrothMessingTangent}
\wh U(\mu\iv) (\kappa[\bb])\cong\Def_\BX(\kappa[\bb]).
\end{equation}
 Indeed, we associate to $\gamma\in \wh U(\mu\iv)(\kappa[\bb])$ the lift ${}^{\gamma}X\in\Def_\BX(\kappa[\bb])$ which corresponds to the filtration $\gamma(\Fil^1\bM_{\kappa[\bb]})$; \emph{cf.} Lemma~\ref{lem:TangentSp}.

Now we give the square-zero PD structure on $\bb$. (We still assume that $p>2$.)
We define the $\wh U(\mu\iv)(\kappa[\bb])$-action on $\Def_\BX(B)$ to be the one induced from the natural $\wh U(\mu\iv)(\kappa[\bb])$-action on the Hodge filtration of $X_B$ via the Grothendieck-Messing deformation theory.

\begin{lemsub}\label{lem:GrothMessingTangent}
In the above setting setting, the actions of $\wh U(\mu\iv)(\kappa[\bb])$ and $\Def_\BX(\kappa[\bb])$ on $\Def_\BX(B)$, which are defined above, coincide via the isomorphism (\ref{eqn:GrothMessingTangent}).
\end{lemsub}
\begin{proof}
Let us give the square-zero PD structure on the kernel of  $B':=B\times_R B \thra R$ so that  both projections $B'\rightrightarrows B$ are PD morphisms.
Let ${}^{\gamma}X\in\Def_{\BX}(\kappa[\bb])$ be the deformation corresponding to $\gamma\in \wh U(\mu\iv)(\kappa[\bb])$. Then the action of ${}^{\gamma}X$ maps $ X_B$ to the  pull-back of $ X_B\times_{ \BX}{}^{\gamma}X\in\Def_{ \BX}(B')$ via the second projection $B'\thra B$. Meanwhile, $ X_B\times_{ \BX}{}^{\gamma}X$ corresponds to the filtration
\begin{equation}\label{eqn:SmallThActionFil}
\Fil^1_{ X_B}\times_{\Fil^1\bM_{\kappa}} \big[\gamma(\Fil^1\bM_{\kappa[\bb]})\big]
\subset \DD(X_R)(B) \times_{\bM_{\kappa}}\bM_{\kappa[\bb]} \cong \DD(X_R)(B'),
\end{equation}
using the notation as above. Now from the isomorphism (\ref{eqn:SmallThickening}) it  follows that the image of the filtration (\ref{eqn:SmallThActionFil}) under the second projection is  $\gamma\Fil^1_{ X_B}$.
\end{proof}

\begin{proof}[Proof of Proposition~\ref{prop:LiftingTate}]
If $B=\kappa[\ep]$ then the proposition is clear from Lemma~\ref{lem:TangentSp}. Now assume that $B\thra R$ is a small thickening (i.e., the $B$-length of $\bb$ is $1$). We let $f:\Spec R\ra\Def_{\BX,G}$ denote the map induced by $X_R$. 
For any lift $\tilde f:\Spec B\ra\Def_{\BX,G}$ of $f$ (with $X_B:=\tilde f^*X _{G}$) the Hodge filtration
\[
 \Fil^1_{X_B}\subset \DD(X_B)(B)\cong \DD(X_R)(B)
\]
is a $\set\mu$-filtration with respect to $(t_\alpha(B))$; indeed, $ \Fil^1_{X_B}$ and $(t_\alpha(B))$ are respectively the images of $\Fil^1\bM _{G}$ and $(\hat t_\alpha^\univ(A _G))$. (Note that we have a natural isomorphism $\DD(X_B)(B) = \tilde f^*\bM _{G}$.)
So  we obtain a map
\[
\set{\tilde f\in\Def_{\BX,G}(B)\text{ lifting }f} \ra \set{\{\mu\}\text{-filtrations in }\DD(X_R)(B)\text{ lifting }\Fil^1_{X_R}}
,\]
sending $\tilde f$ to $\Fil^1_{\tilde f^*X _{G}}$. By Lemma~\ref{lem:GrothMessingTangent}, this map is a morphism of $\wh U_G(\mu\iv)(\kappa[\bb])$-torsors, so it has to be a bijection. This proves the proposition when $B\thra R$ is a small thickening.

Now let $B\thra R$ be any square-zero thickening with $B\in\art W$, and consider an quotient $R'$ of $B$ which surjects onto $R$. Then $\bb':=\ker (B\thra R')$ is a square-zero ideal, so by giving the ``square-zero PD structure'', $\bb'$ is a PD subideal of $\bb$. We fix a lift $f'\in\Def_{\BX,G}(R')$ of $f$ and set $(X_{R'} , (t'_\alpha)):=(f^{\prime*}X _{G,b},(f^{\prime*}\hat t_\alpha^\univ))$.

Since $\bb'$ is a (nilpotent) PD subideal, we have a natural isomorphism $\DD(X_R)(B) \cong \DD(X_{R'})(B)$, which matches $(t_\alpha(B))$ and $(t_\alpha'(B))$.\footnote{This can be seen as follows. For any maps $\tilde f\in\Def_{\BX,G}(B)$ lifting $f'$ (and hence, $f$) and $(X_B,(\tilde t_\alpha))$ pulling back the universal objects, we have natural isomorphisms $\DD(X_R)(B) \cong \DD(X_{R'})(B)\cong \DD(X_B)(B)$ matching $(t_\alpha(B))$, $(t_\alpha'(B))$, and $(\tilde t_\alpha(B))$.} Therefore, the proposition for $B\thra R$ with $B\in\art W$ is obtained by filtering it with successive small thickenings.

When $B$ is a complete local noetherian $W/p^m$-algebra (for some $m$) with residue field $\kappa$, we  filter $B\thra R$ by square-zero thickenings of artin local rings $B_n\thra R_n$ (e.g., $B_n:=B/\m_B^n$ and $R_n:=B_n\otimes_BR$), and the proposition follows from applying the artin local case of the proposition to the lift $X_{B_n}$ of $X_{R_n}$ for any $n$.
\end{proof}

\section{Moduli of $p$-divisible groups with Tate tensors}\label{sec:RZ}
In this section, we  state the main results on the Hodge-type analogue $\RZ^\Lambda_{G,b}$ of Rapoport-Zink spaces (Theorem~\ref{thm:RZHType}). This construction recovers EL and PEL Rapoport-Zink spaces if $(G,b)$ is associated to an unramified EL and PEL datum   (\S\ref{subsec:ELnPEL}).

Let $\kappa$ be an algebraically closed field of characteristic~$p$, and set $W:=W(\kappa)$ and $K_0:=\Frac W$. The most interesting case is when $\kappa = \Fpbar$.
Let $ \BX$ be a $p$-divisible group  over $\kappa$. 

We begin with the  review of the formal moduli schemes classifying $p$-divisible groups up to quasi-isogeny  \cite[Ch.II]{RapoportZink:RZspace}, which is the starting point of the construction of $\RZ^\Lambda_{G,b}$.

\begin{defn}[{\cite[Definition~2.15]{RapoportZink:RZspace}}]\label{def:RZ}
Let $\RZ_\BX:\Nilp_W\ra\Sets$ be a covariant functor defined as follows: for any $R\in\Nilp_W$, $\RZ_\BX(R)$  is the set of isomorphism classes of $(X, \iota)$, where $X$ is a $p$-divisible group over $R$ and $\iota: \BX_{R/p} \dra X_{R/p}$ is a quasi-isogeny. (We take the obvious notion of isomorphism of $(X,\iota)$.)

For $h\in \Z$, let $\RZ_\BX(h):\Nilp_W\ra\Sets$ be a subfunctor of $\RZ_\BX$ defined by requiring that the quasi-isogeny $\iota$ has height $h$.
%
\end{defn}

\begin{rmksub}\label{rmk:RZ}
Let $\wt \BX$ be any $p$-divisible group over $W$ which lifts $\BX:=\BX^\Lambda_b$. Then  for any $(X,\iota)\in\RZ_\BX(R)$, $\iota$ uniquely lifts to
\begin{equation}
\tilde\iota:\wt\BX_R\dra X
\end{equation}
by rigidity of quasi-isogenies; \emph{cf.} (\ref{eqn:DrinfeldRigidity}). (Recall that $pR$ is a nilpotent ideal killed by some power of $p$.) So we may regard $\RZ_\BX(R)$ as the set of isomorphism classes of $(X,\tilde\iota)$ where $\tilde\iota:\wt\BX_R\dra X$ is a quasi-isogeny.
\end{rmksub}

\begin{thm}[Rapoport, Zink]
The functor $\RZ_\BX$ can be represented by a separated formal scheme which is locally formally  of finite type (\emph{cf.} Definition~\ref{def:ForFType}) and formally smooth over $W$. We also denote by $\RZ_\BX$ the  formal scheme representing $\RZ_\BX$. For $h\in\Z$, the subfunctor $\RZ_\BX(h)$ can be represented by an open and closed formal subscheme (also denoted by $\RZ_\BX(h)$). Furthermore, any irreducible component of the underlying reduced scheme $(\RZ_\BX)_{\red}$ of $\RZ_\BX$ is projective.
\end{thm}
\begin{proof}
The representability of $\RZ_\BX$ is proved in \cite[Theorem~2.16]{RapoportZink:RZspace}. It is clear that $\RZ_\BX(h)$ is an open and closed formal subscheme of $\RZ_\BX$. 
The assertion on the irreducible components of $(\RZ_\BX)_{\red}$ is proved in \cite[Proposition~2.32]{RapoportZink:RZspace}. 
\end{proof}

For each $h\in\Z$, we can write $\RZ_\BX(h)$ as the direct limit of subfunctors representable by closed  schemes, as follows. Let $(X_{\RZ_\BX(h)},\iota_{\RZ_\BX(h)})$ denote the universal $p$-divisible group up to height-$h$ quasi-isogeny.

\begin{defn}\label{def:RZab}
We fix a $W$-lift $\wt \BX$ of $\BX$, and view $\RZ_\BX(R)$ as a set of $\set{(X,\tilde\iota:\wt\BX_R\dra X)}/\cong$ (\emph{cf.} Remark~\ref{rmk:RZ}).
Then for any $m>0$ and $n\in \Z$, we define $\RZ_\BX(h)^{m,n}\subset \RZ_\BX(h)$ to be the subfunctor defined as follows:  for any $W/p^m$-algebra $R$, $\RZ_\BX(h)^{m,n}(R)$ is the set of $(X,\tilde\iota)\in\RZ_\BX(h)(R)$ such that $p^n\tilde\iota:\wt\BX_R\dra X$ is an isogeny of $p$-divisible group. (We set $\RZ_\BX(h)^{m,n}(R)=\emptyset$ if $p^mR\ne0$.)
 \end{defn}

As explained in \cite[\S2.22]{RapoportZink:RZspace}, $\RZ_\BX(h)^{m,n}$ can be realised as a closed subscheme of certain grassmannian, hence it can be represented by a projective scheme over $W/p^m$. And it is a closed subscheme of $\RZ_\BX(h)$ such that $\RZ_\BX(h) = \varinjlim_{m,n}\RZ_\BX(h)^{m,n}$.  

We now recall the deformation-theoretic interpretation of the completed local ring of $\RZ_\BX$.
\begin{lemsub}\label{lem:LocMod}
For $x=(X_x,\iota_x) \in\RZ_\BX(\kappa)$, the formal completion $(\RZ_\BX)\wh{_x}$ at $x$  represents the functor $\Def_{X_x}$, using the notation above.
\end{lemsub}
\begin{proof}
We have a morphism $(\RZ_\BX)\wh{_x}\ra\Def_{X_x}$ given by forgetting the quasi-isogeny.
By rigidity of quasi-isogenies (\ref{eqn:DrinfeldRigidity}), we have a natural morphism of functors $\Def_{X_x} \ra (\RZ_\BX)\wh{_x}$ defined by sending $X\in\Def_{X_x}(R)$ to $(X,\iota)\in\RZ_\BX(R)$ where $\iota: \BX_{R/p}\ra X_{R/p}$ is the unique quasi-isogeny that lifts $\iota_x$. It also follows from rigidity that the composition $(\RZ_\BX)\wh{_x} \ra \Def_{X_x}\ra (\RZ_\BX)\wh{_x}$ is an identity morphism. To finish the proof, note that  both are representable by formal power series rings over $W$ with same dimension. (The dimension of $(\RZ_\BX)\wh{_x}$ can be obtained from \cite[Proposition~3.33]{RapoportZink:RZspace}.)
\end{proof}

\subsection{}\label{subsec:RZHodge}
Let us return to the setting of \S\ref{subsec:FilGIsoc}. Let $(G,b)$ and $\Lambda$ be as in Definition~\ref{def:HodgeAssump}, so we have a $p$-divisible group $\BX:=\BX_b^\Lambda$ with the contravariant Dieudonn\'e module $\bM^\Lambda_b$. We choose finitely many tensors $(s_\alpha)\subset \Lambda^\otimes$ which defines $G$ as a subgroup of $\GL(\Lambda)$; \emph{cf.} Proposition~\ref{prop:Chevalley}.
We also choose a $W$-lift $\wt \BX(=\wt\BX^\Lambda_b)$ of $ \BX$ as in Remark~\ref{rmk:X0}.
\begin{defn}\label{def:t}
Let $(X,\tilde\iota)\in \RZ_\BX(R)$ with $R\in\Nilp_W$, where $\tilde\iota:\wt\BX_R\dra X$ is a quasi-isogeny; \emph{cf.} Remark~\ref{rmk:RZ}. We define $s_{\alpha,\DD}:\triv \ra \DD(X)^\otimes[\ivtd p]$ to be the composition
\[ \xymatrix@1{\triv \ar[rr]^-{1\mapsto 1\otimes s_\alpha} && \DD(\wt\BX_R)^\otimes[\ivtd p] \ar[rr]_-{\sim}^-{\DD(\tilde\iota)\iv}&& \DD(X)^\otimes[\ivtd p]}, \]
where the first morphism is the pull-back of the map $\triv\ra\DD(\wt \BX)^\otimes$ which induces  $1\mapsto1\otimes s_\alpha$ on the $W$-sections.
\end{defn}
Note that $s_{\alpha,\DD}:\triv \ra \DD(X)^\otimes[\ivtd p]$ only depends on $(X,\iota)\in \RZ_\BX(R)$ but not on the choice of $\wt \BX$. Indeed, the morphism
\[ \xymatrix@1{\triv \ar[rr]^-{1\mapsto 1\otimes s_\alpha} && \DD( \BX_{R/p})^\otimes[\ivtd p] \ar[rr]_-{\sim}^-{\DD(\iota)\iv}&& \DD(X_{R/p})^\otimes[\ivtd p]} \]
uniquely determines $s_{\alpha,\DD}$.

It is clear that each  $s_{\alpha,\DD}$ is Frobenius-equivariant, but it may \emph{not} come from a morphism of crystals $t_\alpha:\triv\ra\DD(X)^\otimes$. Even if it does, such a morphism $t_\alpha$ of (integral) crystals  may not be uniquely determined by $s_{\alpha,\DD}$ due to the existence of non-zero $p$-torsion morphism when the base ring $R$ is not nice enough. (See Appendix in \cite{MR700767} for such an example.) To deal with this problem, we work only with a morphism of (integral) crystals $t_\alpha:\triv\ra\DD(X)^\otimes$ giving rise to $s_{\alpha,\DD}$ which is ``liftable'' in some suitable sense that we now describe.

Let $\Nilp_W^\sm$ denote the category of formally smooth formally finitely generated $W/p^m$-algebras $A$ for some $m$, endowed with the $J$-adic topology where  $J$ is the Jacobson radical of $A$. For example, $ (W/p^m)[[u_1,\cdots,u_m]][v_1,\cdots,v_n]\in\Nilp_W^\sm$, and we give the $(u_1,\cdots,u_m)$-adic topology. (We emphasise here that $R\in\Nilp_W$ is always given a discrete topology, while $A\in \Nilp_W^\sm$ is equipped with the $J$-adic topology.)

\begin{rmksub}\label{rmk:RZGsm}
Let $A$ be a formally smooth  formally finitely generated algebra over either $W$ or $W/p^m$. In this remark, we give a convenient moduli interpretation of the set of $A$-points of $\RZ_\BX$. Recall that we have
\[\Hom_W(\Spf A, \RZ_\BX)=\varinjlim_n\RZ_\BX(A/J^n),\]
so we may view $\Hom_W(\Spf A,\RZ_\BX)$ as the set of isomorphism classes of $(X,\iota)$, where $X$ is a $p$-divisible group over $A$ and $\iota$ is a quasi-isogeny defined over $\Spf A/p$.  By rigidity of quasi-isogenies (\ref{eqn:DrinfeldRigidity}), giving such $\iota$ is equivalent to giving an quasi-isogeny $\iota_{A/J}:\BX_{A/J}\dra X_{A/J}$ for some ideal of definition $J\subset A$ containing $p$. (Recall that $J/J^n$ is nilpotent and killed by some power of $p$.)

We remind that there is a natural equivalence of categories between the categories of $p$-divisible groups over $\Spec A$ and $\Spf A$, so we may view the $p$-divisible group $X$ either over $\Spec A$ or $\Spf A$. On the other hand, the quasi-isogeny $\iota$ over $\Spf A/p$ may \emph{not} be defined over $\Spec A/p$, and we only require $\iota$ to be defined over $\Spf A/p$, not over $\Spec A/p$. And we have just shown that quasi-isogenies over $\Spf A/p$ are uniquely determined by their restriction to $\Spec A/J$ for some ideal of definition $J$ containing $p$.
\end{rmksub}

Recall that there exists a $p$-adically separated and complete formally smooth $W$-algebra $\wt A$ which lifts $A$; indeed, we apply \cite[Lemma~1.3.3]{dejong:crysdieubyformalrigid} to obtain a $p$-adic $W$-lift $\wt A$ of $A/p$, which is formally smooth over $W$ by construction, so we may view $\wt A$ as a lift of $A$. By formal smoothness over $W$, any two such lifts are (not necessarily canonically) isomorphic. 

\begin{defn}\label{def:RZG}
For any $A\in\Nilp_W^\sm$, we define
\[\RZ\com{s_\alpha}_{\BX,G}(A)\subset\Hom_W(\Spf A, \RZ_\BX)
\]
as follows: Let $f:\Spf A \ra \RZ_\BX$ be a morphism, and $X$ a $p$-divisible group over $\Spec A$ which pulls back to  $f^*X_{\RZ_\BX}$ over $\Spf A$. Then we have $f\in \RZ\com{s_\alpha}_{\BX,G}(A)$ if and only if there exist morphisms of integral crystals (over $\Spec A$)
\[
t_\alpha:\triv\ra\DD(X)^\otimes
\]
such that
\begin{enumerate}
\item\label{def:RZG:qisog}
For some ideal of definition $J$ of $A$ (or equivalently by Lemma~\ref{Lem:Dwork}, for any ideal of definition $J$),
the pull-back of $t_\alpha$ over $A/J$ induces the map of isocrystals $s_{\alpha,\DD}:\triv\ra\DD(X_{A/J})^\otimes[\ivtd p]$.
\item \label{def:RZG:P}
We choose a formally smooth $p$-adic $W$-lift $\wt A$ of $A$, endowed with the standard PD structure on $\ker (\wt A\thra A)=p^m\wt A$ for some $m$. Let $(t_\alpha(\wt A))$ denote the $\wt A$-section of $(t_\alpha)$ (\emph{cf.} Definition~\ref{def:sections}). Then the $\wt A$-scheme
\[\cP_{\wt A}:=\nf\Isom_{\wt A}\big[(\DD(X)(\wt A), (t_\alpha(\wt A))],[\wt A\otimes_{\Zp}\Lambda^*,(1\otimes s_\alpha)]\big),\]
classifying isomorphisms matching $(t_\alpha(\wt A))$ and $(1\otimes s_\alpha)$, is a $G$-torsor. (\emph{Cf.} (\ref{eqn:P}).) Note that this condition is independent of the choice of $\wt A$, since any choices of $\wt A$ are isomorphic.
\item\label{def:RZG:Kottwitz}
The Hodge filtration $\Fil^1_X\subset \DD(X)(A)$ is a $\set\mu$-filtration with respect to  $(t_\alpha(A))\subset \DD(X)(A)^\otimes$, where $\set\mu$ is the unique $G(W)$-conjugacy class of cocharacters such that $b\in G(W)p^{\sig^*\mu\iv}G(W)$.
\end{enumerate}
We thus obtain a functor $\RZ\com{s_\alpha}_{\BX,G}:\Nilp_W^\sm\ra\Sets$. For $(X,\iota)\in\RZ\com{s_\alpha}_{\BX,G}(A)$, we call $(t_\alpha)$ as above \emph{crystalline Tate tensors} or \emph{Tate tensors} on $X$.

If $\wt A$ is formally smooth and formally finitely generated over $W$, we write
\[\RZ\com{s_\alpha}_{\BX,G}(\wt A):=\varprojlim_m\RZ\com{s_\alpha}_{\BX,G}(\wt A/p^m) \subset \Hom_W(\Spf \wt A,\RZ_\BX) = \varinjlim_n\RZ_\BX(\wt A/\wt J^n),\]
where $\wt J$ is the Jacobson radical of $\wt A$. When $\tilde f\in \RZ\com{s_\alpha}_{\BX,G}(A)$ and $\wt X$ is the $p$-divisible group over $\wt A$ given by $\tilde f^*X_{\RZ_\BX}$, then we have a morphism $\tilde t_\alpha:\triv\to\DD(\wt X)^\otimes$ of crystals over $\Spf (\wt A,(p))$ satisfying the same conditions (\ref{def:RZG:qisog})--(\ref{def:RZG:Kottwitz}), where $(A,J,(t_\alpha))$ is replaced by $(\wt A,\wt J,(\tilde t_\alpha))$ everywhere.\footnote{Recall that we define $\DD(\wt X)$ over $\Spf (\wt A,(p))$ to be the projective system $\{\DD(\wt X_{\wt A/p^m})\}$. On the other hand, note that for any $m\geqslant1$  we have $\DD(\wt X_{\wt A/p^m})(\wt A) = \DD(\wt X_{\wt A/p})(\wt A)$ with extra structure, so we indeed have $\DD(\wt X)(\wt A) = \DD(\wt X_{\wt A/p})(\wt A)$ and all the objects appearing in the definition of $\tilde f\in\RZ^{(s_\alpha)}_{\BX,G}$ except the Hodge filtration $\Fil^1_{\wt X}\subset \DD(\wt X)(\wt A)$  depends only on $\wt X_{\wt A/p}$. }
\end{defn}
We give more motivations for the definition in Remark~\ref{rmk:RZGjustification}.
In the meantime, let us  emphasise that $(t_\alpha)$ are required to be morphisms of crystals over $\Spec A$ (\emph{not} just over $\Spf A$), while we work with the quasi-isogeny $\iota$ defined only over $\Spf A$ (or equivalently by Remark~\ref{rmk:RZGsm}, the quasi-isogeny defined over $\Spec A/J$).  In particular, we do \emph{not} claim that an arbitrary quasi-isogeny $\iota$ (defined only over $\Spf A$) induces an $F$-equivariant map $\triv\to \DD(X)^\otimes[\ivtd p]$ over $\Spec A$.
We will still show later (in Lemma~\ref{lem:UniqTensors}) that \emph{if} the crystalline Tate tensors $(t_\alpha)$ in Definition~\ref{def:RZG} \emph{do} exist (which is an assumption), then $(t_\alpha)$ are  uniquely determined by $f\in \Hom_W(\Spf A,\RZ_\BX)$; in other words, $(t_\alpha)$ are uniquely determined by the quasi-isogeny $\iota:\BX_{A/J}\to X_{A/J}$ defined over $\Spec A/J$.  

\begin{exasub}\label{exa:RZ}
If $G = \GL(\Lambda)$, we may choose $(s_\alpha)$ to be the empty set. Then we claim that $\RZ_\BX$ represents $\RZ^\emptyset_{\BX,G}$.  Indeed, we only need to check Definition~\ref{def:RZG}(\ref{def:RZG:Kottwitz}), which is clear since for any $(X,\iota)\in\RZ_\BX(R)$ the dimension of $X$ is constant on $R$ and consistent with $\set\mu$ associated to $(G,b)$; \emph{cf.} Remark~\ref{rmk:muFilGLn}.

In \S\ref{subsec:ELnPEL}, we will recall how to attach $(G,b)$ and $(s_\alpha)$ to an EL (respectively, PEL) datum, and compare $\RZ\com{s_\alpha}_{\BX,G}$ with the moduli functor considered by Rapoport and Zink. Roughly speaking, we will verify the following in the (P)EL case (\emph{cf.} Proposition~\ref{prop:ELnPEL}):
\begin{enumerate}
\item  The existence of the crystalline Tate tensors $(t_\alpha)$ satisfying Definition~\ref{def:RZG}(\ref{def:RZG:P}) corresponds to the existence of an action of the given semi-simple $\Qp$-algebra $B$ on $X$ up to isogeny (and the quasi-polarisation  with certain properties in the PEL case).
\item Definition~\ref{def:RZG}(\ref{def:RZG:qisog}) means that the quasi-isogeny $\iota$ is ``$B$-linear'' (and commutes with the quasi-polarisation in the PEL case).
\item Definition~\ref{def:RZG}(\ref{def:RZG:Kottwitz}) corresponds to the ``Kottwitz determinant condition''.
\end{enumerate}
Note that the tensors $(t_\alpha)$ coming from endomorphisms or quasi-polarisations over $\Spf A$ automatically give Tate tensors of crystals defined over $\Spec A$; \emph{cf.} \cite[Proposition~2.4.8]{dejong:crysdieubyformalrigid}.
\end{exasub}

\begin{rmksub}\label{rmk:RZGjustification}
Let us now explain the motivations and intuitions for the conditions in Definition~\ref{def:RZG} defining $\RZ\com{s_\alpha}_{\BX,G}(A)\subset\Hom_W(\Spf A, \RZ_\BX)$ for $A\in\Nilp_W^\sm$.
We consider $f:\Spf A\ra\RZ_{\BX}$, which corresponds to $(X,\iota)$ as in Remark~\ref{rmk:RZGsm}.

Let us begin with the existence of $(t_\alpha)$ satisfying Definition~\ref{def:RZG}(\ref{def:RZG:qisog}). Recall that $t_\alpha:\triv\ra\DD(X)^\otimes$ is a morphism of crystals over $\Spec A$, not over $\Spf A$. Since the quasi-isogeny $\iota$ may not be defined over $\Spec A$, we may not have an isomorphism between $\DD(X)[\ivtd p]$ and $\DD(\BX_{A/p})[\ivtd p]$. In particular, it is unclear whether for any $f:\Spf A\ra\RZ_{\BX}$ corresponding to $(X,\iota)$, the tensors $(s_{\alpha,\DD})$ induce Frobenius-equivariant morphisms $\triv\ra \DD(X)^\otimes[\ivtd p]$.\footnote{If the tensor $s_{\alpha,\DD}$ corresponds to a quasi-polarisation or an endomorphism of $\BX$, then any $f:\Spf A\ra\RZ_{\BX}$ corresponding to $(X,\iota)$ gives rise to $t_\alpha:\triv\ra\DD(X)^\otimes[\ivtd p]$ defined over $\Spec A$ by \cite[Proposition~2.4.8]{dejong:crysdieubyformalrigid}. On the other hand, the author does not know whether such a statement should hold for more general tensors $(s_{\alpha,\DD})$.} Therefore,  the existence of $(t_\alpha)$ satisfying Definition~\ref{def:RZG}(\ref{def:RZG:qisog}) can be seen as \emph{integrality} and \emph{algebraisability} of $(s_{\alpha,\DD})$ in the following sense: for any ideal of definition $J$ containing $p$ the tensors $(s_{\alpha,\DD})$ are \emph{integral} with respect to the $F$-crystal lattice $\DD(X_{A/J})^\otimes $ of $\DD(\BX_{A/J})^\otimes[\ivtd p]$ (defined via $\iota$), and $t_\alpha:\triv\ra\DD(X)^\otimes$ is an \emph{algebraisation} of $s_{\alpha,\DD}:\triv\to\DD(X_{A/J})^\otimes$ in the sense that $t_\alpha$ is a common lift of $s_{\alpha,\DD}$ independent of the choice of $J$.

It is natural to expect such integrality and algebraisability condition for tensors $(s_{\alpha,\DD})$ to appear in the definition of $\RZ\com{s_\alpha}_{\BX,G}$. To explain, assume that $(\BX,(s_{\alpha,\DD}))$ arise from a mod~$p$ point of Hodge-type Shimura varieties with good reduction embedded in some Siegel modular variety. Once we represent $\RZ\com{s_\alpha}_{\BX,G}$ by a formally smooth formal scheme (\emph{cf.} Theorem~\ref{thm:RZHType}) and the Rapoport-Zink uniformisation (\emph{cf.} \cite{Kim:Unif}),  then for any  $(X,\iota)\in\RZ\com{s_\alpha}_{\BX,G}(A)$ there exists an $A$-valued point in the integral canonical model such that $X$ is the $p$-divisible group associated to the pull-back of the universal abelian variety and $(t_\alpha)$ can be obtained by pulling back the universal  de~Rham tensors on universal abelian scheme on the integral canonical model (obtained in \cite[Corollary~2.3.9]{Kisin:IntModelAbType}).

Now, we choose a formally smooth $W$-lift $\wt A$ of $A$.
Definition~\ref{def:RZG}(\ref{def:RZG:P}) is needed for defining $\{\mu\}$-filtrations in $\DD(X)(\wt A/p^m)$ for any $m\geqslant1$; in particular, in $\DD(X)(A)$. We need to obtain the $G$-torsor $\cP_{\wt A}$ over $\wt A$, not just over $A$, because we would like any point in $\RZ\com{s_\alpha}_{\BX,G}(A)$ to be liftable over $\wt A/p^m$ for $m\gg1$. Finally,  Definition~\ref{def:RZG}(\ref{def:RZG:Kottwitz}) assert that the Hodge filtration $\Fil^1_X\subset \DD(X)(A)$ should be a $\{\mu\}$-filtration.
\end{rmksub}

\begin{lemsub}
\label{Lem:Dwork}
Let $R\in\Nilp_W$ and $I$ be a nilpotent ideal of $R$. Let $X$ be a $p$-divisible group and consider Frobenius-equivariant morphisms of isocrystals
\[t,t':\triv\ra\DD(X)^\otimes[\ivtd p].\]
Then we have $t = t'$ if and only if the equality holds over $R/I$.
\end{lemsub}
In particular, Definition~\ref{def:RZG}(\ref{def:RZG:qisog}) for some ideal of definition $J$ implies Definition~\ref{def:RZG}(\ref{def:RZG:qisog}) for any ideal of definition $J'$; indeed, if $J'\subset J$ then we apply the lemma to $R:=A/J'$ and $I:=J/J'$.
\begin{proof}
We may assume that $pR=0$. Then by Frobenius-equivariance, one can replace $X$ by $\sig^{n*}X$ for some $n$, while $\sig^{n*}X$ only depends on $X_{R/I}$ if $\sig^n(I) = 0$; \emph{cf.} the proof of \cite[Corollary~5.1.2]{dejong:crysdieubyformalrigid}.
\end{proof}

\begin{lemsub}\label{lem:UniqTensors}
Let $X$ be a $p$-divisible group over  $A\in\Nilp^\sm_W$. Then given any morphisms of isocrystals $t,t':\triv\to\DD(X)^\otimes[\ivtd p]$, we have $t=t'$ if and only if they restrict to the same morphism of isocrystals over $A/J$ for some ideal of definition $J\subset A$.

In particular, if $(X,\iota)\in\RZ\com{s_\alpha}_{\BX,G}(A)$ for some $A\in\Nilp^\sm_W$ (using the convention as in Remark~\ref{rmk:RZGsm}), 
then the Tate tensors $t_\alpha:\triv\ra\DD(X)^\otimes$ in Definition~\ref{def:RZG} are uniquely determined by $(X,\iota)$, and induce Frobenius-equivariant morphisms on the isocrystals.
\end{lemsub}
Note that this lemma holds if we replace $A$ with a  formally smooth formally finitely generated $W$-algebra $\wt A$; indeed, for a $p$-divisible group $\wt X$ over $\wt A$, $t:\triv\to\DD(\wt X)^\otimes$ only depends on its restriction $t:\triv\to\DD(\wt X_{\wt A/p})^\otimes$ over $\Spec \wt A/p$.
\begin{proof}
Let us choose a $p$-adic $W$-flat lift $\wt A$ of $A$, and a lift of Frobenius endomorphism $\sig:\wt A\to\wt A$. Given an ideal of definition $J\subset A$, let $D\thra A/J$ be the $p$-adically completed PD hull of $\wt A\thra A/J$. Then $\sigma$ extends to a lift of Frobenius of $D$. Using Lemma~\ref{lem:CrysConn} (together with Remarks~\ref{rmk:isoc} and \ref{rmk:Fisoc}), the first claim can be reduced to the injectivity of the natural map
\[\wt A[\ivtd p] \to D[\ivtd p].\]
To see the injectivity, note that $D[\ivtd p]$ can be embedded into some affinoid algebra containing $\wt A[\ivtd p]$; \emph{cf.}
 \cite[\S5.5]{dejong:crysdieubyformalrigid}. 

To show the rest of the lemma, we first note that the Tate tensors $(t_\alpha)$ are uniquely determined by the map it induces on the isocrystals $t_\alpha:\triv\to\DD(X)^\otimes[\ivtd p]$. (Using the ``dictionary'' given by Lemma~\ref{lem:CrysConn} and Remark~\ref{rmk:isoc}, the claim follows from the $W$-flatness of $\wt A$.) Therefore, if $(t_\alpha)$ and $(t_\alpha')$ are two sets of Tate tensors for $(X,\iota)\in\RZ\com{s_\alpha}_{\BX,G}(A)$ (using the notation of Remark~\ref{rmk:RZGsm}), then $t_\alpha$ and $t_\alpha'$ induce the same map on the isocrystals $\triv\to\DD(X_{A/J})^\otimes$ by Definition~\ref{def:RZG}(\ref{def:RZG:qisog}). By the first part of the lemma, we obtain $t_\alpha = t'_\alpha$ for any $\alpha$.
The Frobenius equivariance of $(t_\alpha)$ can be similarly obtained by applying the first part of the lemma to $(t_\alpha)$ and $ (F(\sigma^*t_\alpha))$, where $F$ is the crystalline Frobenius map $F:\sigma^*\DD(X)^\otimes[\ivtd p]\riso \DD(X)^\otimes[\ivtd p]$.
\end{proof}

\subsection{Example: unramified EL and PEL cases}\label{subsec:ELnPEL}
To a (not necessarily unramified) EL or PEL datum, Rapoport and Zink formulated a suitable moduli problem of $p$-divisible groups with extra structure, and constructed a representing formal  scheme \cite[Theorem~3.25]{RapoportZink:RZspace}.\footnote{See also \cite{Fargues:AsterisqueLLC} for the exposition that is just focused on the unramified case of type A and C.}
In this section, we show that when $(G,b)$ comes from an unramified EL or PEL datum, the formal moduli schemes  constructed by Rapoport and Zink represents $\RZ\com{s_\alpha}_{\BX,G}$ for some suitable choice of $(\Lambda,(s_\alpha))$.

Let us first recall the setting of \cite[Ch.3]{RapoportZink:RZspace} in the unramified case. Let $\fo_B$ be a product of matrix algebras over finite unramified extensions of $\Zp$, and $\Lambda$ be a faithful $\fo_B$-module which is finite flat over $\Zp$.
We consider the following data (\emph{cf.} \cite[\S1.38]{RapoportZink:RZspace}, \cite[Ch.2]{Fargues:AsterisqueLLC})\footnote{To make the comparison with Definition~\ref{def:RZG} more direct, we work over $\Zp$ instead of over $\Qp$ as in the aforementioned references.}:
\begin{description}
 \item[unramified EL case] $(\fo_F,\fo_B,\Lambda, G)$, where $(\fo_B,\Lambda)$ is as before, $\fo_F$ is the centre of $\fo_B$, and $G = \GL_{\fo_B}(\Lambda)$ is a reductive group over $\Zp$.
 \item[unramified PEL case] $(\fo_F,\fo_B,*,\Lambda, (,), G)$, where $(\fo_F,\fo_B,\Lambda, G)$ is an unramified EL-type Rapoport-Zink datum, $(,)$ is a \emph{perfect} alternating $\Zp$-bilinear form on $\Lambda$, $*:a\mapsto a^*$ is an involution on $\fo_B$ such that $(av,w) = (v,a^*w)$ for any $v,w\in \Lambda$, and
 \[G=\mathrm{GU}_{\fo_B}(\Lambda,(,))):= \GL_{\fo_B}(\Lambda)\cap \mathrm{GU}_{\Zp}(\Lambda,(,))\]
 where the (scheme-theoretic) intersection takes place inside $\GL_{\Zp}(\Lambda)$.
\end{description}
For $G$ as above, consider $b\in G(K_0)$ such that we have a $p$-divisible group $\BX:=\BX^\Lambda_b$  (\emph{cf.} Definition~\ref{def:HodgeAssump}), and choose a conjugacy class $\set\mu$ so that $b\in G(W)p^{\sig^*\mu\iv}G(W)$ (\emph{cf.} Definition~\ref{def:LocShData}). In the PEL case, we additionally assume that $\ord_p(c(b)) = -1$, where $c:G\ra \Gm$ is the similitude character.
This assumption is to ensure that the pairing $(,)$ induces a quasi-polarisation of $\BX$ via crystalline Dieudonn\'e theory and  $a\mapsto a^*$ corresponds to the Rosati involution; \emph{cf.} \cite[\S3.20]{RapoportZink:RZspace}.\footnote{Note that our choice of $b$ is the transpose-inverse of the Frobenius matrix (\emph{cf.} the remark above Definition~\ref{def:HodgeAssump}), so in our convention we have $\ord_p(c(b)) = -1$, which differs by sign from \cite[\S3.20]{RapoportZink:RZspace}.}

Under this setting, Rapoport and Zink formulated a concrete moduli problem for $p$-divisible groups, and constructed a formal moduli scheme $\breve\M:=\breve\M_{G,b}$, which turns out to be a formally smooth closed formal subscheme of $\RZ_\BX$. See \cite[Definition~3.21, Theorem~3.25, \S3.82]{RapoportZink:RZspace} for more details. 

For the unramified EL  case, we choose a $\Zp$-basis $(s_\alpha)$ of $\fo_B$ and view them as elements in $\Lambda\otimes\Lambda^* \subset \Lambda^\otimes$. For the unramified PEL-type case, we additionally include a tensor  $s_0\in\Lambda^{\otimes2}\otimes\Lambda^{*\otimes2} \subset \Lambda^\otimes$ associated to the pairing $(,)$ up to similitude, as in  Example~\ref{exa:GSp}.
Now  we consider the subfunctor $\RZ\com{s_\alpha}_{\BX,G}$ of $\RZ_\BX$ using this choices $(\Lambda,(s_\alpha))$; \emph{cf.} Definition~\ref{def:RZG}.
\begin{propsub}\label{prop:ELnPEL}
Assume that  $p>2$.  In the unramified EL and PEL cases, the subfunctor $\RZ\com{s_\alpha}_{\BX,G}$ of $\RZ_\BX$ can be represented by the formal moduli scheme $\breve\M$ constructed by Rapoport and Zink (\emph{cf.}  \cite[Theorem~3.25, \S3.82]{RapoportZink:RZspace}). 
\end{propsub}
\begin{proof}
Let us first handle the unramified EL case.  By considering simple factors of $B$ and applying the Morita equivalence, it suffices to handle the case when $\fo_B=\fo_F=W(\kappa_0)$ and  $\Lambda = \fo_F^n$, where $\kappa_0$ is a finite extension  of $\Fp$. Let us assume this.

Let $(X,\iota)$ denote a $p$-divisible group over $A\in\Nilp^\sm_W$ with quasi-isogeny defined over $A/J$ for some (or any) ideal of definition $J$ containing $p$. 
We choose $(s_\alpha)$ corresponding to a $\Zp$-basis of $\fo_B$.
If $(X,\iota)$ corresponds to $\Spf A\ra\breve\M$, then the induced $\fo_B$-action on $X$ gives rise to Tate tensors $(t_\alpha)$. If we have $(X,\iota)\in\RZ\com{s_\alpha}_{\BX,G}(A)$, then the Tate tensors $(t_\alpha)$ do correspond to an $\fo_B$-action  by the full faithfulness of the Dieudonn\'e theory over $A/p$ (\emph{cf.} \cite{dejong:crysdieubyformalrigid}) and the Grothendieck-Messing deformation theory.

From now on, we assume that $X$ is equipped with an $\fo_B$-action corresponding to a certain set of morphisms of crystals $t_\alpha:\triv\to\DD(X)^\otimes$, and we will translate Definition~\ref{def:RZG} in terms of the $\fo_B$-action and Kottwitz determination condition.

Firstly, Definition~\ref{def:RZG}(\ref{def:RZG:qisog}) is equivalent to the $B$-linearity of the quasi-isogeny $\iota$ over $A/J$ for any ideal of definition containing $p$ by full faithfulness up to isogeny of the crystalline Dieudonn\'e functor over $A/J$ (\emph{cf.} \cite[Corollary~5.1.2]{dejong:crysdieubyformalrigid}). 

We next show that under   Definition~\ref{def:RZG}(\ref{def:RZG:qisog}) (or equivalently, under the $B$-linearity of $\iota$ over $A/J$)
 the scheme $\cP_{\wt A}$ (with the notation as in Definition~\ref{def:RZG}(\ref{def:RZG:P})) is a $\GL_{\fo_B}(\Lambda)$-torsor. 
Since $\fo_B\otimes_{\Zp} \wt A = \prod_{\tau:\kappa_0\hra \Fpbar}\wt A$,
it follows that $P_{\wt A}$ is a $\GL_{\fo_B}(\Lambda)$-torsor if and only if  $\DD(X)(\wt A)$ is a free $\fo_B\otimes_{\Zp} \wt A$-module with the same rank as $\Lambda$. On this other hand, the rank can be computed after the scalar extension to $D[\ivtd p]$, where $D\thra A/J$ is the $p$-adically completed PD hull of $\wt A\thra A/J$ (since $\wt A[\ivtd p]\to D[\ivtd p]$ is injective by the proof of Lemma~\ref{lem:UniqTensors}),
the existence of $B$-linear quasi-isogeny $\iota$ over $A/J$ implies the desired freeness and rank equality.

Let us
show that  $\Fil^1_X$ is a $\set\mu$-filtration with respect to $(t_\alpha(A))$ (Definition~\ref{def:RZG}(\ref{def:RZG:Kottwitz})) if and only if  the ``Kottwitz determinant condition'' holds for  $X_{A/J^n}$ for each $n$ \cite[Definition~3.21(iv)]{RapoportZink:RZspace}. For this, it suffices to show that for a fixed $n$ and $R:=A/J^n$, $\Fil^1_{X_R}$ is a $\set\mu$-filtration if and only if the ``Kottwitz determination condition'' holds. Since the claim is \'etale-local on $\Spec R$,\footnote{The notion of $\set\mu$-filtration is \'etale-local on the base. The Kottwitz determinant condition can be phrased in terms of the ranks of certain vector bundles, and ranks can be computed \'etale-locally.} we may assume that the torsor $\cP_R$ is trivial.
Since $\fo_B =W(\kappa_0)$ for some finite extension  $\kappa_0$ of $\Fp$, we can decompose
\begin{align*}
  \DD(X_R)(R) &= \prod_{\tau\in\Hom(\kappa_0,\Fpbar)} \DD(X_R)(R)_\tau\\
  \Fil^1_{X_R} &= \prod_{\tau\in\Hom(\kappa_0,\Fpbar)} \Fil^1_{X_R,\tau}.
\end{align*}
It follows that the $G(R)$-conjugacy classes of (minuscule) cocharacters $\mu$ exactly correspond to certain integer tuples $(a_\tau)_\tau$ with $a_\tau\in [0, \rk_{R\otimes_{\Zp}\fo_B}\DD(X_R)(R)_\tau]$, 
and $\Fil^1_{X_R}$ is a $\set\mu$-filtration if and only if
$a_\tau = \rk_R\Fil^1_{X_R,\tau}$, which is exactly the Kottwitz determinant condition. 

Let us turn to the  unramified PEL case. Since $\ord_p(c(b)) = 1$, there exists $u\in W\starr$ such that $c(b)= p\iv\sig(u)\iv u$. Then we obtain an $F$-equivariant perfect pairing
\[u(,):\bM^\Lambda_b\otimes \bM^\Lambda_b \xra{(,)} \bM_{c(b)\iv} \xra u \bM_p =  \triv(-1),\]
where $\bM_z$ for $z\in\Gm(K_0)$ denotes the  $F$-crystal on $W$ with $F$ given by multiplication by $z$. (Recall that $\bM^\Lambda_b = W\otimes_{\Zp}\Lambda^*$, so the similitude character for the pairing $(,)$ on $\bM^\Lambda_b$ is $c\iv$.) Then $u(,)$ induces a principal  quasi-polarisation $\lambda_0:\BX \ra \BX^\vee$, and  $\Zp\starr\cdot\lambda_0$ is well defined independent of the choice of $u$.

Conversely, given an $\fo_B$-linear principal quasi-polarisation  $\lambda:X\ra X^\vee$ of a $p$-divisible  group over $A\in\Nilp^\sm_W$, we obtain a tensor $t_\lambda:\triv\ra\DD(X)^{\otimes2}\otimes\DD(X)^{*\otimes2}$ only depending on $\Zp\starr\cdot\lambda$; \emph{cf.} Example~\ref{exa:GSp}. By the full faithfulness result as before, the existence of a principal quasi-polarisation  $\lambda:X\ra X^\vee$ is equivalent to the existence of a certain Tate tensor $t_\lambda$.
If $s_0\in\Lambda^\otimes$ is the tensor corresponding to $\Zp\starr\cdot(,)$ and  $t_0$ the Tate tensor on $\BX$ corresponding to $s_0$, then we have $t_0 = t_{\lambda_0}$.

We choose a $W$-lift $\wt \BX$ of $\BX$ which lifts the $\fo_B$-action and $\lambda_0$. We  let $\lambda_0$ and $t_0$  also denote their lifts over $W$. For any $(X',\iota')\in\RZ_\BX(R')$ with $R'\in\Nilp_W$, we choose the unique lift $\tilde \iota':\BX_{R'}\dra X'$.

We now claim the following are equivalent:
\begin{enumerate}
\item\label{prop:ELnPEL:RZ}
For any $n$,  the  quasi-isogeny $\tilde \iota:\wt\BX_{A/J^n}\dra X_{A/J^n}$ matches the Tate tensors $t_0$ for $\wt\BX$ and  $t_\lambda$ for $X$. Furthermore, the scheme $\cP_{\wt A}$, constructed as in Definition~\ref{def:RZG}(\ref{def:RZG:P}) using $(t_\alpha(\wt A))$ and $t_\lambda(\wt A)$,  is a $\mathrm{GU}_{\fo_B}(\Lambda)$-torsor;
\item\label{prop:ELnPEL:WK}
The quasi-polarisation $\lambda$ is $\fo_B$-linear, and for any $n$ the following diagram commutes up to the multiple by some Zariski-locally constant function $c:\Spec A/J^n\ra \Qp\starr$:
\begin{equation}\label{eqn:polar}
 \xymatrix{
 X_{A/J^n} \ar[rr]^-{\lambda} & &
 X^\vee_{A/J^n} \ar@{-->}[d]_{\tilde\iota^\vee}\\
  \wt\BX_{A/J^n} \ar[rr]^-{\lambda_0} \ar@{-->}[u]_{\tilde\iota} & &
  \wt\BX_{A/J^n}^\vee
}.
\end{equation}
\end{enumerate}
Granting this claim, it follows that $\breve\M$ represents $\RZ\com{s_\alpha}_{\BX,G}$ in the unramified PEL case; indeed, the ``Kottwitz determinant condition'' and Definition~\ref{def:RZG}(\ref{def:RZG:Kottwitz}) can be matched in the identical way as in the unramified EL case.

Let us first show that (\ref{prop:ELnPEL:RZ})$\Rightarrow$(\ref{prop:ELnPEL:WK}).
For this, we may replace $\Spf A$ by some \'etale  covering (of formal schemes) to assume that the torsor $\cP_{\wt A}$ is trivial.
Then the $\fo_B$-linearity of $\lambda$ follows from  the full faithfulness of the Dieudonn\'e theory over $A/p$ (\emph{cf.} \cite{dejong:crysdieubyformalrigid}) and Grothendieck-Messing deformation theory.

Now, let $(,)_0:\DD(\wt \BX)^{\otimes2} \ra \triv(-1)$ and $(,)_\lambda:\DD(X)^{\otimes2} \ra \triv(-1)$ denote the perfect symplectic $F$-equivariant pairing induced by the principal quasi-polarisations.
To simplify the notation, set $R:=A/J^n$.
By definition of $t_0$ and $t_\lambda$, the quasi-isogeny $\tilde\iota$ over $R$  matches $t_0$ and $t_\lambda$ if and only if there exists a Zariski-locally constant function $c:\Spec R\ra \Qp\starr$ such that  $c(,)_{\lambda,R}=(,)_{0,R}\circ(\DD(\tilde \iota)^{\otimes2})$ as pairings on $\DD(X_R)[\ivtd p]$. (Indeed, $c$ is $\Qp\starr$-valued because the automorphism of the $F$-isocrystal $\triv(-1)$ over a finite-type $\kappa$-scheme is $\Qp\starr$ -- a very degenerate case of \cite[Main~Theorem~2]{dejong:crysdieubyformalrigid}.) On the other hand, $\DD(\lambda_0\iv\tilde\iota^\vee\lambda):\DD(\wt\BX_R)[\ivtd p]\ra\DD(X_{R})[\ivtd p]$ is the ``transpose'' of $\DD(\tilde\iota)$ in the sense that
\[(,)_{0,R}\circ(\DD(\tilde\iota)\otimes\id) = (,)_{\lambda,R}\circ(\id\otimes \DD(\lambda_0\iv\tilde\iota^\vee\lambda)):\DD(X_{R})[\ivtd p]\otimes\DD(\wt\BX_R)[\ivtd p] \ra \triv(-1).\]
Therefore, we have $c(,)_{\lambda,R}=(,)_{0,R}\circ(\DD(\tilde \iota)^{\otimes2})$ if and only if we have $\DD(\lambda_0\iv\tilde\iota^\vee\lambda\tilde\iota) = c\id$. By full faithfulness of Dieudonn\'e theory up to isogeny over $R$, this is equivalent to $\tilde\iota^\vee\lambda_R\tilde\iota = c\lambda_{0,R}$. This shows that (\ref{prop:ELnPEL:RZ})$\Rightarrow$(\ref{prop:ELnPEL:WK}).

To show (\ref{prop:ELnPEL:WK})$\Rightarrow$(\ref{prop:ELnPEL:RZ}), it remains to show that (\ref{prop:ELnPEL:WK}) implies that $\cP_{\wt A}$ is a $\mathrm{GU}_{\fo_B}(\Lambda)$-torsor. Since $\breve\M$ is formally smooth, we may lift $(X,\iota)$ so that $A = \wt A$ is formally smooth and formally finitely generated over $W$. Then, it suffices to show that $\cP_{A/J^n}$ is a $\mathrm{GU}_{\fo_B}(\Lambda)$-torsor for each $n$, which follows from \cite[Theorem~3.16]{RapoportZink:RZspace}.
\end{proof}

\subsection{``Closed points'' and deformation theory for $\RZ\com{s_\alpha}_{\BX,G}$}
We choose $(G,b)$, $\Lambda$, and $(s_\alpha)\subset\Lambda^\otimes$ as before. Let $(G,[b],\set{\mu\iv})$ denote the unramified Hodge-type local Shimura datum associated to $(G,b)$; \emph{cf.} Definition~\ref{def:LocShData}. In this section, we show that the moduli functor  $\RZ\com{s_\alpha}_{\BX,G} \subset \RZ_\BX$ interpolates $X^G(b)\subset X^{\GL(\Lambda)}(b)$ on $\kappa$-points and $\Def_{X_x,G}\subset \Def_{X_x}$ on ``formal neighbourhoods''. 
\begin{rmksub}\label{rmk:fpqcclosed}
This remark is to justify why we restrict our focus on the set of \emph{closed} points and the formal neighbourhoods thereof.
Since the underlying reduced scheme $\RZ_{\BX}^\red$ of $\RZ_\BX$ is Jacobson (being locally of finite type over $\kappa$), any closed subscheme  $Z\subset \RZ_{\BX}^\red$ is determined by the subset $Z(\kappa)\subset \RZ_\BX(\kappa)$ of closed points. 

Now, given a formal closed subscheme $\ZZ\subset \RZ_\BX$, we can recover the completed local ring $\wh\OO_{\ZZ,\eta}$ at any $\eta\in\ZZ^\red$ (not necessarily a closed point) from the closed point $z\in\overline{\{\eta\}}$ and the quotient $\wh\OO_{\ZZ,z}$ of $\wh\OO_{\RZ_\BX,z}$ as follows. Choosing a preimage $\hat\eta\in\Spec\wh\OO_{\ZZ,z} \subset \Spec \wh\OO_{\RZ_\BX,z}$ of $\eta$, we can recover $\wh\OO_{\ZZ,\eta}$ as the image of the  map $\wh\OO_{\RZ_\BX,\eta} \hra (\wh\OO_{\RZ_\BX,z})\wh{_{\hat\eta}}\thra (\wh\OO_{\ZZ,z})\wh{_{\hat\eta}}$.
\end{rmksub}

 By Proposition~\ref{prop:LR}, we have a natural bijection
\begin{equation}\label{eqn:RZG:ADL}
X^G(b)\cong \RZ\com{s_\alpha}_{\BX,G}(\kappa).
\end{equation}
Recall that $X^G(b)$ satisfies functorial properties with respect to $(G,b)$ (Lemma~\ref{lem:FunctAffDL}). We also have the following cartesian diagram:
\begin{equation}
\xymatrix{
X^G(b) \ar@{^{(}->}[r] \ar[d]^\cong & X^{\GL(\Lambda)}(b) \ar[d]^\cong\\
\RZ\com{s_\alpha}_{\BX,G}(\kappa) \ar@{^{(}->}[r] &\RZ_\BX(\kappa)
},
\end{equation}
where the vertical isomorphisms are given by Proposition~\ref{prop:LR}, and the horizontal map on the bottom row is the forgetful map. 

Consider $(X_x,\iota_x)\in \RZ\com{s_\alpha}_{\BX,G}(\kappa)$ corresponding to a closed point $x\in\RZ_\BX(\kappa)$. We define the ``formal completion'' $(\RZ\com{s_\alpha}_{\BX,G})\wh{_x}$ of $\RZ\com{s_\alpha}_{\BX,G}$ at $x$ to be the set-valued functor on the category of formal power series rings over $W/(p^m)$ (for some $m$) defined as follows: for any formal power series ring $A$ over $W/(p^m)$, $(\RZ\com{s_\alpha}_{\BX,G})\wh{_x}(A)$ is the subset of elements in $\RZ\com{s_\alpha}_{\BX,G}(A)$ which lift $x$. Equivalently, we have the following description
 \begin{equation}
(\RZ\com{s_\alpha}_{\BX,G})\wh{_x} \cong\RZ\com{s_\alpha}_{\BX,G} \times_{\RZ_\BX} (\RZ_\BX)\wh{_x},
 \end{equation}
 where the right hand side is viewed as the fibre product of functors on the category of formal power series rings over $W/(p^m)$.
If $\RZ\com{s_\alpha}_{\BX,G}$ can be represented by a formal scheme formally smooth and locally formally of finite type over $W$, then $(\RZ\com{s_\alpha}_{\BX,G})\wh{_x}$ can be represented by its formal completion at $x$.

Choose a coset representative $g_x\in G(K_0)$ of $x\in X^G(b)\subset G(K_0)/G(W)$, and write $b_x:=g_x\iv b \sig(g_x)$ (so that we have $(X_x, (t_{\alpha,x}))\cong (\BX^\Lambda_{b_x}, (1\otimes s_\alpha))$). Then for a suitable choice of $\mu\in\set\mu$,  the universal deformation with Tate tensors $(X _{G,x},(\hat t^\univ_{\alpha,x}))$ of $(X_x,(t_{\alpha,x}))$, together with the quasi-isogeny $\iota_x:\BX\dra X_x$, defines an element $\RZ\com{s_\alpha}_{\BX,G}(A _{G,x})$, where $A _{G,x}$ was defined in \S\ref{subsec:FaltingsDeforConstr}. (This can be verified by the explicit construction of $\bM _{G,x}$. Note that we have added the subscript $x$ to the notation of \S\ref{subsec:FaltingsDeforConstr} to indicate that we are deforming $X_x$.) Now, Theorem~\ref{thm:FaltingsDef} shows that the isomorphism $\Def_{X_x}\riso(\RZ_\BX)\wh{_x}$ (\emph{cf.} Lemma~\ref{lem:LocMod}) induces the following isomorphism
\begin{equation}\label{eqn:RZG:FaltingsDefor}
\Def_{X_x,G}\riso (\RZ\com{s_\alpha}_{\BX,G})\wh{_x}
\end{equation}
of functors on formal power series rings over $W/(p^m)$; note that the deformations coming from $\Def_{X_x,G}$ automatically satisfy the condition on the Hodge filtrations (Definition~\ref{def:RZG}(\ref{def:RZG:Kottwitz})) thanks to the explicit construction of the universal deformation over $\Def_{X_x,G}$. In particular, $(\RZ\com{s_\alpha}_{\BX,G})\wh{_x}$ can be pro-represented by a formal power series ring over $W$.
If $\RZ\com{s_\alpha}_{\BX,G}$ can be represented by a formal scheme which is formally smooth and locally formally of finite type over $W$, then the above isomorphism gives an identification of the formal completion at a closed point $x$ with $\Def_{X_x,G}$.

For  $A\in\Nilp^\sm_W$, let $f:\Spf A\to\RZ_\BX$ such that
$f\in\RZ\com{s_\alpha}_{\BX,G}(A)$. Then for any closed point $x$ in $\Spf A$, $f$ induces
\[
\hat f_x:\Spf \wh A_x \to\Def_{X_x,G}
\]
by (\ref{eqn:RZG:FaltingsDefor}).
If we let $X$ denote the $p$-divisible group over $A$ corresponding to $f$, then we have a  morphism of crystals $t_\alpha:\triv\to\DD(X)^\otimes$ for each $\alpha$, which exists by the definition of $\RZ\com{s_\alpha}_{\BX,G}$; \emph{cf.} Definition~\ref{def:RZG}.
\begin{lemsub}\label{lem:FormalNbdTensors}
The pull-back of $t_\alpha$  over $\Spec \wh A_x$ (for a closed point $x$ in $\Spf A$) coincides with the pull-back $(\hat f_x)^*(\hat t^\univ_{\alpha,x})$ of the universal deformation of Tate tensors. The same statement holds if we replace $A$ with a formally smooth formally finitely generated $W$-algebra $\wt A$.
\end{lemsub}
\begin{proof}
This  follows from Lemma~\ref{lem:UniqTensors} applied to $\hat f_x\in\RZ^{(s_\alpha)}_{\BX,G}(\hat A_x)$.
\end{proof}

Let us record some functorial properties that  $\Def_{X_x,G}$ enjoys. To explain, let $(G',b')$ and $\Lambda'$ be as in Definition~\ref{def:HodgeAssump}. (We do \emph{not} assume the existence of any ``equivariant'' morphism $\Lambda\ra\Lambda'$.) We write $\BX':=\BX^{\Lambda'}_{b'}$.

For any $x'\in X^{G'}(b')$ and the corresponding element $(X'_{x'},\iota)\in\RZ_{\BX'}(\kappa)$ (via the embedding $X^{G'}(b') \hra X^{\GL(\Lambda')}(b')$,
we have a natural isomorphism
\begin{equation}
\Def_{X_x,G}\times\Def_{X_{x'},G'}\riso \Def_{X_x\times X_{x'},G\times G'},
\end{equation}
defined by taking the product of deformations.
This isomorphism is compatible with the morphism $\RZ_\BX\times\RZ_{\BX'}\ra \RZ_{\BX\times\BX'}$ defined by taking the product of $p$-divisible groups and quasi-isogenies.

Let $f:G\ra G'$ be a homomorphism over $\Zp$ which takes $b$ to $b'$, and consider the map  $X^G(b)\ra X^{G'}(b')$ associated to $f$ by Lemma~\ref{lem:FunctAffDL}. We choose $x\in X^G(b)$ and let $x'\in X^{G'}(b')$ denote its image by this natural map. We choose a coset representative $g_x\in G(K_0)$ of $x \in X^G(b)\subset G(K_0)/G(W)$, and write $b_x:=g_x\iv b \sig(g_x)$ and $b'_{x'}:= f(b_x)$. Then by the choice of $g_x$ (and $f(g_x)$), we obtain  isomorphisms
\[\Def_{X_x,G}\cong \Def_{\BX^\Lambda_{b_x},G},\quad \Def_{X'_{x'},G'}\cong \Def_{\BX^{\Lambda'}_{b'_{x'}},G'}.
 \]
 So by Proposition~\ref{prop:FunctDefor} we obtain a morphism
\begin{equation}\label{eqn:FunctDefor}
 \Def_{X_x,G}\ra \Def_{X'_{x'},G'}
\end{equation}
By Remark~\ref{rmk:CosetRep} it follows that the morphism above does not depend on the choice of the coset representative $g_x$ of $x \in X^G(b)\subset G(K_0)/G(W)$.

\subsection{Main Statements}
We are ready to state the main result, which asserts that the subset $\RZ\com{s_\alpha}_{\BX,G}(\kappa)\subset\RZ_\BX(\kappa)$ and the subspaces $(\RZ\com{s_\alpha}_{\BX,G})\wh{_x}\subset(\RZ_\BX)\wh{_x}$ for $x\in\RZ\com{s_\alpha}_{\BX,G}(\kappa)$ patch to give a closed formal subscheme $\RZ^\Lambda_{G,b}\subset\RZ_\BX$, and the functorial properties enjoyed by $X^G(b)\cong \RZ\com{s_\alpha}_{\BX,G}(\kappa)$ and $\Def_{X_x,G}\cong (\RZ\com{s_\alpha}_{\BX,G})\wh{_x}$ patch to give the corresponding functorial properties for $\RZ^\Lambda_{G,b}$.
\begin{thmsub}\label{thm:RZHType}
Let $p>2$. Then there exists a closed formal subscheme $\RZ^\Lambda_{G,b}\subset \RZ_\BX$ which is formally smooth over $W$ and represents the functor $\RZ\com{s_\alpha}_{\BX,G}$ for any choice of  $(s_\alpha)\subset\Lambda^\otimes$ with pointwise stabiliser  $G$. More precisely, for any $s_\alpha$ there exist ``universal Tate tensors''
\[
t_\alpha^\univ:\triv\ra\DD((X_{\RZ_\BX})|_{\RZ^\Lambda_{G,b}})^\otimes,
\]
such that for $A\in\Nilp^\sm_W$, a map $f:\Spf (A,J)\ra \RZ_\BX$ factors through $\RZ^\Lambda_{G,b}$ if and only if $f\in\RZ\com{s_\alpha}_{\BX,G}(A)$, in which case  $(t_\alpha)$ as in Definition~\ref{def:RZG} recovers $(f^*t_\alpha^\univ)$.

For another pair $(G',b')$ and $\Lambda'\in\prep(G')$ that give rise to a $p$-divisible group (as in Definition~\ref{def:HodgeAssump}), we consider the  closed formal subscheme $\RZ^{\Lambda'}_{G',b'}\subset\RZ_{\BX'}$ which was just constructed. Then the following properties hold:
\begin{enumerate}
\item\label{thm:RZHType:Prod}
The morphism $\RZ_{\BX}\times_{\Spf W} \RZ_{\BX'}\ra \RZ_{\BX\times\BX'}$, defined by the product of $p$-divisible groups with quasi-isogeny, induces an isomorphism
\[\RZ^\Lambda_{G,b}\times_{\Spf W} \RZ^{\Lambda'}_{G',b'}\riso \RZ^{\Lambda\times\Lambda'}_{G\times G',(b,b')} \]
such that it induces the product decomposition of affine Deligne-Lusztig sets (Lemma~\ref{lem:FunctAffDL}) and the deformation spaces (Proposition~\ref{prop:FunctDefor})  via the isomorphisms (\ref{eqn:RZG:ADL}) and (\ref{eqn:RZG:FaltingsDefor}), respectively.
\item\label{thm:RZHType:Funct}
Let $f:G\ra G'$ be a homomorphism which takes $b$ to $b'$. Then there exists a (necessarily unique) morphism $\RZ^\Lambda_{G,b}\ra \RZ^{\Lambda'}_{G',b'}$ which induces the maps that fit in the following cartesian diagrams:
\[\xymatrix@R=12pt{
X^G(b) \ar[r]^-{\text{Lem~\ref{lem:FunctAffDL}}} \ar[d]^{\text{\eqref{eqn:RZG:ADL}}}_\cong &
X^{G'}(b') \ar[d]^{\text{\eqref{eqn:RZG:ADL}}}_\cong &
\Def_{X_x,G} \ar[r]^-{\text{\eqref{eqn:FunctDefor}}} \ar[d]^{\text{\eqref{eqn:RZG:FaltingsDefor}}}_\cong
& \Def_{X_{x'},G'} \ar[d]^{\text{\eqref{eqn:RZG:FaltingsDefor}}}_\cong\\
\RZ^\Lambda_{G,b}(\kappa) \ar[r] &
\RZ^{\Lambda'}_{G',b'}(\kappa) &
(\RZ^\Lambda_{G,b})\wh{_x} \ar[r] &
(\RZ^{\Lambda'}_{G',b'})\wh{_{x'}}
},\]
where $x'\in\RZ^{\Lambda'}_{G',b'}(\kappa)$ is the image of $x$, and the arrows on the top row are the natural maps induced by $f$.

Furthermore, if $G' = \GL(\Lambda)$, then the natural inclusion $(G,b)\ra (\GL(\Lambda),b)$ induces the natural inclusion $\RZ^\Lambda_{G,b} \ra \RZ^\Lambda_{\GL(\Lambda),b} = \RZ_\BX$ (\emph{cf.} Example~\ref{exa:RZ}). 
\end{enumerate}
\end{thmsub}
%

We will prove this theorem in \S\ref{sec:fpqcDesc} and \S\ref{sec:RepPf}. 

\begin{rmksub}
When the pair $(G,b)$ and the choice of $(\Lambda,(s_\alpha))$ correspond to the unramified EL or PEL case (\emph{cf.} \S\ref{subsec:ELnPEL}), then $\RZ^\Lambda_{G,b}\subset\RZ_\BX$ coincides with the formal moduli subscheme $\breve\M\subset\RZ_\BX$ constructed by Rapoport and Zink, often referred to as an EL or PEL Rapoport-Zink space.
\end{rmksub}

\begin{rmksub}
Just as in the (P)EL case, we will construct (in \S\ref{subsec:Tower}) coverings of the rigid generic fibre $(\RZ^\Lambda_{G,b})^\rig$ by adding suitable level structure, and obtain a rigid analytic tower with Hecke $G(\Qp)$-action. We will verify in \S\ref{sec:ExtraStr} that this covering has the extra structure expected for the
``local Shimura variety'' $\{\mathbb{M}(G,[b],\{\mu\iv\})^{\KK}\}_{\KK\subset G(\Zp)}$ associated to $(G_{\Qp},[b],\{\mu\iv\})$ as stated in \cite[\S5.1]{RapoportViehmann:LocShVar}. In this viewpoint, we can view Theorem~\ref{thm:RZHType} as a construction of a (formally smooth) formal model of the local Shimura variety $\mathbb{M}(G,[b],\{\mu\iv\})^{G(\Zp)}$ with ``hyperspecial level structure'', when the local Shimura datum is of Hodge type.

Formal smoothness of  $\RZ^\Lambda_{G,b}$ is compatible with our expectation that the conjectural local Shimura variety $\mathbb{M}(G,[b],\{\mu\iv\})^\KK$ should admit a formally smooth formal model if $G$ is an unramified reductive group over $\Qp$  and $\KK\subset G(\Qp)$ is hyperspecial. (This is a local analogue of good reduction of Shimura varieties).
\end{rmksub}

\begin{rmksub}
%
The argument to ``algebraise and glue'' the deformation spaces $\Def_{X_x,G}$ for $x\in X^G(b)$ in  \S\ref{sec:fpqcDesc} and \S\ref{sec:RepPf} makes use of the functor $\RZ^{(s_\alpha)}_{\BX,G}$ using $\Nilp^\sm_W$ as ``test rings'' (\emph{cf.}  Definition~\ref{def:RZG}). Implicit in this approach is  the expectation that $\RZ^\Lambda_{G,b}$ should be formally smooth (as explained in the previous remark). 
\end{rmksub}

\begin{rmksub}
\label{rmk:RZHTypeCor:b}
Choose $g\in G(K_0)$ such that $b':=g\iv b\sig(g)$ also defines a $p$-divisible group $\BX':=\BX^\Lambda_{b'}$. We let $\iota_g:\BX\dra \BX'$ denote the quasi-isogeny induced by $g:\bM^\Lambda_{b'}[\ivtd p] \riso\bM^\Lambda_b[\ivtd p]$. Then we have a natural isomorphism $\RZ\com{s_\alpha}_{\BX',G}\riso \RZ\com{s_\alpha}_{\BX,G}$ by post-composing $\iota_g$ (or rather, the suitable base change of it). 
In particular, we obtain an isomorphism $\RZ^\Lambda_{G,b'}\riso \RZ^\Lambda_{G,b}$.
\end{rmksub}
\begin{rmksub}
\label{rmk:RZHTypeCor:Lambda}
Applying Theorem~\ref{thm:RZHType}(\ref{thm:RZHType:Funct}) to $f=\id:(G,b)\ra(G,b)$, we obtain a canonical isomorphism $\RZ^\Lambda_{G,b}\riso \RZ^{\Lambda'}_{G,b}$ which respects the identifications of $\kappa$-points with $X^G(b)$ and the formal completion at a $\kappa$-point $x$ with $\Def^\Lambda_{G,b_x}$.

Combining this with Remark~\ref{rmk:RZHTypeCor:b}, it follows that $\RZ^{\Lambda}_{G,b}$ only depends on the associated unramified Hodge-type local Shimura datum $(G,[b],\set{\mu\iv})$ up to (usually non-canonical) isomorphism. Furthermore, if $(G',[b'],\set{\mu^{\prime-1}})$ is another unramified Hodge-type local Shimura datum, then for any map $f:(G,[b],\set{\mu\iv}) \ra (G',[b'],\set{\mu^{\prime-1}})$,  we have a morphism $\RZ^\Lambda_{G,b}\ra\RZ^{\Lambda'}_{G',b'}$ for some suitable choice of $(b,\Lambda)$ and $(b',\Lambda')$.
\end{rmksub}
\begin{rmksub}
\label{rmk:RZHTypeCor:Funct}
In the setting of Theorem~\ref{thm:RZHType}(\ref{thm:RZHType:Funct}), the association
\[ [(G,b)\ra (G',b')]\rightsquigarrow [\RZ^\Lambda_{G,b}\ra\RZ^{\Lambda'}_{G',b'}] \]
respects compositions and products of morphisms (in the obvious sense), since the analogous statement is true for $X^G(b)$ and $\Def^\Lambda_{G,b_x}$.
\end{rmksub}
\begin{rmksub}
If $f:G\ra G'$ is a closed immersion (so we write $b' = b$), then the associated map $\RZ^\Lambda_{G,b}\ra\RZ^{\Lambda'}_{G',b}$ is a closed immersion. Indeed, by Remark~\ref{rmk:RZHTypeCor:Lambda} we may take $\Lambda = \Lambda'$, and the natural inclusion $\RZ^\Lambda_{G,b}\hra\RZ_\BX$ factors as $\RZ^\Lambda_{G,b} \ra \RZ^\Lambda_{G',b}\hra\RZ_\BX$ by Remark~\ref{rmk:RZHTypeCor:Funct}.
\end{rmksub}

\begin{rmksub}\label{rmk:ExcIsom}
If $(G,b)$ comes from an unramified EL or PEL datum, then it follows from Proposition~\ref{prop:ELnPEL} that our construction of $\RZ^\Lambda_{G,b}$ is compatible with the original construction of Rapoport and Zink.
On the other hand, the ``functoriality'' aspect of Theorem~\ref{thm:RZHType} produces morphisms between EL and PEL Rapoport-Zink spaces which cannot be constructed as a formal consequence of the (P)EL moduli problem. For example, consider an exceptional isomorphism $\GSp(4)\cong\GSpin(3,2)$ of split reductive $\Z_{(p)}$-groups, and set $G:=\GSp(4)_{\Zp}$. Let  $(G,[b],\set{\mu\iv})$ be a unramified Hodge-type local Shimura datum, coming from a (global) Shimura datum for $\GSp(4)_{/\Q}$ as in Example~\ref{exa:ShVar}. Recall that we have another faithful $G$-representation $\Lambda$; namely, the even Clifford algebra  for a split rank-$5$ quadratic space over $\Zp$. It turns out that there exists a perfect alternating form $\psi$ such that the natural $G$-action on $\Lambda$ induces a closed immersion $f:G=\GSpin(3,2)\hra\GSp(\Lambda,\psi)=:G'$ and the local Shimura datum $(G',[f(b)],\set{f\circ\mu\iv})$ is of PEL type. 
Then Theorem~\ref{thm:RZHType} produces a closed immersion $\RZ_{G,b}\hra\RZ_{G',f(b)}$.
\end{rmksub}

\begin{rmksub}
Since the first version of this paper became available, there has been a lot of progress in understanding $\RZ_{G,b}^\Lambda$. For example, given the formal closed subscheme $\RZ_{G,b}^\Lambda\subset\RZ_\BX$  as in Theorem~\ref{thm:RZHType}, then we can describe $\RZ_{G,b}^\Lambda(\kappa')$ group-theoretically, where $\kappa'$ is either a perfect field extension of $\kappa$ (\emph{cf.} \cite[Proposition~3.11]{Zhu:Satake}) or a 
finitely generated extension of $\kappa$ (\emph{cf.} \cite[Theorem~2.4.10]{HowardPappas:GSpin}).
\end{rmksub}

\section{Descent and Extension of Tate tensors from a complete local ring to a global base}\label{sec:fpqcDesc}
In this section,  we prove the technical results (especially, Propositions~\ref{prop:descent} and \ref{prop:extension}) which allow us to ``globalise'' the Faltings deformation spaces.

Let $\Nilp_W^\ft$ be the category of  finitely generated $W/p^m$-algebras for some $m$. We first define the following subfunctor $\RZ^\Lambda_{G,b}$ of $\RZ_\BX$, where $\BX= \BX^\Lambda_b$. (So we have $\RZ_\BX = \RZ^{\Lambda}_{\GL(\Lambda),b}$.) We will show in \S\ref{sec:RepPf} that this subfunctor can be represented by a closed formal scheme of $\RZ_\BX$ which satisfies the desired properties stated in Theorem~\ref{thm:RZHType}.

\begin{defn}\label{def:RZGloc}
In the setting of \S\ref{subsec:RZHodge}, we define a functor $\RZ^\Lambda_{G,b}:\Nilp_W^\ft\ra\Sets$ as follows: for any  $R\in\Nilp^\ft_W$, $\RZ^\Lambda_{G,b}(R)\subset \RZ_\BX(R)$ consists of  $(X,\iota)\in\RZ_\BX(R)$ such that  for any $x:\Spec\kappa\ra\Spec R$, we have $(X_x,\iota_x)\in\RZ\com{s_\alpha}_{\BX,G}(\kappa)$ and the map $\Spf \wh R_x\ra\Def_{X_x}$, induced by $X_{\wh R_x}$, factors through $\Def_{X_x,G}$.
\end{defn}

For $(X,\iota)\in\RZ_{G,b}^\Lambda(R)$  with $R\in\Nilp^\ft_W$, we set
\begin{equation}\label{eqn:tx}
\hat t_{\alpha,x}:\triv\ra\DD(X_{\wh R_x})^\otimes
\end{equation}
to be the pull-back of the universal Tate tensors $\hat t_{\alpha,x}^\univ$ (\ref{eqn:tuniv}) via $\Spf \wh R_x\ra\Def_{X_x,G}$.

From the definition of $\RZ^\Lambda_{G,b}$, it is not clear whether $\RZ^\Lambda_{G,b}$ is formally smooth, and whether we have $\RZ\com{s_\alpha}_{\BX,G}(A)\cong \varprojlim_n\RZ^\Lambda_{G,b}(A/J^n)$  for $A\in\Nilp^\sm _W$ with ideal of definition $J$, where $\RZ\com{s_\alpha}_{\BX,G}$ is defined in Definition~\ref{def:RZG} (\emph{cf.} Proposition~\ref{prop:descent}(\ref{prop:descent:devissage})). Also it is unclear how to rule out the possibility that $\RZ^\Lambda_{G,b}$ is the disjoint union of $\Def_{X_x,G}$ when the closure of $\RZ_{G,b}^\Lambda(\kappa)$ in $\RZ_{\BX}^\red$ is positive-dimensional (\emph{cf.} Proposition~\ref{prop:extension}).
These issues will be resolved by the technical results proved in this section.

Let us first state our descent result:
\begin{prop}\label{prop:descent}
Assume that  $p>2$.
\begin{enumerate}
\item\label{prop:descent:t} Let $(X,\iota)\in\RZ^\Lambda_{G,b}(R)$ for $R\in\Nilp^\ft_W$. Then for each $\alpha$,  there exists a unique morphism of crystals
\[t_\alpha:\triv\ra\DD(X)^\otimes\]
which induces $s_{\alpha,\DD}$ on the isocrystals (\emph{cf.} Definition~\ref{def:t}), and pulls back to $\hat t_{\alpha,x}$ for each  closed point $x$  in $\Spec R$ (\emph{cf.} (\ref{eqn:tx})).
\item\label{prop:descent:liftability}
Let $R\in\Nilp^\ft_W$ and choose a presentation $A/J\cong R$ where $A$ is formally smooth formally finitely generated over $W$ and $J$ is an ideal of definition. 
Let $(X,\iota)\in \RZ^\Lambda_{G,b}(R)$ and let $f:\Spec R\to\RZ_\BX$ denote the corresponding morphism.
Then there exists a lift $\tilde f:\Spf A\to\RZ_\BX$ of $f$  such that $\tilde f\in\RZ\com{s_\alpha}_{\BX,G}(A)$. Furthermore, if we let $\wt X$ denote the $p$-divisible group over $A$ obtained by pulling back the universal $p$-divisible group via $\tilde f$, then for any $\alpha$ the Tate tensor $\tilde t_\alpha:\triv\to\DD(\wt X)^\otimes$ (which exists by the definition of $\RZ\com{s_\alpha}_{\BX,G}$; \emph{cf.} Definition~\ref{def:RZG}) pulls back to $t_\alpha$ constructed in (\ref{prop:descent:t}) via restriction.
\item\label{prop:descent:devissage} Let  $A$ be a formally smooth formally finitely generated $W$-algebra with ideal of definition  $J$.  We consider a projective system $(X\com i,\iota\com i)\in\RZ_\BX(A/J^i)$, which corresponds to a map $\tilde f:\Spf (A,J) \ra\RZ_\BX$. Then we have $(X\com i,\iota\com i)\in\RZ^\Lambda_{G,b}(A/J^i)$ for each $i$ if and only if   $\tilde f\in \RZ\com{s_\alpha}_{\BX,G}(A)$. The same holds if we replace $A$ with $ A/p^m\in\Nilp^\sm_W$.
\end{enumerate}
\end{prop}

We will prove this proposition in \S\ref{subsec:LiftablePD}--\S\ref{subsec:Pfdescent}, and for now we  remark that the claims (\ref{prop:descent:t})--(\ref{prop:descent:devissage}) will be proved simultaneously; we first prove a weaker analogue of (\ref{prop:descent:t}) (\emph{cf.} Proposition~\ref{prop:DescentDevissage}), from which we deduce  (\ref{prop:descent:devissage}) (\emph{cf.} Proposition~\ref{prop:Pfdescent}, Lemma~\ref{lem:Pfdescent}), and using (\ref{prop:descent:devissage}) we prove the rest of Proposition~\ref{prop:descent} (\emph{cf.} Lemma~\ref{lem:PfdescentLiftability}).

Let us now begin the proof of Proposition~\ref{prop:descent}.

\subsection{Preparation: Liftable PD thickenings}\label{subsec:LiftablePD}
\begin{defnsub}\label{def:LiftablePD}
Let $B\thra R$ be a PD thickening  of $\Z/p^m$-alegbras for some $m\geqslant1$. We call $B$ a \emph{liftable} PD thickening of $R$ if there exist a $\Zp$-flat $p$-adic PD thickening $D^\fl\thra R$ and a surjective PD morphism $D^\fl\thra B$.

The most important class of liftable PD thickening to us is \emph{square-zero liftable PD thickenings}, by which we mean a square-zero thickening $B\thra R$ equipped with the ``square-zero PD structure'' on $\ker (B\thra R)$, such that $B\thra R$ is lifable as a PD thickening. (If we assume that $p>2$, then any square-zero PD structure is compatible with the standard PD structure on $p\Zp$.)
\end{defnsub}
The following lemma is obvious from the definition.
\begin{lemsub}\label{lem:LiftablePDRefine}
Let $R''\thra R'\thra R$ be square-zero thickenings of $\Z/p^m$-algebras for some $m\geqslant 1$. If $R''\thra R$ is a square-zero liftable PD thickening, then so are $R''\thra R'$ and $R'\thra R$. 
\end{lemsub}

Note that there exists a ring $R$ of characteristic $p$ which does not admit a $\Zp$-flat $p$-adic PD thickening (so the trivial thickening $R\xra{=} R$ is not liftable for such $R$). See \cite[Remark~4.1.6]{ScholzeWeinstein:RZ} for such an example, which is the quotient of a perfect $\Fp$-algebra by an infinitely generated ideal. On the other hand, the trivial thickening of a reduced $\kappa$-algebra is liftable, as $W(R)\thra R$ is a $W$-flat $p$-adic  PD thickening. 

\begin{lemsub}\label{lem:PDtorsion}
Let $R'\thra R$ be a square-zero thickening of finitely generated $W/p^m$-algebras for some $m\geqslant 1$. We choose a formally smooth $W$-algebra $A \thra R'$ and let $D\thra R$ denote the $p$-adically completed PD hull of $A\thra R$.

Then $R'\thra R$ is a square-zero liftable PD thickening if and only if the natural PD morphism $D\thra R'$ factors through the image $D^\fl$ of $D$ in $D[\ivtd p]$.
\end{lemsub}
The author knows whether $D\thra R$ factors through $D^\fl$ if $R$ is finitely generated over $W/p^m$ for some $m$, or if there is any counterexample.
\begin{proof}
Let $D_{\tor}:= \ker(D\thra D^\fl)$. If $D\thra R'$ factors through $D^\fl$, then we have $D_{\tor}\subset \ker(D\thra R)$, and it is a PD subideal  as the PD structure preserves the $p$-power torsion elements. Therefore, $D^\fl$ is the $W$-flat $p$-adic  PD thickening, which admits a PD surjection onto $R'$.

Conversely, if $R'\thra R$ is liftable, then by assumption there exist a $W$-flat $p$-adic PD thickening $D'\thra R$ and a surjective PD morphism $D'\thra R'$. By choosing a lift $A\to D'$ of $A\thra R'$, we can natural extending it to a PD morphism $D\to D'$. Therefore  $D\thra R'$ factors through $D^\fl$, as it factors through a flat $W$-algebra $D'$.
\end{proof}

The aim of this section is to show that ``sufficiently many'' rings in characteristic~$p$ admit lots of liftable PD thickenings  (although we cannot handle all finitely generated $\kappa$-algebras). In general, it seems quite hard to produce a $p$-torsion free PD thickening, and our approach is to embed our finitely generated $\kappa$-algebra into some f-semiperfect ring and use some construction of $p$-torsion free PD thickening available for some f-semiperfect rings.

Recall that a ring $\wt R$ of characteristic $p$ is called \emph{semiperfect} if the Frobenius map $\sig:\wt R\ra \wt R$ is surjective. To a semiperfect ring $\wt R$ we can form a \emph{inverse perfection}
\begin{equation}
\wt R^\flat:= \varprojlim_\sig\wt  R,
\end{equation}
complete with respect to the natural projective limit topology. Let $\wt J\subset \wt R^\flat$ denote the kernel of the natural projection $\wt R^\flat \thra \wt R$. 

The following definition is from \cite[Definition/Proposition~4.1.2]{ScholzeWeinstein:RZ} and \cite{Lau:DieudonneSemiperfect}:
\begin{defnsub}
A semiperfect ring $\wt R$ is called \emph{f-semiperfect} if there is a finitely generated ideal $\wt J_0\subset \wt R^\flat$ such that for some $n\gg0$ we have $\sig^n(\wt J_0)\subset \wt J\subset \wt J_0$.  

A semiperfect ring $\wt R$ is \emph{balanced} if $\big(\ker(\sigma)\big)^p=0$. This is equivalent to require that  $\wt J:=\ker (\wt R^\flat\thra \wt R)$ satisfies $\wt J^p = \sigma(\wt J)$; \emph{cf.} \cite[Lemma~4.6]{Lau:DieudonneSemiperfect}.
\end{defnsub}

Note that any f-semiperfect ring $\wt R_0$ admits a surjective map $\psi:\wt R_0\thra \wt R$ to a balanced f-semiperfect ring $\wt R$ such that $\sigma^N(\ker(\psi))=0$ for some $N\gg0$; \emph{cf.} \cite[Definition/Proposition~4.1.2]{ScholzeWeinstein:RZ}.

\begin{lemsub}\label{lem:fslift}
Let $\wt R$ Be a balanced semiperfect ring. Then there exists a $\Zp$-flat quotient $\wt W(\wt R)$ of $W(\wt R^\flat)$  such that $\wt W(\wt R)/p = \wt R$.
\end{lemsub}
Note that $\wt W(\wt R)$ was denoted as $W_{\mathrm{PD},0}(\wt R)$ in \cite[proof Lemma~4.1.7]{ScholzeWeinstein:RZ}. We will see from the construction that we have $W(\wt R^\flat) \thra \wt W(\wt R) \thra W(\wt R)$.
\begin{proof}
This lemma can be extracted from \cite[proof Lemma~4.1.7]{ScholzeWeinstein:RZ}. (See also  \cite[proof of Lemma~4.6]{Lau:DieudonneSemiperfect}.) To explain, let $\wt J:=\ker(\wt R^\flat\thra\wt R)$, and  define
\[
[\wt J]:=\{\sum_{i\geqs0}[r_i]p^i=(r_0,r_1^p,\cdots, r_i^{p^i},\cdots) |\ r_i\in \wt J\} \subset W(\wt J) \subset W(\wt R^\flat),
\]
which turns out to be an ideal an ideal of $W(\wt R^\flat)$ since we have $\sig(\wt J)=\wt J^p$ by balancedness of $\wt R$. (The non-trivial part is to show that  $[\wt J]$ is stable under  addition if we have $\sig(\wt J)=\wt J^p$. See  \cite[proof of Lemma~4.6]{Lau:DieudonneSemiperfect} for the verification.) 

Now, we set $\wt W(\wt R):=W(\wt R^\flat)/[\wt J]$. Clearly (from the definition of $[\wt J]$), there is no non-zero $p$-torsion in $\wt W(\wt R)$, and we have $\wt W(\wt R)/p = \wt R^\flat/\wt J = \wt R$.
\end{proof}

\begin{lemsub}\label{lem:fsperf}
Let $\psi:\wt R'\thra \wt R$ be a nilpotent thickening of balanced f-semiperfect rings such that we have $(\ker (\psi))^p = 0$. We give a PD structure on $\ker(\psi)$ determined by $f^{[p]} = 0$ for any $f\in\ker(\psi)$.

Then there exists a unique PD structure on $\ker (\wt W(\wt R')\thra \wt R)$ such that the natural projection $\wt W(\wt R')\thra \wt R'$ is a PD morphism. In particular, $\psi:\wt R'\thra \wt R$ is a liftable PD thickening (\emph{cf.} Definition~\ref{def:LiftablePD}).
\end{lemsub}
\begin{proof}
 Since $\sigma(\ker(\psi)) \subset \ker(\psi)^p = 0$ (i.e., $\psi$ factors $\sigma:\wt R\to \wt R$), it follows that $\psi$ induces an isomorphism $\wt R'^\flat\riso \wt R^\flat$. We let $\wt J':=\ker (\wt R^\flat\thra \wt R')$ and $\wt J:=\ker (\wt R^\flat\thra \wt R)$. Since both $\wt R$ and $\wt R'$ are balanced, the ideals $[\wt J],[\wt J']\subset W(\wt R)$ make sense (\emph{cf.} proof of  Lemma~\ref{lem:fslift}), and we have 
$\wt J \supset \wt J' \supset \wt J^p = \sigma(\wt J)$. 

We write $\fa:=\ker (\wt W(\wt R')\thra \wt R)$, which is generated by $p$ and the image of $[\wt J]$.
Since $\wt W(\wt R'):=W(\wt R^\flat)/[\wt J']$ is $\Zp$-flat, there exists at most one PD structure on $\fa$, and the only candidate of the PD structure is the restriction of the unique divided power of $\fa$ in   $\wt W(\wt R')[\ivtd p]=(W(\wt R^\flat)/[\wt J'])[\ivtd p]$. Therefore, to show that $\fa$ has a PD structure, it suffices to show that for any $\alpha\in (p,[\wt J]) \subset W(\wt R^\flat)$ we have $\frac{\alpha^p}{p!} \in \fa + [\wt J'][\ivtd p]\subset W(\wt R^\flat)[\ivtd p]$. (Note that the same claim for $\frac{\alpha^n}{n!}$ with $n\geqslant p$ follows from the case with $n=p$.) 

Since any $\alpha\in(p,[\wt J])$ can be written as $\alpha = [r_0] + p\beta$, for $r_0\in \wt J$ and $\beta\in W(\wt R^\flat)$, we have 
\[
\frac{\alpha^p}{p!} = \frac{([r_0]+p\beta)^p}{p!} = \sum_{i=0}^p\frac{[r_0^i]}{i!}\frac{p^{p-i}}{(p-i)!}\beta^{p-i}\equiv \frac{[\sigma(r_0)]}{p!}\bmod{pW(\wt R^\flat)}.
\]
Since we have  $\sigma(\wt J)\subset \wt J'$ by assumption,
$\frac{[\sigma(r_0)]}{p!}$ maps to zero in $\wt W(\wt R')[\ivtd p]$, and we have $\frac{\bar\alpha^p}{p!}\in p \wt W(\wt R')$ for any $\bar\alpha\in\fa$. Therefore, $\fa\subset \wt W(\wt R')$ is a PD ideal and the mod~$p$ reduction $\wt W(\wt R')\thra\wt R'$ is a PD morphism.
\end{proof}

\begin{corsub}\label{cor:fsperf-red}
Let $\kappa$ be an algebraically closed field of characteristic~$p$, and let $B$ be a \emph{reduced} ring which is formally finitely generated over $\kappa$ (\emph{cf.} Remark~\ref{rmk:ForFType}). Let $J\subset B$ be an ideal contained in an ideal of definition for $B$ such that $R:=B/J$ is reduced. (Note that we have $B = \varprojlim_i B/J^i$.)

Then there exists a sequence of square-zero liftable PD thickenings (in the sense of Definition~\ref{def:LiftablePD}):
\[B\thra \cdots \thra B_{i+1} \thra B_i \thra \cdots \thra B_0 = R\]
such that $B = \varprojlim_i B_i$. 
\end{corsub}
In practice, $B$ will be either formally smooth over $\kappa$ (with $J$ some suitable closed subideal of some ideal of definition), or a reduced complete local noetherian $\kappa$-algebra with residue field $\kappa$ (with maximal ideal $J$, so $R = \kappa$).
\begin{proof}
\begin{subequations}
We write $\wt B:=\varinjlim_\sig B$, and $\wt R_0:=\wt B/J\wt B$ (which is clearly f-semiperfect). Then we have
\begin{equation}\label{eqn:R0flat}
\wt R_0^\flat :=\varprojlim_\sigma \wt R_0 \cong \varprojlim_n \wt B/\sigma^n(J)\wt B \cong \varprojlim_m \wt B/J^m\wt B,
\end{equation}
where the last isomorphism follows since for any $n$, we have $\sigma^{n+N}(J)\wt B\subset J^{p^{n+N}}\wt B\subset \sigma^n(J)\wt B$, if $J$ can be generated by $p^N$ elements; \emph{cf.} \cite[proof of Proposition~4.1.2]{ScholzeWeinstein:RZ}. In particular, $B$ injects into $\wt R^\flat_0$.

Let $\wt J:= \bigcup_n \sig^{-n}(J^{p^n}\wt R^\flat_0)$, which an ideal of $\wt R_0^\flat$ because it is a rising union of ideals. Then we have $\wt J^p = \sigma(\wt J)$ and $J\wt R_0^\flat \subset \wt J \subset \sigma^{-N}(J\wt R_0^\flat)$, if $J$ can be generated by $p^N$ elements. We set $\wt R = \wt R_0/\wt J$. Note that $\wt R$ is a balalnced f-semiperfect ring, and we have 
\begin{equation}\label{eqn:Rflat}
\wt R_0^\flat = \wt R^\flat = \varprojlim_n \wt R^\flat/\sigma^n(\wt J) = \varprojlim_n \wt R^\flat/(\wt J)^{p^n} ;\text{  \emph{cf.} \cite[Proposition~4.1.2]{ScholzeWeinstein:RZ}.}
\end{equation}
\end{subequations}

We next show that $R$ injects into $\wt R$. Indeed, the kernel of $B \thra R \ra \wt R$ consists of elements $f\in B$ such that $\sig^n(f)=f^{p^n}\in J^{p^n}$ for some $n$, but this condition forces $f\in J$ as $J$ is its own radical.

Observe that $\set{(\wt J)^i}$ is a fundamental system of neighbourhoods of $0$ in $\wt R^\flat$, as it is cofinal with $\set{J^i\wt R^\flat}$; \emph{cf.} (\ref{eqn:R0flat}) and (\ref{eqn:Rflat}).  We set $\wt B_i:=\wt R^\flat/\wt J^i$,  and let $B_i\subset \wt B_i$ denote the image of $B$. Note that $\wt B_i$ is balanced for any $i$ since $\sigma((\wt J)^i) = (\sigma(\wt J))^i = (\wt J)^{pi}$. As $B$ injects into $\wt R^\flat$, we have $B=\varprojlim_i B_i$. (Recall that $B$ is $J$-adically complete, and $\{J^i\wt R^\flat\}$ and $\{(\wt J)^i\}$ are cofinal.) 

It remains to show that for any $i$  the PD thickening $B_{i+1}\thra B_i$ (with respect to the ``square-zero PD structure'') is liftable, which we deduce from the liftability of the square-zero PD thickening of balanced f-semiperfect rings $\wt B_{i+1}\thra \wt B_i$ (\emph{cf.} Lemma~\ref{lem:fsperf}).
We choose a formally smooth formally finitely generated $W$-algebra $A$ surjecting onto $B$ and $\ker(A\thra R)$ is an ideal of definition. (Indeed, we choose a polynomial algebra $A_0$ over $W$ which surjects onto $B/J^2$, and let $A$ denote the completion of $A_0$ with respect to $\ker(A_0\thra R)$.) Let $D_i\thra B_i$ denote the $p$-adically completed  PD hull of $A\thra B\thra B_i$, then we obtain a natural PD morphism $D_i\thra B_{i+1}$. By Lemma~\ref{lem:PDtorsion}, it suffices to show that $D_i\ra B_{i+1}$ factors through the image $D_i^\fl$ of $D_i$ in $D_i[\ivtd p]$.

We choose a lift $A\ra W(\wt R^\flat)$ of $A\thra B\hra\wt R^\flat$. Then by the universal property of $D_i$, the map $A\to W(\wt R^\flat)\thra \wt W(\wt B_{i+1})$ naturally extends to a PD morphism $D_i \to \wt W(\wt B_{i+1})$; \emph{cf.} Lemma~\ref{lem:fsperf}. Now, we consider the following diagram
\begin{equation}\label{eqn:fsperf-red}
\xymatrix{
& A \ar@{->>}[d] \ar[r]\ar@{.>}[ld]
& W(\wt R^\flat) \ar@{->>}[r]
& \wt W(\wt B_{i+1}) \ar@{->>}[d] \\
D_i\ar@{.>>}[r] \ar@{.>}[urrr]
& B \ar@{->>}[r] 
& B_{i+1} \ar@{^{(}->}[r]
& \wt B_{i+1}
}.
\end{equation}

Note that the two possible solid arrows $A\to\wt B_{i+1}$ coincide since the composed map $A\to W(\wt R^\flat)\thra \wt B_i$ factors through $B$ and $B_{i+1}$ is the image of $B$ in $\wt B_{i+1}$. Now the universal property of $D_i$ applied to the two possible solid arrows $A\to \wt B_{i+1}$ shows that two possible maps $D_i\to \wt B_{i+1}$ coincides; i.e., $D_i\to \wt W(\wt B_{i+1}) \thra \wt B_{i+1}$ factors through $B_{i+1}$. Furthermore, the natural map $D_i\to\wt B_{i+1}$ factors through the $\Zp$-flat closure $D_i^\fl$ (as $D_i\to \wt W(\wt B_{i+1})$ factors through $D_i^\fl$ by $\Zp$-flatness of $\wt W(\wt B_{i+1})$). This shows that $\ker(D_i\thra B_{i+1})$ contains the $p$-power torsion of $D_i$.
\end{proof}
\begin{rmksub}
Unfortunately, Corollary~\ref{cor:fsperf-red} does not show that any finitely generated $\kappa$-algebra $R$ admits a $p$-adic $W$-flat PD thickening, as there are examples of $R$ that cannot occur as one of $B_i$'s as in Corollary~\ref{cor:fsperf-red}.
\end{rmksub}
\begin{corsub}\label{cor:fsperf-appl}
Let $W:=W(\kappa)$ for an algebraically closed field $\kappa$ of characteristic $p>2$.
Let $A$ be a formally smooth formally finitely generated  $W$-algebra, and let  $J\subset A$  be an ideal contained in an ideal of definition for $A$, such that $R:=A/J$ is a reduced $\kappa$-algebra. (Note that we have $A = \varprojlim_iA/J^i$.)  Then  there exists a sequence of square-zero liftable PD thickenings $\set{A_{m,i}}$ indexed by $(m,i)\in\Z_{\geqslant1}^2$, where each $A_{m,i}$ is a nilpotent thickening of $A_{1,1}:=R$ killed by $p^m$,  such that  $A_{m+1,i}\thra A_{m,i}$ and $A_{m,i+1}\thra A_{m,i}$ are liftable square-zero PD thickenings for any $(m,i)$, and we have $A/p^m = \varprojlim_i A_{m,i}$ for any $m\geqslant 1$.  
\end{corsub}
\begin{proof}
We set $B=A/p$, which is a reduced $\kappa$-algebra (\emph{cf.} \cite[Lemma~1.3.3]{dejong:crysdieubyformalrigid}),
 so by Corollary~\ref{cor:fsperf-red} we obtain $A_{1,i}:=B_i$.
Furthermore, we construct $\wt R^\flat$, $\wt B_i$, etc, as in the proof of Corollary~\ref{cor:fsperf-red}, and we have the commutative diagram (\ref{eqn:fsperf-red}).

For the chosen lift $A\to W(\wt R^\flat)$ of $A/p\hra \wt R^\flat$, we set 
$A_{m,i}$ to be the image of $A$ in $\wt W(\wt B_i)/p^m$. Then we have $A/p^m = \varprojlim_i A_{m,i}$ for any $m\geqslant1$. (Injectivity is clear from $A/p^m\hra W(\wt R^\flat)/p^m = \varprojlim_i \wt W(\wt B_i)/p^m$, and surjectivity can be checked modulo~$p$; \emph{cf.} Corollary~\ref{cor:fsperf-red}.)

Finally, the liftability of square-zero PD thickenings $A_{m',i'}\thra A_{m,i}$ with $(m',i')\in\{(m+1,i),(m,i+1)\}$  can be verified via diagram chasing analogous to (\ref{eqn:fsperf-red}). To explain, let $D_{m,i}\thra A_{m,i}$ denote the $p$-adically completed PD hull of $A\thra A_{m,i}$, so we obtain the PD morphisms $D_{m,i}\to A_{m',i'}$ for $(m',i')\in\{(m+1,i),(m,i+1)\}$, where the targets are given the ``square-zero PD structure'' on the kernel of the projection onto $A_{m,i}$. And we have the following diagrams
\[
\xymatrix{
A \ar@{->>}[dr] \ar[r]\ar@{.>}[d]
& W(\wt R^\flat)  \ar@{->>}[r]
& \wt W(\wt B_{i}) \ar@{->>}[d] \\
D_{m,i}\ar@{.>>}[r] \ar@{.>}[urr]
& A_{m+1,i} \ar@{^{(}->}[r]
& \wt W(\wt B_{i})/p^{m+1}
}
\xymatrix{
A \ar@{->>}[dr] \ar[r]\ar@{.>}[d]
& W(\wt R^\flat)  \ar@{->>}[r]
& \wt W(\wt B_{i+1}) \ar@{->>}[d] \\
D_{m,i}\ar@{.>>}[r] \ar@{.>}[urr]
& A_{m,i+1} \ar@{^{(}->}[r]
& \wt W(\wt B_{i+1})/p^m
}
.
\]
By the same diagram chasing as in the proof of Corollary~\ref{cor:fsperf-red}, it follows that the map $D_{m,i}\to A_{m',i'}$ for $(m',i')\in\{(m+1,i),(m,i+1)\}$ factors through the $\Zp$-flat closure $D_{m,i}^\fl$ by $\Zp$-flatness of $\wt W(
\wt B_{i})$ and $\wt W(\wt B_{i+1})$, which implies the liftability of $A_{m',i'}\thra A_{m,i}$ for $(m',i')\in\{(m+1,i),(m,i+1)\}$ by Lemma~\ref{lem:PDtorsion}.
\end{proof}

\subsection{Proof of Proposition~\ref{prop:descent}}\label{subsec:Pfdescent}
Proving the following proposition is the key step of the proof of Proposition~\ref{prop:descent}:
\begin{propsub}\label{prop:DescentDevissage}
Let $R'\thra R$ be a square-zero thickening of finitely generated $W/p^m$-algebras for some $m$. We give the square-zero PD structure on $\ker (R'\thra R)$. Let $(X,\iota)\in\RZ^\Lambda_{G,b}(R)$.

Then there exist a unique section $t_\alpha(R')\subset\DD(X)(R')^\otimes$ for each $\alpha$, such that its image in $\DD(X_{\wh R_x})(\wh R'_x)^\otimes$ coincides with $\hat t_{\alpha,x}(\wh R'_x)$ for any  closed point $x$  in $\Spec R$, where $(\hat t_{\alpha,x})$ is as in (\ref{eqn:tx})).
\end{propsub}

Before we prove the proposition, let us record the following lemma:
\begin{lemsub}\label{lem:fpqcdescent}
Let $R'$ be a noetherian ring. Then the product $\prod_x\wh R'_x$ over all closed points $x\in \Spec R'$ is faithfully flat over $R'$.
\end{lemsub}
\begin{proof}
By \cite[Theorem~2.1]{Chase:DirectProducts}, any \emph{direct product} of flat modules over a noetherian ring is again flat, which shows that $\prod_x \wh R'_x$ is flat over $R'$ when $R'$ is noetherian. Now, given any $\eta\in\Spec R'$ and a closed point  $\m$ specialising $\eta$, there exists $\hat\eta\in\Spec\wh R_x$ mapping to $\eta$, and we can view $\hat\eta$ as a point in  $\Spec (\prod_x \wh R_x)$. This shows the desired faithful flatness.
\end{proof}
Let us  outline the strategy to prove Proposition~\ref{prop:DescentDevissage}. By Lemma~\ref{lem:fpqcdescent}, we need to show that the collection $\{\hat t_{\alpha,x}(\wh R'_x)\}_x$ respects the descent datum with respect to $R'\to\prod_x\wh R'_x$, where $x$ runs over all  closed points of $\Spec R$. We want to obtain this compatibility from the quasi-isogeny $\iota:\BX_{R/p}\to X_{R/p}$, which is only possible if $R'\thra R$ is a \emph{liftable} PD thickening in the sense of Definition~\ref{def:LiftablePD}. In particular, we can prove Proposition~\ref{prop:DescentDevissage} if both $R$ and $R'$ are one of $A_{m,i}$ as in Corollary~\ref{cor:fsperf-appl}. In general, we use some fibre product construction to reduce to the square-zero thickenings appearing in Corollary~\ref{cor:fsperf-appl}.

\begin{lemsub}\label{lem:DescentDevissage-Liftable}
Proposition~\ref{prop:DescentDevissage} holds if $R'\thra R$ is a liftable PD thickening.
\end{lemsub}
\begin{proof}
Let us set up the notation. We choose a formally smooth formally finitely generated $W$-algebra $A\thra R'$.
 Let $D\thra R$ be the $p$-adically completed PD hull of $A\thra R'\thra R$. By liftability, the natural PD morphism $D\thra R'$ factors through the $W$-flat closure $D^\fl$ of $D$; \emph{cf.} Lemma~\ref{lem:PDtorsion}.

For a closed point $x$ in $\Spec R$, we consider the square-zero thickening $\wh R'_x\thra \wh R_x$ (with the ``square-zero PD structure). The completion $\wh A_x$ (viewing $x$ as a closed point of $\Spf A$) surjects onto $\wh R'_x$, and if we set 
\begin{equation}\label{eqn:PDFormalNbhd}
D^\fl_x := \varprojlim_m D^\fl/p^m\otimes_{A/p^m} \wh A_x/p^m
\end{equation}
turns out to be a $W$-flat $p$-adic PD thickening of $\wh R_x$, admitting a PD surjection onto $\wh R'_x$. (Indeed, 
 since $A/p^m\to\wh A_x/p^m$ is flat for any $m$, the PD surjection $D^\fl/p^m\thra R'$ for $m\gg1$ naturally extends to a PD surjection $D^\fl/p^m\otimes_{A/p^m} \wh A_x/p^m \thra R'\otimes_{A/p^m} \wh A_x/p^m = \wh R'_x$; \emph{cf.} \cite[Corollary~3.22]{Berthelot-Ogus}. The $W$-flatness of $D^\fl_x$ follows from the $W$-flatness of $D^\fl$.) 

Let us choose a finite-rank direct factor  $\E\subset \DD(X)^\otimes$ where all $(t_\alpha)$ factor through, so that $\E(D^\fl)$ is a finitely generated projective  $D^\fl$-module. Then  $\prod_x(\hat t_{\alpha,x}(D^\fl_x))$ defines an element in $\prod_x\E(D^\fl_x)\cong \E(D^\fl)\otimes_{D^\fl}(\prod_x D^\fl_x)$, and its image in $\E(R')\otimes_{R'}(\prod_x\wh R'_x)$ is given by $\prod_x(\hat t_{\alpha,x}(\wh R'_x))$. Here, $x$ runs over all closed points $x$ in $\Spec R$. 
Now, by faithfully flat descent and Lemma~\ref{lem:fpqcdescent}, we have a left exact sequence
\begin{equation}\label{eqn:equaliser}
\xymatrix@1{
\E(R') \ar[r] & \E(R')\otimes_{R'}(\prod_x\wh R'_x) \ar@<-.3ex>[r] \ar@<.3ex>[r] & \E(R')\otimes_{R'}(\prod_x\wh R'_x)^{\otimes2}
}.
\end{equation}
So to prove the lemma, it suffices to show that $\prod_x(\hat t_{\alpha,x}(\wh R'_x))$  is in the equaliser of the double arrows. We verify this using its lift $\prod_x(\hat t_{\alpha,x}(D^\fl_x))\in \E(D^\fl)\otimes_{D^\fl}(\prod_x D^\fl_x)$.

We first claim that $\prod_x(\hat t_{\alpha,x}(D^\fl_x))$ lies in the image of
\[
\E(D^\fl)[\ivtd p]\to \E(D^\fl)[\ivtd p]\otimes_{D^\fl}(\prod_x D^\fl_x).
\]
In fact, recall that the quasi-isogeny $\iota:\BX_{R/p}\to X_{R/p}$ induces the  morphism of $F$-isocrystals
$s_{\alpha,\DD}:\triv\to\DD(X_{\wh R_x})^\otimes[\ivtd p]$ (\emph{cf.} Definition~\ref{def:t}).
We claim that $\hat t_{\alpha,x}:\triv\to\DD(X_{\wh R_x/p})^\otimes$ induces $s_{\alpha,\DD}$ on the isocrystals. 
Indeed,  by \cite[Lemma~4.3]{deJong-Messing} any map of $F$-isocrystals $\triv\to\DD(X_{\wh R_x})^\otimes[\ivtd p]$ is determined by its fibre at the closed point $x$,\footnote{Although \cite[Lemma~4.3]{deJong-Messing} is stated for isocrystals over $\wh R_x/p$, we can apply it to isocrystals over $\wh R_x$ since  the restriction by $\Spec \wh R_x/p\hra\Spec \wh R_x$ is an equivalence of categories on the categories of (iso)crystals.} and $s_{\alpha,\DD}$ and $\hat t_{\alpha,x}$ coincide at any closed point $x$ in $\Spec R$ by definition of $\RZ^\Lambda_{G,b}$ (\emph{cf.} Definition~\ref{def:RZGloc}).
This shows that  after inverting $p$, $\prod_x(\hat t_{\alpha,x}(D_x^\fl))$ coincides with the image of  ``$s_{\alpha,\DD}(D^\fl[\ivtd p])$''$\in\E(D^\fl)[\ivtd p]$, so   we have
\[
\prod_x(\hat t_{\alpha,x}(D_x^\fl)) \in \ker \left(\E(D^\fl)\otimes_{D^\fl}(\prod_x D_x^\fl) \rightrightarrows \E(D^\fl)\otimes_{D^\fl}(\prod_x D^\fl_x)^{\otimes2}\right).
\]

We now conclude since the natural map induced by the  reduction
\begin{multline*}\label{eqn:equaliserR}
 \ker \left(\E(D^\fl)\otimes_{D^\fl}(\prod_x D_x^\fl) \rightrightarrows \E(D^\fl)\otimes_{D^\fl}(\prod_x D^\fl_x)^{\otimes2}\right)\\
 \stackrel{\sigma^{n*}}{\longrightarrow}  \ker \left(\E(R')\otimes_{R'}(\prod_x \wh R'_x) \rightrightarrows \E(R')\otimes_{R'}(\prod_x \wh R'_x)^{\otimes2}\right) \liso \E(R'),
\end{multline*}
sends $\prod_x(\hat t_{\alpha,x}(D_x^\fl))$ to $\prod_x(\hat t_{\alpha,x}(\wh R'_x))$, so we obtain the desired tensor  $t_\alpha(R')\in\E(R')$ from $\prod_x(\hat t_{\alpha,x}(\wh R'_x))$ by faithfully flat descent theory (\emph{cf. }(\ref{eqn:equaliser})).
\end{proof}

To prove Proposition~\ref{prop:DescentDevissage} in general, we need more lemmas on lifting $(X,\iota)\in\RZ^\Lambda_{G,b}(R)$ under nilpotent thickenings (\emph{cf.} Lemmas~\ref{lem:ForSmFT}, \ref{lem:SchlessingerRimFType}).
\begin{lemsub}\label{lem:ForSmFT}
Let $R'\thra R$ and $(X,\iota)\in\RZ^\Lambda_{G,b}(R)$ be as in Proposition~\ref{prop:DescentDevissage}, and \emph{assume} that   there exists $t_\alpha(R')\in\DD(X)(R')^\otimes$ for each $\alpha$ glueing $\hat t_{\alpha,x}(\wh R'_x)$ for any closed point $x$ in $\Spec R$. (In other words, we assume that the conclusion of Proposition~\ref{prop:DescentDevissage} holds for $(X,\iota)$ and $R'\thra R$.)
\begin{enumerate}
\item\label{lem:ForSmFT:red}
The following $R'$-scheme is a $G$-torsor:
 \[
 \cP_{R'}:=\nf\Isom_{R'}\big([\DD(X)(R'), (t_\alpha(R'))], [R'\otimes_{\Zp}\Lambda^*,(1\otimes s_\alpha)] \big),
 \]
\item \label{lem:ForSmFT:Fil}
The Hodge filtration $\Fil^1_X\subset \DD(X)(R)$ is a $\set\mu$-filtration with respect to the image $(t_\alpha(R))\subset \DD(X)(R)^\otimes$ of $(t_\alpha(R'))$.
\item \label{lem:ForSmFT:GM}
Let $(X',\iota')\in\RZ_\BX(R')$ be a lift of $(X,\iota)$. Then we have $(X',\iota')\in\RZ^\Lambda_{G,b}(R')$ if and only if the Hodge filtration $\Fil^1_{X'}\subset \DD(X')(R')\cong\DD(X)(R')$ is a $\set\mu$-filtration with respect to $(t_\alpha(R'))$.
\end{enumerate}
Furthermore, $(X,\iota)$ can be lifted to some $(X',\iota')\in\RZ^\Lambda_{G,b}(R')$.
\end{lemsub}
By Lemma~\ref{lem:DescentDevissage-Liftable}, the assumption of the proposition is satisfied if $R'\thra R$ is a liftable PD thickening.
\begin{proof}
To show (\ref{lem:ForSmFT:red}),  it suffices to show that $\cP_{R'}\times_{\Spec R'}\Spec \wh R'_x$  is a (necessarily trivial) $G$-torsor for any closed point $x$. Let $\hat f_x:\Spf \wh R_x \to \Def_{X_x,G} = \Spf A_{G,x}$ be the map corresponding to the deformation $X_{\wh R_x}$ of $X_x$, and let $[\bM_{G,x},(\hat t^{\univ}_{\alpha,x}(A_{G,x}))]$ be the universal deformation of $[\DD(X_x)(W),(t_{\alpha,x}(W))]$; \emph{cf.} \S\ref{subsec:FaltingsDeforConstr}. Then for any lift $\hat f'_x:\Spf \wh R'_x\to\Spf A_{G,x}$ of $\hat f_x$, we have an isomorphism $\DD(X_{\wh R_x})(\wh R'_x) \cong \hat f_x'^*\bM_{G,x}$ preserving the tensors, and by construction we have an isomorphism $\bM_{G,x} \cong A_{G,x}\otimes_{\Zp}\Lambda$ matching $(\hat t^{\univ}_{\alpha,x}(A_{G,x}))$ with $(1\otimes s_\alpha)$; \emph{cf.} (\ref{eqn:UnivDefCrys}), (\ref{eqn:tuniv}).

By  (\ref{lem:ForSmFT:red}), we have a notion of $\set\mu$-filtration on $\DD(X)(R)$ and $\DD(X)(R')$ compatible with the notion of $\set\mu$-filtration on the formal neighbourhood of a closed point $x$ in $\Spec R$ defined by $(\hat t_{\alpha,x})$. Note that $\wh R_x\otimes_R\Fil^1_X = \hat f_x^*(\Fil^1\bM_{G,x})$ is a $\{\mu\}$-filtration with respect to $(\hat t_{\alpha,x}(\wh R_x))$ for any closed point $x$ in $\Spec R$ (\emph{cf.} (\ref{eqn:UnivDefCrys})), so we obtain (\ref{lem:ForSmFT:Fil}) by Corollary~\ref{cor:muFilviaCompl}.

By Corollary~\ref{cor:muFilviaCompl} and the definition of $\RZ^\Lambda_{G,b}$ (\emph{cf.} Definition~\ref{def:RZGloc})), (\ref{lem:ForSmFT:GM}) is equivalent to the claim that for any closed point $x$ in $\Spec R$, $X'_{\wh R_x'}$ comes from a lift $\hat f'_x:\Spf \wh R'_x\to \Spf A_{G,x}$ if and only if its Hodge filtration is a $\set\mu$-filtration, which is proved in Proposition~\ref{prop:LiftingTate}.

Since $\Fil^1_X\subset\DD(X)(R)$ can be lifted to a $\set\mu$-filtration of $\DD(X)(R')$ by Lemma~\ref{lem:muFil}, the existence of lift $(X',\iota')\in\RZ^\Lambda_{G,b}(R')$ of $(X,\iota)$ now follows from (\ref{lem:ForSmFT:GM}) and the Grothendieck-Messing deformation theory.
\end{proof}

Let $R$, $R'$, and $B$ be finitely generated $W/p^m$-algebras for some $m\geqslant 1$, and we have morphisms $B\to R$ and $R'\to R$ of $W/p^m$-algebras.  Assume that $B\to R$ is surjective with the kernel $\bb$ killed by the nil-radical of $B$, and $R'_{\red}\thra R_{\red}$ is surjective. We set $B':=B\times_R R'$, which is finitely generated over $W/p^m$. By assumption, we can regard any closed point $x\in\Spec R$ as a point of $\Spec R'$, $\Spec B$ and $\Spec B'$. Since $\RZ_\BX$ is representable by a formal scheme, we have a natural bijection (\emph{cf.} \cite[Corollary~5.4]{Artin:VersalDefAlgStack}).
\begin{equation}\label{eqn:SchlessingerRimRZ}
\RZ_\BX(B')\riso\RZ_\BX(B)\times_{\RZ_\BX(R)}\RZ_\BX(R').
\end{equation}
\begin{lemsub}\label{lem:SchlessingerRimFType}
In the above setting, the natural map $\RZ^\Lambda_{G,b}(B')\ra\RZ^\Lambda_{G,b}(B)\times_{\RZ^\Lambda_{G,b}(R)}\RZ^\Lambda_{G,b}(R')$, obtained by restricting the isomorphism~(\ref{eqn:SchlessingerRimRZ}),  is a bijection.
\end{lemsub}
\begin{proof}
The lemma follows from the isomorphism 
\[\Def_{X_x,G}(\wh B'_x) \riso \Def_{X_x,G}(\wh B_x)\times_{\Def_{X_x,G}(\wh R_x)}\Def_{X_x,G}(\wh R'_x),\]
which holds since $\Def_{X_x,G}$ is a formal scheme and we have  $\wh B'_x\riso \wh B_x\times_{\wh R_x}\wh R'_x$.
\end{proof}

\begin{proof}[Proof of Proposition~\ref{prop:DescentDevissage}]
Let $R'\to R$ be a square-zero thickening of finitely generated $W/p^{m_0}$-algebras for some $m_0\geqslant1$. We choose a formally smooth formally finitely generated $W$-algebra $A$ surjecting onto $R'$ with $\ker(A\thra R')$ an ideal of definition. (To obtain this, we choose a polynomial algebra $A_0$ over $W$ surjecting onto $R'$ and let $A$ be the completion of $A_0$ for $\ker (A_0\thra R')$; \emph{cf.} \cite[\S1.3]{dejong:crysdieubyformalrigid}.)

Let us choose a sequence of square-zero liftable PD thickenings $\{A_{m,i}\}$ of $A$ as in Corollary~\ref{cor:fsperf-appl}. 
Since we have $A/p^{m_0} = \varprojlim_i A_{m_0,i}$ by construction of $A_{m_0,i}$ (\emph{cf.}  Corollary~\ref{cor:fsperf-appl}), 
 there exist $(m_0,i_0)$ such that $A_{m_0,i_0}\thra R'\thra R \thra R_\red = A_{1,1}$. We consider the following sequence of square-zero liftable PD thickenings
 \[
 A_N:=A_{m_0,i_0} \thra \cdots\thra A_{j+1}\thra A_j\thra \cdots A_0 = R_\red
 \]
obtained by refining the sequence $A_{m_0,i_0} \thra A_{m_0,i_0-1}\thra \cdots A_{m_0,1}\thra A_{m_0-1,1}\thra \cdots \thra R_{\red}$ 
so that $\ker (A_{j+1}\thra A_j)$ is killed by the nil-radical for any $j$. Since any refinement of  square-zero \emph{liftable} PD thickening is still liftable (\emph{cf.} Lemma~\ref{lem:LiftablePDRefine}), $A_{j+1}\thra A_j$ is a liftable square-zero PD thickening for any $j$.
 
 Recall that Proposition~\ref{prop:DescentDevissage} is already proved for the thickenings $A_{j+1}\thra A_j$ by liftability (\emph{cf.} Lemma~\ref{lem:DescentDevissage-Liftable}). To utilise this special case  to obtain the general case of Proposition~\ref{prop:DescentDevissage},  we will show that any $(X,\iota)\in\RZ^\Lambda_{G,b}(R)$ can be lifted to $(X_{A_N},\iota)\in\RZ^\Lambda_{G,b}(A_{N})$ by $A_N\thra R$.

We write $J:=\ker (A\thra R)$. For any $j$, we set $I_{j}:=\ker (A\thra A_{j})$ and define $R_j:=A/(J+I_{j})$, which is a simultaneous quotient of $R$ and $A_{j}$. Let $(X_{R_{j}},\iota)\in \RZ^\Lambda_{G,b}(R_{j})$ denote the pull-back of $(X,\iota)\in\RZ^\Lambda_{G,b}(R)$ by $R\thra R_{j}$. (To ease the notation, we do not keep track of the base ring of $\iota$ under nilpotent thickenings as $\iota$'s uniquely lift by rigidity of quasi-isogenies.) 
 
 Assuming that there exists $(X_{A_{j}},\iota)\in \RZ^\Lambda_{G,b}(A_{j})$ lifting $(X_{R_{j}},\iota)$ for some $j$, we want to construct a simultaneous lift $(X_{A_{j+1}},\iota)\in \RZ^\Lambda_{G,b}(A_{j+1})$  of $(X_{A_{j}},\iota)$ and $(X_{R_{j+1}},\iota)$. (Note that the induction hypothesis is automatic for $j=0$ since $A_{0} = R_{0} = R_\red$.)
 
We define
\[
B_{j+1}:= A/((J\cap I_{j})+I_{j+1}) \riso R_{j+1} \times_{R_{j}}A_{j} = A/(J+ I_{j+1})\times_{A/(J+ I_{j})}A/I_{j};
\]
indeed, for any ring $A$ and ideals $\fa, \fa'\subset A$ the diagonal map $A/(\fa\cap\fa') \ra A/\fa\times_{A/(\fa+\fa')}A/\fa'$ is an isomorphism,\footnote{The injectivity of the map $A/(\fa\cap\fa') \ra A/\fa\times_{A/(\fa+\fa')}A/\fa'$  is clear. It remains to show that for $a,a'\in  A$ such that $a\equiv a'\bmod{\fa+\fa'}$,  the element $(a\bmod\fa,a'\bmod{\fa'})$ is in the image of $A/(\fa\cap\fa')$. For this, we may replace $(a,a')$ with $(0,a'-a)$, where $a'-a\in\fa+\fa'$. Now the claim follows from the isomorphism $(\fa+\fa')/\fa' \cong \fa/(\fa\cap\fa')$. \label{footnote:FibreProd}} and we apply it to $\fa=J+I_{j+1}$ and $\fa'= I_{j}$. Let us give a diagram of the rings that appeared thus far:
\begin{equation}\label{eqn:DiagramLiftablePD}
\xymatrix{
A_N\ar@{->>}[r]\ar@{->>}[dr]\ar@{->>}[d]
&\cdots\ar@{->>}[r]& 
A_{j+1}\ar@{->>}[r] \ar@{->>}[dr]&
B_{j+1} = R_{j+1} \times_{R_{j}}A_{j} \ar@{->>}[r]\ar@{->>}[d] &
A_j \ar@{->>}[d]\\
R'\ar@{->>}[r]&R=R_N \ar@{->>}[r] &\cdots\ar@{->>}[r]&
R_{j+1} \ar@{->>}[r] &
R_j}
\end{equation}
And we have $(X_{R_{j+1}},\iota)\in\RZ^\Lambda_{G,b}(R_{j+1})$ and $(X_{A_{j}},\iota)\in \RZ^\Lambda_{G,b}(A_{j})$, both of which lift $(X_{R_{j}},\iota)\in \RZ^\Lambda_{G,b}(R_j)$.

Since $\ker(A_{j+1}\thra A_j)$ is killed by the nil-radical by assumption, we can apply Lemma~\ref{lem:SchlessingerRimFType} to obtain  $(X_{B_{j+1}},\iota)\in\RZ^\Lambda_{G,b}(B_{j+1})$ which simultaneously lifts  $(X_{R_{j+1}},\iota)$ and $(X_{A_j},\iota)$. Now, since $A_{j+1}\thra A_j$ is a square-zero liftable PD thickening, so is $A_{j+1}\thra B_{j+1}$ (by applying Lemma~\ref{lem:LiftablePDRefine} to the top row of (\ref{eqn:DiagramLiftablePD}). Therefore, by Lemmas~\ref{lem:ForSmFT} and \ref{lem:DescentDevissage-Liftable}, there exists a lift $(X_{A_{j+1}},\iota)\in\RZ^\Lambda_{G,b}(A_{j+1})$ of $(X_{B_{j+1}},\iota)$. This completes the induction claim.

By induction on $j$, we eventually obtain a lift $(X_{A_N},\iota)\in\RZ^\Lambda_{G,b}(A_N)$ of $(X,\iota)\in\RZ^\Lambda_{G,b}(R)$. (Since $A_N\thra R$, we have $R=R_N $.) 
Note that the trivial thickening $A_{N}\xra{\id}A_{N}$ is liftable in this case (\emph{cf.} Corollary~\ref{cor:fsperf-appl}, Lemma~\ref{lem:LiftablePDRefine}), so  we get $(t_\alpha(A_N))\subset\DD(X_{A_N})(A_{N})^\otimes$ glueing $\{\hat t_{\alpha,x}(\wh A_{N,x})\}_x$ by  Lemma~\ref{lem:DescentDevissage-Liftable} (applied to the trivial thickening $A_{N}\xra{\id}A_{N}$). Since we chose $A_N$ to also surject onto $R'$, we have $R'\otimes_{A_N}\DD(X_{A_N})(A_N)\cong \DD(X)(R')$. Therefore, for any $\alpha$ the image $t_{\alpha}(R')\in \DD(X)(R')^\otimes$ of $t_{\alpha}(A_N)$ glues $\{\hat t_{\alpha,x}(\wh R'_x)\}_x$, as claimed in Proposition~\ref{prop:DescentDevissage}.
\end{proof}

To sum up, we have proved Proposition~\ref{prop:DescentDevissage} for any square-zero thickenings of $\Nilp_W^\ft$ (not just for square-zero liftable PD thickenings), so it follows from Lemma~\ref{lem:ForSmFT} that $\RZ^\Lambda_{G,b}$ is formally smooth as a functor on $\Nilp_W^\ft$.

Let us now prove Proposition~\ref{prop:descent}(\ref{prop:descent:devissage}).
\begin{propsub}\label{prop:Pfdescent}
Let $A$ be a formally smooth formally finitely generated $W$-algebra with ideal of definition $J$. For any $\tilde f:\Spf(A,J)\to \RZ_\BX$,
we have $\tilde f\in \RZ^{(s_\alpha)}_{\BX,G}(A)$ if and only if  $\tilde f|_{\Spec A/J^i}\in\RZ_\BX(A/J^i)$ lies in $\RZ^\Lambda_{G,b}(A/J^i)$ for any  $i\geqslant1$.
\end{propsub}
This proposition shows Proposition~\ref{prop:descent}(\ref{prop:descent:devissage}) for  formally smooth $W$-algebras. Later, we will deduce Proposition~\ref{prop:descent}(\ref{prop:descent:devissage}) for $\Nilp^\sm_W$ from this case.
\begin{proof}
The ``only if'' direction is trivial, and we show the ``if'' direction now.
Assume that we have $(X^{(i)},\iota^{(i)})\in\RZ^\Lambda_{G,b}( A/J^i)$, where $(X^{(i)},\iota^{(i)})$ corresponds to $\tilde f|_{\Spec A/J^i}$.
To show that $\tilde f\in\RZ^{(s_\alpha)}_{\BX,G}(A)$, we need to produce morphisms $\tilde t_\alpha:\triv\to \DD(\wt X)^\otimes$ of crystals over $\Spf (A,(p))$ satisfying Definition~\ref{def:RZG}. (Recall that this morphism is determined by its restriction $\tilde t_\alpha:\triv\to\DD(\wt X_{A/p})^\otimes$ over $\Spec A/p$.) Now, by Lemma~\ref{lem:CrysConn} (applied to morphisms of crystals over $\Spec A/p$), giving a morphism $\tilde t_\alpha:\triv\to \DD(\wt X)^\otimes$ of crystals over $\Spf (A,(p))$ is equivalent to giving a horizontal section $\tilde t_\alpha(A)\in \DD(\wt X)(A)$ with respect to the crystalline connection $\nabla$ (\ref{eqn:CrysConn}).

Let us first construct a candidate for $\tilde t_\alpha(A)$. For each $(X^{(i)},\iota^{(i)})\in\RZ^\Lambda_{G,b}(A/J^i)$, we obtain $(t_\alpha(A/J^i))\subset\DD(X^{(i)})(A/J^i)$ by Proposition~\ref{prop:DescentDevissage}. We claim that  $t_\alpha(A/J^i)$'s are compatible with respect to the natural projections induced by $A/J^{i+1}\thra A/J^i$, so we get a section
\[
\tilde t_\alpha(A):=\varprojlim_i t_\alpha( A/J^i) \in \varprojlim_i\DD(X^{(i)})(A/J^i) \cong \DD(\wt X)(A).
\]
Indeed, $t_\alpha(A/J^i)$'s are determined by $\{\hat t_{\alpha,x}((A/J^i)\wh{_x})\}$'s where $x$ runs over closed points of $\Spf A$ (\emph{cf.} Proposition~\ref{prop:DescentDevissage}), and these formal local tensors are compatible with varying $i$ by the construction of $\hat t_{\alpha,x}$; \emph{cf.}  (\ref{eqn:tx}). Now, we obtain the desired compatibility with varying $i$ as $t_\alpha(A/J^i)$'s are uniquely determined by $\{\hat t_{\alpha,x}((A/J^i)\wh{_x})\}_x$ for all closed points $x$ in $\Spec A/J^i$.

We have just obtained a section $\tilde t_\alpha(A)\in \DD(\wt X)(A)$. Furthermore,  the image of $\tilde t_\alpha(A)$ in $\DD(\wt X_{\wh A_x})(\wh A_x)$ coincides with the pull-back $\hat{\tilde f}_x^*(\hat t^{\univ}_{\alpha,x})$ of the universal Tate tensor $\hat t^{\univ}_{\alpha,x}$ (\emph{cf.} (\ref{eqn:tuniv}), (\ref{eqn:tx})) via the map $\hat{\tilde f}_x:\Spf \wh A_x \to\Def_{X_x,G}$ induced by $\tilde f$.

 We next claim that for any $\alpha$,  $\tilde t_\alpha(A)$ is horizontal with respect to the crystalline connection $\nabla$ (i.e., $\nabla(\tilde t_\alpha(A))=0$). Since the crystalline connection $\nabla$ is $J$-adically continuous (\emph{cf.} Remark~\ref{rmk:JadicConn}), we can check the vanishing of $\nabla(\tilde t_\alpha(A))=0$ on the formal neighbourhood of any closed point $x$ in $\Spf A$.
 On the other hand, since the image of $\tilde t_\alpha(A)$ in $\DD(\wt X_{\wh A_x})(\wh A_x)^\otimes$ coincides with $(\hat{\tilde f}_x^*(\hat t^{\univ}_{\alpha,x})(\wh A_x)$, which is horizontal (\emph{cf.} (\ref{eqn:tuniv})),
we conclude that $\tilde t_\alpha(A)$ is horizontal.

By horizontality of $\tilde t_\alpha(A)\in \DD(\wt X)(A)$,
we obtain a morphism of crystals $\tilde t_\alpha:\triv\to\DD(\wt X)^\otimes$ over $\Spf(A,(p))$ associated to  $\tilde t_\alpha(A)$. Now, we verify Definition~\ref{def:RZG} for $\tilde f:\Spf A\to\RZ_\BX$ and $(\tilde t_\alpha)$.\footnote{I.e., we verify  Definition~\ref{def:RZG} for $\tilde f_m:\Spf A/p^m\to\RZ_{\BX}$ induced by $\tilde f$ and $(\tilde t_\alpha)$ for any $m$.}
We first note that (\ref{def:RZG:P}) and  (\ref{def:RZG:Kottwitz}) of Definition~\ref{def:RZG} can be verified for $\hat{\tilde f}_x$ and $(\tilde t_\alpha|_{\wh A_x})$
for each closed point $x$ of $\Spf A$. (Indeed, the $A$-scheme $\cP_{A}$ is a $G$-torsor if and only if this holds over any $\wh A_x$. Likewise, Corollary~\ref{cor:muFilviaCompl} show that Definition~\ref{def:RZG}(\ref{def:RZG:Kottwitz}) can be verified for $\wh A_x$ at each  $x$.) On the other hand, since $\hat{\tilde f}_x:\Spf \wh A_x \to\Def_{X_x,G}$ and $\tilde t_\alpha|_{\wh A_x}=\hat{\tilde f}_x^*(\hat t^{\univ}_{\alpha,x})$, it follows that $\hat{\tilde f}_x$ and $(\tilde t_\alpha|_{\wh A_x})$ satisfy
(\ref{def:RZG:P}) and  (\ref{def:RZG:Kottwitz}) of Definition~\ref{def:RZG}; \emph{cf.} (\ref{eqn:RZG:FaltingsDefor}).

It remains to show that show that both $\tilde t_\alpha$ and $s_{\alpha,\DD}$ define the same map of $F$-isocrystals $\triv\to\DD(X_{R/p})^\otimes[\ivtd p]$ (\emph{cf.} Definition~\ref{def:RZG}(\ref{def:RZG:qisog})).
Let $D\thra R/p$ denote the $p$-adically completed PD hull of $A\thra R/p$. Then by \cite[Proposition~3.21]{Berthelot-Ogus}, the $p$-adic completion of $D\otimes_A\wh A_x$ (for any closed point $x$ in $\Spf A$) is the $p$-adically completed PD hull of $\wh A_x\thra \wh R_x/p$. Since the maps $\triv\to\DD(X_{R/p})^\otimes[\ivtd p]$ induced by $\tilde t_\alpha$ and $s_{\alpha,\DD}$ are determined by their sections over $D$, they can be compared at the formal neighbourhood at each closed point $x$ in $\Spf A$. But since any map of isocrystals over excellent local ring is determined by the fibre at the closed point (\emph{cf.} \cite[Lemma~4.3]{deJong-Messing}), it suffices to show that $\tilde t_\alpha$ and $s_{\alpha,\DD}$ induce the same map $\triv\to\DD(X_x)^\otimes[\ivtd p]$ on the fibre of any closed point $x$ in $\Spf A$. But this is clear from the definition of $\RZ^\Lambda_{G,b}(R)$; \emph{cf.} Definition~\ref{def:RZGloc}.
 \end{proof}

The following lemma handles the remaining case of Proposition~\ref{prop:descent}(\ref{prop:descent:devissage}).
\begin{lemsub}\label{lem:Pfdescent}
Let $B\in\Nilp^{\sm}_W$  with  the maximal ideal of definition $J_B$. 
Then for any $f:\Spf(B,J_B)\to \RZ_\BX$,
we have $f\in \RZ^{(s_\alpha)}_{\BX,G}(B)$ if and only if  $f|_{\Spec B/J_B^i}\in\RZ_\BX(B/J_B^i)$ lies in $\RZ^\Lambda_{G,b}(B/J_B^i)$ for any  $i\geqslant1$.
\end{lemsub}
\begin{proof}
The ``only if'' direction is trivial. To show the ``if'' direction, we choose a formally smooth $p$-adic $W$-lift $A$ of $B$ (so that we have $A/p^m = B$ for some $m$), and set $J:=\ker (A\thra B/J_B)$, which is the maximal ideal of definition of $A$. Assume that we have $(X^{(i)},\iota^{(i)})\in\RZ^\Lambda_{G,b}(B/J_B^i)$, where $(X^{(i)},\iota^{(i)})$ corresponds to $f|_{\Spec B/J_B^i}$. By Proposition~\ref{prop:Pfdescent}, it suffices to find a lift $(\wt X^{(i)},\iota^{(i)})\in\RZ^\Lambda_{G,b}(A/J_B^i)$ of $(X^{(i)},\iota^{(i)})$  so that $\{(\wt X^{(i)},\iota^{(i)})\}_i$ forms a projective system. Granting this, the projective system $\{(\wt X^{(i)},\iota^{(i)})\}_i$ corresponds to a map $\tilde f:\Spf A\to \RZ_\BX$ lifting $f$, and by Proposition~\ref{prop:Pfdescent} we have $\tilde f\in  \RZ^{(s_\alpha)}_{\BX,G}(A)$. Therefore we have $f\in\RZ^{(s_\alpha)}_{\BX,G}(B)$.

It remains to construct a lift $(\wt X^{(i)},\iota^{(i)})\in\RZ^\Lambda_{G,b}(A/J^i)$ inductively on $i$. In other words, given  a lift $(\wt X^{(i)},\iota^{(i)})$ of $(X^{(i)},\iota^{(i)})$ for some $i$, we want to find $(\wt X^{(i+1)},\iota^{(i+1)})\in\RZ^\Lambda_{G,b}(A/J^{i+1})$ simultaneously lifting $(X^{(i+1)},\iota^{(i+1)})$ and $(\wt X^{(i)},\iota^{(i)})$. For this, we consider
\[
R_i:= A/(J^i\cap(p^m,J^{i+1})) \riso A/J^i \times_{B/J_B^i} B/J_B^{i+1} = A/J^i \times_{A/(p^m,J^i)}A/(p^m,J^{i+1}),
\]
where the isomorphism follows since  $(p^m, J^i)=(p^m,J^{i+1})+J^i$. (\emph{Cf.} footnote \#\ref{footnote:FibreProd} on p.~\pageref{footnote:FibreProd}.)
By Lemma~\ref{lem:SchlessingerRimFType}, there is a unique element $(X_{R_i},\iota_{R_i/p})\in \RZ^\Lambda_{G,b}(R_i)$ simultaneously lifting $(X^{(i+1)},\iota^{(i+1)})$ and $(\wt X^{(i)},\iota^{(i)})$. 

Now, by the formal smoothness property of $\RZ^\Lambda_{G,b}$ applied to $ A/J^{i+1}\thra R_i$, we can find a lift $(\wt X^{(i+1)},\iota^{(i+1)})\in \RZ^\Lambda_{G,b}( A/J^{i+1})$ of $(X_{R_i},\iota_{R_i/p})$. By repeating this for all $i$, we obtain a desired  projective system $\{(\wt X^{(i)},\iota^{(i)})\in\RZ^\Lambda_{G,b}(A/J^i)\}$ lifting $\{(X^{(i)},\iota^{(i)})\in\RZ^\Lambda_{G,b}(B/J_B^i)\}$.
\end{proof}

We are now ready to conclude the proof of Proposition~\ref{prop:descent}::
 \begin{lemsub}\label{lem:PfdescentLiftability}
 Proposition~\ref{prop:Pfdescent} implies Proposition~\ref{prop:descent}(\ref{prop:descent:t}) and (\ref{prop:descent:liftability}).
 \end{lemsub}
 \begin{proof}
Let $(X,\iota)\in \RZ^\Lambda_{G,b}(R)$ for some finitely generated $W/p^m$-algebra $R$ for some $m$. We want to deduce the following claims from  Proposition~\ref{prop:Pfdescent}:
\begin{description}
\item[Proposition~\ref{prop:descent}(\ref{prop:descent:t})]
For any $\alpha$, there exists a unique morphism of crystals
\[t_\alpha:\triv\to\DD(X)^\otimes\]
that  glues $\{\hat t_{\alpha,x}\}$ (\ref{eqn:tx}) and induces $s_{\alpha,\DD}:\triv\to\DD(X)^\otimes[\ivtd p]$ on the $F$-isocrystals.
\item[Proposition~\ref{prop:descent}(\ref{prop:descent:liftability})]
Let $f:\Spec R\to \RZ_\BX$ be the map corresponding to $(X,\iota)$. Then for any  formally smooth formally finitely generated $W$-algebra $A$ surjecting onto $R$, there exists a lift $\tilde f:\Spf A\to\RZ_\BX$ with
$\tilde f\in\RZ^{(s_\alpha)}_{\BX,G}(A)$.
Furthermore, if we let $\wt X$ denote the $p$-divisible group over $ A$ corresponding to $\tilde f$, then  the unique tensor $t_\alpha:\triv\to\DD(X)^\otimes$ (as in Proposition~\ref{prop:descent}(\ref{prop:descent:t})) lifts to a map $\tilde t_\alpha:\triv\to \DD(\wt X)^\otimes$ (as in  Definition~\ref{def:RZG}).
\end{description}

Given any $R\in\Nilp^\ft_W$, we choose a formally smooth formally finitely generated $W$-algebra $A$ so that $A/J\cong R$ for some ideal of definition $J$.  (For example, $A$ can be constructed as a completion of some polynomial algebra.) Let $(X,\iota)\in\RZ^\Lambda_{G,b}$, which corresponds to a map $f:\Spec R\to\RZ_\BX$. 
By the formal smoothness of $\RZ^\Lambda_{G,b}$ as a functor on $\Nilp^\ft_W$ (\emph{cf.} Lemma~\ref{lem:ForSmFT}, Proposition~\ref{prop:DescentDevissage}), we obtain a projective system of lifts $\{(X^{(i)},\iota^{(i)})\in \RZ^\Lambda_{G,b}(\wt A/\wt J^i)\}$ of $(X,\iota)\in\RZ^\Lambda_{G,b}(R)$. Then by Proposition~\ref{prop:Pfdescent}, such a projective system gives rise to $\tilde f\in\RZ^{(s_\alpha)}_{X,G}(A)$, and furthermore, if we let $\wt X$ denote the pull-back of the universal $p$-divisible group by $\tilde f$, then we have Tate tensors $\tilde t_\alpha:\triv\to\DD(\wt X)^\otimes$ satisfying Definition~\ref{def:RZG}. By  Definition~\ref{def:RZG}(\ref{def:RZG:qisog}) and Lemma~\ref{lem:FormalNbdTensors}, the restriction of the Tate tensors $\tilde t_\alpha:\triv\to\DD(\wt X)^\otimes$ over $\Spec R$ satisfies  the requirements for $(t_\alpha)$.

Now, it remains to show the uniqueness of $t_\alpha:\triv\to\DD(X)^\otimes$.
Let $t_\alpha':\triv\to\DD(X)^\otimes$ be another morphism satisfying the requirements for $t_\alpha$ as in Proposition~\ref{prop:descent}(\ref{prop:descent:t}). 
Let $D\thra R$ denote the $p$-adically completed PD hull of $A\thra R$, and let $D_x$ be the $p$-adic completion of $D\otimes_{A}\wh A_x$ for any closed point $x$ in $\Spec R$, which is a $p$-adic PD thickening of $\wh R_x$. (The argument is identical for $D^\fl_x$ in the proof of Lemma~\ref{lem:DescentDevissage-Liftable}.)  Then by the compatibility with $s_{\alpha,\DD}$, it follows that $u_\alpha:=t_\alpha(D)-t'_\alpha(D)\in\DD(X)(D)^\otimes$ is killed by some power of $p$, and its image in $\DD(X_{\wh R_x})(D_x)^\otimes$ is zero for any closed point $x$. We choose a finitely generated $p$-power torsion $ A$-submodule $M\subset \DD(X)(D)^\otimes$ containing all $u_\alpha$'s, then $M\otimes_{A} \wh A)_x\subset \DD(X_{\wh R_x})(D_x)^\otimes$. Therefore, $u_\alpha\in M$ has the property that its image in $M\otimes_{A} \wh A_x$ is zero for any closed point $x$ in $\Spec R$. This forces $u_\alpha = 0$ since $M$ injects into
\[
M\otimes_{A}(\prod_x \wh A_x) \cong \prod_x (M\otimes_{A}\wh A_x)
\]
by faithful flatness of $A\to \prod_x \wh A_x$; \emph{cf.} Lemma~\ref{lem:fpqcdescent}. (We have the isomorphism in the displayed equation  since $M$ is finitely presented over $A$.) This shows the desired uniqueness of $t_\alpha$ by Lemma~\ref{lem:CrysConn}.
 \end{proof}
 \begin{proof}[Proof of Proposition~\ref{prop:descent}]
 Proposition~\ref{prop:descent}(\ref{prop:descent:devissage}) follows from Proposition~\ref{prop:Pfdescent} and Lemma~\ref{lem:Pfdescent}. This implies the rest of  Proposition~\ref{prop:descent} by Lemma~\ref{lem:PfdescentLiftability}.
 \end{proof}

\subsection{}
We extend $\RZ^\Lambda_{G,b}$ to a functor on $\Nilp_W$ as follows.
\begin{defnsub}\label{def:RZGcolim}
For any $R'\in\Nilp_W$,  $(X',\iota')\in\RZ_\BX(R')$ lies in $\RZ^\Lambda_{G,b}(R') $ if and only if for some finitely generated $W$-subalgebra $R\subset R'$, there exists $(X,\iota)\in\RZ^\Lambda_{G,b}(R)$ which descends $(X',\iota')$. 
\end{defnsub}

\begin{lemsub}\label{lem:lpf}
The functor $\RZ^\Lambda_{G,b}$ on $\Nilp_W$ commutes with  filtered direct limits.
\end{lemsub}
\begin{proof}
Let $\set{R_\xi}$ be a filtered direct system of rings in $\Nilp_W$, and set $R:=\varinjlim_\xi R_\xi$. We want to show that for any $(X,\iota)\in\RZ^\Lambda_{G,b}(R)$  there exists $(X_\xi,\iota_\xi)\in\RZ^\Lambda_{G,b}(R_\xi)$ which pulls back to $(X,\iota)$. Such $(X_\xi,\iota_\xi)$ is essentially unique in the sense that if there is another point $(X_{\xi'},\iota_{\xi'})\in \RZ^\Lambda_{G,b}(R_{\xi'})$ for some $\xi'$, there exists $\xi''\geqslant \xi, \xi'$ such that both $(X_\xi,\iota_\xi)$ and $(X_{\xi'},\iota_{\xi'})$ map to the same point in $\RZ^\Lambda_{G,b}(R_{\xi''})$; indeed, this property can be checked for $\RZ_\BX$ as $\RZ^\Lambda_{G,b}$ is a subfunctor of $\RZ_\BX$, and $\RZ_\BX$ commutes with  filtered direct limits in $\Nilp_W$ as it is locally formally of finite type over $W$.

Let us first handle the case when $R$ and $R_\xi$ are finitely generated over $W$. In that case, there exists $\xi$ such that the natural map $R_\xi\ra R$ is surjective and admits a section $R\ra R_\xi$.\footnote{Indeed, we can choose $\xi_0$ such that  $R_{\xi_0}$ surjects onto $R$. Let $I:=\ker( R_{\xi_0}\thra R)$, and choose finitely many generators $f_i\in I$. Then since $R=\varinjlim_\xi R_\xi$, each $f_i$ should map to $0$ in  $R_{\xi_i}$ for some $\xi_i\geqslant \xi_0$. Therefore, there exists $\xi\geqslant\xi_0$ so that the transition map $R_{\xi_0}\to R_\xi$ factors through $R_{\xi_0}/I\cong R$.} Now, we set $(X_\xi,\iota_\xi)\in\RZ^\Lambda_{G,b}(R_\xi)$ to be the pull-back of $(X,\iota)$, which has the desired property (by the essential uniqueness of $(X_\xi,\iota_\xi)$).

The general case can be reduced to this special case; indeed, we may replace $R$ with some finitely generated $W$-subalgebra  $R_0\subset R$, and work with a filtered direct system of finitely generated $W$-subalgebras of $R_\xi$'s whose direct limit is $R_0$.
\end{proof}

\begin{defnsub}\label{def:RZGt}
Let $R\in\Nilp_W$, and choose a finitely generated $W$-subalgebra $R_0\subset R$ $(X,\iota)\in\RZ^\Lambda_{G,b}(R)$  descends to $(X_{R_0},\iota)\in\RZ^\Lambda_{G,b}(R_0)$. Then  we define
\[
t_\alpha:\triv\ra\DD(X)^\otimes
\]
by pulling back $t_\alpha:\triv\ra\DD(X_{R_0})^\otimes$, constructed in Proposition~\ref{prop:descent}(\ref{prop:descent:t}). By the uniqueness part of Proposition~\ref{prop:descent}(\ref{prop:descent:t}) (and Lemma~\ref{lem:lpf}), it follows that $t_\alpha$ on $X$ is independent of the choice of $R_0\subset R$ and $(X_{R_0},\iota)$.
\end{defnsub}

\begin{corsub}\label{cor:deformation}
The functor $\RZ^\Lambda_{G,b}$ on $\Nilp_W$ is formally smooth. Furthermore, the statement of Lemma~\ref{lem:ForSmFT} holds even when $R'\thra R$ be a square-zero thickening of (not necessarily finitely generated) $W/p^m$-algebras for some $m$.
\end{corsub}
\begin{proof}
Given  $(X,\iota)\in\RZ^\Lambda_{G,b}(R)$ and a square-zero thickening $R'\thra R$ of $W/p^m$-algebras, we can find a  finitely generated $W/p^m$-subalgebra $R'_0\subset R'$ such that $(X,\iota)$ descends to  $(X_{R_0},\iota_{R_0/p})\in\RZ^\Lambda_{G,b}(R_0)$, where $R_0\subset R$ is the image of $R'_0$. By Lemma~\ref{lem:ForSmFT} and Proposition~\ref{prop:DescentDevissage}, there is an $(X_{R_0'},\iota_{R_0'/p})\in \RZ^\Lambda_{G,b}(R_0')$ lifting $(X_{R_0},\iota_{R_0/p})$, and its base change via $R_0'\to R'$ is an $R'$-lift of $(X,\iota)$ .

The rest of the statements of Lemma~\ref{lem:ForSmFT} can be similarly generalised for any square-zero thickening $R'\thra R$ of $W/p^m$-algebras as all the objects involved can be descended over some finitely generated $W/p^m$-subalgebras of $R'$ and $R$.
\end{proof}

Let us now move on to the other main result of this section on ``effectiveness'':
\begin{prop}\label{prop:extension}
Let $S$ be a  complete local noetherian $W/p^m$-algebra with residue field $\kappa$, and let $(X_{S},\iota_{S/p})\in\RZ_\BX(S)$. (In particular, $\iota_{S}$ is defined over $\Spec S/p$.)
Then we have $(X_{S},\iota_{S/p})\in\RZ^\Lambda_{G,b}(S)$ if and only if $(X_{S/\m_S^i}, \iota_{S/ (p,\m_S^i)})\in \RZ^\Lambda_{G,b}(S/\m_S^i)$ for any $i$.
\end{prop}
The ``only if'' direction  is trivial, and we will now prove the ``if'' direction for now on. Conceptually, this proposition  asserts that formal-locally defined Tate tensors extend over some finite-type base ring if the quasi-isogeny does.

For the Artin representability theorem (\cite[Corollary~5.4]{Artin:VersalDefAlgStack}), we  need to generalise this proposition for complete local noetherian $W/p^m$-algebras $S$ with residue field finitely generated over $\kappa$. We obtain such a generalisation in Lemma~\ref{lem:eff}, using the idea in Remark~\ref{rmk:fpqcclosed}

\subsection{Preparation}\label{subsec:SmearOut}

For any $x\in X^G(b) = \RZ^\Lambda_{G,b}(\kappa)$, we write $A_{\GL,x}:= \widehat\OO_{\RZ_\BX,x}$, which can be viewed as a universal deformation ring of $X_x$, and we let $A_{G,x}$  denote the quotient of $A_x$  pro-representing $\Def_{X_x,G}$.

We consider $(X_{S},\iota_{S/p})\in\RZ_\BX(S)$ given by a (necessarily local) map $f_S:A_{G,x}\to S$ for some $x\in \RZ^\Lambda_{G,b}(\kappa)$. In particular, $X_S$ is the pull-back of the universal deformation of $(X_x,(t_{\alpha,x}))$ and the universal deformation of Tate tensors pull back to give morphisms of crystals $\hat t_{\alpha,x}:\triv\ra\DD(X_S)^\otimes$ over $\Spec S$; \emph{cf.} (\ref{eqn:tx}). Therefore, by construction of the universal deformation in \S\ref{subsec:FaltingsDeforConstr}, the following $S$-scheme 
\[
 \cP_{S}:=\nf\Isom_{S}\big([\DD(X_{S})(S), (\hat t_{\alpha,x}(S))], [S\otimes_{\Zp}\Lambda^*,(1\otimes s_\alpha)] \big),
 \]
is a trivial $G$-torsor, and the Hodge filtration $\Fil^1_{X_S}\subset \DD(X_{S})(S)$ is a $\set\mu$-filtration with respect to $(\hat t_{\alpha,x}(S))$. 

To prove Proposition~\ref{prop:extension}, we need to construct $(X,\iota)\in\RZ^\Lambda_{G,b}(R)$ for some finitely generated $W/p^m$-subalgebra $R\subset S$, such that $(X,\iota)$ pulls back to $(X_{S},\iota_{S/p})$. As a starting point, we begin with a finitely generated $W$-subalgebra $R\subset S$ with the following properties, which exists by standard argument:
\begin{enumerate}
\item
There exists  $(X,\iota)\in\RZ_\BX(R)$ which pulls back to $(X_{S},\iota_{S/p})$ over $S$; this is possible since $\RZ_\BX$ is locally formally of finite type over $W$.
\item\label{subsec:SmearOut:tS}
The (finitely many) tensors $(\hat t_{\alpha,x}(S))\subset\DD(X_{S})(S)^\otimes$ lie in the image of $\DD(X)(R)^\otimes$; indeed, this can  be arranged by considering a finite-rank direct factor of $\DD(X_{S})(S)^\otimes$ containing $(\hat t_{\alpha,x}(S))$, and possibly by increasing $R$ by adjoining finitely many elements in $S$. We let 
\[
(t_\alpha(R))\subset\DD(X)(R)^\otimes
\]
 denote the tensors which (injectively) map to $(\hat t_{\alpha,x}(S))$.
\item\label{subsec:SmearOut:P}
The following $R$-scheme is a $P_G(\mu)$-torsor for some $\mu\in\set\mu$:
\[
 \cP_{\Fil^\bullet_X}:=\nf\Isom_R\big([\DD(X)(R), (t_\alpha(R)),\Fil^1_X], [R\otimes_{\Zp}\Lambda^*,(1\otimes s_\alpha),\Fil^1_\mu] \big);
\]
in other words, the Hodge filtration $\Fil^1_X\subset\DD(X)(R)$ is a $\set\mu$-filtration with respect to $(t_\alpha(R))$.
In fact, the natural $P_G(\mu)$-action on $\cP_{\Fil^\bullet_X}$ is already transitive, 
and by increasing $R$ if necessary we can ensure that $\cP_{\Fil^\bullet_X}$ is smooth with non-empty fibre everywhere.  (\emph{Cf.} \cite[EGA~IV{$_{\text{4}}$}, Proposition~17.7.8, Th\'eor\`eme~8.10.5]{GD:EGA}.)

This assumption also ensures that the following $R$-scheme is a $G$-torsor:
\[\cP_{R}:=\nf\Isom_R\big([\DD(X)(R), (t_\alpha(R))], [R\otimes_{\Zp}\Lambda^*,(1\otimes s_\alpha)] \big).\] 
\end{enumerate}
If $R\subset S$ satisfies the above conditions, then  any finitely generated $W$-subalgebra of $S$ containing $R$ satisfies the same conditions.

To prove Proposition~\ref{prop:extension}, it suffices to prove the following proposition.
\begin{prop}\label{prop:SmearOut}
There exists a finitely generated $W$-subalgebra $R\subset S$ such that the properties listed in \S\ref{subsec:SmearOut} are satisfied and we have $(X,\iota)\in\RZ^\Lambda_{G,b}(R)$. 
\end{prop}
As a biproduct of the proof, we will also obtain that $(t_\alpha(R))$ coincides with the $R$-sections of $(t_\alpha)$ constructed in Proposition~\ref{prop:descent}.

\subsection{Proof of Proposition~\ref{prop:SmearOut}}
Assume that $S$ is \emph{reduced} for the moment. Since the $\kappa$-subalgebra $R\subset S$ is also reduced, we can consider the perfections $\wt R:=\varinjlim_\sig R$ and $\wt S:=\varinjlim_\sig S$. We  extend the injective map $R\hra S$ to an injective map $\wt R\hra \wt S$.
 
We consider $(1\otimes s_\alpha)\subset W(\wt R)[\ivtd p]\otimes_{\Zp}\Lambda^\otimes\cong \DD(X_{\wt R})(W(\wt R))^\otimes[\ivtd p]$, where the isomorphism is induced by the quasi-isogeny $\iota_{\wt R}:\BX_{\wt R}\dra X_{\wt R}$. 

\begin{lemsub}\label{lem:IntPerf}
In the above setting, we have $(1\otimes s_\alpha)\subset\DD(X_{\wt R})(W(\wt R))^\otimes$. In particular, $s_{\alpha,\DD}$ for each $\alpha$ (\emph{cf.}  Definition~\ref{def:t}) comes from a unique morphism of integral crystals, which we  also denote by $s_{\alpha,\DD}:\triv\ra\DD(X_{\wt R})^\otimes$.
Furthermore, the $\wt R$-sections  $(s_{\alpha,\DD}(\wt R))\in\DD(X_{\wt R})(\wt R)^\otimes$ of $(s_{\alpha,\DD})$ coincide with the image of $(t_\alpha(R))$ in $\DD(X_{\wt R})(\wt R)^\otimes$.
\end{lemsub}
\begin{proof}
Let us first show that the isomorphism
\[
\DD(\iota_{\wt S}):\DD(X_{\tilde S})(W(\wt S))[\ivtd p] \riso W(\wt S)[\ivtd p]\otimes_W\DD(\BX)(W)
\]
 matches $(t_\alpha(W(\wt S)) \subset \DD(X_{\wt S})(W(\wt S))^\otimes$ with
$(1\otimes s_\alpha)\subset W(\wt S)[\ivtd p]\otimes_{\Zp}\Lambda^\otimes$.
Recall that any $F$-equivariant homomorphism between \emph{constant} $F$-isocrystals over $\wt S$ (i.e., $F$-isocrystals obtained as a scalar extension from $F$-isocrystals over $\kappa$) is defined over $\kappa$ (\emph{cf.}
\cite[Lemma~3.9]{RapoportRichartz:Gisoc}). Therefore, to show that $\DD(\iota_{\wt S})$ preserves the tensors, it suffices to check the claim for $\DD(\iota_{\kappa})$; i.e., after reducing modulo $W(\m_{\tilde S})$. And the claim for $\DD(\iota_{\kappa})$ follows from the assumption that $x\in \RZ^\Lambda_{G,b}(\kappa)$.

This shows $\DD(\iota_{\wt S})$ matches $(\hat t_{\alpha,x}(W(\wt S)) $ with
$(1\otimes s_\alpha)$. Since we have $W(\wt R) = W(\wt R)[\ivtd p]\cap W(\wt S)$, it follows the that each of $1\otimes s_\alpha$  lies in $\DD(X_{\wt R})(W(\wt R))^\otimes$. The integrality claim for $(s_{\alpha,\DD})$ now follows from the  standard dictionary  (\emph{cf.} \cite[IV~\S4]{Grothendieck:BTMontreal}).
To verify the last claim, we may compare the images of both in $\DD(X_{\wt S})(\wt S)^\otimes$, where the claim is obvious from the construction; \emph{cf.} \S\ref{subsec:SmearOut}(\ref{subsec:SmearOut:tS}).
\end{proof}

For any  closed point $y\in\Spec R$ (which may not lie in the image of $\Spec S \ra \Spec R$), we have a natural map $R\hra \wt R\thra \kappa(y)=\kappa$. By taking the image of $t_\alpha(W(\wt R))$ in $\DD(X_{\wt R})(W(\wt R))\otimes_{W(\wt R)}W \cong \DD(X_{y})(W)$, we  obtain
\begin{equation}
(t_{\alpha,y})\subset\DD(X_{y})(W)^\otimes,
\end{equation}
which induces $(s_{\alpha,\DD})$ on the isocrystals.

\begin{lemsub}\label{lem:PtWiseSmearOut}
For any closed point $y$ in $\Spec R$, we have $(X_{x'},\iota_{x'})\in\RZ\com{s_\alpha}_{\BX,G}(\kappa)$. Furthermore, the images of $(t_{\alpha,y})$ and $(t_\alpha(R))$ in $\DD(X_{y})(\kappa)$ coincide.
\end{lemsub}
\begin{proof}
The only assertion which may not directly follow from the construction is to verify Definition~\ref{def:RZG}(\ref{def:RZG:P}) for $(X_{y},\iota_{y})$.
We consider the following scheme over $W(\wt R)$:
\[
 \cP_{W(\wt R)}:=\nf\Isom_{W(\wt R)}\big([\DD(X_{\wt R})(W(\wt R)), (t_\alpha(W(\wt R)))], [W(\wt R)\otimes_{\Zp}\Lambda^*,(1\otimes s_\alpha)] \big).
\]
Note that its mod~$p$ fibre $\cP_{\wt R}$ is a $G$-torsor since it is the pull-back of $\cP_R$ (\emph{cf.} \S\ref{subsec:SmearOut}(\ref{subsec:SmearOut:P})), and  $\cP_{W(\wt R)[\frac{1}{p}]}$ is a trivial $G$-torsor by the quasi-isogeny $\iota_{\wt R}$.  The  assumption on $(X_{\wt R},\iota_{\wt R})$ implies that the base change  $\cP_{W(\wt S)}$ is a $G$-torsor.  (To see this, note that $[\DD(X_{\wt R})(W(\wt R)), (t_\alpha(W(\wt R)))]$ pulls back to $[\DD(X_{\wt S})(W(\wt S)), (\hat t_{\alpha,x}(W(\wt S)))]$, which is in turn the pull-back of the the  crystalline Dieudonn\'e modules with Tate tensors $[\bM_{G,x},(\hat t^\univ_{\alpha,x}(A_{G,x}))]$ by any lift $A_{G,x}\to W(\wt S)$ of $A_{G,x}\xra{f_S} S \to \wt S$ (where $f_S$ is the map corresponding to the deformation $X_S$). Now we can pull back the isomorphism $[\bM_{G,x},(\hat t^\univ_{\alpha,x}(A_{G,x}))] \cong [A_{G,x}\otimes_{\Zp}\Lambda,(1\otimes s_\alpha)]$ over $W(\wt S)$, which exists by construction (\emph{cf.} \S\ref{subsec:FaltingsDeforConstr}).

For any closed point $y$ in $\Spec \wt R$, we set $\cP_{W(y)}$ to be the fibre of $\cP_{W(\wt R)}$ by $W(y):W(\wt R)\to W$ (induced by $y$). We want to show that $\cP_{W(y)}$ is a $G$-torsor over $W$ for any closed point $y$ of $\Spec R$. Since each fibre of $\cP_{W(y)}$ is known to be a $G$-torsor, it suffices to show the $W$-flatness of $\cP_{W(y)}$, for which it suffices to show that $\cP_{W(y)}\times_{\Spec W} \Spec W/p^m$ is flat over $W/p^m$ for any $m\geqslant1$. Therefore, it suffices to show that $\cP_{W_m(\wt R)}$ is flat over $W_m(\wt R):=W(\wt R)/p^m$ for any $m\geqslant 1$. This can be shown by local flatness criterion.

We consider the following map for any $m\geqslant 0$
\[
\gamma_m: \OO_{\cP_{\wt R}} \thra p^m\OO_{\cP_{W(\wt R)}}/p^{m+1}\OO_{\cP_{W(\wt R)}},
\]
induced by the surjective map $\OO_{\cP_{W(\wt R)}}\thra p^m\OO_{\cP_{W(\wt R)}}$ given by multiplication by $p^m$. By local flatness criterion (\emph{cf.} \cite[Theorem~22.3]{matsumura:crt}), $\cP_{W_m(\wt R)}$ is flat over $W_m(\wt R)$ if and only if $\gamma_{m'}$ is an isomorphism for any $m'<m$ and $P_{\wt R}$ is flat over $\wt R$. 
Since we already have the $\wt R$-flatness of $P_{\wt R}$, it remains to show the injectivity of $\gamma_m$ for any $m\geqslant 0$. Now, let us consider the following commutative diagram
\[\xymatrix{
\OO_{\cP_{\wt R}} \ar@{->>}[d]_-{\gamma_m} \ar@{^{(}->}[r] &
\OO_{\cP_{\wt S}} \ar[d]^-{\cong}_-{\wt S\otimes\gamma_m}\\
p^m\OO_{\cP_{W(\wt R)}}/p^{m+1}\OO_{\cP_{W(\wt R)}} \ar[r] &
p^m\OO_{\cP_{W(\wt S)}}/p^{m+1}\OO_{\cP_{W(\wt S)}}
}.\]
The right vertical arrow is an isomorphism since $\cP_{W(\wt S)}$ is flat over $W(\wt S)$, and it coincides with $\wt S\otimes\gamma_m$ since $\OO_{\cP_{W(\wt S)}} = W(\wt S)\otimes_{W(\wt R)}\OO_{\cP_{W(\wt R)}}$. 
The top arrow is injective since it is the scalar extension of the injective map $\wt R\hra \wt S$ by a flat ring extension $\wt R\to \OO_{\cP_{\wt R}}$. By this diagram, $\ker (\gamma_m)\subset \OO_{\cP_{\wt R}}$ maps injectively into the kernel of the isomorphism $\wt S\otimes\gamma_m$, which shows that $\gamma_m$ should also be injective.
\end{proof}

\begin{propsub}\label{prop:SmearOutRed}
Proposition~\ref{prop:SmearOut} holds when $S$ is reduced.
\end{propsub}
\begin{proof}
We work in the setting of \S\ref{subsec:SmearOut} with $S$ reduced. 
For any closed point $y\in\Spec R$ (corresponding to a maximal ideal $\m_y\subset R$), we have $(X_y,\iota_y)\in\RZ^\Lambda_{G,b}(\kappa)$ by Lemma~\ref{lem:PtWiseSmearOut}. Now 
we want to show that the map $\hat f_y:\Spf \wh R_y\ra\Def_{X_y}$, induced by $X_{\wh R_y}$, factors through $\Def_{X_y,G}$. 

Applying Corollary~\ref{cor:fsperf-red} to $(B,J):=(\wh R_y,\m_y)$, we obtain a sequence of square-zero liftable  PD thickenings of artin local rings  (in the sense of Definition~\ref{def:LiftablePD}):
\[
\wh R_y \thra \cdots \thra R_{y,i+1}\thra R_{y,i} \thra \cdots \thra R_{y,0}=\kappa.
\]
We choose a formal power series ring $A$ over $W$ with a surjection $A\thra \wh R_y$, and let $D_{y,i}\thra R_{y,i}$ denote the $p$-adically completed PD hull of $\wt A\thra \wh R_y\thra R/\m_y^i$. We chose $\{R_{y,i}\}$ so that the natural PD morphism  $D_{y,i}\to R_{y,i+1}$ factors through the $\Zp$-flat closure $D_{y,i}^\fl$ of $D_{y,i}$, which is also a PD thickening of $R_{y,i}$ (\emph{cf.} Corollary~\ref{cor:fsperf-red}, Lemma~\ref{lem:PDtorsion}).  

We set $X\com i_y:=X_{R_{y,i}}$, and let $(t_\alpha(R_{y,i}))\subset \DD(X\com i_y)(R_{y,i})^\otimes$ denote the image of $(t_\alpha(R))$. (Note that we have not constructed a map of crystals $t_\alpha:\triv\to\DD(X)^\otimes$, and we only have a section $t_\alpha(R)$.) To prove the proposition, we need to prove that for any $i$ we have $X\com i_y\in\Def_{X_y,G}(R_{y,i})$, and $(t_\alpha(R_{y,i})) \in \DD(X\com i_y)(R_{y,i})^\otimes$  coincides with the $R_{y,i}$-section of $\hat t^{(i)}_{\alpha,y}:\triv\ra\DD(X\com i_y)^\otimes$,  where $\hat t^{(i)}_{\alpha,y}=f_y^{(i)*}(\hat t^\univ_{\alpha,y})$ is the pull-back of the  universal Tate tensor $\hat t^\univ_{\alpha,y}$ by $f_y^{(i)}:\Spec R_{y,i}\to\Def_{X_x,G}$.  

We show this claim by induction on $i$. The base case with $i=0$ is exactly Lemma~\ref{lem:PtWiseSmearOut}. Now, we assume the claim for $i$ (i.e.,  $X\com i_y\in\Def_{X_y,G}(R_{y,i})$ and $t_\alpha(R_{y,i}) = \hat t^{(i)}_{\alpha,y}(R_{y,i})$) and want to deduce the claim for $i+1$. Since $R_{y,i+1}\thra R_{y,i}$ is liftable, we have $(\hat t^{(i)}_{\alpha,y}(D_{y,i}^\fl))\subset \DD(X^{(i)}_y)(D_{y,i}^\fl)^\otimes$ lifting $(\hat t^{(i)}_{\alpha,y}(R_{y,i+1}))$.

Let us consider the following commutative diagram:
\begin{equation}\label{eqn:SmearOutDiagram}
\xymatrix@R=3pt@C=7pt{
1\otimes s_\alpha &
\ar@{|->}[l]_-{\text{Lemma~\ref{lem:IntPerf}}}^-{\DD(\iota_{\wt R})} t_\alpha(W(\wt R)) \ar@{|->}[r] &
t_\alpha(\wt R) &
\ar@{|->}[l]_-{\text{Lemma~\ref{lem:IntPerf}}} t_\alpha(R)\\
 W(\wt R)[\ivtd p]\otimes_{\Zp}\Lambda^\otimes &
\ar@{_{(}->}[l] \DD(X_{\wt R})(W(\wt R))^\otimes  \ar@{->>}[r]&
\DD(X_{\wt R})(\wt R)^\otimes &
\ar@{_{(}->}[l] \DD(X_R)(R)^\otimes
\ar@{->>}[dddd]_{t_\alpha(R)\mapsto t_{\alpha}(R_{y,i+1})}\\
 & & &\\
 & & & \\
  & & & \\
\Lambda^\otimes \ar[uuuu]\ar[r]&
D^\fl_{y,i}[\ivtd p]\otimes_{\Zp}\Lambda^\otimes &
\ar@{_{(}->}[l]\DD(X\com i_y)(D^\fl_{y,i})^\otimes \ar[r]&
\DD(X\com i_y)(R_{y,i+1})^\otimes \\
s_\alpha \ar@{|->}[r] & 1\otimes s_\alpha&
\ar@{|->}[l]_-{\text{Lem~\ref{Lem:Dwork}}}^-{\DD(\iota_{y,i})} \hat t^{(i)}_{\alpha,y}( D^\fl_{y,i}) \ar@{|->}[r] &
\hat t^{(i)}_{\alpha,y}(R_{y,i+1})
}.
\end{equation}
The commutativity of this diagram follows from the fact that $\iota_{\wt R}:\BX_{\wt R}\dra X_{\wt R}$ descends to $\iota_R:\BX_R\dra X_R$, which lifts
  $ \iota_{R_{y,i}}:\BX_{R_{y,i}}\dra X_{R_{y,i}}$.

Indeed, the diagram shows that $t_\alpha(R_{y,i+1})$ (respectively, $\hat t^{(i)}_{\alpha,y}(R_{y,i+1})$) is uniquely determined by $s_\alpha$ by chasing the top row and the right vertical arrow (respectively, by chasing the bottom row), so we have $t_\alpha(R_{y,i+1}) = \hat t^{(i)}_{\alpha,y}(R_{y,i+1})$.

Now Proposition~\ref{prop:LiftingTate} shows that $X\com{i+1}_y\in\Def_{X_y,G}(R_{y,i+1})$ since the Hodge filtration of $X\com{i+1}_y$ corresponds to a $\set\mu$-filtration with respect to $(t_\alpha(R_{y,i+1}))$ by assumption (\emph{cf.} \S\ref{subsec:SmearOut}). The equality $\hat t^{(i+1)}_{\alpha,y}(R_{y,i+1}) = t_\alpha (R_{y,i+1})$ follows from  $\hat t^{(i)}_{\alpha,y}(R_{y,i+1}) = t_\alpha (R_{y,i+1})$ and the fact that $\hat t^{(i+1)}_{\alpha,y}$ lifts $\hat t^{(i)}_{\alpha,y}$ (where $\hat t^{(i+1)}_{\alpha,y}$ is the pull-back of $\hat t^\univ_{\alpha,y}$). 

By induction on $i$, we have just shown that $X^{(i)}_y\in\Def_{X_y,G}(R_{y,i})$ for any $i\geqslant 0$. so we obtain $\hat f_y:\Spf \wh R_y\to\Def_{X_y,G}$, as desired.
\end{proof}
We are ready to prove Proposition~\ref{prop:SmearOut}.
\begin{proof}[Proof of Proposition~\ref{prop:SmearOut}] 
 Let us first handle the case when $S$ admits the following sequence of square-zero \emph{liftable} PD thickenings (in the sense of Definition~\ref{def:LiftablePD}):
\begin{equation} \label{eqn:fsperf-SmearOut}
S:= A_N \thra \cdots\thra A_{j+1}\thra A_j\thra \cdots A_0 = S/\n,
\end{equation}
where $\n$ is the nilradical of $S$.

We choose $R\subset S$, $(X,\iota)\in\RZ_\BX(R)$, and $(t_\alpha(R))\subset\DD(X)(R)^\otimes$ as in \S\ref{subsec:SmearOut}. 
or $0\leqslant j\leqs N$, we set $R_j$ to be the image of $R$ in $A_j$. We choose a polynomial ring $A_0$ over $W$ which surjects onto $R$ (so it surjects onto $R_j$ for each $j$). By adding more variables if necessary, we also assume that the completion $A_S$ of $A_0$ with respect to the kernel of $A_0\thra R\ra S\thra \kappa$ surjects onto $S$. For any $0\leqslant j\leqslant N$ we define $D_j\thra R_j$  to be the $p$-adically completed PD hull of $A_0\thra R_j$, and $D_{S,j}\thra A_j$ to be the $p$-adically completed PD hull of $A_S\thra S\thra A_j$. By the assumpiton on $A_j$ (\emph{cf.} Corollary~\ref{cor:fsperf-appl}), the natural PD surjection $D_{S,j}\thra A_{j+1}$ factors through the flat closure $D_{S,j}^\fl$ of $D_{S,j}$. Now, since the map $D_j\thra R_{j+1}\hra A_{j+1}$ factors as  $D_j \to D_{S,j}^\fl\thra A_{j+1}$, the $p$-power torsion of $D_j$ lies in the kernel of $D_j\thra R_{j+1}$ and we get $D_j^\fl\thra R_{j+1}$. (In other words, $R_{j+1}\thra R_j$ is a square-zero liftable PD thickening for any $0\leqslant j<N$; \emph{cf.} Lemma~\ref{lem:PDtorsion}.) We set $D_S:=D_{S,N} \thra A_N = S$, which also factors through $D_S^\fl$ (by considering $D_S\to D_{S,N-1}^\fl\thra A_N=S$).

By Proposition~\ref{prop:SmearOutRed}, when $j=0$ we have $(X_{R_0},\iota)\in\RZ^\Lambda_{G,b}(R_0)$, and the image $t_\alpha(R_0)\in\DD(X_{R_0})(R_0)^\otimes$ of $t_\alpha(R)$ (as in \S\ref{subsec:SmearOut}) coincides with the section of $t\com{0}_\alpha:\triv\ra\DD(X_{R_0})^\otimes$ constructed in Proposition~\ref{prop:descent}.

We now proceed inductively on $j$. Let us assume (as the induction hypothesis) that we have  $(X^{(j)},\iota)\in\RZ^\Lambda_{G,b}(R_j)$ and $t\com j_\alpha(R_j) = t_\alpha(R_j)$ in $\DD(X_{R_j})(R_j)^\otimes$, where  $t_\alpha(R_j)$ is the image of  $t_\alpha(R)$ and  $t_\alpha\com j:\triv\ra\DD(X_{R_j})^\otimes$ is as constructed in Proposition~\ref{prop:descent}(\ref{prop:descent:t}). We first  claim that the natural isomorphism  $R_{j+1}\otimes_R \DD(X)(R)\cong \DD(X_{R_j})(R_{j+1})$ matches $(t_\alpha(R_{j+1}))$ and $(t^{(j)}_\alpha(R_{j+1}))$, by the following diagram:
\[
\xymatrix@R=3pt@C=7pt{
1\otimes s_\alpha & 
\ar@{|->}[l]_-{\DD(\iota_{S/p})} \hat t_{\alpha,x}(D_{S}^\fl) \ar@{|->}[r] & 
\hat t_{\alpha,x}(S) &
t_\alpha(R) \ar@{|->}[l]_-{\text{\S\ref{subsec:SmearOut}(\ref{subsec:SmearOut:tS})}}
\\
D_{S}^\fl[\ivtd p]\otimes_{\Zp}\Lambda^\otimes & 
\ar@{_{(}->}[l] \DD(X_{S})(D^\fl_{S})^\otimes  \ar@{->>}[r]&
\DD(X_{S})(S)^\otimes &
\DD(X_R)(R)^\otimes \ar@{_{(}->}[l]\ar@{->>}[dddd]
\\
 & & &\\
 & & & \\
  & & & \\
\Lambda^\otimes \ar[uuuu]\ar[r]&
D^\fl_{j}[\ivtd p]\otimes_{\Zp}\Lambda^\otimes &
\ar@{_{(}->}[l]\DD(X_{R_j})(D^\fl_{j})^\otimes \ar[r]&
\DD(X_{R_j})(R_{j+1})^\otimes 
\\
s_\alpha \ar@{|->}[r] & 1\otimes s_\alpha&
\ar@{|->}[l]_-{\DD(\iota_{R_j/p})} \hat t^{(j)}_{\alpha}( D^\fl_{j}) \ar@{|->}[r] &
\hat t^{(j)}_{\alpha}(R_{j+1})
},
\]
where $D_S\thra S$ is the PD hull of $A_S\thra S$, and $D_S^\fl$ is the $\Zp$-flat closure of $S$, and we used the notations from \S\ref{subsec:SmearOut}. 
Note that $\iota_{S/p}:\BX_{S/p}\dra X_{S/p}$ matches $\hat t_{\alpha,x}:\triv\to\DD(X_S)^\otimes$ and $s_{\alpha ,\DD}$ by \cite[Lemma~4.3]{deJong-Messing}, and $\iota_{R_j/p}:\BX_{R_j/p}\dra X_{R_j/p}$ matches $t_{\alpha}^{(j)}:\triv\to \DD(X_{R_j})^\otimes$ and $s_{\alpha,\DD}$  because $(X_{R_j},\iota_{R_j/p})$ and $(t^{(j)}_\alpha)$ can be lifted over some $A\in\Nilp_W^\sm$ (\emph{cf.} Proposition~\ref{prop:descent}(\ref{prop:descent:liftability})), where the desired compatibility follows from Definition~\ref{def:RZG}(\ref{def:RZG:qisog}).

By the diagram, the image of $t_\alpha(R)$ in $\DD(X_{R_j})(R_{j+1})$ coincides with $t^{(j)}_\alpha(R_{j+1})$. Since the Hodge filtration $\Fil^1_{X_R}$ is a $\set\mu$-filtration for $(t_\alpha(R))$ by assumption (\emph{cf.} \S\ref{subsec:SmearOut}(\ref{subsec:SmearOut:P})), it follows that $(X_{R_{j+1}}, \iota_{R_{j+1}/p})\in\RZ^\Lambda_{G,b}(R_{j+1})$ by Lemma~\ref{lem:ForSmFT}(\ref{lem:ForSmFT:GM}). Therefore, by induction on $j$ and the assumption that $S=A_N$ (so $R=R_N$), we conclude that $(X,\iota) = (X^{(N)},\iota^{(N)})\in\RZ^\Lambda_{G,b}(R)$.

Now we handle the general case of the proposition, without assuming the existence of (\ref{eqn:fsperf-SmearOut}).
Let $A_S$ be a formal power series ring over $W$ surjecting onto $S$, and we set $J:=\ker (A_S\thra S/\n)$. Applying Corollary~\ref{cor:fsperf-appl} to $(A_S,J)$, we obtain a sequence square-zero liftable PD thickenings $\{A_{m,i}\}$ indexed by $(m,i)\in\Z_{\geqslant1}^2$ with $A_{1,1} = S/\n$.  Since we have $A/p^{m} = \varprojlim_i A_{m,i}$ for any $m\geqslant 1$  by construction of $A_{m_0,i}$,  there exists $(m',i')$ such that we have $A_{m',i'}\thra S\thra S/\n$. Let $S':=A_{m',i'}$.

Now, given $(X_S,\iota_{S/p})\in\RZ_\BX(S)$ where $X_S$ is given by $f_S:A_{G,x}\to S$, we can choose a lift $f'_{S'}:A_{G,x}\to S':=A_{m_0,i_0}$ of $f_S$, which corresponds to some $p$-divisible group $X'_{S'}$ over $S'$. Since $S'\thra S$ is a nilpotent thickening, the quasi-isogeny $\iota_{S/p}$ uniquely lifts to a quasi-isogeny $\iota'_{S'/p}$ over $S'/p$, so we have $(X'_{S'},\iota'_{S'/p})\in\RZ^\Lambda_{G,b}(S')$ lifting $(X_S,\iota_{S/p})$. Since $S':=A_{m',i'}$ admits a sequence of thickenings as in (\ref{eqn:fsperf-SmearOut}) where $A_j:=A_{m_j,i_j}$ for some suitable $(m_j,i_j)$ for any $j$, we have just shown the existence of $(X',\iota')\in\RZ^\Lambda_{G,b}(R')$ that pulls back to $(X'_{S'},\iota'_{S'/p})$, where $R'\subset S'$ is a finitely generated $W/p^{m'}$-subalgebra. By setting $R\subset S$ to be the image of $R'$, we obtain $(X,\iota):=(X'_R,\iota'_{R/p})\in\RZ^\Lambda_{G,b}(R)$, which pulls back to $(X_S,\iota_{S/p})$.
\end{proof}

\section{Construction of the moduli of $p$-divisible groups with Tate tensors}\label{sec:RepPf}

 We give a proof of Theorem~\ref{thm:RZHType} in this section, applying the technical results proved in \S\ref{sec:fpqcDesc}.

\subsection{Artin representability theorem}\label{subsec:Artin}
Let $p>2$.  We choose a $W$-lift $\wt \BX$ of $ \BX$ as in Remark~\ref{rmk:X0}, and define $\RZ^\Lambda_{G,b}(h)^{m,n}:=\RZ^\Lambda_{G,b}\times_{\RZ_\BX}\RZ_\BX(h)^{m,n} $ as a functor on the category of $W/p^m$-algebras; i.e., for a $W/p^m$-algebra $R$, we set $\RZ^\Lambda_{G,b}(R):=\RZ^\Lambda_{G,b}(R)\cap\RZ_\BX(h)^{m,n}(R)$. Similarly, we define $\RZ^\Lambda_{G,b}(h):=\RZ^\Lambda_{G,b}\times_{\RZ_\BX}\RZ_\BX(h)$ as a functor on $\Nilp_W$.

We first prove (Theorem~\ref{thm:ArtinRepTh}) that $\RZ^\Lambda_{G,b}(h)^{m,n}$ is a separated algebraic space locally of finite type over $W/p^m$ using the Artin representability theorem  \cite[Corollary~5.4]{Artin:VersalDefAlgStack}\footnote{See also \cite{Conrad-deJong:Approx} for some clarifications.}. See \cite[\S2]{Knutson:AlgSp} for the definition of algebraic spaces.

To apply the Artin representability theorem as stated in \cite[Corollary~5.4]{Artin:VersalDefAlgStack}, it suffices to verify the following conditions\footnote{The conditions that  we state here are slightly stronger than the ones given in \cite[Corollary~5.4]{Artin:VersalDefAlgStack}.}, some of  which will be more precisely stated when they are verified:
\begin{enumerate}
\item $\RZ^\Lambda_{G,b}(h)^{m,n}$ is separated; i.e., the diagonal map
\[\RZ^\Lambda_{G,b}(h)^{m,n} \ra \RZ^\Lambda_{G,b}(h)^{m,n}\times\RZ^\Lambda_{G,b}(h)^{m,n}\]
is representable by a closed immersion. More concretely, for any  $W/p^m$-algebra $R$ and given two points $x,y\in\RZ^\Lambda_{G,b}(h)^{m,n}(R)$, the locus over which $x$ and $y$ coincide is a closed subscheme of $\Spec R$. This follows since $\RZ^\Lambda_{G,b}(h)^{m,n}$ is a subfunctor of a separated scheme $\RZ_\BX(h)^{m,n}$.
\item $\RZ^\Lambda_{G,b}(h)^{m,n}$ commutes with filtered direct limits of $W/p^m$-algebras (i.e., locally of finite presentation over $W/p^m$); indeed, this follows because both $\RZ_\BX(h)^{m,n}$ and $\RZ^\Lambda_{G,b}$ commute with  filtered direct limits (\emph{cf.} Lemma~\ref{lem:lpf}).
\item $\RZ^\Lambda_{G,b}(h)^{m,n}$ is an fppf sheaf; \emph{cf.} Lemma~\ref{lem:fppf}.
\item\label{subsec:Artin:Eff} $\RZ^\Lambda_{G,b}(h)^{m,n}$  satisfies the ``effectivity property''; namely,
for any  complete local noetherian $W/p^m$-algebra $(R,\m_R)$ such that its residue field is a finitely generated field extension of $\kappa$, the following natural map is bijective:
\[ \RZ^\Lambda_{G,b}(h)^{m,n}(R) \ra \varprojlim_i\RZ^\Lambda_{G,b}(h)^{m,n}(R/\m_R^i);\]
\emph{cf.} Lemma~\ref{lem:eff}.
\item $\RZ^\Lambda_{G,b}(h)^{m,n}$ satisfies some suitable generalisation of Schlessinger's criterion (i.e., Conditions (S1$'$) and (S2) in \cite[\S2]{Artin:VersalDefAlgStack}); namely, Lemma~\ref{lem:SchlenssingerRim} holds and the tangent space at any $x\in\RZ^\Lambda_{G,b}(h)^{m,n}(\kappa)$ is finite dimensional over $\kappa$. Finiteness of tangent spaces is obvious since $\RZ^\Lambda_{G,b}(h)^{m,n}$ is a subfunctor of $\RZ_\BX(h)^{m,n}$.
\item If $n\gg0$ then for any $m\geqs1$ there exists an obstruction theory for $\RZ^\Lambda_{G,b}(h)^{m,n}$ in the sense of \cite[(2.6), (4.1)]{Artin:VersalDefAlgStack}. Indeed, we verify some variant of this (exploiting the flexibility to increase $n$), which would still imply the representability of $\RZ^\Lambda_{G,b}(h)^{m,n}$; \emph{cf.} Lemma~\ref{lem:ObsThy}, the proof of Theorem~\ref{thm:ArtinRepTh}.
\end{enumerate}
We have already verified the first two conditions, so it remains to verify the remaining four conditions.

To show that $\RZ^\Lambda_{G,b}(h)^{m,n}$ is an fppf sheaf, it suffices to show that $\RZ^\Lambda_{G,b}$ is an fppf sheaf (as the height of a quasi-isogeny can be computed fppf-locally, and whether $p^n\tilde\iota$ is an isogeny can be verified fppf-locally). Since $\RZ^\Lambda_{G,b}$ is a subfunctor of $\RZ_\BX$, which is an fppf sheaf (as it can be represented by a formal scheme), the following lemma shows that $\RZ^\Lambda_{G,b}$ is an fppf sheaf.
\begin{lemsub}\label{lem:fppf}
Let $(X,\iota)\in\RZ_\BX(R)$ for $R\in\Nilp_W$. Assume that for a faithfully flat $R$-algebra $R'$ the pull-back $(X_{R'},\iota_{R'/p})$ lies in $\RZ^\Lambda_{G,b}(R')$. Then we have $(X,\iota)\in\RZ^\Lambda_{G,b}(R)$.
\end{lemsub}
\begin{proof}
We may assume that $R$ is finitely generated. Since whether $(X,\iota)\in\RZ^\Lambda_{G,b}(R)$ is decided by the pull-back over all artinian quotients of $R$, we may assume that  $R\in \art W$. Then $(X,\iota)$ defines a map $\Spec R \ra (\RZ_\BX)\wh{_x}\cong \Def_{X_x}$ for some closed point $x\in\RZ_\BX(\kappa)$. This map factors through $\Def_{X_x,G}$ if and only if it does after pre-composing with a faithfully flat map $\Spec R'\ra\Spec R$, since $\Def_{X_x,G}$ is a closed formal subscheme of $\Def_{X_x}$. 
\end{proof}

Next, we will verify the ``effectivity property'' (\ref{subsec:Artin:Eff}).
\begin{lemsub}\label{lem:eff}
The functor $\RZ^\Lambda_{G,b}(h)^{m,n}$  satisfies the ``effectivity property'', as stated above in \S\ref{subsec:Artin}(\ref{subsec:Artin:Eff}).
\end{lemsub}
\begin{proof}
Let $(R,\m_R)$ be a complete local noetherian $W/p^m$-algebra  such that its residue field is a finitely generated field extension of $\kappa$. We want to show that the following natural map is bijective:
\[ \RZ^\Lambda_{G,b}(h)^{m,n}(R) \ra \varprojlim_i\RZ^\Lambda_{G,b}(h)^{m,n}(R/\m_R^i).\]
Note that this property holds for $\RZ_\BX(h)^{m,n}$  (\emph{cf.} \cite[\S2.22]{RapoportZink:RZspace}), which shows the injectivity. It remains to show the surjectivity.

Let $(X,\iota)\in \RZ_\BX(h)^{m,n}(R)$, and assume that over $R/\m_R^i$ for each $i$ we have $(X_{R/\m_R^i},\iota)\in \RZ^\Lambda_{G,b}(R/\m_R^i)$. We want to show that we have  $(X,\iota)\in\RZ^\Lambda_{G,b}(R)$. (By the effectivity property for $\RZ_\BX(h)^{m,n}$, this implies the desired surjectivity. Also recall that $\RZ^\Lambda_{G,b}(h)^{m,n}(R):= \RZ^\Lambda_{G,b}(R) \cap \RZ_\BX(h)^{m,n}(R)$.)

Let $\eta\in |\RZ_\BX(h)^{m,n}|$ denote the point in the underlying topological space where $(X_{R/\m_R},\iota_{R/\m_R})\in\RZ_\BX(h)^{m,n}(R/\m_R)$ is supported. We claim that the closure $\overline{\{\eta\}}$ of $\eta$ contains a closed point $x$ in $\RZ^\Lambda_{G,b}(\kappa)$. Indeed, by assumption there exists a finitely generated $\kappa$-subdomain $\overline R_0\subset R/\m_R$ where the $R/\m_R$-point $(X,\iota)$ extends to $(X_{\overline R_0},\iota)\in\RZ^{\Lambda}_{G,b}(\overline R_0)$. Then the corresponding map of schemes $\Spec \overline R_0 \to \RZ_\BX(h)^{m,n}$ sends the generic point to $\eta$. Now, we may choose $x$ in the image of a closed point in $\Spec \overline R_0$.

By Lemma~\ref{lem:fppf}, we have $(X,\iota)\in\RZ^\Lambda_{G,b}(R)$ if and only if for some faithfully flat $R$-algebra $R'$ we have $(X_{R',}\iota_{R'/p})\in\RZ^\Lambda_{G,b}(R')$.
Let us write $A_x:=\widehat \OO_{\RZ_\BX,x}$ for $x$ as in the previous paragraph. Let us now construct a faithfully flat $R$-algebra $R'$ such that the natural map $\Spec R'\to\RZ_\BX(h)^{m,n}$ factors through $\Spec A_x$.
We choose a preimage $\hat\eta\in\Spec A_x$ of $\eta$, and let $\kappa'$ be the compositum of $R/\m_R$ and the residue field $\kappa(\hat\eta)$ of $A_x$ at $\hat\eta$. Let  $C$ and $C'$  be  Cohen rings $R/\m_R$ and  $\kappa'$, respectively, and we choose a $C$-algebra structure on $R$ and $C'$ (lifting the natural one mod~$p$).
Then for the completion  $R'$ of $R\otimes_CC'$,  the natural map $\Spec R'\to\RZ_\BX(h)^{m,n}$ factors through $\Spec (A_x)\wh{_{\wh\eta}}$.

Let $A_{G,x}$ denote the quotient of $A_x$ corresponding to $\Def_{X_x,G}$.
By construction we have  $(X_{R'/\m_{R'}^i},\iota)\in\RZ^\Lambda_{G,b}(R'/\m_{R'}^i)$ for all $i$ (i.e., the map $A_x\to R'/\m_{R'}^i$ factors through $A_{G,x}$ for any $i$), so $A_{x}\to R'$ factors through $A_{G,x}$. Let $S\subset R'$ to be the image of $A_{G,x}$. Then the residue field of $S$ is $\kappa$, and  the $R'$-point $(X_{R'},\iota)\in \RZ_\BX(h)^{m,n}(R')$ is actually defined over $S$. By Proposition~\ref{prop:extension}, this $S$-point lies in $\RZ^\Lambda_{G,b}(S)$, which in turn shows that $(X,\iota)\in\RZ_\BX(h)^{m,n}(R)$ also lies in $\RZ^\Lambda_{G,b}(R)$.
\end{proof}
\begin{lemsub}[\emph{Cf.} Condition (S1$'$) in {\cite[(2.2)]{Artin:VersalDefAlgStack}}]\label{lem:SchlenssingerRim}
Assume that $p>2$, and consider $B,R, R'\in\Nilp_W$ such that $B\thra R$ is a square-zero thickening with the kernel annihilated by the nilradical of $B$, and $R'\ra R$ is a $W$-algebra map that induces a surjective map $R'\ra R_{\red}$. Set $B':=B\times_R R'$. Let $\cF$ be one of $\RZ^\Lambda_{G,b}$, $\RZ^\Lambda_{G,b}(h)$ and $\RZ^\Lambda_{G,b}(h)^{m,n}$, where in the last case we assume $p^mB=0$ and $p^mR'=0$. Then the natural map
\begin{equation}\label{eqn:SchlenssingerRim}
\cF(B') \ra \cF(B)\times_{\cF(R)}\cF(R')
\end{equation}
is a bijection. 
\end{lemsub}
\begin{proof}
The map (\ref{eqn:SchlenssingerRim}) is injective, because the analogous maps for  $\RZ_\BX$, $\RZ_\BX(h)$, and $\RZ_\BX(h)^{m,n}$ are bijections.

To show the surjectivity of (\ref{eqn:SchlenssingerRim}) for $\RZ^\Lambda_{G,b}$, we may assume that both $B$ and $R'$ are finitely generated over $W$, in which case the claim follows from Lemma~\ref{lem:SchlessingerRimFType}.
To show the surjectivity of (\ref{eqn:SchlenssingerRim}) for  $\RZ^\Lambda_{G,b}(h)$ and $\RZ^\Lambda_{G,b}(h)^{m,n}$, we observe that a cofibre product of quasi-isogenies of height $h$ is again of height $h$, and that a cofibre product of isogenies is again an isogeny.
\end{proof}

\begin{lemsub}[``Obstruction theory'']\label{lem:ObsThy}
Let $\mathfrak{U}\subset \RZ_\BX(h)$ be a quasi-compact open formal subscheme, and choose an integer $n$ large enough so that  the natural map $\wh\Omega_{\mathfrak{U}/W}|_{\mathfrak{U}\cap\RZ_\BX(h)^{1,n}} \ra \Omega_{\mathfrak{U}\cap\RZ_\BX(h)^{1,n}/\kappa}$ is an isomorphism.\footnote{Such $n$ exists as $\mathfrak{U}$ is noetherian; indeed, as $\mathfrak{U}\times_{\Spf W}\Spec \kappa = \varinjlim_n(\mathfrak{U}\cap\RZ_\BX(h)^{1,n})$, we may choose $n$ so that $\mathfrak{U}\cap\RZ_\BX(h)^{1,n}$ contains the closed subscheme of $\mathfrak{U}\times_{\Spf W}\Spec \kappa$ cut out by the square of the maximal ideal of definition. Note that $\RZ_\BX(h)$ may not be quasi-compact.}
Let $B\thra R$ be any square-zero thickening such that its kernel $\bb$ is killed by the nilradical of $B$, and let $(X,\tilde\iota)\in \mathfrak{U}(R)\cap \RZ^\Lambda_{G,b}(h)^{m,n}(R)$. 

Then there exists a $B$-point $(X_B,\tilde\iota)\in \RZ^\Lambda_{G,b}(h)^{m,n}(B)$ lifting $(X,\tilde\iota)$ if and only if there exists a $B$-point $(X_B,\tilde\iota)\in \RZ_\BX(h)^{m,n}(B)$ lifting $(X,\tilde\iota)$.
\end{lemsub}
Note that $\RZ_\BX(h)^{m,n}$ has an obstruction theory that satisfies the conditions in \cite[(2.6), (4.1)]{Artin:VersalDefAlgStack} (given by the theory of cotangent complex, for example). The lemma shows that any obstruction theory for the $\mathfrak{U}\cap\RZ_\BX(h)^{m,n}$ (which exists) also provides an obstruction theory for $\mathfrak{U}\cap\RZ^\Lambda_{G,b}(h)^{m,n}$ satisfying  \cite[(2.6), (4.1)]{Artin:VersalDefAlgStack} if $n$ is large enough (depending on $\mathfrak{U}$).
\begin{proof}
It suffices to prove the ``if'' direction. 
Let $(X,\tilde\iota)\in\mathfrak{U}(R)$, and assume that there exists a $B$-point $( X_B,\tilde\iota)\in \RZ_\BX(h)^{m,n}(B)$ lifting $(X,\tilde\iota)$. Set $R_0:=R_{\red}$, and write $(X_{R_0},\tilde\iota)\in\mathfrak{U}(R_0)$ denote the pull-back.

Let $f_0:\Spec R_0\ra\mathfrak{U}$ denote the map induced by $(X_{R_0},\tilde\iota)$.
Then  the set of $B$-points of $\RZ_\BX(h)^{m,n}$ lifting $(X,\tilde\iota)$, which is non-empty by assumption, is a torsor under the $R_0$-module $f_0^*(\wh\Omega_{\mathfrak{U}/W}^*)\otimes_{R_0}\bb$ by the assumption on $n$ in the statement. (Recall that any lift of $(X,\tilde\iota)$ lies in $\mathfrak{U}$.)
Therefore, any $B$-lift $(X'_B,\tilde\iota')\in\RZ_\BX(B)$ of $(X,\tilde\iota)$ actually lie in $\RZ_\BX(h)^{m,n}(B)$, as the set of such $B$-lifts is also a torsor under the same $R_0$-module $f_0^*(\wh\Omega_{\mathfrak{U}/W}^*)\otimes_{R_0}\bb$. Now if we also have $(X,\tilde\iota)\in\RZ^\Lambda_{G,b}(R)$, then  Corollary~\ref{cor:deformation} produces a $B$-point
\[(X_B,\tilde\iota)\in \RZ^\Lambda_{G,b}(B)\cap\RZ_\BX(h)^{m,n}(B) = \RZ^\Lambda_{G,b}(h)^{m,n}(B)\]
lifting $(X,\tilde\iota)$, as desired.
\end{proof}

We are ready to prove the following:
\begin{thmsub}\label{thm:ArtinRepTh}
The functor $\RZ^\Lambda_{G,b}(h)^{m,n}$ can be represented by  a separated  scheme locally of finite type over $\Spec W/p^m$.
\end{thmsub}

\begin{proof}
Choose a quasi-compact open $\mathfrak{U}\subset\RZ_\BX(h)$ which contains $\RZ_\BX(h)^{m,n}$. (Recall that $\RZ_\BX(h)^{m,n}$ is quasi-compact.) Then we have verified the criterion in \cite[Corollary~5.4]{Artin:VersalDefAlgStack} to show that  $\mathfrak{U}\cap\RZ^\Lambda_{G,b}(h)^{m,n'}$ is a separated algebraic space locally of finite type over $W/p^m$ for any $n'\gg n$. Since we have
\[\RZ^\Lambda_{G,b}(h)^{m,n} = (\mathfrak{U}\cap\RZ^\Lambda_{G,b}(h)^{m,n'})\times_{\RZ_\BX(h)^{m,n'}}\RZ_\BX(h)^{m,n},\]
it follows that $\RZ^\Lambda_{G,b}(h)^{m,n}$ is also a separated algebraic space locally of finite type over $W/p^m$.

Since the natural inclusion $\RZ^\Lambda_{G,b}(h)^{m,n}\hra \RZ_\BX(h)^{m,n}$ is a monomorphism which is locally of finite type, it is separated and  locally quasi-finite. (By looking at \'etale charts, this claim reduces to the  case of schemes, which is standard.) By \cite[Th\'eor\`eme~(A.2)]{Laumon-MoretBailly}, $\RZ^\Lambda_{G,b}(h)^{m,n}$ is a scheme.
\end{proof}

\subsection{Closedness}
\begin{thmsub}\label{thm:q-cpct}
The natural monomorphism $\RZ^\Lambda_{G,b}(h)^{m,n}\hra \RZ_\BX(h)^{m,n}$ is a closed immersion of schemes for any $m$, $n$, and $h$. Also, the functor $\RZ^\Lambda_{G,b}$ is representable by a formal scheme locally formally of  finite type over $W$, and the natural inclusion $\RZ^\Lambda_{G,b}\hra \RZ_\BX$ is a  closed immersion of formal schemes.
\end{thmsub}

Note that $\RZ^\Lambda_{G,b}\to \RZ_\BX$ is a closed immersion if and only if $\RZ^\Lambda_{G,b}(h)^{m,n}\hra \RZ_\BX(h)^{m,n}$ is a closed immersion for any $m$ and $n$. Since a proper monomorphism is a closed immersion, we want to show that  $\RZ^\Lambda_{G,b}(h)^{m,n}\ra\RZ_\BX(h)^{m,n}$ is  proper. For this, we need to show that $\RZ^\Lambda_{G,b}(h)^{m,n}$ is quasi-compact  (\emph{cf.} Corollary~\ref{cor:q-compImg}), and verify the valuative criterion for properness (Lemma~\ref{lem:ValCrit}). 

\begin{lemsub}\label{lem:ValCrit}
Let $R$ be a $\kappa$-algebra which is a discrete valuation ring,  and $L:=\Frac(R)$. 
Let $(X,\iota)\in\RZ_\BX(h)^{m,n}(R)$ be such that  $(X_L, \iota_L)\in \RZ^\Lambda_{G,b}(h)^{m,n}(L)$. Then we have  $(X,\iota)\in\RZ^\Lambda_{G,b}(h)^{m,n}(R)$.
\end{lemsub}
\begin{proof}
It suffices to show that $(X,\iota)\in\RZ^\Lambda_{G,b}(R)$ in the setting of the statement.
As both $\RZ_\BX$ and $\RZ^\Lambda_{G,b}$ commute with filtered direct limits in $\Nilp_W$, we may assume that $L$ is a finitely generated field extension of $\kappa$, and $(X,\iota)\in\RZ_\BX(R)$ extends to $(X_{R_0},\iota_{R_0})\in\RZ_\BX(R_0)$ for some \emph{smooth} $\kappa$-subalgebra $R_0\subset R$ with $L=\Frac R_0$.  By smoothness, there exists a $p$-adic topologically smooth $W$-algebra $A_0$ with $A_0/p = R_0$. 
By localising $A_0$ at the ideal corresponding to the closed point of $\Spec R$ and $p$-adically completing it, we obtain a $p$-adic flat $W$-algebra $A$ with $A/p=R$. 

By the assumption on $(X, \iota)$, the map of isocrystals $s_{\alpha,\DD}:\triv\ra\DD(X)^\otimes[\ivtd p]$ comes from a unique map of integral  crystals $t_\alpha:\triv\ra \DD(X_L)^\otimes$ on the generic fibre. So we obtain the following horizontal section for any $\alpha$:
\[ (t_\alpha(\wh A_{(p)})) \subset \DD(X_L)(\wh A_{(p)})^\otimes \cap  \DD(X)(A)^\otimes[\ivtd p] = \DD(X)(A)^\otimes, \]
since we have $A = A[\ivtd p]\cap \wh A_{(p)} $. We rename this section as $t_\alpha(A)\in\DD(X)(A)^\otimes$.

Since $(X,\iota)$ is defined over $R_0$, it follows that  $t_\alpha(A)\in\DD(X)(A)^\otimes[\ivtd p]$ lies in the image of  $\DD(X_{R_0})(A_0)^\otimes[\ivtd p]$. Since we also have $A_0 = A_0[\ivtd p] \cap A$ (as $A_0/p = R_0 \hra A/p = R$ by assumption), it follows that $t_\alpha(A)$ is the image of  $t_\alpha(A_0)\in\DD(X_{R_0})(A_0)^\otimes$. Since $(t_\alpha(A_0))$ are horizontal (as they are after inverting $p$), it follows that $(t_\alpha)$ on $X_L$ extend to maps of crystals $t_\alpha:\triv \ra \DD(X_{R_0})^\otimes$ by Lemma~\ref{lem:CrysConn}.

Now consider the following finitely generated $A_0$-scheme
\[\cP_{A_0}:=\nf\Isom_{A_0}\big[(\DD(X_{R_0})(A_0), (t_\alpha(A_0))],[A_0\otimes_{\Zp}\Lambda^*,(1\otimes s_\alpha)]\big).\]
By construction, $\cP_{A_0}$ restricts to a trivial $G$-torsor over $A_0[\ivtd p]$ since the quasi-isogeny over $R_0$ gives a splitting.

Let us now  show that the pull-back $\cP_A$ of $\cP_{A_0}$ is a $G$-torsor over $A$. It suffices to show that its pull-back $\cP_{A'}$ is a $G$-torsor for some suitable faithfully flat $A$-algebra $A'$. To construct such $A'$, let $R'$ be a complete discretely valued  $R$-algebra whose residue field is an algebraic closure $\kappa'$ of the residue field of $R$, and we identify $R'=\kappa'[[u]]$. We set $A':=W(\kappa')[[u]]$ and choose a lift $A\ra A'$ of $R\ra R'$.

We now show that $\cP_{A'}$ is a (necessarily trivial) $G$-torsor. We already have that $\cP_{A'[1/p]}$ is a  trivial $G$-torsor. Since $(X_L,\iota_L)\in\RZ^\Lambda_{G,b}(L)$, it follows that $\cP_{W(\bar L')}$ is a $G$-torsor where $L':=\Frac(R')$, so $\cP_{\wh A'_{(p)}}$ is a $G$-torsor. (Note that $\wh A'_{(p)}$ is a $p$-adic discrete valuation ring with $\wh A'_{(p)}/p = L'$, so  we have a faithfully flat map $\wh A'_{(p)}  \ra W(\ol L')$.) Therefore, we have that $\cP_{U'}$ is a $G$-torsor, where $U'\subset \Spec A'$ is the complement of the closed point.

By \cite[Th\'eor\`eme~6.13]{Colliot-Thelene-Sansuc:QuadFiber}, $\cP_{U'}$ extends to some $G$-torsor $\cP'_{A'}$ over $A'$. But since $A'$ is strictly henselian, $\cP'_{A'}$ is a trivial $G$-torsor, which implies that $\cP_{U'}$ is a trivial $G$-torsor. Therefore there exists an isomorphism of vector bundles
\[\varsigma:\OO_{U'}\otimes_{A'}\DD(X_{A'})(A') \riso \OO_U\otimes_{\Zp}\Lambda^*\]
matching $t_\alpha(A')|_{U'}$ with $1\otimes s_\alpha$. Since $\varsigma$ is defined away from a codimension-$2$ subset in a normal scheme, $\varsigma$ extends to an $A'$-section of $\cP_{A'}$ by taking the global section. This shows an isomorphism $G_{A'}\cong \cP_{A'}$.\footnote{This argument is adapted from ``Step~5'' of the proof of \cite[Proposition~1.3.4]{Kisin:IntModelAbType}.}

Furthermore, since $\cP_{A_0}$ pulls back to a smooth scheme over $A$, it has to be smooth over some open formal subscheme $\Spec A_0'\subset \Spec A_0$ containing the closed point of $\Spec R$, with non-empty fibres at any point in $\Spec A_0'$; i.e., the restriction $\cP_{A_0'}$ is a $G$-torsor. 
By replacing $A_0$ with the $p$-adic completion of $A_0'$, we may assume that $\cP_{A_0}$ is a $G$-torsor.

By Lemma~\ref{lem:muFil}, the  Hodge filtration $\Fil^1_{X_{R_0}}\subset \DD(X_{R_0})(R_0)$ is a $\set\mu$-filtration. (Indeed, $\Fil^1_{X_L}$ is a $\set\mu$-filtration so $\Fil^1_{X_{R_0}}$ is a $\set\mu$-filtration over the closure of the generic point, which is $\Spec R_0$.) This shows that $(X_{R_0},\iota_{R_0})\in\RZ\com{s_\alpha}_{\BX,G} (R_0)$, so we have $(X,\iota)\in\RZ^\Lambda_{G,b}(R)$.
\end{proof}

It remains to show the quasi-compactness of $\RZ^\Lambda_{G,b}(h)^{m,n}$. We begin with the following proposition.
\begin{propsub}\label{prop:Alteration}
Let $R$ be a smooth domain over $\kappa$, and let  $(X,\iota)\in \RZ_\BX(R)$.
Assume that there exists a dense subset of closed points $\Sigma\subset \Spec R$, such that for any $x\in\Sigma$ we have $(X_x,\iota_x)\in\RZ^\Lambda_{G,b}(\kappa)$. Then there exists a dense open subscheme $\Spec R'\subset \Spec R$ such that $(X_{R'},\iota_{R'})\in\RZ^\Lambda_{G,b}(R')$.
\end{propsub}

\begin{proof}
Let us choose a formally smooth $p$-adic $W$-lift $A$ of $R$.
By the standard dictionary \cite[Corollary~2.2.3]{dejong:crysdieubyformalrigid}, the morphism $s_{\alpha,\DD}:\triv\ra\DD(X)^\otimes[\ivtd p]$, constructed in Definition~\ref{def:t}, corresponds to a horizontal section $t_\alpha(A)\in \DD(X)(A)^\otimes[\ivtd p]$. Let us first show that $t_\alpha(A)\in\DD(X)(A)^\otimes$; i.e., $t_\alpha(A)$ is the $A$-section of a (unique) map of  crystals $t_\alpha:\triv\ra\DD(X)^\otimes$. 

We endow the $p$-adic filtration with $A[\ivtd p]$ and  identify the associated graded algebra $\gr^\bullet A[\ivtd p] \cong R\llpar u\rrpar$ by sending $p$ to $u$. Then we define $\set{t_\alpha\com m\in\DD(X)(R)^\otimes}_{m\in\Z}$ so that $\sum_{m\in\Z}t_\alpha\com m u^m$ is the image of $t_\alpha(A)$ via the map
\[
\DD(X)(A)^\otimes[\ivtd p] \ra \gr^\bullet \left(\DD(X)(A)^\otimes[\ivtd p]\right) \cong \bigoplus_{m\in \Z,\ m\gg-\infty} u^m\DD(X)(R)^\otimes,
\]
where $\gr^\bullet$ is with respect to the $p$-adic filtration.

Note that $t_\alpha(A)\in\DD(X)(A)^\otimes$ if and only if $t_\alpha\com m=0$ for any $m<0$. On the other hand, if $m<0$ then $t_\alpha\com m$ vanishes at a dense set of points $\Sigma$, so $t_\alpha\com m = 0$; indeed, for any $x\in\Sigma$, any map $\tilde x:A\ra W$ lifting $R\thra R/\m_x\cong \kappa$ pulls back $t_\alpha(A)$ to a ($p$-integral) tensor in  $\DD(X_x)(W)^\otimes$ because $(X_x,\iota_x)\in\RZ\com{s_\alpha}_{\BX,G}(\kappa) = \RZ^\Lambda_{G,b}(\kappa)$.

We next consider the following $A$-scheme (as in Definition~\ref{def:RZG}(\ref{def:RZG:P})):
\[ \cP_A:= \nf\Isom_A \big([\DD(X)(A),(t_\alpha(A))],[A\otimes_{\Zp}\Lambda^*, (1\otimes s_\alpha) ]\big).\]
By construction, each fibre of $\cP_A$ at a point of $\Spec A$ is either a $G$-torsor or empty. Since the fibre $\cP_x$ at $x\in\Sigma$ is a $G$-torsor and  $\cP_{A[\frac{1}{p}]}$ is a trivial $G$-torsor, it follows that the fibre $\cP_\eta$ at the generic point $\eta$ of $\Spec R \subset \Spec A$ is a $G$-torsor (by semi-continuity of fibre dimensions, for example).

By generic flatness, we find a localisation $R'$ of $R$ such that $\cP_{R'}$ is a $G$-torsor, and we choose an $A$-algebra $A'$ which lifts $R'$. (If $R' = R[1/ f]$ then we let $A'$ to be the $p$-adic completion of $A[1/\tilde f]$ where $\tilde f$ is some lift of $f$.) We want to show that $\cP_{A'}$ is a $G$-torsor. Indeed, $\cP_{R'}$ and $\cP_{A'[\frac{1}{p}]}$ are $G$-torsors so it remains to show that $\cP_{A'}$ is flat over $A'$. By local flatness criterion \cite[Theorem~22.3]{matsumura:crt} it suffices to show that the following surjective map of $\OO_{\cP_{R'}}$-modules is an isomorphism for each $m$:
\begin{equation}\label{eqn:LocFlCrit}
\OO_{\cP_{R'}} \cong (p^mA'/p^{m+1}A')\otimes_{R'}\OO_{\cP_{R'}} \thra p^m\OO_{\cP_{A'}}/ p^{m+1}\OO_{\cP_{A'}}.
\end{equation}
Let $\Sigma':=\Sigma\cap\Spec R'$, which is a dense set of closed points. For any $x'\in\Sigma'$, we choose a lift $\tilde x':\Spec W\ra\Spec A'$. By the defining condition of $\Sigma'$ it follows that the map (\ref{eqn:LocFlCrit}) pulls back to an isomorphism
\[
\OO_{\cP_{x'}} \riso p^m\OO_{\cP_{\tilde x'}}/ p^{m+1}\OO_{\cP_{\tilde x'}}.
\]
Since $\bigcup_{x'\in\Spec R'}\cP_{x'}$ is dense in $\cP_{R'}$, it follows that the kernel of (\ref{eqn:LocFlCrit}) is supported in some proper closed subset of $\cP_{R'}$. On the other hand, $\OO_{\cP_{R'}}$ is a domain, which forces  (\ref{eqn:LocFlCrit})  to be an isomorphism.

To see that $\Fil^1_{X_{R'}}$ is a $\set\mu$-filtration with respect to $(t_\alpha(R'))$, it suffices to check this  at dense set of points (namely, $\Sigma'$). This shows $(X_{R'},\iota_{R'})\in\RZ\com{s_\alpha}_{\BX,G}(R')$. 
\end{proof}

Now, Theorem~\ref{thm:q-cpct} follows from Lemma~\ref{lem:ValCrit} and the corollary below:
\begin{corsub}\label{cor:q-compImg}
The scheme $\RZ^\Lambda_{G,b}(h)^{m,n}$ is quasi-compact for any $m,n,h$.
\end{corsub}

\begin{proof}
We may assume $m=1$ since quasi-compactness only depends on the underlying topological space.  
Let $Z_0\subset \RZ_\BX(h)^{1,n}_\red$ denote the reduced closed subscheme whose underlying topological space is the Zariski closure of the image of $|\RZ^\Lambda_{G,b}(h)^{1,n}|$. Since $\RZ_\BX(h)^{1,n}$ is quasi-compact (\emph{cf.} \cite[\S2.22]{RapoportZink:RZspace}),  $Z_0$ is necessarily quasi-compact. Then by Proposition~\ref{prop:Alteration} (applied to the natural inclusion of the smooth locus of $Z_0$ into $\RZ_\BX$), there exists a dense open subscheme $U_0\subset Z_0$ (so $U_0$ is necessarily quasi-compact) such that the open immersion $U_0 \hra Z_0$ factors through $\RZ^\Lambda_{G,b}(h)^{1,n}_{\red}$. We can check that the map thus obtained $U_0\to \RZ^\Lambda_{G,b}(h)^{1,n}_{\red}$ is an \'etale monomorphism, so it is an open immersion. 

To show quasi-compactness of $\RZ^\Lambda_{G,b}(h)^{1,n}$, it remains to show  quasi-compactness of the (reduced) complement $\RZ^\Lambda_{G,b}(h)^{1,n}_{\red}\setminus U_0$. (Indeed, if the complement is quasi-compact, then $|\RZ^\Lambda_{G,b}(h)^{1,n}|$ can be covered by $|U_0|$ and  finitely many affine open subsets  $|U_i|$'s of $|\RZ^\Lambda_{G,b}(h)^{1,n}|$ covering the complement of $|U_0|$.)
For this, we observe that the injective map of topological spaces
\begin{equation}\label{eqn:specialisation}
|\RZ^\Lambda_{G,b}(h)^{1,n}| \to |Z_0|
\end{equation}
is bijective, which follows from the valuative criterion (Lemma~\ref{lem:ValCrit}) and the existence of  $U_0$ as above. 
Now, let $Z_1\subset Z_0$ denote the reduced complement of $U_0$. By applying  Proposition~\ref{prop:Alteration} to the smooth locus of $Z_1$ and the set of closed points $\Sigma\subset |V_1|$, there exists a dense open subscheme $U_1\subset Z_1$ (necessarily quasi-compact) such that the open immersion $V_1\hra Z_1$ can be factored by an open immersion $U_1\to \RZ^\Lambda_{G,b}(h)^{1,n}_{\red}\setminus U_0$, being an \'etale monomorphism. 
We may repeat this process to obtain a strictly decreasing sequence of closed subsets $|Z_0|\supsetneq|Z_1|\supsetneq\cdots$ such that the preimage of $|Z_0|\setminus|Z_r|$ in $|\RZ^\Lambda_{G,b}(h)^{1,n}|$ is a quasi-compact open subset for any $r\geqslant1$. Now, note that $Z_r=\emptyset$ for some $r$ since  $Z_0$ is noetherian, which shows that $|\RZ^\Lambda_{G,b}(h)^{1,n}|$ is quasi-compact.
\end{proof}

\subsection{Independence of auxiliary choices and functoriality}\label{subsec:indep}
We now finish the proof of Theorem~\ref{thm:RZHType}. We have constructed a closed formal scheme $\RZ^\Lambda_{G,b}\subset\RZ_\BX$ (Theorem~\ref{thm:q-cpct}), which enjoys the following properties:
\begin{enumerate}
\item $\RZ^\Lambda_{G,b}\subset\RZ_\BX$ represents $\RZ\com{s_\alpha}_{\BX,G}$ as in the statement of Theorem~\ref{thm:RZHType}.
 The universal tensors $t^\univ_\alpha:\triv\ra\DD(X_{\RZ_\BX}|_{\RZ^\Lambda_{G,b}})^\otimes$ can be obtained by glueing the unique Tate tensors over some affine open covering of $\RZ^\Lambda_{G,b}$. This claim follows from Proposition~\ref{prop:descent}(\ref{prop:descent:devissage}).
\item $\RZ^\Lambda_{G,b}\subset\RZ_\BX$ does not depend on the choice of $(s_\alpha)\in\Lambda^\otimes$; indeed, the subset $\RZ^\Lambda_{G,b}(\kappa)\subset\RZ_\BX(\kappa)$ and the completion at any $\kappa$-point do not depend on $(s_\alpha)$.
\item If $G=\GL(\Lambda)$ then $\RZ^\Lambda_{G,b} = \RZ_\BX$; \emph{cf.} Example~\ref{exa:RZ}.
\item For any closed reductive $\Zp$-subgroup $G'\subset G$ with $b\in G'(K_0)$, the closed formal subscheme $\RZ^\Lambda_{G',b}\subset\RZ_\BX$ is contained in $\RZ^\Lambda_{G,b}$. Indeed, this claim amounts to verifying analogous claims on the set of $\kappa$-points and the completions thereof; \emph{cf.} Lemma~\ref{lem:FunctAffDL}, Proposition~\ref{prop:FunctDefor}.
\end{enumerate}

It  remains to verify the functoriality assertions; namely, (\ref{thm:RZHType:Prod}) and (\ref{thm:RZHType:Funct})  in Theorem~\ref{thm:RZHType}. These assertions will immediately follow from Lemma~\ref{lem:ProdRZG} and Proposition~\ref{prop:FunctRZG}, hence we conclude the proof of Theorem~\ref{thm:RZHType}. 

Let  $(G',b')$ and $\Lambda'$ be another datum as in Definition~\ref{def:HodgeAssump}, and write $\BX':=\BX^{\Lambda'}_{b'}$. We have constructed a natural formal closed subscheme $\RZ^{\Lambda'}_{G',b'}\subset\RZ_{\BX'}$.

Recall that $\BX^{\Lambda\times\Lambda'}_{(b,b')} \cong \BX^\Lambda_b\times\BX^{\Lambda'}_{b'} = \BX\times\BX'$.
Then we  have a closed immersion $\RZ_\BX\times_{\Spf W}\RZ_{\BX'}\hra\RZ_{\BX\times\BX'}$, defined by the product of deformations up to quasi-isogeny, 
and a closed formal subscheme $\RZ^{\Lambda\times\Lambda'}_{G\times G',(b,b')}\subset \RZ_{\BX\times\BX'}$.

\begin{lemsub}\label{lem:ProdRZG}
We have $\RZ^{\Lambda\times\Lambda'}_{G\times G',(b,b')} = \RZ^\Lambda_{G,b}\times_{\Spf W}\RZ^{\Lambda'}_{G',b'}$ as closed formal subschemes of $\RZ_{\BX\times\BX'}$; in particular, Theorem~\ref{thm:RZHType}(\ref{thm:RZHType:Prod}) holds.
\end{lemsub}
\begin{proof}
As both are closed formal subschemes of $\RZ_{\BX\times\BX'}$, it suffices to show the equality of the set of $\kappa$-points and the completions thereof, which follows from Lemma~\ref{lem:FunctAffDL} and Proposition~\ref{prop:FunctDefor}.
\end{proof}

\begin{propsub}\label{prop:FunctRZG}
Given a map $f:G\ra G'$ which maps $b$ to $b'$, there exists a map $\RZ^\Lambda_{G,b}\ra\RZ^{\Lambda'}_{G',b'}$ which induces the desired maps on the set of $\kappa$-points and completions thereof as described in Theorem~\ref{thm:RZHType}(\ref{thm:RZHType:Funct}).
\end{propsub}
The proposition in the case when $f$ is an identity map asserts that the formal scheme $\RZ^\Lambda_{G,b}$ depends only on $(G,b)$, not on the auxiliary choice of $(\Lambda,(s_\alpha))$, up to canonical isomorphism.

\begin{proof}
We follow the structure of the proof of Proposition~\ref{prop:FunctDefor}.
The case when $f$ is a closed immersion and $\Lambda=\Lambda'$ was already handled at the beginning of \S\ref{subsec:indep}. For a natural projection $\pr_2:(G\times G',(b,b'))\ra (G',b')$, the natural projection
\[
\RZ^{\Lambda\times\Lambda'}_{G\times G',(b,b')} \cong \RZ^\Lambda_{G,b}\times_{\Spf W}\RZ^{\Lambda'}_{G',b'} \thra \RZ^{\Lambda'}_{G',b'}
\]
has the desired properties on $\kappa$-points and completions thereof. (The same holds for the first projection.)

Now, let $f:(G,b)\ra(G',b')$ be any morphism, and consider the graph morphism
\[
(1,f):(G,b) \ra (G\times G',(b,b')),
\]
which is a closed immersion on the reductive $\Zp$-groups. By letting $G$ act via $(1,f)$ on the faithful $G\times G'$-representation $\Lambda\times\Lambda'$, we obtain the closed subspace $\RZ^{\Lambda\times\Lambda'}_{G,b}\subset\RZ_{\BX\times\BX'}$. We claim that we have the following commutative diagram:
\begin{equation}\label{eqn:FunctRZG}
\xymatrix{
 & & \RZ^\Lambda_{G,b}\ar@{-->}[d]^{\exists !}\\
\RZ^{\Lambda\times\Lambda'}_{G,b}\ar[r] \ar@/^/[urr]^-{\cong }& \RZ^{\Lambda\times\Lambda'}_{G\times G',(b,b')} \cong \RZ^\Lambda_{G,b}\times_{\Spf W}\RZ^{\Lambda'}_{G',b'} \ar@{->>}[ur]^-{\pr_1} \ar@{->>}[r]_-{\pr_2}& \RZ^{\Lambda'}_{G',b'}
},
\end{equation}
where the solid arrows are already defined.  By looking at the sets of $\kappa$-points and the completions thereof (\emph{cf.} Proposition~\ref{prop:FunctDefor}; especially, the diagram in the proof),  it follows that $\pr_1$ restricts to an isomorphism  $\RZ^{\Lambda\times\Lambda'}_{G,b}\riso\RZ^\Lambda_{G,b}$ as claimed in the diagram. Therefore, the broken arrow is well-defined and satisfies the desired properties on $\kappa$-points and the completions thereof.
\end{proof}

\section{Extra structures on the moduli of $p$-divisible groups}\label{sec:ExtraStr}
We assume that $p>2$, and  set $\kappa=\Fpbar$, $W=\wh\Z_p^{\ur}$, and $K_0=\wh\Q_p^{\ur}$. We fix $(G,b)$ as in Definition~\ref{def:HodgeAssump} (with associated unramified Hodge-type local Shimura datum $(G,[b],\set{\mu\iv})$). With some suitable choice of $\Lambda$ (which gives rise to $\BX:=\BX^\Lambda_b$), we construct the closed  formal subscheme $\RZ^\Lambda_{G,b}\subset \RZ_\BX$. From now on, we write $\RZ_{G,b}:=\RZ^\Lambda_{G,b}$ as it does not depend on $\Lambda$ up to canonical isomorphism.  

In this section, we define a Weil descent datum, the action of $J_b(\Qp)$, ``\'etale realisations'' of crystalline Tate tensors,  the rigid analytic tower $\set{\RZ_{G,b}^\KK}$, and the Grothendieck-Messing period map -- in other words, we construct ``local Shimura varieties'' as conjectured in Rapoport and Viehmann \cite[\S5]{RapoportViehmann:LocShVar}.  Since $\RZ_{G,b}$ is locally formally of finite type over $\Spf W$ (\emph{cf.} Theorem~\ref{thm:RZHType}),  Berthelot's construction of rigid generic fibre $\RZ_{G,b}^\rig$ can be applied; \emph{cf.} \cite{berthelot:cohorigide}, \cite[\S7]{dejong:crysdieubyformalrigid}. We then construct the period morphism on the rigid generic fibre $\RZ_{G,b}^\rig$, which is an   \'etale morphism (highly transcendental in general).
When $\RZ_{G,b}$ is an EL or PEL Rapoport-Zink space, the extra structure on $\RZ_{G,b}$ that we define is compatible with the one defined by Rapoport and Zink in \cite{RapoportZink:RZspace}.  (We leave readers to verify this, which is more or less straightforward.)

The category of rigid analytic varieties can naturally be viewed as a full subcategory of the category of adic spaces (\emph{cf.} \cite[\S1.1.11]{Huber:EtCohBook}), so we may regard all the rigid analytic varieties as adic spaces\footnote{In light of  the recent work on ``infinite-level Rapoport-Zink spaces'' in \cite{ScholzeWeinstein:RZ} and \cite[\S6]{Scholze:PerfectoidSurvey}, the theory of adic spaces is the most natural framework to study the rigid analytic tower $\set{\RZ_{G,b}^\KK}$.}.

For the  EL and PEL case, Scholze and Weinstein \cite{ScholzeWeinstein:RZ} constructed an infinite-level Rapoport-Zink space.
We construct an ``infinite-level Rapoport-Zink space'' $\RZ_{G,b}^\infty$ associated to $(G,b)$ using the infinite-level Rapoport-Zink space for $\GL_{\Qp}(\Lambda[\ivtd p])$ and  the rigid analytic tower $\set{\RZ_{G,b}^\KK}$. This construction is rather \emph{ad hoc}, and there should be a more natural construction, as alluded in the introduction of \cite{ScholzeWeinstein:RZ}.

\subsection{More notation on adic spaces and  $p$-divisible groups}\label{subsec:SW}

We will work with the notion of adic spaces in the sense of Huber. (See \cite[\S2]{ScholzeWeinstein:RZ} for basic definitions.) Although it is possible to work with classical rigid analytic geometry for most part of this section (except \S\ref{subsec:InfLevel}),  the flexibility of the theory of adic spaces could be useful (for example, to define geometric points).

For any formal scheme $\XX$ locally formally of finite type over $\Spf W$, we let $\XX^\rig$ denote the rigid analytic generic fibre constructed by Berthelot (\emph{cf.} \cite[\S7.1]{dejong:crysdieubyformalrigid}), and we implicitly view it as an adic space.

Let us recall the functorial characterisation of $\XX^\rig$; \emph{cf.} \cite[Proposition~7.1.7]{dejong:crysdieubyformalrigid}. For any adic space $\Y$ (topologically) of finite type over $K_0$, we have
\begin{equation}\label{eqn:RigGenFib}
\Hom_{K_0}(\Y,\XX^\rig) \liso \varinjlim_{\YY;\ \YY^\rig=\Y} \Hom_W(\YY,\XX),
\end{equation}
where $\YY$ runs through formal models of $\Y$. This property uniquely determines $\XX^\rig$ by the rigid analytic Yoneda lemma \cite[Lemma~7.1.5]{dejong:crysdieubyformalrigid}, so the adic space associated to Berthelot's generic fibre $\XX^\rig$ coincides with Huber's construction of the generic fibre of the adic space associated to the formal scheme $\XX$; \emph{cf.} \cite[Proposition~2.2.2]{ScholzeWeinstein:RZ}.

\begin{rmksub}\label{rmk:GenFibPt}
Let $K$ be a complete  extension of $K_0$ (with rank~$1$ valuation), and let $\fo_K$ denote its valuation. (We will often use $C$ instead of $K$ for algebraically closed complete extension of $\Qp$.) 
Then one can check without difficulty that any point $x:\Spa(K,\fo_K) \ra \XX^\rig$ comes from a unique map $x: \Spf\fo_K\ra\XX$ of formal schemes, also denoted by $x$.
\end{rmksub}

Let $X$ be a $p$-divisible group over $\XX$.
Then, $X[p^n]^\rig$ is a finite \'etale covering of  $\X:=\XX^\rig$, and it is an abelian group object in the category of adic spaces.
\begin{defnsub}\label{def:TateModSheaf}
Let $T(X)$ denote the lisse $\Zp$-sheaf on $\X$ defined by the projective system $\set{X[p^n]^\rig}$, and define $V(X)$ to be the the lisse $\Qp$-sheaf associated to $T(X)$; i.e., $T(X)$ viewed in the isogeny category.  (See \cite[Definition~8.1]{Scholze:CdR} for the definition of lisse $\Zp$-sheaf on an adic space.) 
\end{defnsub}

As lisse $\Zp$- or $\Qp$- sheaves, it is possible to form tensor products, symmetric and alternating products, and duals (so $T(X)^\otimes$ and $V(X)^\otimes$ make sense).
The formation of $T(X)$ and $V(X)$ commutes with any base change $\YY\ra\XX$ for reasonable formal scheme $\YY$. In particular, for any geometric point $\bar x:\Spa(C,\fo_C)\ra \X$ the  fibre $T(X)_{\bar x}$, as a $\Zp$-module, only depends on the pull-back $X_{\bar x}$ of $X$ by $\bar x:\Spf\fo_C\ra\XX$. (Here, we use the convention as in Remark~\ref{rmk:GenFibPt}.)

\begin{rmksub}\label{rmk:RZad}
Let us make  explicit the adic space generic fibre $\RZ_\BX^\rig$ of $\RZ_\BX$.  For any analytic space (or adic space) $\X$ topologically of finite type over $K_0$, the set $\Hom_{K_0}(\X,\RZ_\BX^\rig)$ can be interpreted as the set of equivalence classes of $f\in \Hom_W(\XX,\RZ_\BX)$ for any formal model $\XX$ of $\X$, where for any morphism $\pi:\XX'\ra\XX$ of formal models of $\X$, $f\in\Hom_W(\XX,\RZ_\BX)$ is equivalent to $f\circ\pi\in\Hom_W(\XX',\RZ_\BX)$.

If we set $X_{\RZ_\BX}$ to be the universal $p$-divisible group over $\RZ_\BX$ and $f^\rig:\X \ra \RZ_\BX^\rig$ to  the map given by $f:\XX\to\RZ_\BX$, then we have  $T(f^*X_{\RZ_\BX}) \cong f^{\rig*}(T(X_{\RZ_\BX}))$; in particular, the $\Zp$-local system $T(f^*X_{\RZ_\BX})$ on $\X$ is independent of the choice of formal model.
A similar discussion holds for $\RZ_{G,b}^\rig$ in the place of $\RZ_\BX^\rig$.
\end{rmksub}

Let $\X$ be a connected component of $\RZ^\rig_{G,b}$.
For a geometric point $\bar x$ of $\X$ (i.e., $\bar x:\Spa(C,\fo_C)\ra \X$ for some algebraically closed complete extension $C$ of $K_0$),
 let $\pi_1^{\fet}(\X,\bar x)$ denote the algebraic fundamental group\footnote{In \cite[\S3]{Scholze:CdR} the algebraic fundamental group is called the ``pro-finite fundamental group''.}  of $\X$ with base point $\bar x$ in the terminology of  \cite{deJong:FundGp}. 

Quite formally, one  obtains a natural equivalence of categories from the category of lisse $\Z_\ell$-sheaves on $\X$ to the category of finitely generated $\Z_\ell$-modules with continuous  $\pi_1^{\fet}(\X,\bar x)$-action, where the equivalence is defined by $\sF\rightsquigarrow \sF_{\bar x}$. (\emph{Cf.} \cite[\S4]{deJong:FundGp}, \cite[Proposition~3.5]{Scholze:CdR}.) Similarly, one  obtains a natural equivalence of categories from the category of lisse $\Q_\ell$-sheaves on $\X$ (i.e., lisse $\Z_\ell$ sheaves viewed up to isogeny\footnote{Lisse $\Qell$-sheaves are more restrictive objects than ``local systems of $\Qell$-vector spaces'' as in \cite[Definition~4.1]{deJong:FundGp}, which involve ``analytic \'etale coverings'' of $\X$ (not just finite \'etale coverings). In particular, the geometric fibre $\sF_{\bar x}$ of a lisse $\Qell$-sheaf has an action of $\pi_1^{\fet}(\X,\bar x)$, not just (the adic space version of) the analytic fundamental group defined in \cite[\S2]{deJong:FundGp}.}) to the category of finite-dimensional $\Q_\ell$-vector spaces with continuous  $\pi_1^{\fet}(\X,\bar x)$-action. Here, $\ell$ can be any prime (including $\ell=p$).

Let $\triv$ denote either  the constant  rank-$1$  $\Zp$- or $\Qp$- sheaf.
\begin{defnsub}
An \emph{\'etale Tate tensor} on $X$ is a morphism $t_{\et}:\triv\ra V(X)^\otimes$ of lisse $\Qp$-sheaves on $\X$. An \'etale Tate tensor is called \emph{integral} if it restricts to a map $\triv\ra  T(X)^\otimes$ of lisse $\Zp$-sheaves on $\X$.
\end{defnsub}
It  follows that when $\X$ is connected, giving an \'etale Tate tensor $t_{\et}$ is equivalent to giving an $\pi_1^{\fet}(\X,\bar x)$-invariant element $t_{\et,\bar x}\in V(X)^\otimes_{\bar x}$, and $t_{\et}$ is integral if and only if $t_{\et,\bar x}\in T(X)^\otimes_{\bar x}$ for a single geometric point $\bar x$.

We now claim that crystalline Tate tensors have `` \'etale realisations''.
\begin{thmsub}\label{thm:EtTate}
Assume that that $\XX$ is formally smooth and locally formally of finite type over $\Spf W$, and let  $t:\triv\ra\DD(X)^\otimes$ be a morphism of crystals which is Frobenius-equivariant  up to isogeny and such that $t(R)\in\Fil^0\DD(X)(R)^\otimes$.
Then there exists a unique morphism $t_{\et}:\triv \ra T(X)^\otimes$ of lisse $\Zp$-sheaves on $\X$, such that at each geometric point $\bar x$ supported at a classical point $x$ with residue field $K$, the (classical) crystalline comparison isomorphism matches  $t_{\et,\bar x}\in T(X_{\bar x})^\otimes$ with $t_x:\triv\ra\DD(X_{x})^\otimes$ (where $t_x$ is the fibre of $t$ at $x$).
\end{thmsub}

If $\bar x$ is as in the theorem, then  $t_{\et,\bar x}\in T(X)_{\bar x}^\otimes$ is invariant under the  $\pi_1^{\fet}(\X',\bar x)$-action, not just  the $\Gal(\ol K_0/K)$-action, where $\X'\subset\X$ is the connected component containing $\bar x$. Indeed, one can see that the requirement for $t_{\et}$ in Theorem~\ref{thm:EtTate} uniquely determines $t_{\et}$.

The main idea of the proof of   Theorem~\ref{thm:EtTate} is to construct $t_{\et}$ using the (relative) crystalline comparison for $p$-divisible groups over $\XX$, and show the $p$-integrality using the theory of Kisin modules (over a $p$-adic discrete valuation ring).
Although the proof is quite ``standard'', 
it  takes a long digression to set up the notation. We give a proof in \S\ref{sec:Ccris}.

It should be possible to compare the fibres $t_{\et,\bar x}\in T(X)_{\bar x}^\otimes$ and $t_{\bar x}:\triv\ra\DD(X_{\bar x})^\otimes$ at any geometric point $\bar x$ of $\X$, using the theory of vector bundles over the Fargues-Fontaine curve. We will not work in this generality, as geometric points supported at classical points are sufficient to uniquely determine $t_{\et}$.

\subsection{The action of $J_b(\Qp)$}\label{subsec:ActionJ}
Recall from (\ref{eqn:J}) that $J_b(\Qp)$ is the group of quasi-isogenies $\gamma: \BX \dra  \BX$ which preserves the  tensors $(s_{\alpha,\DD})$. Then $\RZ_{G,b}$ has a natural left $J_b(\Qp)$-action defined as follows: for any $(X,\iota)\in\RZ_{G,b}(R)$ for $R\in\Nilp_W$ and $\gamma\in J_b(\Qp)$, we have $\gamma(X,\iota) = (X,\iota\circ \gamma\iv)\in\RZ_\BX(R)$. To see $\gamma(X,\iota)\in\RZ_{G,b}(R)$, it suffices to observe that for $R=\Fpbar$ we recovers the natural $J_b(\Qp)$-action on $X^G(b)\cong \RZ_{G,b}(\Fpbar)$ (\emph{cf.} Proposition~\ref{prop:LR}), and $\gamma$ induces $\gamma:(\RZ_\BX)\wh{_x}\riso(\RZ_\BX)\wh{_{\gamma x}}$ (as $\gamma$ does not modify the underlying $p$-divisible group.) By functoriality of adic space generic fibre, $J_b(\Qp)$ naturally acts on $\RZ_{G,b}^\rig$.

The $J_b(\Qp)$-action on $\RZ_{G,b}$ has a kind of ``continuity''  property in the sense of \cite[D\'efinition~2.3.10]{Fargues:AsterisqueLLC}; indeed, the proof of \cite[Proposition~2.3.11]{Fargues:AsterisqueLLC} works in the more general setting of ours.

\subsection{Weil descent datum}\label{subsec:WeilDesc}
Let $(G,[b],\set{\mu\iv})$ denote the unramified Hodge-type local Shimura datum associated to $(G,b)$.
The following definition is the local analogue of the reflex field for a Shimura datum. (\emph{Cf.} \cite[\S1.31]{RapoportZink:RZspace}.)
\begin{defnsub}
The \emph{(local) reflex field} or \emph{(local) Shimura field} for $(G,[b],\set{\mu\iv})$ is the subfield $E=E(\mu)\subset K_0$ which is the field of definition of the $G(K_0)$-conjugacy class of the cocharacter $\mu$. Note that $E$ is a finite unramified extension of $\Qp$ (as $\mu$ descends over some finite subextension of $\Qp$ in $K_0$).
\end{defnsub}

Put $d:=[E:\Qp]$, and let $q=p^d$ be the cardinality of the residue field of $E$.
Let $\tau = \sig^d\in\Gal(K_0/E)$ denote the $q$-Frobenius element (i.e., the lift of the $q$th power map on $\Fpbar$).

\begin{rmksub}\label{rmk:mu-int}
By the definition of $E$, we can deduce the following. For a cocharacter  $\mu\in\{\mu\}$ (defined over $W$), we have $\mu^\tau\in\set\mu$ where $\mu^{\tau}:=\tau^*\mu$.
\end{rmksub}

For any formal scheme $\XX$ over $\Spf W$, we write $\XX^\tau:=\XX\times_{\Spf W,\tau}\Spf W$. We similarly define $\X^\tau$ for an adic space over $(K_0,W)$.
\begin{defnsub}
Let $\XX$ be a formal scheme over $\Spf W$. A \emph{Weil descent datum} on $\XX$ over $\fo_E$ is an isomorphism over $\Spf W$:
\[
\Phi:\XX \riso \XX^\tau.
\]
Similarly, we  define a Weil descent datum over $E$ for a rigid analytic space $\X$ over $K_0$ as an isomorphism $\Phi:\X\riso\X^\tau$ over $K_0$.
\end{defnsub}

For any positive integer $r$, let $E_r\subset K_0$ denote the (unramified) subextension of degree $r$ over $E$. Then for any Weil descent datum $\Phi$ over $\fo_E$ for $\XX$, we can define a Weil descent datum over $\fo_{E_r}$ as follows:
\begin{equation}\label{eqn:WeilIteration}
\Phi^r:\XX\xrightarrow[\sim]{\Phi}\XX^\tau \xrightarrow[\sim]{\Phi^\tau}\cdots \xrightarrow[\sim]{\Phi^{\tau^{r-1}}}\XX^{\tau^r}. \end{equation}
The same construction works for rigid analytic spaces over $K_0$.

Let $\XX_0$ be a formal scheme over $\Spf \fo_E$, and $\XX:=\XX_0\times_{\Spf\fo_E}\Spf W$. Then there exists a natural Weil descent datum over $\fo_E$ on $\XX$. We say that a Weil descent datum $\Phi$ over $\fo_E$ is \emph{effective} if there exists a formal scheme $\XX_0$ over $\fo_E$ such that $\Phi$ is isomorphic to the one naturally associated to $\XX_0$. We similarly define effective Weil descent data over $E$ for adic spaces $\X$ over $K_0$.

Let $\XX$ be a formal scheme locally formally of finite type over $\Spf W$, equipped with a Weil descent datum $\Phi$ over $\fo_E$. Then on the adic space generic fibre $\X:=\XX^\rig$ we obtain a Weil descent datum  $\Phi^\rig:\X \riso \X^\tau$  induced by $\Phi$. If $\Phi$ is effective, then so is $\Phi^\rig$.

We  define a Weil descent datum on $\RZ_{G,b}$ over $\fo_E$  by restricting the natural Weil descent datum on $\RZ_\BX$, which we now recall. For $R\in\Nilp_W$ with the structure morphism $f:W\ra R$, we define $R^\tau\in \Nilp_W$ to be $R$ as a ring with structure morphism $f\circ\tau$. Then we have $\RZ_{G,b}^\tau(R) = \RZ_{G,b}(R^\tau)$. 
\begin{defnsub}\label{def:WeilDesc}
For any $(X,\iota)\in\RZ_\BX(R)$, we define $(X^\Phi,\iota^\Phi) \in \RZ_\BX(R^\tau)$, where $X^\Phi$ is $X$ viewed as a $p$-divisible group over $R^\tau$, and $\iota^\Phi$ is defined as follows:
\[
\iota^\Phi: \BX_{R^\tau/p} = (\tau^* \BX)_{R/p} \stackrel{\Frob^{-d} }{\dra}  \BX_{R/p} \stackrel{\iota}{\dra}X_{R/p} = X^\Phi_{R/p},
\]
where $\Frob^d:\BX \ra \tau^*\BX$ is the relative $q$-Frobenius (with $q=p^d$). This defines a Weil descent datum $\Phi:\RZ_\BX\riso\RZ_\BX^\tau$ over $\fo_E$; \emph{cf.} \cite[\S3.48]{RapoportZink:RZspace}.

Note that for $x=(X_x,\iota_x)\in\RZ_{G,b}(\Fpbar)$,  we have $x^\Phi:=(X_x^\Phi,\iota_x^\Phi)\in\RZ_{G,b}(\Fpbar^\tau)$; indeed, Definition~\ref{def:RZG}(\ref{def:RZG:Kottwitz}) is satisfied for $(X_x^\Phi,\iota_x^\Phi)$ by Remark~\ref{rmk:mu-int}.
Then it is clear from the construction that for $x\in\RZ_{G,b}(\Fpbar)$, the morphism $\Phi:(\RZ_\BX)\wh{_x}\riso(\RZ_\BX^\tau)\wh{_{x^\Phi}}$ induces $\Phi:(\RZ_{G,b})\wh{_x}\riso(\RZ_{G,b}^\tau)\wh{_{x^\Phi}}$.
Therefore we have $(X^\Phi,\iota^\Phi)\in\RZ_{G,b}(R^\tau)$ for any $R\in\Nilp_W$ by definition of $\RZ_{G,b}$ (\emph{cf.} Definitions~\ref{def:RZGloc}, \ref{def:RZGcolim}), so we get a Weil descent datum  $\Phi:\RZ_{G,b}\riso\RZ_{G,b}^\tau$  over $\fo_E$ defined by sending $(X,\iota)\in\RZ_{G,b}(R)$ to $(X^\Phi,\iota^\Phi) \in \RZ_{G,b}(R^\tau)$.
\end{defnsub}

Although the Weil descent datum for $\RZ_{G,b}$ is \emph{not} effective, it induces a natural action of the Weil group $W_E$ on the $\ell$-adic cohomology of  $\RZ_{G,b}^\rig$. Alternatively, one can ``complete'' the components of $\RZ_{G,b}$ so that the Weil descent datum would become effective (\emph{cf.} \cite[Theorem~3.49]{RapoportZink:RZspace}).

The Weil descent datum commutes with the natural action of $J_b(\Qp)$, as the relative $q$-Frobenius $\Frob^d:\BX\ra\tau^*\BX$ commutes with any quasi-isogenies. In particular, we have a $W_E\times J_b(\Qp)$-action on the $\ell$-adic cohomology of $\RZ_{G,b}^\rig$.

\subsection{\'Etale Tate tensors and rigid analytic tower}\label{subsec:Tower}
For any open compact subgroup $\KK\subset G(\Zp)$, we will construct a finite \'etale cover $\RZ_{G,b}^\KK$ of $\RZ_{G,b}^\rig$  that naturally fits into a  $G(\Qp)$-equivariant  tower $\set{\RZ_{G,b}^\KK}$ with Galois group $G(\Zp)$.

For any geometric point $\bar x$ of $\RZ_{G,b}^\rig$ we let $\pi_1^{\fet}(\RZ_{G,b}^\rig,\bar x)$ denote the algebraic fundamental group of the connected component of $\RZ_{G,b}^\rig$ containing $\bar x$.

Let $X_{G,b}$ denote the universal $p$-divisible group over $\RZ_{G,b}$. By Theorem~\ref{thm:EtTate}, we have a morphism of lisse $\Zp$-sheaves
\[t_{\alpha,\et}:\triv\ra T(X_{G,b})^\otimes\]
corresponding to each $t_\alpha$.

\begin{propsub}\label{prop:PEt}
Let $\bar x$ be a geometric point of $\RZ_{G,b}^\rig$ supported at a classical point $x$.
Then the following $\Zp$-scheme
\[\cP_{\et,\bar x}:=\nf\Isom_{\Zp} \big( [\Lambda, (s_\alpha)] ,[T(X_{G,b})_{\bar x}, (t_{\alpha,\et,\bar x})]  \big)\]
is a  trivial $G$-torsor. (Here, we view $\Lambda$ and $T(X_{G,b})_{\bar x}$ as vector bundles over $\Spec\Zp$.)
\end{propsub}
\begin{proof}
Note that any $G$-torsor over $\Zp$ is trivial; indeed, since $\Zp$ is a henselian local ring, a $G$-torsor over $\Zp$ is trivial if its special fibre is trivial. But  any $G$-torsor over a finite field is trivial if $G$ is  reductive (so its fibres are  connected by definition).

It remains to show that $\cP_{\et,\bar x}$ is a $G$-torsor. Let $K$ be the residue field at $x$, and let $X_{x_0}$ denote the pull-back of $X_{G,b}$ by $x_0:\Spec \kappa\ra \Spf \fo_K \xra x \RZ_{G,b}$. Then by \cite[Proposition~1.3.4]{Kisin:IntModelAbType}, we have a $W$-linear isomorphism
\[W\otimes_{\Zp}T(X_{G,b})_{\bar x}\cong \DD(X_{x_0})(W)^*\]
matching $(1\otimes t_{\alpha,\et,\bar x})$ and $(t_{\alpha,x_0}(W))$.  Therefore, $(\cP_{\et,\bar x})_W:=\cP_{\et,\bar x}\times_{\Spec \Zp}\Spec W $ is isomorphic to the $G$-torsor $\cP_{W}$ defined using $(\DD(X_{x_0})(W),(t_{\alpha,x_0}(W)))$.
\end{proof}

Let $\KK\com0:=G(\Zp)$ and $\KK\com i:=\ker(G(\Zp)\ra G(\Z/p^i))$ for any $i>0$. 
\begin{defnsub}
We set $\RZ_{G,b}^{\KK\com0}:=\RZ_{G,b}^\rig$.
For any $i> 0$ we define the following rigid analytic covering of $\RZ_{G,b}^\rig$:
\[\RZ_{G,b}^{\KK\com i}
= \nf\Isom_{\RZ_{G,b}^\rig}\big([\Lambda/p^i\Lambda,(s_\alpha)],[X_{G,b}[p^i]^\rig, (t_{\alpha, \et})], \big);\]
i.e., for an analytic space (or adic space) $\X$ over $K_0$, its $\X$-point $u$ classifies isomorphisms of $\Z/p^i$-local systems matching tensors after pulling back to $\X$. Here, we use the identification $X_{G,b}[p^i]^\rig\cong T(X_{G,b})/(p^i)$ to view the mod~$p^i$ reduction of $(t_{\alpha,\et})$ as tensors of $X_{G,b}[p^i]^\rig$. Since $\RZ_{G,b}^{\KK\com i} $ is an open and closed subspace of $\nf\Isom_{\RZ_{G,b}^\rig}(\Lambda/p^i\Lambda,X_{G,b}[p^i]^\rig)$, one can see that $\RZ_{G,b}^{\KK\com i} $ is a finite \'etale Galois cover of $\RZ_{G,b}^\rig$. 
When $G=\GL(\Lambda)$, this definition of level structure recovers the usual one.

We let the finite group  $G(\Z/p^i) = \KK\com0/\KK\com i$ act on the right on $\RZ_{G,b}^{\KK\com i}$ as follows: an element $g\in G(\Z/p^i)$ acts as $\varsigma^{(i)}\mapsto g\iv\circ \varsigma^{(i)}$ on sections $\varsigma^{(i)}$ of $\RZ_{G,b}^{\KK\com i}$. This makes $\RZ_{G,b}^{\KK\com i}$ an \'etale Galois cover of $\RZ_{G,b}^\rig$ with Galois group $G(\Z/p^i)$. When $G=\GL(\Lambda)$,  this action is compatible with the natural  action as defined in \cite[\S5.34]{RapoportZink:RZspace}.

For any open subgroup $\KK\subset \KK\com0$ which contains $\KK\com i$ for some $i\geqs0$, we set
\[\RZ_{G,b}^\KK = \RZ_{G,b}^{\KK\com i}/(\KK/\KK\com i).\]
This definition is independent of the choice of $i\gg0$. The  $J_b(\Qp)$-action and the Weil descent datum over $E$ on $\RZ_{G,b}^\rig$ pull back to  $\RZ_{G,b}^\KK$.
\end{defnsub}

Let us now define the ``right $G(\Qp)$-action'' of the rigid analytic tower $\set{\RZ_{G,b}^\KK}$ (i.e., Hecke correspondences). We follow \cite[\S5.34]{RapoportZink:RZspace} and \cite[\S2.3.9.3]{Fargues:AsterisqueLLC}. Let $g\in G(\Qp)$, and choose $\KK\subset G(\Zp)$ so that $g\iv \KK g \subset G(\Zp)$. For a fixed $g$, the assumption on $\KK$ can be arranged by replacing $\KK$ by some finite index open subgroup; indeed, for an open compact subgroup $\KK_0\subset G(\Zp)$, $\KK:= \KK_0\cap g \KK_0 g\iv$ satisfies this assumption. By a (right) $G(\Qp)$-action on the tower $\set{\RZ_{G,b}^\KK}$, we mean a collection of isomorphisms
\[ [g]: \RZ_{G,b}^\KK \riso \RZ_{G,b}^{g\iv \KK g}, \]
for any $g\in G(\Qp)$ and  $\KK\subset G(\Zp)$ with $g\iv \KK g \subset G(\Zp)$, which commutes with the map $\RZ_{G,b}^{\KK'}\thra \RZ_{G,b}^\KK$ for $\KK'\subset \KK$,  and we have $[g']\circ [g] = [gg']$ for any $g,g'\in G(\Qp)$ whenever it makes sense.

Let us first describe the map $[g]$ on $K$-points, where $K$ is a finite extension of $K_0$. Recall that $\RZ_{G,b}^\rig(K) = \Hom(\Spf\fo_K,\RZ_{G,b})$, so a point $u\in \RZ_{G,b}^\KK(K)$ can be interpreted as a $p$-divisible group $X_u:=u^* X_{G,b}$ over $\fo_K$, a quasi-isogeny $\iota: \BX\dra X_{u,\Fpbar}$, and a $\Gal(\Kbar/K)$-stable right coset $\tilde u \KK$ of isomorphisms
\[
\tilde u:   \Lambda\riso T(X_u)_{\bar\eta},
\]
where $\bar\eta:\Spa (\wh{\Kbar},\fo_{\wh{\Kbar}})\ra \Spa(K,\fo_K)$ is a geometric point. Since $\tilde u\KK$ is $\Gal(\Kbar/K)$-stable, the $\Gal(\Kbar/K)$-action on $\Lambda$ via $\tilde u$ has its image in $\KK$.

Since we assumed that $g\iv \KK g\subset G(\Zp)$, it follows that $g\Lambda\subset \Lambda[\ivtd p]$ is stable under the action of $\KK$, so $\tilde u(g\Lambda)\subset V(X_u)_{\bar\eta}$ is $\Gal(\Kbar/K)$-stable. This means that we can find a $p$-divisible group $X_{u\cdot g}$ over $\fo_K$ with quasi-isogeny $\jmath_g:X_u \dra X_{u\cdot g}$ such that  $g\Lambda$ is the image of the following map
\begin{equation}\label{eqn:LevelHeckeCorr}
T(X_{u\cdot g})_{\bar\eta}\hra V(X_{u\cdot g})_{\bar\eta} \xleftarrow[\jmath_g^*]{\sim}  V(X_u)_{\bar\eta} \xleftarrow[\tilde u]{\sim} \Lambda[\ivtd p].
\end{equation}
Indeed, for $n$ so that $p^n\Lambda\subset g\Lambda$,  $ g\Lambda/p^n\Lambda$
corresponds to the geometric generic fibre of some finite flat $\fo_K$-subgroup $\mathfrak{G}$ of $X_u[p^n]$. We set $X_{u\cdot g}:= X_u/\mathfrak{G}$ and
\begin{equation}
\jmath_g:X_u \stackrel{p^{-n}}{\dra}X_u \thra X_u/\mathfrak{G}=:X_{u\cdot g}.
\end{equation}
Then the pair $(X_{u\cdot g},\jmath_g)$ satisfies the desired property  (\ref{eqn:LevelHeckeCorr}). Now we obtain the following  $K$-valued point of $\RZ_\BX^\rig$:
\begin{equation}\label{eqn:BTHeckeCorr}
(X_{u\cdot g}, \jmath_{g,\Fpbar}\circ\iota)\in\Hom_W(\Spf\fo_K,\RZ_\BX)\cong \RZ_\BX^\rig(K).
\end{equation}

\begin{lemsub}\label{lem:HeckeCorrPt}
In the above setting,
let us write $(X,\iota):=(X_u,\iota)$ and $(X',\iota'):=(X_{u\cdot g},\jmath_{g,\Fpbar} \circ\iota)$; \emph{cf.}  (\ref{eqn:BTHeckeCorr}).
Then $(X', \iota')$ corresponds to  a $\Spf \fo_K$-point of $\RZ_{G,b}$.
\end{lemsub}
\begin{proof}
By construction, we have \'etale Tate tensors $(t'_{\alpha,\et})\subset T(X')^\otimes$ and an  isomorphism $\Lambda\riso T(X')_{\bar\eta}$ matching $(t'_{\alpha,\et})$ and $(s_\alpha)$. Let $S$ be the $p$-adically completed PD hull of some surjection $W[u]\thra \fo_K$. Then by Kisin theory, one associate $(t'_\alpha(S))\subset\DD(X')(S)^\otimes$ from $(t'_{\alpha,\et})$ such that its pointwise stabiliser is isomorphic to $G_S$; indeed,  $(t'_\alpha(S))$ can be constructed using Theorems~1.2.1 and 1.4.2 in \cite{Kisin:IntModelAbType}, and the assertion on the pointwise stabiliser follows from \cite[Proposition~1.3.4]{Kisin:IntModelAbType}.

By compatibility between Kisin modules and $D_{\dR}(V(X')) \cong D_{\dR}(V(X))$   \cite[Theorem~1.2.1(1)]{Kisin:IntModelAbType}, the isomorphism $\DD(X')(\fo_K)[\ivtd p]\riso \DD(X)(\fo_K)[\ivtd p]$ induced by $\jmath_g$ sends the tensor $(t'_\alpha(\fo_K))$ to $(t_\alpha(\fo_K))$, where $t'_\alpha(\fo_K)$ is the image of $t'_\alpha(S)$. This shows that the Hodge filtration $\Fil_{X'}\subset\DD(X')(\fo_K)$ is a $\{\mu\}$-filtration (via valuative criterion and Lemma~\ref{lem:muFil}). Therefore, by compatibility between Kisin modules and $D_{\cris}(V(X')) \cong D_{\cris}(V(X))$, it follows that  $x_0':=(X'_{\Fpbar},\iota'_{\Fpbar})\in\RZ_{G,b}(\Fpbar)$.
Finally, \cite[Proposition~1.5.8]{Kisin:IntModelAbType} shows that the map $\Spf \fo_K\ra(\RZ_\BX)\wh{_{x'_0}}$, defined by $(X',\iota')$, factors through $(\RZ_{G,b})\wh{_{x'_0}}$.
\end{proof}

Now we can lift  $(X_{u\cdot g}, \jmath_{g,\Fpbar}\circ\iota)\in\RZ_\BX^\rig(K)$ to $\RZ_{G,b}^{g\iv \KK g}(K)$ by adding the level structure corresponding to the right $g\iv \KK g$-coset of the isomorphism:
\begin{equation}\label{eqn:HeckeCorr}
\Lambda\xrightarrow[g]{\sim} g\Lambda\xrightarrow[\text{(\ref{eqn:LevelHeckeCorr})}]{\sim}  T(X_{u\cdot g})_{\bar\eta}  .
\end{equation}
By construction, the associated right $g\iv \KK g$-coset is $\Gal(\Kbar/K)$-stable, so we obtain a map $[g]:\RZ^\KK_{G,b}(K)\ra \RZ^{g\iv\KK g}_{G,b}(K)$.
If $g\in G(\Zp)$ then this action clearly recovers the natural ``Galois action'' of the covering.

The construction (\ref{eqn:BTHeckeCorr}) and (\ref{eqn:HeckeCorr}) can be generalised to $\X$-valued points in a functorial way for topologically finite-type $K_0$-analytic space (or adic space) $\X$. Then the $p$-divisible group $X_u$ is defined over some formal model $\XX$ of $\X$. By replacing $\XX$ with some admissible blow up if necessary, we can find a finite flat group scheme $\mathfrak{G}$ of $X_u$ whose rigid analytic generic fibre gives the local system corresponding to $g\Lambda/p^n\Lambda$; \emph{cf.} \cite{BoschLuetkebohmert:Rigid2}. Now by rigid analytic Yoneda lemma \cite[Lemma~7.1.5]{dejong:crysdieubyformalrigid}, we obtain a morphism $[g]:\RZ_{G,b}^\KK\ra\RZ_\BX^\rig$, which factors through $\RZ_{G,b}^\rig$ by considering the image of classical points (\emph{cf.} Lemma~\ref{lem:HeckeCorrPt}). And by considering the suitable generalisation of (\ref{eqn:HeckeCorr}), we  obtain a map $[g]:\RZ^\KK_{G,b}\ra \RZ^{g\iv \KK g}_{G,b}$.

Assume furthermore that $g^{\prime -1}\KK g' \subset G(\Zp)$ for some $g'\in G(\Qp)$. (This can be arranged by shrinking $\KK$ further if necessary.) Then  we can show that the map $[g']:\RZ_{G,b}^\KK \ra \RZ_{G,b}^{g^{\prime-1}\KK g'}$ is equal to the composition
\[\RZ_{G,b}^\KK \xra{[g]}\RZ_{G,b}^{g\iv \KK g} \xra{[g'g\iv]} \RZ_{G,b}^{g^{\prime-1}\KK g'}.\]
By setting $g'=\id\in G(\Qp)$, iwe see that $[g]:\RZ_{G,b}^\KK \ra \RZ_{G,b}^{g\iv \KK g}$ is an isomorphism.

Now the following proposition is immediate from the construction:
\begin{propsub}
The assignment $g\mapsto ([g]:\RZ_{G,b}^\KK \ra \RZ_{G,b}^{g\iv \KK g})$ defines a right $G(\Qp)$-action on the tower $\{\RZ_{G,b}^\KK\}$ extending the Galois action of $G(\Zp)$, which commutes with the natural $J_b(\Qp)$-action and the Weil descent datum over $E$.
\end{propsub}
This shows that on the ``$\ell$-adic'' cohomology of the tower $\set{\RZ_{G,b}^\KK}$, we have a natural action of $W_E\times J_b(\Qp) \times G(\Qp)$.

\subsection{Period morphisms}\label{subsec:Period}
Set $\E_{G,b}:=\DD(X_{G,b})_{\RZ_{G,b}}$, which is a vector bundle on $\RZ_{G,b}$ equipped with a filtration $\Fil^1_{X_{G,b}}$. From the universal Tate tensors $t_\alpha:\triv\ra \E_{G,b}^\otimes$, we get morphisms of rigid analytic $F$-isocrystals $t_\alpha^\rig:\triv\ra(\E_{G,b}^\rig)^\otimes$.  Note that the (universal) quasi-isogeny $\iota_{\red}:( \BX)_{(\RZ_{G,b})_{\red}} \dra (X_{G,b})_{(\RZ_{G,b})_{\red}}$ induces an isomorphism of vector bundles on $\RZ_{G,b}^\rig$
\[\E_{G,b}^\rig \riso \cO_{\RZ_{G,b}^\rig} \otimes_{\Zp} \Lambda \]
which matches $(t_\alpha^\rig)$ with the maps $(1\mapsto 1\otimes s_\alpha)$; indeed, the rigid analytic $F$-isocrystal $\E_{G,b}^\rig$ only depends on $\DD((X_{G,b})_{(\RZ_{G,b})_{\red}})[\ivtd p]$, as explained in \cite[\S5.3]{dejong:crysdieubyformalrigid}.

Let $\Flag_{G,\set\mu}$ denote the projective rigid analytic variety over  $K_0$ obtained from the analytification of $\Flag_{G_{K_0},\set\mu}^{K_0\otimes\Lambda^*, (1\otimes s_\alpha)}$ (\emph{cf.} \S\ref{subsec:muFil}). It follows that the Hodge filtration $(\Fil^1_{X_{G,b}})^\rig\subset\E_{G,b}^\rig$ defines a natural map
\begin{equation}
\pi: \RZ_{G,b}^\rig \ra \Flag_{G,\set{\mu}},
\end{equation}
which we call the \emph{period map}. By letting $J_b(\Qp)$ act on $\Flag_{G,\set{\mu}}$ via embedding $J_b(\Qp)\subset G(K_0)$ the period map $\pi$ is $J_b(\Qp)$-equivariant. In order to have compatibility with Weil descent data, one has to modify the target of the period map as in the case of (P)EL Rapoport-Zink spaces. To explain, the map  $\aleph:\RZ_{G,b}\ra\Delta:=\Hom_\Z(X^*(G)^{\Gal(\Qpbar/\Qp)},\Z)$ in the (P)EL case \cite[\S3.52]{RapoportZink:RZspace} can be generalised to the unramified Hodge-type case by \cite[Lemma~2.2.9]{ChenKisinViehmann:AffDL};  indeed, \emph{loc.~cit.} gives a functorial map $\Hom_W(\Spf A,\RZ_{G,b})\ra\pi_1(G)$ for formally smooth formally finitely generated $W$-algebra $A$ (noting that for a maximal torus $T\subset G_W$, the natural projection $X_*(T)\thra \pi_1(G)$ defines a map $\pi_1(G) \ra\Delta$ via evaluation).\footnote{One can explicitly describe $\aleph$ on $\Fpbar$-point as follows: it is the map that sends $gG(W)\in X^G(b)$ to the homomorphism $[\chi\mapsto \ord_p\chi(g)]$ where $\chi:G_{\Qp}\ra \Gm$ is a homomorphism over $\Qp$ and $g$ is any representative of $gG(W)$.}
 We then define a Weil descent datum on $\Flag_{G,\set{\mu}}\times\Delta$ using the same formula as in \cite[\S5.43]{RapoportZink:RZspace}. One can show that $(\pi,\aleph)$ is compatible with the Weil descent datum, generalising the (P)EL case \cite[\S5.46]{RapoportZink:RZspace}.

\begin{propsub}\label{prop:PeriodEt}
The period map $\pi$ is \'etale in the sense of \cite[\S5.9]{RapoportZink:RZspace}.
\end{propsub}

\begin{proof}[Proof of Proposition~\ref{prop:PeriodEt}]
The proof is almost identical to the proof of \cite[Proposition~5.15]{RapoportZink:RZspace}, if we use Corollary~\ref{cor:deformation} in  place of the Grothendieck-Messing deformation theory.

As in the case of schemes of finite type over a field, \'etaleness can be checked via infinitesimal lifting property for nilpotent thickenings supported at classical points by \cite[Proposition~5.10]{RapoportZink:RZspace}. To unwind this criterion, let $A'\thra A$ be a square-zero thickenings of local $K_0$-algebras which are finite dimensional over $K_0$. Let $A^\circ\subset A$ and $A^{\prime\circ}\subset A'$ respectively denote the subrings of power-bounded elements \cite[Definition~2.1.1]{ScholzeWeinstein:RZ}. Then we claim that the dotted arrow in the  commutative diagram below can be uniquely filled:
\[\xymatrix{
\Spa(A,A^\circ) \ar[r] \ar@{^{(}->}[d] &
\RZ_{G,b}^\rig \ar[d]_\pi \\
\Spa (A', A^{\prime\circ}) \ar[r] \ar@{.>}[ur]^{\exists !} &
\Flag_{G,\set\mu}
}\]
Let us translate this diagram in more concrete terms.
Let $\Fil^1_{A'}\subset A'\otimes_{\Zp}\Lambda$ be a $\set\mu$-filtration such that for a finite flat $W$-subalgebra $R_0\subset A$ there exists a map $f:\Spf R_0\ra\RZ_{G,b}$ such that the isomorphism
\begin{equation}\label{eqn:PeriodQIsog}
A\otimes_{R_0} \DD(X)(R_0) \xrightarrow[\DD(\iota)_A]{\sim} A\otimes_{\Zp}\Lambda,
\end{equation}
takes the Hodge filtration $A\otimes_{R_0}\Fil^1_{X}$ to $A\otimes_{A'}\Fil^1_{A'}$, where $(X,\iota)$ is the pull-back of the universal object $(X_{G,b},\iota)$ by $f$, and $\DD(\iota)_A$ is the isomorphism induced by $\iota$. The existence of the dotted arrow means the existence of a finite flat $W$-subalgebra $R'\subset A'$ and a map $f':\Spf R' \ra\RZ_{G,b}$ lifting $f$ in some suitable sense, such that  the Hodge filtration $A'\otimes_{R'}\Fil^1_{X'}$ corresponds to $\Fil^1_{A'}$ by the isomorphism $\DD(\iota')_{A'}$, where $(X',\iota')$ is the pull-back of $(X_{G,b},\iota)$ by $f'$. (Note that the uniqueness of the dotted arrow follows from the Grothendieck-Messing deformation theory.)

We choose a finite flat $W$-subalgebra $R'\subset A'$, and let $R\subset A$ denote the image of $R'$ in  $A$. Assume that  $R$ contains $R_0$.
Note that the pull-back  of the universal quasi-isogeny induces an isomorphism
\begin{equation}\label{eqn:PeriodQIsog2}
A'\otimes_{R'} \DD(X_R)(R') \xrightarrow[\DD(\iota)_{A'}]{\sim} A'\otimes_{\Zp}\Lambda,
\end{equation}
where we give the square-zero PD structure on $R'\thra R$.

By increasing $R'$ if necessary, we may assume that the intersection
\[\Fil^1_{R'}:=\Fil^1_{A'}\cap \DD(X_R)(R')\]
is a  $\set\mu$-filtration with respect to $(t_\alpha(R'))$, where $(t_\alpha)$ is the pull-back of the universal Tate tensors over $\RZ_{G,b}$. To see this, note that $A^{\prime\circ}$ is the preimage of the valuation ring of the residue field of $A'$. Then by valuative criterion for properness applied to the projective $R'$-scheme  $\Flag_{G,\set\mu}^{\DD(X_R)(R'), (t_\alpha(R'))}$, the $A'$-point corresponding to $\Fil^1_{A'}$ uniquely extends to an $A^{\prime\circ}$-point, which has to be defined over some finite $R'$-subalgebra $R''\subset A^{\prime\circ}$ (as $A^{\prime\circ}$ is the union of such $R''$'s). We rename $R''$ to be $R'$.
Now, the existence of $(X',\iota')$ lifting $(X_R,\iota)$ follows from Corollary~\ref{cor:deformation}.
\end{proof}
\begin{rmksub}
One defines  \'etale maps for adic spaces to be maps locally of finite presentation satisfying the usual infinitesimal lifting property for formal \'etale-ness using any affinoid $(K_0,W)$-algebras as test objects; \emph{cf.} \cite[Definition~1.6.5]{Huber:EtCohBook}. By  \cite[Example~1.6.6(ii)]{Huber:EtCohBook} and \cite[Proposition~5.10]{RapoportZink:RZspace}, this definition coincides with the definition of \'etale morphisms given in Proposition~\ref{prop:PeriodEt}. 
\end{rmksub}

\subsection{Infinite-level Rapoport-Zink spaces}\label{subsec:InfLevel}
In this section, all the rigid analytic spaces are regarded as adic spaces in the sense of \cite[Definition~2.1.5]{ScholzeWeinstein:RZ}. 

Since we will not directly work with the definitions of (pre)perfectoid spaces, we refer  to \cite[\S2.1]{ScholzeWeinstein:RZ} for basic definitions. Roughly speaking, a \emph{preperfectoid space} over $\Spa(K_0,W)$ is an adic space over $\Spa(K_0,W)$ which becomes a perfectoid space after base change over any perfectoid extension $(K,\fo_K)$ of $(K_0,W)$ and take the ``$p$-adic completion''; \emph{cf.} \cite[Definition~2.3.9]{ScholzeWeinstein:RZ}. In particular, preperfectoid spaces may be non-reduced as explained in \cite[Remark~2.3.5]{ScholzeWeinstein:RZ}.

Scholze and Weinstein  \cite[Theorem~D]{ScholzeWeinstein:RZ} constructed a preperfectoid space $\RZ_\BX^\infty$ over $\RZ_\BX^\rig$, which can be viewed as the ``infinite-level'' Rapoport-Zink space. (In \cite{ScholzeWeinstein:RZ} $\RZ_\BX^\infty$ is denoted as $\M_\infty$.) By definition, $\RZ^\infty_\BX$  parametrises $\Zp$-equivariant morphism over $\RZ_\BX^\rig$
\[ \Lambda \ra (\varprojlim X_{\RZ_\BX}[p^n])^{\ad}_{(K_0,W)} \]
which induces an isomorphism $\Lambda\riso (\varprojlim X_{\RZ_\BX}[p^n])^{\ad}_{(K_0,W)} (K,K^+)$ of $\Zp$-modules on the fibres at each point $\Spa(K,K^+)\ra\RZ_\BX^\rig$.
Here,  $(\varprojlim X_{\RZ_\BX}[p^n])^{\ad}_{(K_0,W)}$ is the generalised adic space over $(K_0,W)$ associated to the formal scheme
$\varprojlim X_{\RZ_\BX}[p^n]$, constructed in 
 \cite[\S2.2]{ScholzeWeinstein:RZ}.

For any open compact subgroup $\KK'\subset \GL_{\Zp}(\Lambda)(\Zp)$, there exists a natural projection $\RZ_\BX^\infty\ra\RZ_\BX^{\KK'}$ respecting the tower. It may not be known whether $\RZ_\BX^\infty$ represents the projective limit of $\RZ_\BX^{\KK'}$ as sheaves (or even, whether one should expect this)\footnote{Indeed, taking projective limits of adic spaces is  problematic as explained in \cite[\S2.4]{ScholzeWeinstein:RZ}.}. In other words,  although any map $\Spa(A,A^+)\ra\RZ_\BX^\infty$ gives rise to an isomorphism $\Lambda\riso T(X_{\RZ_\BX})_{(A,A^+)}$ of lisse $\Zp$-sheaves on $\Spa(A,A^+)$, obtained from the natural maps $\RZ_\BX^\infty\thra\RZ_\BX^{\KK'}$,  it is not known whether the ``converse'' holds.

Instead of working with a problematic notion of projective limit, Scholze and Weinstein \cite[Theorem~6.3.4]{ScholzeWeinstein:RZ} showed that  a weaker notion of equivalence $\RZ_\BX^\infty \sim \varprojlim \RZ_\BX^{\KK'}$ holds, where $\sim$ is defined in \cite[Definition~2.4.1]{ScholzeWeinstein:RZ}. To simplify the description of the equivalence, note that any projection  $\RZ_\BX^\infty\ra\RZ_\BX^{\KK'}$ has the property that there exists an affinoid open cover $\set{\Spa(A_\xi,A_\xi^+)}$ of $\RZ_\BX^\infty$ whose image in $\RZ_\BX^{\KK'}$ is an affinoid open cover $\set{\Spa(A_{\KK',\xi},A_{\KK',\xi}^+)}$ for each $\KK'$. (This follows from the cartesian square in the  first paragraph of the proof of \cite[Theorem~6.3.4]{ScholzeWeinstein:RZ}.) By the equivalence $\RZ_\BX^\infty \sim \varprojlim \RZ_\BX^{\KK'}$ we mean that:
\begin{itemize}
\item The natural map on the topological space $|\RZ_\BX^\infty|\ra \varprojlim|\RZ_\BX^{\KK'}|$ is a homeomorphism.
\item The image of $\varinjlim A_{\KK',\xi}$ in $A_\xi$ is dense for each $\xi$.
\end{itemize}

Let $\KK\com i\subset G(\Zp)$ and $\KK^{\prime(i)}\subset \GL_{\Zp}(\Lambda)(\Zp) $ respectively denote the kernel of reduction modulo $p^i$. We define $\RZ_{G,b}^\infty$ to be the ``projective limit'' of $\RZ_\BX^{\infty}\times_{\RZ_\BX^{\KK^{\prime(i)}}}\RZ_{G,b}^{\KK\com i}$; more concretely, we let  $\RZ_{G,b}^\infty$ be the closed adic subspace of $\RZ_\infty$ cut out by the equations defining $\RZ_\BX^{\infty}\times_{\RZ_\BX^{\KK^{\prime(i)}}}\RZ_{G,b}^{\KK\com i}$ for all $i$.

Note that the natural projection $\RZ_\BX^\infty\thra \RZ_\BX^{\KK^{\prime(i)}}$ restricts to $\RZ_{G,b}^\infty\thra \RZ_{G,b}^{\KK\com i}$, which factors as
\[\RZ_{G,b}^\infty\hra \RZ_\BX^{\infty}\times_{\RZ_\BX^{\KK^{\prime(i)}}}\RZ_{G,b}^{\KK\com i}\thra \RZ_{G,b}^{\KK\com i}. \]
Therefore, a morphism $\Spa(A,A^+) \ra \RZ_{G,b}^\infty$ gives rise to an isomorphism $\Lambda\riso T(X_{G,b})_{(A,A^+)}$ of lisse $\Zp$-sheaves on $\Spa(A,A^+)$, which matches $(s_\alpha)$ and $(t_{\alpha,\et})$, where $X_{G,b}$ is the universal $p$-divisible group over $\RZ_{G,b}$. 
\begin{propsub}
The generalised adic space $\RZ_{G,b}^\infty$ is a preperfectoid, and we have $\RZ_{G,b}^\infty\sim \varprojlim_\KK\RZ_{G,b}^\KK$. 
\end{propsub}
\begin{proof}

Since any closed subspace of a preperfectoid space is a preperfectoid space (\cite[Proposition~2.3.11]{ScholzeWeinstein:RZ}) it remains to show  $\RZ_{G,b}^\infty\sim \varprojlim_\KK\RZ_{G,b}^\KK$.

On the underlying topological space we clearly have a natural homeomorphism.
\[
|\RZ_{G,b}^\infty| \riso \varprojlim_\KK |\RZ_{G,b}^\KK|.
\]
To verify the other condition, let $\set{\Spa(A_\xi,A_\xi^+)}$ be an affinoid open covering of $ \RZ_\BX^\infty$ whose image in $\RZ_\BX^{\KK^{\prime(i)}}$ is an affinoid open cover $\set{\Spa(A_{i,\xi},A^+_{i,\xi})}$. Let $\Spa(B_\xi,B_\xi^+)\subset \RZ_{G,b}^\infty$ be the pull-back of $\Spa(A_\xi,A_\xi^+)$, and we similarly define $\Spa(B_{i,\xi},B_{i,\xi}^+)\subset \RZ_{G,b}^{\KK\com i}$. Then we have the following commutative diagram
\[
\xymatrix{
\varinjlim A_{i,\xi} \ar[r] \ar@{->>}[d] &
A_\xi  \ar@{->>}[d] \\
\varinjlim B_{i,\xi} \ar[r] &
B_\xi
},
\]
where the right vertical arrow is a quotient map and the upper horizontal arrow has a dense image. It thus follows that the lower horizontal arrow also has a dense image. This shows $\RZ_{G,b}^\infty\sim \varprojlim_\KK\RZ_{G,b}^\KK$.
\end{proof}

\begin{rmksub}
As remarked in the introduction of  \cite{ScholzeWeinstein:RZ}, it should be possible to obtain an ``infinite-level Rapoport-Zink space'' $\RZ_{G,b}^\infty$ for $(G,b)$ (or at least, an adic space equivalent to $\RZ_{G,b}^\infty$) directly without going through  finite levels, and obtain an ``explicit description'' of $\RZ_{G,b}^\infty$ using the theory of vector bundles on Fargues-Fontaine curves (in the spirit of \cite[Theorem~D]{ScholzeWeinstein:RZ}).
Such a construction should work for more general class of ``local Shimura data'' $(G,[b],\set{\mu\iv})$. 
\end{rmksub}

\section{Digression on crystalline comparison for $p$-divisible groups}\label{sec:Ccris}
The goal of this section is to prove Theorem~\ref{thm:EtTate}, for which we need to recall the basic constructions and crystalline comparison theory for $p$-divisible groups.
We will use the notation and setting as in \S\ref{subsec:SW}. and we additionally assume that $\XX = \Spf R$ is a connected formal scheme which is formally smooth and formally of finite type over $W$, $\wh\Omega_{R/W}$ is free over $R$, and one can take an $R$-basis $d u_i$ such that $u_i\in R\starr$ for all $i$. The choice of $R$ is more general than \cite{Brinon:CrisDR} (where various natural properties of (relative) period rings are proved), but one can rather easily deduce the properties of crystalline period rings that are relevant for us.\footnote{The properties of $\Bcris(R)$ that will be used can be rather easily deduced by the same proof as in \cite{Brinon:CrisDR}. More subtle properties which require refined almost \'etaleness,  such as $R[\ivtd p]$-flatness and the $\pi_1$-invariance, will not be used in this paper, although they are obtained in \cite[\S5]{Kim:ClassifFSm} by slightly  extending refined almost \'etaleness and repeating the proof of \cite{Brinon:CrisDR}.}

In this section we allow $p=2$. Although all the results hold when $\kappa$ is a perfect field (instead of an algebraically closed field) with little modification in the proofs, we continue to assume that $\kappa$ is algebraically closed for the notational simplicity.
\subsection{Crystalline period rings}
Choose a separable closure $\BE$ of $\Frac(R)$, and define $\ol R$ to be the union of normal $R$-subalgebras $R'\subset \BE$ such that $R'[\ivtd p]$ is  finite \'etale over $R[\ivtd p]$.  Set $\wh{\ol R} :=\varprojlim_n\ol R/(p^n)$. (When $R$ is a finite extension of $W$, we have $\ol R = \fo_{\Kbar_0}$.) We let $\bar\eta$ denote the geometric generic point of any of $\Spec R[\ivtd p]$, $\Spec \ol R[\ivtd p]$, and $\Spec \wh{\ol R}[\ivtd p]$. 

Let us briefly discuss the relation between the \'etale fundamental group of $\Spec R[\ivtd p]$ and the algebraic fundamental group of $(\Spf R)^\rig$. For any finite \'etale $R[\ivtd p]$-algebra $A$, let $R_A$ be the normalisation of $R$ in $A$. Since $R$ is excellent\footnote{Note that $R$ is a quotient of some completion of a polynomial algebra over $W$ by \cite[Lemma~1.3.3]{dejong:crysdieubyformalrigid}, and such a ring is known to be excellent (\emph{cf.} \cite[Theorem~9]{Valabrega:FewThms}).}, $R_A$ is finite over $R$. By construction, $(\Spf R_A)^\rig$ is finite \'etale over $(\Spf R)^\rig$, so we obtain a functor $\Spec A\rightsquigarrow (\Spf  R_A)^\rig$ from finite \'etale covers of $\Spec R[\ivtd p]$ to finite \'etale covers of $(\Spf R)^\rig$. This induces a natural map of profinite groups
\begin{equation} \label{eqn:FundGp}
\pi_1^{\fet}((\Spf R)^\rig, \bar x) \ra \pi_1^{\et}(\Spec R[1/p], \bar x),
\end{equation}
for any geometric closed point  $\bar x$ of $\Spec R[1/p]$, which  can also be viewed as a ``geometric point'' of $(\Spf R)^\rig$.

\begin{rmksub}\label{rmk:FundGp}
In the case we care about (such as $T(X)_{\bar x}$ for a $p$-divisible group $X$ over $R$), the action of $\pi_1^{\fet}((\Spf R)^\rig,\bar x)$ factors through $\pi_1^{\et}(\Spec R[1/p], \bar x)$.
\end{rmksub}

We set\footnote{Perhaps, $(\wh{\ol R})^\flat$ would be a more precise notation, but the notation $\ol R^\flat$ would cause no  confusion.}
\[\ol R^\flat:=\invlim_{x\mapsto x^p} \ol R/(p),\]
which is a perfect $\fo_{\Kbar_0}^\flat$-algebra equipped with a natural action of $\pi_1^{\et}(\Spec R[\ivtd p],\bar\eta)$.

For any $(x_n)_{n\in \Z_{\geqs0}}\in \ol R^\flat$, define
\[x^{(n)}:=\lim_{m\to\infty}(\tilde x_{m+n})^{p^m}\]
for any lift $\tilde x_{m+n}\in \wh{\ol R}$ of $x_{m+n}\in \ol R/(p)$. Note that $x^{(n)}$ is well-defined and independent of  the choices involved.
Consider the following $W$-algebra map
\begin{equation}\label{eqn:theta}
\theta:W(\ol R^\flat)\ra \wh{\ol R}\ ,\quad \theta(a_0,a_1,\cdots) := \sum_{n=0}^\infty p^n a_n^{(n)}.
\end{equation}
over the classical map $W(\fo_{\Kbar_0}^\flat) \thra \fo_{\wh\Kbar_0}$. The kernel of $\theta$ (\ref{eqn:theta}) is a principal ideal generated by an explicit element $p-[p^\flat]$, where $p^\flat = (a_n)$ with $a_n =$``$p^{1/n}\bmod p$''. This claim can be obtained from Lemma~\ref{lem:AcrisPresentation} below, since $p-[p^\flat]$ also generates the kernel of $W(\fo_{\Kbar_0}^\flat) \thra \fo_{\wh\Kbar_0}$.

Let $\theta_R:R\otimes_W W(\ol R^\flat)\thra \wh{\ol R}$ be  the $R$-linear extension of $\theta$, and define $\Acris(R)$ to be the $p$-adic completion of the PD envelop of $R\otimes_W W(\ol R^\flat)$ with respect to $\ker(\theta_R)$. We let $\Fil^1\Acris(R)$ denote the kernel of $\Acris(R)\thra \wh{\ol R}$, which is an PD ideal. 
(Note that the notation is incompatible with $\Acris(R)$ for f-semiperfect ring $R$ introduced in \S\ref{subsec:LiftablePD}. In this section, $\Acris(R)$ as in \S\ref{subsec:LiftablePD} will not appear.) 

By choosing a lift of Frobenius $\sig:R\ra R$ (which exists by the formal smoothness of $R$), one defines a lift of Frobenius $\sig$ on $R\otimes_W W(\ol R^\flat)$, which extends to $\Acris(R)$. The universal continuous connection $d:R\ra \wh\Omega_{R/W}$ extends (by  usual divided power calculus) to a $p$-adically continuous connection  $\nabla:\Acris(R)\ra\Acris(R)\otimes_R \wh\Omega_{R/W}$. Finally $\Acris(R)$ has a natural $\pi_1^{\et}(\Spec R[\ivtd p],\bar\eta)$-action, which extends the natural action on $\ol R^\flat$ and fixes $R$.

\begin{lemsub}\label{lem:AcrisPresentation}
The natural map
\[ (R\wh\otimes_W W(\ol R^\flat))\wh\otimes_{W(\fo_{\Kbar_0}^\flat)}\Acris(W) \ra\Acris(R)\]
is an isomorphism, where $\wh\otimes$ denote the $p$-adically completed tensor product.
\end{lemsub}
\begin{proof}
We want to show that the above map is an isomorphism modulo~$p^m$ for each $m$. Since $\Acris(R)/p^m$ is the PD envelop of $(\theta_R\bmod{p^m})$ over $W/p^m$, it suffices to show that $R\otimes_W W_m(\ol R^\flat)$ is flat over $W_m(\fo_{\Kbar_0}^\flat)$ for each $m$ by \cite[Proposition~3.21]{Berthelot-Ogus}. By local flatness criterion, it suffices to show that $\ol R^ \flat$ is flat over $\fo_{\Kbar_0}^\flat$. (Note that $R$ is flat over $W$.) Since $\fo_{\Kbar_0}^\flat$ is a valuation ring (of rank~$1$), $\fo_{\Kbar_0}^\flat$-flatness is equivalent to torsion-freeness, but clearly $\ol R^\flat$ has no nonzero $\fo_{\Kbar_0}^\flat$-torsion.
\end{proof}
Lemma~\ref{lem:AcrisPresentation} allows us to deduce explicit descriptions of $\Acris(R)$ from $\Acris(W)$, which is well-known; \emph{cf.} \cite[\S5]{fontaine:Asterisque223ExpII}.

Since $\Acris(R)$ is an $\Acris(W)$-algebra,  the element $t\in\Acris(W)$, which is ``Fontaine's $p$-adic analogue of $2\pi i$'', can be viewed as an element of $\Acris(R)$. We define
\[\Bcris^+(R):=\Acris(R)[\ivtd p],\quad \Bcris(R):=\Bcris^+(R)[\ivtd t] = \Acris(R)[\ivtd t]. \]

The Frobenius endomorphism $\sig$ and the connection $\nabla$ extends to $\Bcris^+(R)$ and $\Bcris(R)$. For any $r\in \Z$, we define the filtration $\Fil^r\Bcris^+(R)$ (for $r\geqs0$) to be the ideal generated by the $r$th divided power ideal of $\Acris(R)$, and set
\[\Fil^r\Bcris(R):= \sum_{i\geqs-r} t^{-i}\Fil^{i+r}\Bcris^+(R).\]

\begin{lemsub}\label{lem:inv}
We have
\begin{align*}
\Zp &= \Acris(R)^{\sig=1;\nabla=0};\\
\Qp &= (\Fil^0\Bcris(R))^{\sig=1; \nabla=0}.
\end{align*}
\end{lemsub}
\begin{proof}[Idea of the proof]
One may repeat the proof of \cite[Corollaire~6.2.19]{Brinon:CrisDR}. Indeed, the main ingredient of the proof is an explicit description of $\Acris(R)$ in terms of ``$t$-adic expansions'' \cite[Proposition~6.2.13]{Brinon:CrisDR}, which can be deduced, via  Lemma~\ref{lem:AcrisPresentation}, from the classical result on $\Acris(W)$ in \cite[\S5.2.7]{fontaine:Asterisque223ExpII}.
\end{proof}

\subsection{Crystalline comparison  for $p$-divisible groups}
Now, let $X$ be a $p$-divisible group over $R$, and let $\bar\eta: R\ra \BE$ denote the geometric generic point, where $\BE$ is the separable closure of $\Frac R$ that contains $\ol R$. We can consider $T(X)$  as a lisse $\Zp$-sheaf on either $\Spec R[\ivtd p]$ or $(\Spf R)^\rig$. (This will not lead to any serious confusion as observed in Remark~\ref{rmk:FundGp}.)

Then we have
\begin{equation}\label{eqn:TateMod}
T(X)_{\bar\eta}\cong \Hom_{\wh{\ol R}}(\Qp/\Zp, X_{\wh{\ol R}}),
\end{equation}
which defines a natural map
\begin{equation}\label{eqn:IntCrysComp}
\rho_X:T(X)_{\bar\eta} \ra \Hom(\DD(X_{\wh{\ol R}})(\Acris(R)), \Acris(R))
\end{equation}
by sending $f\in T(X)_{\bar\eta}$ to the pull-back morphism $f^*:\DD(X_{\wh{\ol R}})\ra \triv=\DD(\Qp/\Zp)$ evaluated at $\Acris(R)$. First, note that $\DD(X)(\Acris(R))$ is naturally isomorphic to $\Acris(R)\otimes_R\DD(X)(R)$, and this identification respects the Frobenius endomorphism and the connections.
So for any $f\in T(X)_{\bar\eta}$, the morphism $\rho_X(f):\DD(X)(R)\ra \Acris(R)$ respects both the Frobenius action $F$ and the connections $\nabla$, and  $\rho_X(f)$ maps the Hodge filtration $\Fil^1_X\subset \DD(X)(R)$ into $\Fil^1\Acris(R)$.
Furthermore, $\rho_X$ is equivariant under the natural $\pi_1^{\et}(\Spec R[\ivtd p], \bar\eta)$-action.

To summarise, the following map can be obtained by $\Bcris(R)$-linearly extending  $\rho_X$ and dualising it:
\begin{equation}\label{eqn:RatCrysComp}
\Bcris(R)\otimes_R\DD(X)(R) \ra \Bcris(R)\otimes_{\Qp} V(X)_{\bar\eta}^*.
\end{equation}
Furthermore, this map respects the naturally defined Frobenius-actions,  connections,  filtrations, and  $\pi_1^{\et}(\Spec R[\ivtd p], \bar\eta)$-action. (Here, we declare that $V(X)_{\bar\eta}^*$ is horizontal, is fixed by the Frobenius action, and lies in the $0$th filtration, and $\DD(X)(R)$ carries the trivial  $\pi_1^{\et}(\Spec R[\ivtd p], \bar\eta)$-action.)

Let $x:R\ra\fo_K$ be a map where $\fo_K$ is a complete discrete valuation $W$-algebra with residue field $\kappa$, 
and choose a ``geometric point'' $\bar x: R\ra\fo_K \hra \fo_{\wh\Kbar}$. We can extend $\bar x$  to $\wh{\ol R}\ra \fo_{\wh\Kbar}$, also denoted by $\bar x$.

We can repeat the construction of (\ref{eqn:RatCrysComp}) for the $p$-divisible group $X_x$ over $\fo_K$, although $\fo_K$ is not necessarily formally smooth over $W$ (i.e., absolutely unramified). Recall that we have a natural isomorphism of isocrystals $\DD(X_x)[\ivtd p]\cong \DD((X_{x,\kappa})_{\fo_K/p})[\ivtd p]$ induced by
\begin{multline*}
\DD(X_x)\cong \DD(X_{x,\fo_K/p}) \xleftarrow{\Frob^r_{\fo_K/p}}\DD(\sig^{r*}(X_{x,\fo_K/p})) \\ \cong \DD((\sig^{r*}X_{x,\kappa})_{\fo_K/p})\xrightarrow{\Frob^r_{\kappa}}\DD((X_{x,\kappa})_{\fo_K/p}),
\end{multline*}
lifting the identity map on $\DD(X_{x,\kappa})[\ivtd p]$, where $r$ is chosen so that the maximal ideal of $\fo_K/p$ is killed by $p^r$th power, and $\Frob^r_{\fo_K/p}$ and $\Frob^r_{\kappa}$  respectively denote the $r$th iterated relative Frobenius morphisms for $X_{x,\fo_K/p}$ and $X_{x,\kappa}$. With this choice of $r$, we have $\sig^{r*}(X_{x,\fo_K/p}) \cong (\sig^{r*}X_{x,\kappa})_{\fo_K/p}$. 
The resulting isomorphism $\DD(X_x)[\ivtd p]\cong \DD((X_{x,\kappa})_{\fo_K/p})[\ivtd p]$ is independent of the choice of $r$.

From this we get a natural isomorphism:
\begin{equation}\label{eqn:Isotriv2}
\DD(X_x)(\Acris(W))[\ivtd p] \cong \Bcris^+(W)\otimes_{W}\DD(X_{x,\kappa})(W).
\end{equation}
We also have a natural $\Gal(\Kbar/K)$-isomorphism
\begin{equation}\label{eqn:TateMod2}
T(X_x)_{\bar x} \cong \Hom_{\wh\Kbar}(\Qp/\Zp,X_{\bar x}).
\end{equation}
Now, by repeating the construction of the map (\ref{eqn:RatCrysComp}) we obtain
\begin{equation}\label{eqn:RatCrysComp2}
\Bcris(W)\otimes_{W}\DD(X_{x,\kappa})(W) \ra \Bcris(W)\otimes_{\Qp} V(X_x)_{\bar x}^*,
\end{equation}
which respects the naturally defined Frobenius action,  connection,  filtration, and  $\Gal(\Kbar/K)$-action.
\begin{thmsub}\label{thm:Ccris}
The maps (\ref{eqn:RatCrysComp}) and (\ref{eqn:RatCrysComp2}) are isomorphisms.
\end{thmsub}
More general version of this theorem is proved in \cite[Theorem~5.3]{Kim:ClassifFSm}.
\begin{proof}[Idea of the proof]
By Theorem~7 in \cite[\S6]{Faltings:IntegralCrysCohoVeryRamBase}, it follows that (\ref{eqn:RatCrysComp2}) is an isomorphism.
To prove that  (\ref{eqn:RatCrysComp}) is an isomorphism, we repeat the proof of \cite[Theorem~7]{Faltings:IntegralCrysCohoVeryRamBase} to show that the following map
\[
\Acris(R)\otimes_R\DD(X)(R) \ra \Acris(R)\otimes_{\Qp} V(X)_{\bar\eta}^*,
\]
induced by $\rho_X$ (\ref{eqn:IntCrysComp}), is injective with cokernel killed by $t$. One first handles the case when $X=\mu_{p^\infty}$ either by considering the PD completion of $\Acris(R)$ as originally done by Faltings\footnote{See the footnote in the proof of \cite[Theorem~6.3]{Kim:ClassifFSm} for slightly more details.} or by some explicit computation with the Artin-Hasse exponential map as in \cite[\S4.2]{ScholzeWeinstein:RZ}. Now one deduces the general case from this by some  Cartier duality argument, as explained in \cite[\S6]{Faltings:IntegralCrysCohoVeryRamBase}.
\end{proof}

Let us now show that the (relative) crystalline comparison isomorphism (\ref{eqn:RatCrysComp}) interpolates the crystalline comparison isomorphisms at classical points (\ref{eqn:RatCrysComp2}). For $x$ as before, we set $\bar x: R\xra x \fo_K \hra \fo_{\wh\Kbar}$, and choose an extension $\bar x:\wh{\ol R}\ra\wh\Kbar$. (Indeed, we can lift the geometric point $R[\ivtd p]\ra \Kbar$ to $\ol R[\ivtd p] \ra \Kbar$, and $\ol R$ maps to $\fo_{\Kbar}$. We then take the $p$-adic completion.)

By (\ref{eqn:TateMod}) and (\ref{eqn:TateMod2}), we get an isomorphism
\begin{equation}\label{eqn:TateMod3}
 T(X)_{\bar\eta} \riso T(X)_{\bar x} \cong T(X_x)_{\bar x},
 \end{equation}
sending $\Qp/\Zp\ra X_{\wh{\ol R}}$ to its fibre at $\bar x:\wh {\ol R} \ra \fo_{\wh\Kbar}$.

Note  that $\bar x$ induces a map $\bar x^\flat: \ol R^\flat \ra \fo_{\wh\Kbar}^\flat$. Choose $x_0:R \ra W$ such that $x_0$ and $x$ induce the same $\kappa$-point of $R$ (which is possible as $R$ is formally smooth over $W$).
Then the map $x_0\otimes W(\bar x^\flat): R\otimes_W W(\ol R^\flat) \ra W( \fo_{\wh\Kbar}^\flat)$ extends to $\Acris(R)\ra\Acris(W)$, respecting all the extra structure possibly except $\sig$; indeed, $x_0:R\ra W$ may not respect $\sig$.

\begin{lemsub}\label{lem:FunctorialityCrysComp}
The following diagram commutes
\[\xymatrix{
\Bcris(R)\otimes_R \DD(X)(R) \ar[r]^-{\sim}_-{\text{\eqref{eqn:RatCrysComp}}} \ar[d] &
\Bcris(R)\otimes_{\Qp}V(X)_{\bar\eta}^* \ar[d]^-{\text{\eqref{eqn:TateMod3}}} \\
\Bcris(W)\otimes_{W} \DD(X_{x,\kappa})(W) \ar[r]^-{\sim}_-{\text{\eqref{eqn:RatCrysComp2}}} &
\Bcris(W)\otimes_{\Qp} V(X)_{\bar x}^*
},\]
where the left vertical arrow is induced from
\begin{multline*}
\Bcris^+(W)\otimes_{\Acris(R)}\DD(X)(\Acris(R)) \\ \cong \DD(X_x)(\Acris(W))[1/p] \cong \Bcris^+(W)\otimes_{W}\DD(X_{x,\kappa})(W).
\end{multline*}
Here, the second isomorphism is (\ref{eqn:Isotriv2}).
\end{lemsub}
\begin{proof}
Clear from the construction.
\end{proof}

We extend the isomorphism  (\ref{eqn:RatCrysComp}) to the following isomorphism:
\begin{equation}\label{eqn:TensorCrysComp}
\Bcris(R)\otimes_R\DD(X)(R)^\otimes \riso  \Bcris(R)\otimes_{\Qp} (V(X)_{\bar\eta}^*)^\otimes = \Bcris(R)\otimes_{\Qp}V(X)^\otimes_{\bar\eta},
\end{equation}
respecting all the extra structures. Now, given $t:\triv\ra\DD(X)^\otimes$ as in Theorem~\ref{thm:EtTate}, the element
\[t(\Acris(R)) = 1\otimes t(R) \in \Acris(R)\otimes_R\DD(X)(R)^\otimes \subset \Bcris(R)\otimes_R\DD(X)(R)^\otimes\]
is fixed by the Frobenius and $\pi_1^{\et}(\Spec R[\ivtd p],\bar\eta)$-action, is killed by the connection, and lies in the $0$th filtration. By the isomorphism (\ref{eqn:TensorCrysComp}) and Lemma~\ref{lem:inv}, the above element $1\otimes t(R)$  corresponds to an element $t_{\et,\bar\eta}\in V(X)^\otimes_{\bar\eta}$ fixed by the $\pi_1^{\et}(\Spec R[\ivtd p],\bar\eta)$-action. Therefore, by the usual dictionary there exists a unique map of lisse $\Qp$-sheaves
\begin{equation}\label{eqn:EtTate}
t_{\et}:\triv\ra V(X)^\otimes
\end{equation}
such that it induces  the map $1\mapsto t_{\et,\bar\eta}$ on the fibre at $\bar\eta$. By Lemma~\ref{lem:FunctorialityCrysComp}, $t_{\et,\bar x}$ interpolates the \'etale Tate tensors associated to the fibre of $t$ at classical points.

It remains  to show that  $t_{\et}$ is ``integral''; i.e., we have $t_{\et}:\triv \ra T(X)^\otimes$. For this, it suffices to show that $t_{\et,\bar x}\in T(X)_{\bar x}^\otimes$ for some geometric point $\bar x$ (since  $R$ is assumed to be a domain). By formal smoothness, we may choose $\bar x$ that lies over $x:R\ra W$. In the next section, we verify the integrality claim using the theory of Kisin modules.\footnote{The integral refinement of (\ref{eqn:TensorCrysComp}) a la Fontaine-Laffaille does not work in general unless $t$ factors through a factor of $\DD(X_x)^\otimes$ with the gradings concentrated in $[a, a+p-2]$.}

\subsection{Review of Kisin theory}\label{subsec:KisinThy}
For simplicity\footnote{The rest of the discussion can be modified for $\fo_K$ that are finitely ramified.}, we assume that $\fo_K = W$.
We follow the treatment of  \S1.2 and \S1.4 in \cite{Kisin:IntModelAbType}. Let $\Sig:=W[[u]] $ and define $\sig:\Sig\ra\Sig$ by extending the Witt vectors Frobenius by $\sig(u)= u^p$.

\begin{defnsub}
By \emph{Kisin module} we mean a finitely generated free $\Sig$-module $\gM$ equipped with a $\sig$-linear map $\vphi:\gM\ra\gM[\ivtd{p-u}]$ whose linearisation induces an isomorphism $1\otimes \vphi: \sig^*\gM[\ivtd{p-u}] \ra\gM[\ivtd{p-u}]$.

For $i\in\Z$ we  define $\Fil^i(\sig^*\gM[\ivtd p]):=(1\otimes \vphi)\iv((p-u)^i\gM[\ivtd p])$. There is a good notion of subquotients, direct sums, $\otimes$-products, and duals.
\end{defnsub}

We choose $p^\flat\in \fo_{\ol K_0}^\flat$ and define $\Sig\ra W(\fo_{\ol K_0}^\flat) (\subset  \Acris(W))$ by sending $u$ to $[p^\flat]$.
The following can be extracted from the main results of  \cite{kisin:fcrys}:
\begin{thmsub}\label{thm:KisinMod}
There exists a covariant rank-preserving fully faithful exact functor $\gM:L\mapsto \gM(L)$ from the category of $\Gal(\ol K_0/K_0)$-stable $\Zp$-lattices of some crystalline representations to the category of Kisin modules, respecting $\otimes$-products and duals. Furthermore, the functor $\gM$ satisfies the following additional properties:
\begin{enumerate}
\item\label{thm:KisinMod:Dcris}
We have natural $\Gal(\ol K_0/K_0)$-equivariant isomorphisms
\begin{align*}
L & \cong  \Fil^0\big(\Bcris(W)\otimes_{\sig,W}\gM(L)/u\gM(L)\big)^{\varphi=1}\\
&\cong\Fil^0\big(\Bcris(W)\otimes_{\sig,\Sig}\gM(L)\big)^{\varphi=1},\notag
\end{align*}
which identifies $\Dcris(L[\ivtd p])\cong \sig^*(\gM(L)/u\gM(L))[\ivtd p]$.
\item\label{thm:KisinMod:DdR}
We have a natural filtered isomorphism
\[
D_{\dR}(L[1/p]) \cong (\sig^*\gM(L)[\ivtd p])/(u-p),
\]
where  on the target we take the image filtration of $\Fil^\bullet(\sig^*\gM(L)[\ivtd p])$.
\item\label{thm:KisinMod:Vcris}
For two $\Zp$-lattice crystalline $\Gal(\ol K_0/K_0)$-representations $L$ and $L'$, let $\mathfrak{f}:\gM(L)\ra\gM(L')$ be an $\vphi$-equivariant map. Then there exists at most one $\Gal(\ol K_0/K_0)$-equivariant map $f:L\ra L'$ with $\gM(f) = \mathfrak{f}$, and such $f$ exists if and only if the map
\[
\Bcris(W)\otimes_{\sig,\Sig}\gM(L) \xra{1\otimes\mathfrak{f}} \Bcris(W)\otimes_{\sig,\Sig}\gM(L')
\]
is $\Gal(\ol K_0/K_0)$-equivariant, in which case $f[\ivtd p]$ is obtained from the $\vphi$-invariance of the $0$th filtration part of the isomorphism above.
\end{enumerate}
\end{thmsub}
\begin{proof}
The theorem can be read off from the statement and proof of \cite[Propositions~2.1.5,~2.1.12]{kisin:fcrys}.
\end{proof}

We continue to assume that $\fo_K = W$, so we have $\Dcris(L[\ivtd p]) = \DdR(L[\ivtd p])$ as $K$-modules.
We clearly have that  $\gM(\triv) = (\Sig,\sig)$ (where $\triv$ denotes $\Zp$ equipped with the trivial $\Gal(\ol K_0/K_0)$-action). If there is no risk of confusion, we let $\triv$ also denote the Kisin module $(\Sig,\sig)$.
\begin{corsub}\label{cor:KisinTateTensor}
For $\gM:=\gM(L)$, the isomorphisms in Theorem~\ref{thm:KisinMod}(\ref{thm:KisinMod:Dcris}) induce
\[ L[1/p]^{\Gal(\ol K_0/K_0)} \cong \sig^*\gM [1/p]^{\vphi=1} \cong \Fil^0\DdR(L[1/p])\cap\Dcris(L[1/p])^{\vphi=1},\]
which restrict to
\[L^{\Gal(\ol K_0/K_0)} \cong (\sig^*\gM)^{\vphi=1}\cong  \Fil^0\DdR(L[1/p])\cap(W\otimes_{\sig,\Sig}\gM)^{\vphi=1}.\]
\end{corsub}
\begin{proof}
The first isomorphism is obtained by  taking $\Gal(\ol K_0/K_0)$-invariance of the isomorphisms in Theorem~\ref{thm:KisinMod}(\ref{thm:KisinMod:Dcris}). Indeed, we have $(\sig^*\gM[\ivtd p])^{\vphi=1}\subset \Fil^0(\sig^*\gM[\ivtd p]) $ by definition of $\Fil^0(\sig^*\gM[\ivtd p])$.

Let us show the last integrality assertion. Note that  we have $1\otimes\vphi:(\sig^*\gM)^{\vphi=1} \riso \gM^{\vphi=1}$. 
Now Theorem~\ref{thm:KisinMod}(\ref{thm:KisinMod:Vcris}) gives the following isomorphism
\[
L^{\Gal(\ol K_0/K_0)} = \Hom_{\Gal(\ol K_0/K_0)}(\triv, L) \ra \Hom_{\Sig,\vphi}(\triv, \gM(L)) =\gM^{\vphi=1},
\]
sending $f:\triv\ra L$ to $\gM(f):\triv\ra\gM$, which is compatible with the isomorphism  $L[\ivtd p]^{\Gal(\ol K_0/K_0)} \riso \sig^*\gM [1/p]^{\vphi=1} \xrightarrow[1\otimes\vphi]{\sim}\gM[\ivtd p]^{\vphi=1}$ that was just proved.

It remains to show that the mod~$u$ reduction $(\sig^*\gM)^{\vphi=1}\ra \Fil^0\DdR(L[1/p])\cap(W\otimes_{\sig,\Sig}\gM)^{\vphi=1}$ is an isomorphism. Since the map becomes an isomorphism after inverting $p$, it suffices to show that given $m\in \Fil^0\DdR(L[1/p])\cap(W\otimes_{\sig,\Sig}\gM)^{\vphi=1}$ its lift $\tilde m\in(\sig^*\gM)^{\vphi=1}[\ivtd p]$ lies in $(\sig^*\gM)^{\vphi=1}$. 

Note that the natural inclusion $L^{\Gal(\ol K_0/K_0)} \ra L$ corresponds to the following injective map of Kisin modules
\[
\Sig\otimes_{\Zp} L^{\Gal(\ol K_0/K_0)} \cong \Sig\otimes_{\Zp} \gM^{\vphi=1}\xleftarrow[1\otimes\vphi]{\sim}  \Sig\otimes_{\Zp} (\sigma^*\gM)^{\vphi=1} \hra \sigma^*\gM,
\]
and its cokernel is $p$-torsion free. Therefore for $\tilde m \in(\sig^*\gM)^{\vphi=1}[\ivtd p]$, we have $\tilde m\in(\sig^*\gM)^{\vphi=1}$ if and only if its mod~$u$ reduction lies in $W\otimes_{\sig,\Sig}\gM$; i.e.,  any $m\in \Fil^0\DdR(L[1/p])\cap(W\otimes_{\sig,\Sig}\gM)^{\vphi=1}$ can be lifted to $\tilde m\in(\sig^*\gM)^{\vphi=1}$. 
\end{proof}
\begin{thmsub}\label{thm:BK}
For any $p$-divisible group $Y$ over $W$, the isomorphism $\Dcris(V(Y)^*) \cong \DD(Y)(W)[\ivtd p]$ restricts to an isomorphism of $F$-crystals
\[ W\otimes_{\sig,\Sig}\gM(T(Y)^*) \cong \DD(Y)(W).\]
If we invert $p$, then the Hodge filtration on the right hand side induces the the image filtration of $\Fil^\bullet\sig^*\gM(T(Y)^*)[\ivtd p]$ on the left hand side.
\end{thmsub}
\begin{proof}
If $p>2$ or $X^\vee$ is connected, then this is a result of Kisin (\emph{cf.} \cite[Theorem~1.4.2]{Kisin:IntModelAbType}. The remaining case when $p=2$ follows from  \cite[Proposition~4.2(1)]{Kim:ClassifFFGpSchOver2AdicDVR}.
\end{proof}

\begin{proof}[Proof of Theorem~\ref{thm:EtTate}]
Let us first show  Theorem~\ref{thm:EtTate} when $X$ is a $p$-divisible group over $R$, where $R$ is as in the beginning of \S\ref{sec:Ccris}. For $t:\triv\ra\DD(X)^\otimes$ as in Theorem~\ref{thm:EtTate}, let  $t_{\et}:\triv\ra V(X)^\otimes$ denote its \'etale realisation as constructed in (\ref{eqn:EtTate}). We choose   $x:R\ra  W$ and  a ``geometric point'' $\bar x$ supported at $x$.
Set $\gM_x:=\gM(T(X_x)_{\bar x}^*)$ and $\bM_x:=\DD(X_x)(W)$. By Theorem~\ref{thm:BK} we have a natural $F$-equivariant isomorphism
\begin{equation}\label{eqn:BKx}
W\otimes_{\sig,\Sig}\gM_x\cong \bM_x.
\end{equation}
Let $t_{\Sig,\bar x} \in(\gM_x^\otimes[\ivtd p])^{\vphi=1}$ be the tensor corresponding to $t_{\et,\bar x}\in (V(X)_{\bar x}^\otimes)^{\Gal(\ol K_0/K_0)}$ by Corollary~\ref{cor:KisinTateTensor}.

We want to show that $t_{\et}$ is integral, for which it suffices to show that $t_{\Sig,\bar x} \in(\gM_x^\otimes)^{\vphi=1}$ by Corollary~\ref{cor:KisinTateTensor}. Recall that $t_{\et,\bar x}$ is constructed so that it corresponds to the morphism $t_x:\triv\ra\DD(X_x)^\otimes$ by the crystalline comparison isomorphism. It now follows from Theorem~\ref{thm:KisinMod}(\ref{thm:KisinMod:Dcris}) and (\ref{eqn:BKx}) that the following natural isomorphism
\[ K_0 \otimes_{\sig,\Sig[1/p]}\gM_x^\otimes[\ivtd p] \cong \bM_x^\otimes[\ivtd p] \big(\cong \Dcris(V(X)_{\bar x}^*)^\otimes\big)\]
matches $1\otimes t_{\Sig,\bar x}$ with $t_x(W)$. But since $t_x(W)\in \Fil^0\bM_x^\otimes$ (not just in $\bM_x^\otimes$), we obtain $t_{\Sig,\bar x}\in (\gM_x^\otimes)^{\vphi=1}$ from Corollary~\ref{cor:KisinTateTensor} and (\ref{eqn:BKx}). This shows Theorem~\ref{thm:EtTate} when $\XX=\Spf R$.

To prove Theorem~\ref{thm:EtTate} in general, note that that $\XX$ admits a Zariski open covering $\set{\UU_\xi}$  where each $\UU_\xi = \Spf R_\xi$ satisfies the assumption as in the beginning of \S\ref{sec:Ccris}. We have just proved that there exists a morphism $t_{\et}|_{\UU^\rig_\xi}:\triv\ra T(X_{\UU_\xi})^\otimes$ for each $\xi$ that satisfies the condition in the theorem, and these morphisms should coincide at each overlap by uniqueness. So the locally defined tensors $\set{t_{\et}|_{\UU^\rig_\xi}}$ glue to give a tensor $t_{\et}$ on $\X$, which concludes the proof.
\end{proof}

\bibliography{bib}
\bibliographystyle{amsplain}

\end{document}

%% file: newmacros.tex
\usepackage{amssymb,latexsym}

\author[W.\thinspace{}Kim]{Wansu Kim}
\address{Wansu Kim\\%
Department of Mathematics\\%
Korea Institute of Advanced Study\\%
85 Hoegi-ro\\%
Seoul, 02455\\%
South Korea}
\email{wansukim@kias.re.kr}

\usepackage[active]{srcltx}
\usepackage{amsmath}
\usepackage{amsfonts}

\usepackage{amsthm}
\usepackage{amscd}
\usepackage{mathrsfs}

\usepackage{fancyhdr}
\usepackage{graphicx}
\usepackage{psfrag}
\usepackage{calc}
\usepackage{url}

\usepackage{pb-diagram,pb-xy}

\usepackage{times}
\usepackage{verbatim}
\usepackage[cmtip,arrow,matrix,curve,tips,frame]{xy}
\usepackage{hyperref}
\xyoption{matrix}

\usepackage{ifthen}
\usepackage{color}
\usepackage{amsxtra}
\usepackage{amstext}
\usepackage{amssymb}
\usepackage{stmaryrd}
\usepackage{mathrsfs}
\usepackage{epsfig}
\usepackage{pifont}
\usepackage{charter}
\usepackage{avant}
\usepackage{courier}
\usepackage[T1]{fontenc}
\usepackage{textcomp}
\usepackage{bbm}

\numberwithin{equation}{subsection}

\theoremstyle{plain}
\newtheorem{thm}[subsection]{Theorem}
\newtheorem*{thm*}{Theorem}

\newtheorem{thmsub}[equation]{Theorem}

\newtheorem*{exthm*}{Expected Theorem}

\newtheorem{lemsub}[equation]{Lemma}
\newtheorem*{lem*}{Lemma}

\newtheorem{prop}[subsection]{Proposition}

\newtheorem{propsub}[equation]{Proposition}
\newtheorem*{prop*}{Proposition}

\newtheorem{corsub}[equation]{Corollary}
\newtheorem*{cor*}{Corollary}

\newtheorem*{claim*}{Claim}

\newtheorem*{conj*}{Conjecture}

\theoremstyle{definition}
\newtheorem{defn}[subsection]{Definition}
\newtheorem*{defn*}{Definition}
\newtheorem{defnsub}[equation]{Definition}

\newtheorem*{ques*}{Question}

\newtheorem{exasub}[equation]{Example}
\newtheorem*{exa*}{Example}

\theoremstyle{remark}

\newtheorem{rmksub}[equation]{Remark}
\newtheorem*{rmk*}{Remark}

\numberwithin{figure}{subsection}
\numberwithin{table}{subsection}

\newcounter{listnum}

\newcommand{\BE}{\mathbb{E}}
\newcommand{\BX}{\mathbb{X}}

\DeclareMathOperator{\pr}{pr}

\DeclareMathOperator{\rank}{rank}
\DeclareMathOperator{\rk}{rk}

\DeclareMathOperator{\id}{id}

\newcommand{\rig}{\mathrm{rig}}

\DeclareFontEncoding{OT2}{}{} 

\newcommand{\nf}[1]{\underline{#1}}
\newcommand{\ol}[1]{\overline{#1}}
\newcommand{\wt}[1]{\widetilde{#1}}

\newcommand{\tim}{\!\cdot\!}
\newcommand{\geqs}{\geqslant}
\newcommand{\leqs}{\leqslant}

\newcommand{\et}{\text{\rm\'et}}

\newcommand{\fet}{\text{\rm f\'et}}

\newcommand{\red}{\mathrm{red}}

\newcommand{\wh}[1]{\widehat{#1}}

\newcommand{\X}{\mathcal{X}}
\newcommand{\Y}{\mathcal{Y}}

\newcommand{\XX}{\mathfrak{X}}
\newcommand{\YY}{\mathfrak{Y}}
\newcommand{\ZZ}{\mathfrak{Z}}
\newcommand{\UU}{\mathfrak{U}}

\newcommand{\E}{\mathscr{E}}
\newcommand{\sF}{\mathscr{F}}

\newcommand{\cF}{\mathcal{F}}

\newcommand{\cO}{\mathcal{O}}

\newcommand{\M}{\mathcal{M}}

\newcommand{\bM}{\boldsymbol{\mathrm{M}}}

\newcommand{\triv}{\mathbf{1}}

\newcommand{\llpar}{(\!(}
\newcommand{\rrpar}{)\!)}

\DeclareMathOperator{\Spec}{Spec}

\DeclareMathOperator{\Spf}{Spf}

\DeclareMathOperator{\Spa}{Spa}

\newcommand{\starr}{^\times}

\DeclareMathOperator{\Frac}{Frac}

\newcommand{\ra}{\rightarrow}
\newcommand{\xra}[1]{\xrightarrow{#1}}
\newcommand{\xla}[1]{\xleftarrow{#1}}

\newcommand{\hra}{\hookrightarrow}

\newcommand{\thra}{\twoheadrightarrow}

\newcommand{\la}{\leftarrow}

\newcommand{\riso}{\xrightarrow{\sim}}
\newcommand{\liso}{\stackrel{\sim}{\la}}

\newcommand{\dra}{\dashrightarrow}

\newcommand{\Flag}{\mathtt{Fl}}

\newcommand{\com}[1]{^{(#1)}}

\newcommand{\invlim}{\mathop{\varprojlim}\limits}

\newcommand{\Gm}{\mathbb{G}_m}

\newcommand{\set}[1]{\{#1\}}

\newcommand{\iv}{^{-1}}
\newcommand{\ivtd}[1]{\tfrac{1}{#1}}

\newcommand{\vphi}{\varphi}

\newcommand{\sig}{\sigma}
\newcommand{\Sig}{\mathfrak{S}}

\newcommand{\Q}{\mathbb{Q}}

\newcommand{\cC}{\mathcal{C}}
\newcommand{\cD}{\mathcal{D}}

\newcommand{\cP}{\mathcal{P}}

\newcommand{\KK}{{\mathsf K}}

\newcommand{\BC}{\mathbb{C}}

\newcommand{\Sets}{(\mbox{\rm\bf Sets})}

\newcommand{\Qbar}{\overline{\Q}}

\newcommand{\Z}{\mathbb{Z}}

\newcommand{\F}{\mathbb{F}}

\newcommand{\Fpbar}{\overline{\F}_p}

\newcommand{\Qp}{\Q_p}

\newcommand{\Qell}{\Q_{\ell}}

\newcommand{\Qpbar}{\Qbar_p}

\newcommand{\Kbar}{\overline{K}}

\newcommand{\Zp}{\Z_p}

\newcommand{\Fp}{\F_p}

\DeclareMathOperator{\Nilp}{Nilp}
\DeclareMathOperator{\Alg}{Alg}
\newcommand{\ft}{\mathrm{ft}}

\DeclareMathOperator{\Isom}{Isom}

\newcommand{\sm}{\mathrm{sm}}

\newcommand{\Def}{\mathfrak{Def}}
\newcommand{\RZ}{\mathtt{RZ}}

\newcommand{\OO}{\mathcal{O}}
\newcommand{\fo}{\mathscr{O}}
\newcommand{\ep}{\epsilon}

\DeclareMathOperator{\diag}{diag}
\DeclareMathOperator{\Frob}{Frob}

\DeclareMathOperator{\ord}{ord}

\newcommand{\GL}{\mathrm{GL}}
\DeclareMathOperator{\GSp}{GSp}

\newcommand{\ad}{\mathrm{ad}}

\DeclareMathOperator{\Gal}{Gal}

\DeclareMathOperator{\GSpin}{GSpin}

\newcommand{\fa}{\mathfrak{a}}
\newcommand{\bb}{\mathfrak{b}}

\newcommand{\m}{\mathfrak{m}}
\newcommand{\n}{\mathfrak{n}}

\newcommand{\gM}{\mathfrak{M}}

\newcommand{\GLn}{\GL_n}

\newcommand{\ur}{\mathrm{ur}}
\newcommand{\univ}{\mathrm{univ}}

\newcommand{\fl}{\mathrm{fl}}

\DeclareMathOperator{\Hom}{Hom}
\DeclareMathOperator{\QHom}{QHom}

\DeclareMathOperator{\End}{End}

\DeclareMathOperator{\Lie}{Lie}

\newcommand{\bft}{\mathbf{t}}

\newcommand{\Fil}{\mathtt{Fil}}
\newcommand{\gr}{\mathtt{gr}}

\newcommand{\Sh}{\mathtt{Sh}}
\newcommand{\UH}{\mathfrak{H}}
\newcommand{\sS}{\mathscr{S}}

\newcommand{\cris}{\mathrm{cris}}
\DeclareMathOperator{\CRIS}{CRIS}
\newcommand{\Acris}{A_{\cris}}

\newcommand{\Bcris}{B_{\cris}}
\newcommand{\Dcris}{D_{\cris}}
\newcommand{\DdR}{D_{\dR}}

\DeclareMathOperator{\dR}{dR}

\DeclareMathOperator{\tor}{tor}

\DeclareMathOperator{\Rep}{Rep}

\newcommand{\prep}{\Rep_{\Zp}}
\newcommand{\Qprep}{\Rep_{\Qp}}

\DeclareMathOperator{\Qisg}{Qisg}

\newcommand{\DD}{\mathbb{D}}

\newcommand{\art}[1]{\mathfrak{AR}_{#1}}